\documentclass[12pt, reqno]{article}

\title{Classification of scaling limits of uniform quadrangulations with a
  boundary} \author{Erich Baur\footnote{Universit{\'e} de Lyon et ENS de Lyon,
    erich.baur@ens-lyon.fr}\and Gr\'{e}gory
  Miermont\footnote{Université{\'e} de Lyon, ENS de Lyon et Insitut
    universitaire de France, 
    gregory.miermon@ens-lyon.fr}\and Gourab Ray\footnote{University of Cambridge, gourab1987@gmail.com}}
\date{}

\usepackage{bbm}
\usepackage{stmaryrd}
\usepackage{pdfpages}
\usepackage{amsmath,amssymb,amsfonts,amsthm,mathrsfs}
\usepackage{setspace} 
\usepackage{scrextend}
\usepackage{mdframed}

\usepackage{hyperref}
\hypersetup{
    colorlinks=true,
    linkcolor=blue,
    citecolor=red, 
    urlcolor=blue,
    pdfborder={0 0 0}
}            

\usepackage{graphicx} 
\usepackage{xcolor} 
\usepackage{bbm}

\newtheorem{thm}{Theorem}[section]
\newtheorem{lemma}[thm]{Lemma}
\newtheorem{corollary}[thm]{Corollary}
\newtheorem{prop}[thm]{Proposition}
\theoremstyle{definition}
\newtheorem{defn}[thm]{Definition}
\theoremstyle{definition}
\newtheorem{remark}[thm]{Remark}
\newtheorem{remex}[thm]{Remark/Exercise}
\numberwithin{equation}{section}

\renewcommand{\P}{\mathbb P}
\newcommand{\Z}{\mathbb Z}
\newcommand{\E}{\mathbb E}
\newcommand{\R}{\mathbb R}
\newcommand{\N}{\mathbb N}
\newcommand{\eps}{\varepsilon}

\newcommand{\1}{1 \mkern -6mu 1} 

\newcommand{\m}{\textup{$\mathfrak{m}$}}

\newcommand{\cR}{\mathcal{R}}
\newcommand{\cT}{\mathcal{T}}
\newcommand{\cQ}{\mathcal{Q}}

\newcommand{\Wha}{W^{\theta}}
\newcommand{\Xha}{X^{\theta}}
\newcommand{\Zha}{Z^{\theta}}
\newcommand{\Di}{D^{\textup{I}}}
\newcommand{\Wi}{W^{\textup{I}}}
\newcommand{\Zi}{Z^{\textup{I}}}

\newcommand{\Q}{\mathbb Q}

\newcommand{\bq}{{\bf q}}
\newcommand{\dgr}{d_{\textup{gr}}}
\newcommand{\dgh}{d_{\textup{GH}}}
\newcommand{\dha}{d_{\textup{H}}}
\newcommand{\dq}{d_{{\bf q}}}
\newcommand{\dmap}{d_{\textup{map}}}

\newcommand{\uC}{\underline{C}}

\def \e{{\rm e}}

\newcommand{\dis}{\textup{dis}}

\newcommand{\BHP}{\textup{\textsf{BHP}}}
\newcommand{\BP}{\textup{\textsf{BP}}}
\newcommand{\IBD}{\textup{\textsf{IBD}}}
\newcommand{\BM}{\textup{\textsf{BM}}}
\newcommand{\BD}{\textup{\textsf{BD}}}
\newcommand{\ICRT}{\textup{\textsf{ICRT}}}
\newcommand{\CRT}{\textup{\textsf{CRT}}}
\newcommand{\UIHPQ}{\textup{\textsf{UIHPQ}}}
\newcommand{\UIPQ}{\textup{\textsf{UIPQ}}}
\newcommand{\Y}{\textup{\textsf{Y}}}
\newcommand{\Zsf}{\textup{\textsf{Z}}}

\newcommand{\slimGH}{\textup{$\textsf{s-Lim}$}}
\newcommand{\slimLGH}{\textup{$\textsf{s-Lim}_{\textup{loc}}$}}

\newcommand{\f}{\textup{$\mathfrak{f}$}}

\newcommand{\la}{\textup{$\mathfrak{l}$}}
\newcommand{\an}{\textup{$\mathfrak{r}$}}
\newcommand{\q}{\textup{$\mathfrak{q}$}}

\newcommand{\Br}{\textup{$\mathfrak{B}$}}
\newcommand{\Fo}{\textup{$\mathfrak{F}$}}
\newcommand{\La}{\textup{$\mathfrak{L}$}}
\newcommand{\cb}{\textup{Ball}}
\newcommand{\suc}{\textup{succ}}
\newcommand{\arcr}{\curvearrowright}
\newcommand{\arcl}{\curvearrowleft}

\newcommand{\br}{\textup{\textsf{b}}}
\newcommand{\vd}{v^{\bullet}}
\newcommand{\uuX}{\underline{\underline{X}}}
\newcommand{\uuY}{\underline{\underline{Y}}^\sigma}

\renewcommand{\d}{\mathrm{d}}

\begin{document}
\maketitle

\begin{abstract} 
  We study non-compact scaling limits of uniform random planar
  quadrangulations with a boundary when their size tends to
  infinity. Depending on the asymptotic behavior of the boundary size and
  the choice of the scaling factor, we observe different limiting metric
  spaces. Among well-known objects like the Brownian plane or the infinite
  continuum random tree, we construct two new one-parameter families of
  metric spaces that appear as scaling limits: the Brownian half-plane with skewness parameter $\theta$
  and the infinite-volume Brownian disk of perimeter $\sigma$. We also
  obtain various coupling and limit results clarifying the relation between these
  objects.
\end{abstract} 
\footnote{{\it Acknowledgment of support.} The research of EB was
  supported by the Swiss National Science Foundation grant
  P300P2\_161011, and performed within the framework of the LABEX
  MILYON (ANR-10-LABX-0070) of Universit\'e de Lyon, within the
  program ``Investissements d'Avenir'' (ANR-11-IDEX-0007) operated by
  the French National Research Agency (ANR). GM acknowledges support
  of the grants ANR-14-CE25-0014 (GRAAL) and ANR-15-CE40-0013
  (Liouville), and of Fondation Simone et Cino Del Duca. }

\setcounter{tocdepth}{2}

\begin{figure}[ht]

\begin{minipage}{1\linewidth}
\centering
  \parbox{5.1cm}{\includegraphics[width=0.25\textwidth]{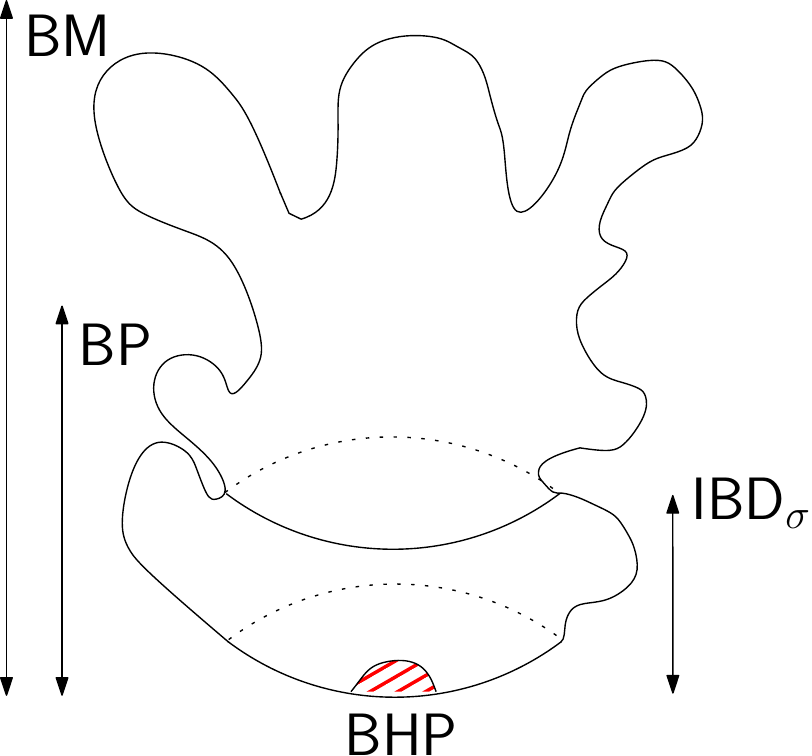}}
  \parbox{4.8cm}{\includegraphics[width=0.20\textwidth]{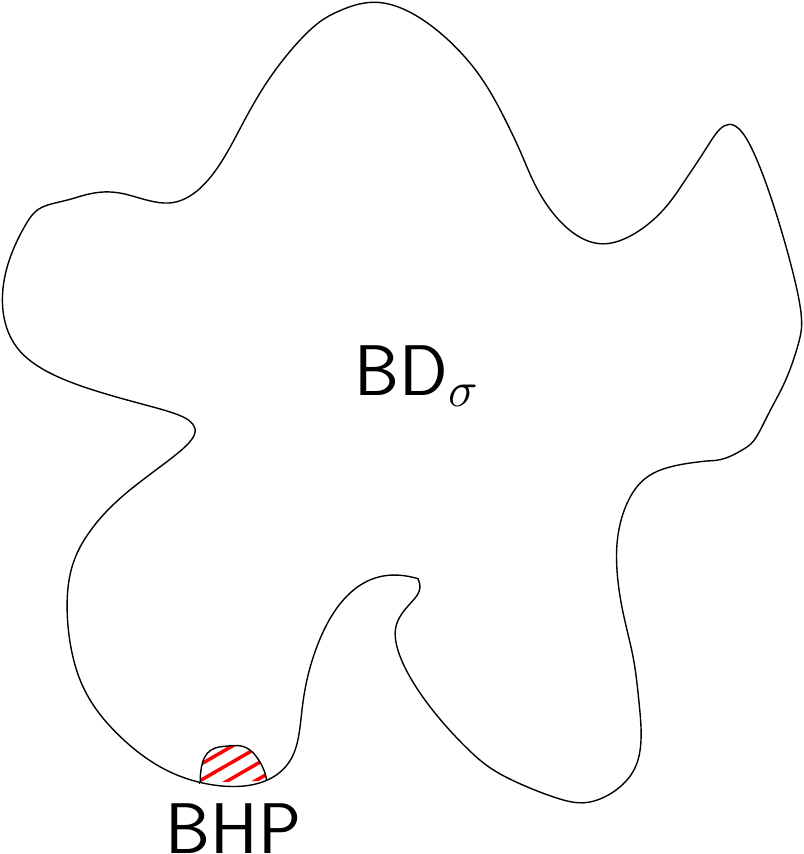}}
  \parbox{4.1cm}{\includegraphics[width=0.25\textwidth]{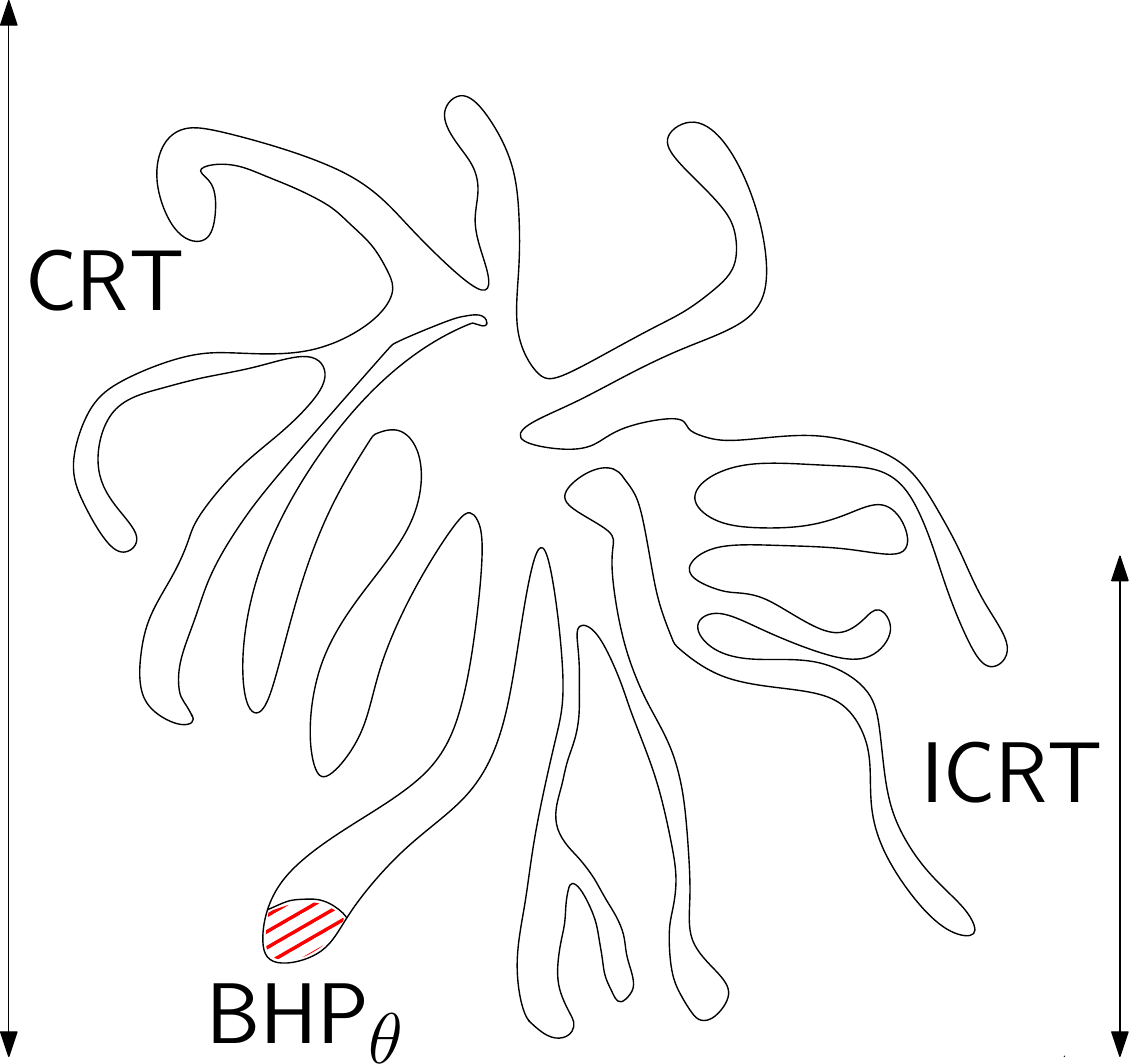}}
\end{minipage}

\caption{Schematic representation of the non-trivial scaling limits in the
  regimes $1\ll \sigma_n\ll\sqrt{n}$ (left), $\sigma_n\approx\sqrt{n}$
  (middle), $\sqrt{n}\ll \sigma_n\ll n$ (right). See the listings in
  Section~\ref{sec:overview}.}
  \label{fig:regimes-schema}
\end{figure}
\tableofcontents
\section{Introduction}
In this work, we obtain a complete classification of possible scaling limits
of finite random planar quadrangulations with a boundary when their size
tends to infinity.

Recall that a planar map is a proper embedding of a finite connected graph
in the two-dimensional sphere. The graph may have loops and multiple
edges. The faces of a map are the connected components of the complement of
its edges. A planar quadrangulation with a boundary is a particular planar
map where its faces have degree four, i.e., are incident to four oriented
edges (an edge is counted twice if it lies entirely in the face), except
possibly one face which may have an arbitrary (even) degree. This face is
referred to as the external face, where the other faces that form
quadrangles are called internal faces.  The boundary of the map is given by
the oriented edges that are incident to the external face, and the number
of such edges is called the size of the boundary, or the {\it perimeter} of
the map. The size of the map is given by the number of internal faces.  We
do not ask for the boundary to be a simple curve. We always
consider rooted maps with a boundary, which means that we distinguish one
oriented edge of the boundary such that the root face lies to the left of
that edge. This edge will be called the root edge, and its origin the root
vertex.  As usual, two (rooted) maps are considered equivalent if they
differ by an orientation- and root-preserving homeomorphism of the sphere.

We are interested in scaling limits of planar maps picked uniformly at
random among all quadrangulations with a boundary when the size and
(possibly) the perimeter of the map tend to infinity. This means that we
view the vertex set of the quadrangulation as a metric space for the graph
distance and consider (under a suitable rescaling of the distance)
distributional limits of such metric spaces, either in the global or local
Gromov-Hausdorff topology.

In~\cite{LG3} and independently in~\cite{Mi2}, it was shown that uniformly
chosen quadrangulations of size $n$, equipped with the graph distance
$\dgr$ rescaled by a factor $n^{-1/4}$, converge to a random compact metric
space called the Brownian map. The latter turns out to be a universal
object which appears as the distributional limit of many classes of random
maps. We refer to the recent overview~\cite{Mi3} for various aspect of the
Brownian map and for more references.

Here we shall deal with quadrangulations of size $n$ having a boundary of
size $2\sigma_n$, and we will distinguish three boundary regimes as $n$
tends to infinity:
$${\bf a})\,\,\sigma_n/\sqrt{n}\rightarrow 0;\quad{\bf b})\,\,\sigma_n/\sqrt{n}\rightarrow \sqrt{2}\sigma\;\hbox{ for some
}\sigma\in(0,\infty);\quad{\bf c})\,\, \sigma_n/\sqrt{n}\rightarrow
\infty.$$ Bettinelli~\cite{Be3} showed that in regime {\bf a}), the boundary
becomes negligible in the scale $n^{-1/4}$, and the Brownian map appears in
the limit when $n$ tends to infinity. In regime {\bf b}), he obtained under the
same rescaling convergence along appropriate infinite subsequences to a
random metric space called the Brownian disk $\BD_\sigma$. Uniqueness of
this limit was later established by Bettinelli and Miermont
in~\cite{BeMi}. For the third regime {\bf c}), it is shown in~\cite{Be3} that a
rescaling by $\sigma_n^{-1/2}$ leads in the limit to Aldous's continuum
random tree $\CRT$~\cite{Al1,Al2}.

The scaling factors considered by Bettinelli~\cite{Be3} ensure that the
diameter of the rescaled planar map stays bounded in
probability. Consequently, the limits he obtains are random
compact metric spaces, and the right notion of convergence is the
Gromov-Hausdorff convergence in the space of (isometry classes of) compact
metric spaces. 

We will study all possible scalings $a_n\rightarrow\infty$ in all the above
boundary regimes. When $a_n$ grows slower then the diameter of the map as
$n$ tends to infinity, the right notion of convergence is the {\it local}
Gromov-Hausdorff convergence. Depending on the ratio of
perimeter and scaling parameter, the boundary will in the limit be
either invisible, or of a size comparable to the full map, or dominate the
map.

In the process we obtain two new one-parameter families of limit 
spaces: the Brownian half-plane $\BHP_\theta$ with parameter
$\theta\in[0,\infty)$ and the infinite-volume Brownian disk $\IBD_\sigma$
with boundary length $\sigma\in(0,\infty)$.

The Brownian disk $\BD_\sigma$ and the Brownian half-plane $\BHP=\BHP_0$
play a central role in this work. The latter can be seen as the
Gromov-Hausdorff tangent cone in distribution of $\BD_\sigma$ at its root,
and also as the scaling limit of the so-called uniform infinite half-planar
quadrangulation $\UIHPQ$, which in turn arises as the local limit in the
sense of Benjamini and Schramm of uniform quadrangulations with $n$ faces
and a boundary growing slower than $n$.

The space $\BHP_\theta$ for $\theta>0$ can be understood as an
interpolation between $\BHP$ (when $\theta\to 0$) and the so-called
infinite continuum random tree $\ICRT$ introduced by Aldous~\cite{Al1}
(when $\theta\to \infty$). The $\IBD_\sigma$ in turn interpolates between
$\BHP$ (when $\sigma\rightarrow \infty$) and the Brownian plane $\BP$
introduced by Curien and Le Gall~\cite{CuLG,CuLG2} (when $\sigma\rightarrow
0$). See also the recent work of Budzinski~\cite{Bu} for a hyperbolic
version of the Brownian plane.  These interpretations are easy consequences
of our results. We refer to Remark~\ref{rem:exercises} and the exercises there for
the exact statements.
 
For a better overview, we begin with a rough list of our main results on
scaling limits of finite-size quadrangulations with a boundary (including
results of~\cite{Be3} and~\cite{BeMi}). We then mention further results
that will be obtained below, including limit statements on
$\BD_{\sigma}$. The precise formulations can be found in
Section~\ref{sec:mainresults}, after a proper definition of the limit
spaces and a reminder on the notion of convergence in
Section~\ref{sec:def}.

As in many works in this context, our approach is based on the Bouttier-Di
Francesco-Guitter bijection~\cite{BoGu,BoDFGu}, which establishes a
one-to-one correspondence between (finite-size) quadrangulations with a 
boundary on the one hand and discrete labeled forests and bridges on the
other hand. The bijection is recalled in
Section~\ref{sec:encoding}. Section~\ref{sec:auxiliaryresults} contains
some more auxiliary results, mostly convergence results on forests and bridges
when their size tends to infinity. The statements proved there are of
some independent interest, but can also be skipped at first
reading. Section~\ref{sec:proofs} contains all the proofs of our main statements.

\subsection{Overview over the main results}
\label{sec:overview}
For any $\sigma_n\in \N=\{1,2,\ldots\}$, we write $Q_n^{\sigma_n}$ for a uniformly distributed rooted
quadrangulation with $n$ inner faces and a boundary of size
$2\sigma_n$. The vertex set of $Q_n^{\sigma_n}$ is denoted by
$V(Q_n^{\sigma_n})$, $\rho_n$ represents the root vertex and $\dgr$ stands
for the graph distance on $V(Q_n^{\sigma_n})$. For any two sequences $(a_n,n\in\N),(b_n,n\in\N)$
of reals, we write $a_n \ll b_n$ or $b_n \gg a_n$ if and only if $a_n/b_n \to 0$ as
$n\to\infty$, and we write $a_n\sim b_n$ if $a_n/b_n\to 1$. 

We denote by $\circ$ the trivial one-point metric space and write $\slimGH$
($\slimLGH$) for the distributional scaling limit of
$(V(Q_n^{\sigma_n}),a_n^{-1}\dgr,\rho_n)$ in the Gromov-Hausdorff topology
(in the local Gromov-Hausdorff topology) as $n$ tends to infinity.

\paragraph{The regime $\sigma_n \ll  \sqrt{n}$.}
\begin{mdframed}
\begin{itemize}
\item If $1\ll a_n \ll \sqrt{\sigma_n}$, then
$\slimLGH=\BHP.$  
\item If $1\ll a_n \sim (1/9)^{1/4}\sqrt{2\sigma_n/\sigma}$, $\sigma \in (0,\infty)$, then
$\slimLGH=\IBD_\sigma.$ 
\item If $\sqrt{\sigma_n} \ll a_n \ll n^{1/4}$, then  
$\slimLGH=\BP.$  
\item If $a_n \sim (8/9)^{1/4}n^{1/4}$, then (see~\cite{Be3})
$\slimGH=\BM.$  
\item If $a_n \gg n^{1/4}$, then $\slimGH=\circ\,.$
\end{itemize} 
\end{mdframed}

\paragraph{The regime $\sigma_n \sim \sigma \sqrt{2n}$, $\sigma\in(0,\infty)$.}
\begin{mdframed}
\begin{itemize}
\item If $1 \ll a_n \ll n^{1/4}$, then
$\slimLGH=\BHP.$
\item If $a_n \sim (8/9)^{1/4}n^{1/4}$, then (see~\cite{Be3}
  and~\cite{BeMi}) $\slimGH=\BD_\sigma.$
\item If $a_n \gg n^{1/4}$, then $\slimGH=\circ\,.$
\end{itemize}
\end{mdframed}

\paragraph{The regime $\sigma_n\gg \sqrt{n}$.}
\begin{mdframed}
\begin{itemize}
\item If $\sigma_n\ll n$ and $\lim_{n\rightarrow\infty}
  (9/4)^{1/4}a_n/\sqrt{2n/\sigma_n}=\sqrt{\theta}\in[0,\infty)$, then 
$\slimLGH=\BHP_\theta.$
\item If $\max\{1,\sqrt{n/\sigma_n}\}  \ll a_n \ll \sqrt{\sigma_n}$, then
$\slimLGH=\ICRT.$
\item If $a_n \sim \sqrt{2\sigma_n}$ (see~\cite{Be3}), then
$\slimGH=\CRT.$
\item If $a_n \gg \sqrt{\sigma_n} $, then $\slimGH=\circ\,.$
\end{itemize}
\end{mdframed}
The new results in these listings are covered by
Theorems~\ref{thm:BP},~\ref{thm:IBD},~\ref{thm:BHP1},~\ref{thm:BHP3}
and~\ref{thm:ICRT} below. 

In the regime $\sigma_n\ll\sqrt{n}$ in the first list, the last three
convergences include the case of bounded $\sigma_n$. In the last regime
$\sigma_n\gg \sqrt{n}$, we allow $\sigma_n$ to grow faster than $n$. The
scaling constants are chosen in such a way that the description of the
limiting objects is the most natural. See also
Figure~\ref{fig:regimes-schema} for a schematic representation.

Figure~\ref{fig:usersmanual} shows all possible regimes in one diagram, in
which the $x$-axis denotes the limiting possible values for the logarithm
of the boundary length $\log(\sigma_n)/\log(n)$ in units of $\log(n)$, and
the $y$-axis corresponds to the limit of the logarithm of the scaling
factor $\log(a_n)/\log(n)$ in units of $\log(n)$. For the specific value
$y=0$, it will be assumed that $a_n=1$, so that we are really in the regime
of local limits, without any rescaling. Similarly, for some specific values
of $(x,y)$, that are shown on the colored lines, we will require some
particular scaling behaviors that are detailed in the list above. For
instance, for $x=1/2$ and $y=1/4$, we really ask that $\sigma_n\sim
\sigma\sqrt{2n}$ for some $\sigma>0$ and $a_n\sim (8/9)^{1/4}n^{1/4}$.

As it is shown in Theorem~\ref{thm:UIHPQ-BHP}, the $\BHP$ can also be
obtained from the $\UIHPQ$ by zooming-out around the root:
$\lambda\cdot\UIHPQ\rightarrow\BHP$ in distribution in the local
Gromov-Hausdorff sense as $\lambda \to 0$. Here,
$\lambda\cdot\UIHPQ$ is obtained from $\UIHPQ$ by keeping the same
set of points, but rescaling the metric by a factor $\lambda$, see
Section~\ref{sec:locGH} below. 

Many of our results, for example those involving the Brownian half-planes
$\BHP_\theta$, $\theta\geq 0$, are based on coupling methods, which yield
in fact stronger statements than those mentioned above. In particular,
couplings will allow us to deduce that the topology of $\BHP_\theta$ is
that of a closed half-plane, whereas $\IBD_\sigma$ is homeomorphic to the
pointed closed disk (Corollaries~\ref{cor:topology-BHP}
and~\ref{cor:topology-IBD}).  

The above results will moreover enable us to determine the limiting
behavior of the Brownian disk $\BD_{T,\sigma}$ of volume $T$ and perimeter
$\sigma$ when zooming-in around its root vertex, or, equivalently by
scaling, by blowing up its volume and perimeter. Depending on the behavior
of the ``perimeter'' function $\sigma(\cdot):(0,\infty)\rightarrow
(0,\infty)$ for large volumes $T$, we observe $\BP$, $\IBD_\varsigma$,
$\BHP_\theta$ or the $\ICRT$ as the distributional limit in the local
Gromov-Hausdorff sense of $\BD_{T,\sigma(T)}$ when
$T\rightarrow\infty$. See Figure~\ref{fig:diagram2} below and
Corollaries~\ref{cor:BD1},~\ref{cor:BD4},~\ref{cor:BD2},~\ref{cor:BD3}.

\begin{figure}[ht]
  \centering
  \includegraphics[width=0.6\textwidth]{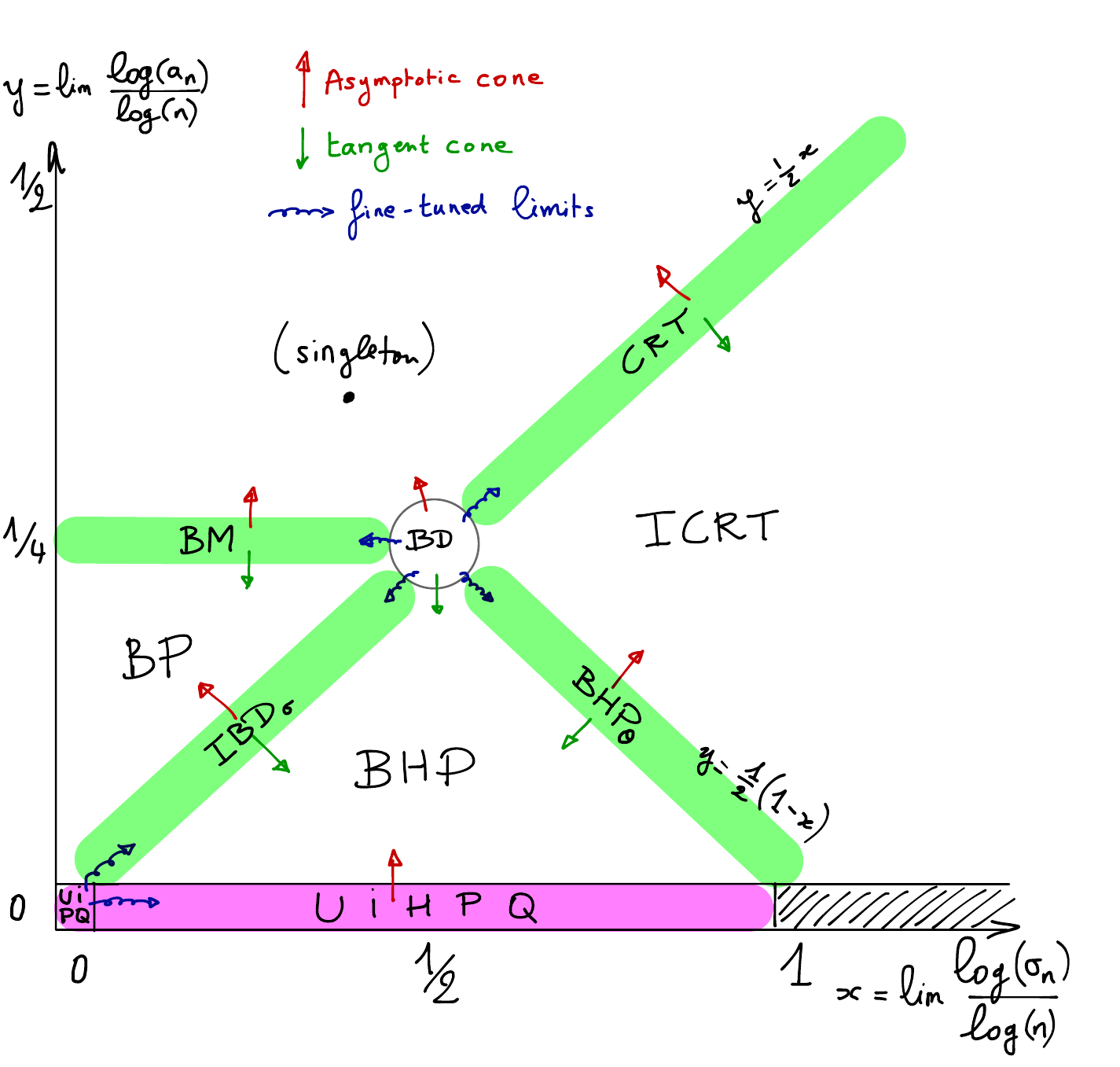}
\caption{The user's manual to this paper, displaying all possible regimes
  and limits for the rescaled pointed space
  $(V(Q_n^{\sigma_n}),a_n^{-1}\dgr,\rho_n)$.}
  \label{fig:usersmanual}
\end{figure}

\section{Definitions}
\label{sec:def}
In this section, we define our limit objects and recall some facts about
the (local) Gromov-Hausdorff convergence and the local limits of maps. 

All our limit metric spaces will be defined in terms of certain random
processes. To make the presentation unified, we will denote by $(X,W)$ the
canonical continuous process in $\mathcal{C}(I,\R)^2$, where $I$ will be an
interval of the form $I=[0,T]$ for some $T>0$, or $I=\R$. The set
$\mathcal{C}(I,\R)$ of continuous functions on $I$ is equipped with the
compact-open topology (topology of uniform convergence over compact subsets
of $I$).

For $t\in I\cap [0,\infty)$, we write
$\underline{X}_t=\inf_{[0,t]}X,$
and in case $I=\R$, we put for $t<0$
$\underline{X}_t=\inf_{(-\infty,t]}X.$

If $Y=(Y_t:t\geq 0)$ is a real-valued process indexed by the positive real
half-line, we write $\Pi(Y)$ for its {\it Pitman transform} defined as
$\Pi(Y)_t=Y_t-2\underline{Y}_t$, $t\geq 0$.  We will often use the fact
that if $B=(B_t,t\geq 0)$ is a standard Brownian motion, then its Pitman
transform $\Pi(B)$ has the law of a three-dimensional Bessel process, and
$\inf_{[t,\infty)}\Pi(B)= -\inf_{[0,t]}B$ for every $t\geq 0$. See~\cite[Theorem 0.1 (ii)]{Pi}.

\subsection{Metric spaces coded by real functions}
\label{sec:continuum-tree-coded}

\paragraph{Real trees. }
If $f$ is an element of $\mathcal{C}(I,\R)$, and $s,t\in I$, we denote
by $\underline{f}(s,t)$ the quantity 
$$
\underline{f}(s,t)=\left\{
\begin{array}{lcl}
\inf_{[s,t]}f & \mbox{ if }& s\leq t\\
\inf_{I\setminus[t,s]}f & \mbox{ if }&s>t
\end{array}\right.,
$$
and for $s,t\in I$ we let
$$d_f(s,t)=f(s)+f(t)-2\max(\underline{f}(s,t),\underline{f}(t,s))\, .$$
This defines a pseudo-metric on $I$, which is a class function for the
equivalence relation $\{d_f=0\}$. Therefore, we can define the
quotient space $\mathcal{T}_f=I/\{d_f=0\}$, on which $d_f$ induces a
true distance, still denoted by $d_f$ for simplicity. Since we assumed
that $I$ contains $0$, it is natural to ``root'' the space
$(\mathcal{T}_f,d_f)$ at the point $\rho$ given by the equivalence
class $[0]=\{s\in I:d_f(0,s)=0\}$ of $0$. 

The metric space $(\mathcal{T}_f,d_f,\rho)$ is called the {\it continuum
  tree coded by $f$}. In more precise terms, it is a rooted $\R$-tree,
which is also compact if $I$ is compact. This fact is well-known in the
``classical case'' where $f$ is a non-negative function on an interval
$[0,T]$, and $f(0)=f(T)=0$, see, e.g.~\cite[Section 3]{LGMi}, but it
remains true in this more general context. This fact will not be used in
this paper, so we do not prove it here.

Note that the space $(\mathcal{T}_f,d_f)$ comes with a natural Borel
$\sigma$-finite measure, $\mu_f$, which is defined as the push-forward
of the Lebesgue measure on $I$ by the canonical projection $p_f:I\to
\mathcal{T}_f$.

\paragraph{Metric gluing of a real tree on another. }
Let $f,g$ be two elements of $\mathcal{C}(I,\R)$. These functions code
two $\R$-trees $\mathcal{T}_f,\mathcal{T}_g$ in the preceding
sense. We now define a new metric space $(M_{f,g},D_{f,g})$ by
informally quotienting the space $(\mathcal{T}_g,d_g)$ by the
equivalence relation $\{d_f=0\}$. Formally, for $s,t\in I$, we let
\begin{equation}
\label{eq:Dfg}
D_{f,g}(s,t)=\inf\left\{\sum_{i=1}^kd_g(s_i,t_i):\begin{array}{l}
k\geq 1, \, s_1,\ldots,s_k,t_1,\ldots,t_k\in I,s_1=s,t_k=t,\\ 
d_f(t_i,s_{i+1})=0\mbox{ for every }i\in \{1,2,\ldots,k-1\}
    \end{array}
    \right\}\, .
\end{equation} 
This defines a pseudo-metric on $I$, and we let $M_{f,g}$ be the
quotient space $I/\{D_{f,g}=0\}$, endowed with the true metric
inherited from $D_{f,g}$ (and again, still denoted by $D_{f,g}$). 
Again, this space is naturally pointed at the equivalence
class of $0$ for $\{D_{f,g}=0\}$, which we still denote by $\rho$. 

Again, the space $(M_{f,g},D_{f,g})$ is naturally endowed with the
measure $\mu_{f,g}$, defined as the push-forward of the Lebesgue
measure on $I$ by the canonical projection $p_{f,g}:I\to M_{f,g}$. 

\subsection{Random snakes}\label{sec:random-snakes}
The definition for most of our limiting random spaces depend on the
notion of a random snake, which we now introduce. Let $f\in
\mathcal{C}(I,\R)$ be a continuous path on an interval $I$. The random
snake driven by $f$ is a random Gaussian process $(Z^f_s,s\in I)$ satisfying
$Z^f_0=0$ a.s.\ and
$$\E[|Z^f_s-Z^f_t|^2]=d_f(s,t)\, .$$ These specifications characterize
the law of $Z^f$: roughly speaking, it can be seen as Brownian motion
indexed by the tree $\mathcal{T}_f$, see, e.g., Section 4 of~\cite{LGMi}.
It is easy to see and well-known that the process $Z^f$ admits a continuous
modification as soon as $f$ is a locally H\"older-continuous function on
$I$. In this case, we always work with this modification.

The snake driven by a random function $Y$ is then defined as the random
Gaussian process $Z^Y$ conditionally given $Y$. In all our applications,
$Y$ will be considered under probability distributions that make it a
H\"older-continuous function with probability one.

More specifically, except for the case of the infinite-volume Brownian
disk, see below, we will either let $Y=X$ for $X$ the canonical process on
$\mathcal{C}(I,\R)$ (namely for the Brownian map and the Brownian plane),
or $Y=X-\underline{X}$ (for the Brownian disk and the Brownian half-planes).

\subsection{Limit random metric spaces}
We apply the preceding constructions to a variety of random
versions of the functions $f,g$.
\subsubsection{Compact spaces}\label{sec:compact-spaces}
In this section the processes considered all take values in
$\mathcal{C}([0,T],\R)$ for some $T>0$. 
\paragraph{Continuum random tree {\normalfont $\CRT_T$, $T>0$.}} The continuum
random tree was introduced by Aldous~\cite{Al1, Al2} and is defined as
follows.

\begin{defn}
  Let $T>0$. The continuum random tree $\CRT_T$ with volume $T$ is the random rooted real tree
  $(\mathcal{T}_X,d_x,\rho)$ for the probability distribution that
  makes the canonical process $X$ of $\mathcal{C}([0,T],\R)$ the
 standard Brownian excursion with duration $T$. 
\end{defn}

The term ``CRT'' usually denotes $\CRT_1$ with volume $T=1$, in
which case $X$ is taken under the law of the normalized Brownian
excursion. We simply write $\CRT$ instead of $\CRT_1$.

Note the scaling relation, for $\lambda,T>0$:  
$$\lambda\cdot\CRT_T =_d \CRT_{\lambda^2 T} \, .$$
This comes from the fact that, if $\mathbbm{e}^T$ is a Brownian excursion with
duration $T$, then $\lambda\mathbbm{e}^{T}(\cdot/\lambda^2)$ has same
distribution as $\mathbbm{e}^{\lambda^2 T}$. 

We should also discuss the role of $\rho$ in the above definition. The
re-rooting property of $\CRT_T$~\cite[(20)]{Al2} states, roughly speaking, that
if $\rho'$ is a random variable with distribution $\mu_X/\mu_X(1)$
(the normalized version of the measure defined above), then
$(\mathcal{T}_X,d_X,\rho')$ has same distribution as
$(\mathcal{T}_X,d_X,\rho)$. In this sense, the point $\rho$ plays no
distinguished role in the construction of $\CRT_T$.

\paragraph{Brownian map {\normalfont $\BM_T$, $T>0$.}} 
The Brownian map is roughly speaking the metric gluing of the tree
coded by a snake driven by a normalized Brownian excursion, on the
tree coded by the excursion itself. 

\begin{defn}
  The Brownian map $\BM_T$ with volume $T$ is the metric space
  $(M_{X,W},D_{X,W},\rho)$ for the probability law that makes $X$ a
  Brownian excursion of duration $T$, and $W$ is the snake driven by
  $X$.  
\end{defn}
We write $\BM$ instead of $\BM_1$.
The scaling properties of Gaussian processes imply easily that for
$\lambda>0$,
$$\lambda\cdot \BM_T=_d \BM_{\lambda^4 T}\, .$$

Just as for $\CRT_T$, the point $\rho$ in $\BM_T$ should be seen as a
random choice according to the normalized measure
$\mu_{X,W}/\mu_{X,W}(1)$, which is known as the re-rooting property
of the Brownian map (Theorem 8.1 of~\cite{LG4}). The latter is a crucial
property for characterizing the Brownian map, see, e.g., the recent
work~\cite{MiSh}.

\paragraph{Brownian disk {\normalfont $\BD_{T,\sigma}$,
    $\sigma\in(0,\infty)$, $T>0$.}} The Brownian disk first appears in~\cite{Be3}
as limiting metric space along suitable infinite subsequences. Uniqueness
of the limit and a concrete description of the metric were obtained
in~\cite{BeMi}.

The description is slightly more elaborate than that of the Brownian
map. For $t\geq 0$, we let $\underline{X}_t=\inf_{[0,t]}X$. 

\begin{defn}
  \label{def:BD}
  The Brownian disk $\BD_{T,\sigma}$ with volume $T$ and boundary length
  $\sigma$ is the metric space $(M_{X,W},D_{X,W},\rho)$ under the
  probability measures that makes $X$ a first-passage Brownian bridge from
  $0$ to $-\sigma$ of duration $T$, and conditionally given $X$,
  $(W_t,0\leq t\leq T)$ has same distribution as $(\sqrt{3}\,
  \gamma_{-\underline{X}_t}+Z_t,0\leq t\leq T)$, where
\begin{itemize}
\item $(Z_t,0\leq t\leq T)=Z^{X-\underline{X}}$ is the random snake driven
  by the reflected process $(X_t-\underline{X}_t,0\leq t\leq T)$, i.e., (a
  continuous modification of) the centered Gaussian process with
  covariances given by
      $$
      \E\left[Z_sZ_t\right]=\min_{[s\wedge t,s\vee
        t]}(X-\underline{X}).
      $$        
\item $(\gamma_x,0\leq x\leq \sigma)$ is a standard Brownian bridge
  with duration $\sigma$, independent of $Z^{X-\underline{X}}$.
\end{itemize}
\end{defn}

The Brownian disks are homeomorphic to the closed unit disk
$\overline{\mathbb{D}}$, where $\mathbb{D}=\{z\in \mathbb{C}:|z|<1\}$,
see~\cite[Proposition 21]{Be3} (cited as Lemma~\ref{lem:proof-prop-refpr}
below).  They enjoy the following scaling property: For $\lambda>0$,
$$\lambda\cdot \BD_{T,\sigma}=_d \BD_{\lambda^4 T,\lambda^2\sigma}\, .$$
If $T=1$, we will simply write $\BD_\sigma$ instead of $\BD_{1,\sigma}$.
Contrary to the Brownian tree or the Brownian map, $\rho$ does not
play the role of a random point distributed according to
$\mu_{X,W}/\mu_{X,W}(1)$. The reason is that $\rho$ is a.s.\ a point
of the boundary of the disk, which is of zero measure, see~\cite{BeMi}
for more details. 

\subsubsection{Non-compact spaces}\label{sec:non-compact-spaces}
In this subsection, all processes take values in $\mathcal{C}(\R,\R)$.
\paragraph{Infinite continuum random tree {\normalfont $\ICRT$.}} 
The $\ICRT$ is process $2$ in~\cite{Al1} and can be defined as
follows.

\begin{defn}
  \label{def:ICRT}
  The infinite continuum random tree $\ICRT$ is the random rooted real tree
  $(\mathcal{T}_X,d_X,\rho)$, for the probability distribution under
  which the canonical process $X$ in $\mathcal{C}(\R,\R)$ is such that
  $(X_t,t\geq 0)$ and $(X_{-t},t\geq 0)$ are two independent standard
  three-dimensional Bessel processes started at $0$. 
\end{defn}

This results in an a.s.\ non-compact real tree, which enjoys the
remarkable self-similarity property that $\lambda\cdot \ICRT=_d\ICRT$
for every $\lambda>0$.

Note that if we let $Y$ be the canonical process in
$\mathcal{C}(\mathbb{R},\mathbb{R})$ such that $(Y_t,t\geq 0)$ and
$(Y_{-t},t\geq 0)$ are two independent standard Brownian motions, then the
random rooted real tree $(\mathcal{T}_Y,d_Y,[0])$ has same distribution as
$\ICRT$. This follows readily from the fact that $\Pi((Y_t,t\geq 0))$ has
the law of a three-dimensional Bessel process.

\paragraph{Brownian plane {\normalfont $\BP$.}} The Brownian plane was
introduced in~\cite{CuLG}. 

\begin{defn}
  \label{def:BP}
  The Brownian plane $\BP$ is the pointed space
  $(M_{X,W},D_{X,W},\rho)$ under the probability distribution such
  that 
  \begin{itemize}
  \item $(X_t,t\geq 0)$ and $(X_{-t},t\geq 0)$ are two independent
    three-dimensional Bessel processes.   
  \item Given $X=(X_t, t\in\mathbb{R})$, $W$ has same distribution as the
    random snake $Z^X$ driven by $X$.
\end{itemize}
\end{defn}

The Brownian plane is a.s.\ homeomorphic to $\R^2$, and is invariant
under scaling: for $\lambda>0$,
$$\lambda\cdot \BP=_d\BP\, .$$

\paragraph{Brownian half-planes {\normalfont $\BHP_{\theta}$,
  $\theta\in[0,\infty)$.}}

The Brownian half-planes are the first truly new limiting metric
spaces that we encounter in this study. 
Recall that $\underline{X}_t=\inf_{[0,t]}X$ for $t\geq 0$ and
$\underline{X}_{t}=\inf_{(-\infty,t]}X$ for $t<0$. 
\begin{defn}
  \label{def:BHP}
  Let $\theta\geq 0$ be fixed. The Brownian half-plane $\BHP_\theta$
  with skewness parameter $\theta$ is the pointed space
  $(M_{X,W},D_{X,W},\rho)$ under the probability distribution such
  that 
  \begin{itemize}
  \item $(X_t,t\geq 0)$ is a standard Brownian motion with linear drift
    $-\theta$, and $(X_{-t},t\geq 0)$ is the Pitman transform
    $\Pi(X')$ of an independent copy $X'$ of $(X_t,t\geq
    0)$. 
  \item Given $X$, $W$ has same distribution as
    $(\sqrt{3}\, \gamma_{-\underline{X}_t}+Z_t,t\in \R)$, where
    \begin{itemize}
    \item $(Z_t,t\in\R)=Z^{X-\underline{X}}$ is the snake driven by the
      process $(X_t-\underline{X}_t,t\in \R)$, i.e., the centered Gaussian
      process with covariances given by
      $$
      \E\left[Z_sZ_t\right]=\min_{[s\wedge t,s\vee t]}(X-\underline{X}).
      $$ 
   \item $(\gamma_x,x\in \R)$ is a two-sided standard Brownian motion
      with $\gamma_0=0$, independent of $Z^{X-\underline{X}}$. 
     \end{itemize}
\end{itemize}
\end{defn}

The scaling property enjoyed by $\BHP_\theta$ is that for $\lambda>0$,
$$\lambda\cdot \BHP_\theta=_d\BHP_{\theta/\lambda^2}\, .$$
This makes the value $\theta=0$ special in the sense that the space is
self-similar in distribution in this case (just as $\ICRT$ or $\BP$). Keep
in mind that we often write $\BHP$ instead of $\BHP_0$.  We will see in
Corollary~\ref{cor:topology-BHP} that for every $\theta\geq 0$, $\BHP_\theta$
is a.s.\ homeomorphic to the closed half-plane
$\overline{\mathbb{H}}=\R\times\R_+$.

\begin{remark}
  Note that a random metric space also called the Brownian half-plane first
  appeared in the recent work~\cite{CaCu}, where it is conjectured that it
  arises as the scaling limit of the uniform infinite half-planar
  quadrangulation $\UIHPQ$, the definition of which is recalled in
  Section~\ref{sec:constr-UIHPQ}.  Theorem~\ref{thm:UIHPQ-BHP} below states
  indeed that the scaling limit of $\UIHPQ$ is the space $\BHP_0$.
  However, an important caveat is that the definition of the Brownian
  half-plane from~\cite{CaCu} is different from ours: it is still of the
  form $(M_{X,W},D_{X,W},\rho)$, but for processes $(X,W)$ having a very
  different law from the one presented in Definition~\ref{def:BHP}
  (with $\theta=0$).  We do not actually prove that the two definitions
  coincide, since we believe that this would require some specific
  work. Nonetheless, we prefer to stick to the name ``Brownian half-plane''
  since we feel that this should be the proper denomination for the scaling
  limit of the $\UIHPQ$. See also Remark~\ref{rem:BHP-char} below.
\end{remark}

\paragraph{Infinite-volume Brownian disk {\normalfont $\IBD_\sigma$,
    $\sigma\in(0,\infty)$.}} The infinite-volume Brownian disk
$\IBD_\sigma$ should be thought of as a Brownian disk $\BD_\sigma$ filled
in with a Brownian plane $\BP$. The definition is a bit elaborate; we
give some explanation in Remark~\ref{rem:def-IBD} below.

Let $(B_t,t\geq 0)$ be a standard Brownian motion with $B_0=0$, and
$T_{x}=\inf\{t\geq 0:B_t<-x\}$ the first hitting time of
$(-\infty,-x)$. Let $R,R'$ be two independent three-dimensional Bessel
processes independent of $B$, and $U_0$ be a uniform random variable
in $[0,\sigma]$, independent of $B,R,R'$. We set
$$Y^\sigma_t=\left\{\begin{array}{l@{\quad\mbox{if }\,}l}
      R'_{-t+T_\sigma-T_{U_0}}+\sigma -U_0  & t \leq T_{U_0}-T_\sigma\\
      B_{T_\sigma+t}+\sigma & T_{U_0}-T_\sigma\leq t \leq 0\\
      B_t & 0\leq t\leq T_{U_0}\\  
      -U_0+R_{t-T_{U_0}} & t \geq T_{U_0}
\end{array}\right. .
$$

\begin{defn}
\label{def:IBD}
  Let $\sigma>0$ be fixed. The infinite-volume Brownian disk
  $\IBD_\sigma$ with boundary length $\sigma$ is the pointed space
  $(M_{X,W},D_{X,W},\rho)$ under the probability distribution such
  that
  \begin{itemize}
  \item $(X_t,t\in \R)$ is given by the process $Y^\sigma$ described above.
  \item Given $X$, $W$ has same distribution as
    $(\sqrt{3}\, \gamma_{-\uuX^\sigma_t}+Z_t,t\in \R)$, where 
    \begin{itemize}
    \item $$\uuX_t=\left\{\begin{array}{l@{\quad\mbox{if }\,}l}
            \min\left\{\inf_{(-\infty,t]}X,\, \inf_{[0,\infty)}
                X+\sigma\right\}
              &t\leq 0\\
              \min_{[0,t]}X & t\geq 0
           \end{array}\right.,$$
  and $\uuX^\sigma=\uuX-\sigma$ on $(-\infty,0)$, $\uuX^\sigma=\uuX$ on $[0,\infty)$.

       \item $(Z_t,t\in \R)=Z^{X-\uuX}$ is the random snake driven by the
         process $X-\uuX$.
    \item $(\gamma_x,0\leq x\leq \sigma)$ is a standard Brownian
      bridge with duration $\sigma$, independent of $Z^{X-\uuX}$. 
    \end{itemize}
\end{itemize}
\end{defn}

The infinite-volume Brownian disks enjoy the scaling property
$$\lambda\cdot \IBD_{\sigma}=_d\IBD_{\lambda^2\sigma}\, ,$$
for $\lambda,\sigma>0$. We will prove in Corollary~\ref{cor:topology-IBD}
below that for every $\sigma>0$, $\IBD_\sigma$ is a.s.\ homeomorphic to the
pointed closed disk $\overline{\mathbb{D}}\setminus \{0\}$.

\begin{remark}
\label{rem:def-IBD}
We give some intuition behind the above definition. Recall that in the case
of the Brownian disk $\BD_{T,\sigma}$, the contour process $X$ is given by
a first-passage Brownian bridge from $0$ to $-\sigma$ of duration
$T$. Here, in the case of $\IBD_\sigma$, we consider a Brownian motion $B$
stopped upon first hitting $-\sigma$. When for the first time level $-U_0$
is hit, where $U_0$ is uniform on $[0,\sigma]$, the encoding of a Brownian
plane ``inside'' the disk starts. The encoding of the latter is given by
two independent three-dimensional Bessel processes $R$ and $R'$, as in the
definition of $\BP$. The part of the Brownian plane encoded by $R'$ as well
as the trees encoded by $B$ along $(U_0,\sigma]$ appear in the definition
of $Y^\sigma$ to the left of zero.
\end{remark}
\begin{figure}[ht]
  \centering
  \includegraphics[width=0.5\textwidth]{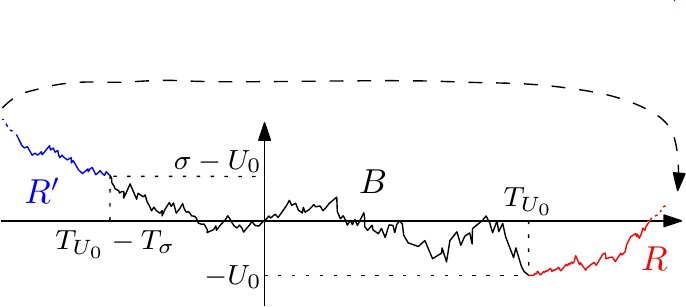}
\caption{The contour process $X=_dY^\sigma$ of the infinite Brownian disk $\IBD_\sigma$.}
  \label{fig:IBD-contour}
\end{figure}

\paragraph{Uniform infinite half-planar quadrangulation {\normalfont
    $\UIHPQ$}.}
The (non-compact) random metric space $\UIHPQ$
$Q_{\infty}^{\infty}=(V(Q_{\infty}^{\infty}),\dgr,\rho)$ is an infinite
rooted random quadrangulation with an infinite boundary. It arises as the
distributional limit of $Q_n^{\sigma_n}$, $1\ll\sigma_n\ll n$, for the
so-called local metric $\dmap$, see Proposition~\ref{prop:Qn-UIHPQ}.  We defer
to Section~\ref{sec:locallimit} for a definition of the metric and to
Section~\ref{sec:constr-UIHPQ} for a precise construction of the
$\UIHPQ$. 

\subsection{Notion of convergence}
\label{S-notionconvergence}
\subsubsection{Gromov-Hausdorff convergence}
Given two pointed compact metric spaces $\mathbf{E}=(E,d,\rho)$ and
$\mathbf{E}' =(E',d',\rho')$, the Gromov-Hausdorff distance between
$\mathbf{E}$ and $\mathbf{E}'$ is given by
$$
\dgh(\mathbf{E},\mathbf{E}') =
\inf\left\{\dha(\varphi(E),\varphi'(E))\vee\delta(\varphi(\rho),\varphi'(\rho'))\right\},
$$
where the infimum is taken over all isometric embeddings $\varphi :
E\rightarrow F$ and $\varphi' : E'\rightarrow F$ of $E$ and $E'$ into the
same metric space $(F,\delta)$, and $\dha$ denotes the Hausdorff distance
between compact subsets of $F$. The space of all isometry classes of
pointed compact metric spaces $(\mathbb{K},\dgh)$ forms a Polish space.

We will use a well-known alternative characterization of the Gromov-Hausdorff
distance {\it via} correspondences, which we recall here for the reader's convenience.
A {\it correspondence} between two pointed metric spaces $\mathbf{E}=(E,d,\rho)$,
$\mathbf{E}'=(E',d',\rho')$ is a subset $\cR\subset E\times E'$ such that
$(\rho,\rho')\in \cR$, and for every $x\in E$ there exists at least one $x'\in E'$ such that
$(x,x')\in E\times E'$ as well as for every $y'\in E'$, there exists at least one
$y\in E$ such that $(y,y')\in\cR$. The distortion of $\cR$
with respect to $d$ and $d'$ is given by
$$
\dis(\cR) = \sup\left\{|d(x,y)-d'(x',y')| : (x,x'), (y,y')\in\cR\right\}.
$$
Then it holds that (see, for example,~\cite{BuBuIv})
$$
\dgh(\mathbf{E},\mathbf{E}') = \frac{1}{2}\inf_\cR\dis(\cR),
$$
where the infimum is taken over all correspondences between $\mathbf{E}$
and $\mathbf{E}'$. 

The convergences listed in the overview above which involve compact limiting spaces,
i.e., $\BM$, $\BD_\sigma$, $\CRT$ and the trivial one-point space, hold
in distribution in $(\mathbb{K},\dgh)$.

\subsubsection{Local Gromov-Hausdorff convergence}
\label{sec:locGH}
For non-compact spaces like $\BHP_\theta$, $\IBD_\sigma$ or 
$\ICRT$, the Gromov-Hausdorff convergence is too restrictive.
Instead, the right notion is convergence in the so-called local
Gromov-Hausdorff sense, which, roughly speaking, requires only convergence
of balls of a fixed radius seen as compact metric spaces.

We give here a quick reminder of this form of convergence; for more
details, we refer to Chapter $8$ of~\cite{BuBuIv}. As in~\cite{CuLG}, we
can restrict ourselves to the case of (pointed) complete and locally
compact length spaces (see our discussion below).

More precisely, a metric space $(E,d)$ is a length space if for every pair
$(x,y)$ of points in $E$, the distance $d(x,y)$ agrees with the infimum over
the lengths of continuous paths from $x$ to $y$. Here, a continuous path
from $x$ to $y$ is a continuous function $\gamma : [0,T]\rightarrow E$ with
$\gamma(0)=x$ and $\gamma(T)=y$ for some $T\geq 0$, and the length of
$\gamma$ is given by
$$
L(\gamma)=\sup_{\tau}\sum_{k=1}^{n-1} d(\gamma(t_k),\gamma(t_{k+1})),
$$
where the supremum is taken over all subdivisions $\tau$ of $[0,T]$ of the
form $0=t_1<t_2<\dots<t_n=T$ for some $n\in\N$. Note that in a complete
and locally compact length space $(E,d)$, there exists between any two
points $x,y\in E$ with $d(x,y)<\infty$ a continuous path of minimal length,
see~\cite[Theorem 2.5.23]{BuBuIv}.

Now let $\mathbf{E}=(E,d,\rho)$ be a pointed metric space, that is a metric
space with a distinguished point $\rho\in E$. We denote by
$B_r(\mathbf{E})$ the closed ball of radius $r$ around $\rho$ in
$\mathbf{E}$. Equipped with the restriction of $d$, we view
$B_r(\mathbf{E})$ as a pointed compact metric space, with distinguished
point given by $\rho$. By a small abuse of notation, we shall also view
$B_r(\mathbf{E})$ as a set and write $x\in B_r(\mathbf{E})$ if $x\in E$ is
at distance at most $r$ from $\rho$.

Given pointed complete and locally compact length spaces $(\mathbf{E}_n)_n$ and
$\mathbf{E}$, the sequence $(\mathbf{E}_n)_n$ converges to $\mathbf{E}$ in the local
Gromov-Hausdorff sense if for every $r\geq 0$,
$$
\dgh(B_r(\mathbf{E}_n),B_r(\mathbf{E}))\rightarrow 0\quad\textup{ as }n\rightarrow\infty.
$$
This notion of convergence is metrizable (see~\cite[Section 2.1]{CuLG} for
a definition of the metric) and turns the space $\mathbb{K}_{bcl}$ of
isometry classes of pointed boundedly compact length spaces into a Polish
space.  

We are interested in limits of quadrangulations; however, as
discrete planar maps the latter are clearly not length
spaces. Following~\cite{CuLG}, we may nonetheless interpret a (finite or
infinite) quadrangulation $Q$ as a pointed complete and locally finite
length space ${\mathbf Q}$. Namely, we replace each edge of $Q$ by an
Euclidean segment of length one such that two segments can intersect only
at their endpoints, and they do so if and only if the corresponding edges
in $E$ share one or two vertices.

The resulting metric space ${\mathbf Q}$ is then a union of copies of the
interval $[0,1]$, one for each edge of $Q$. The distance between two points
is simply given by the length of a shortest path between them. With the
root vertex of $Q$ as distinguished point, this new metric space ${\mathbf
  Q}$ is a (pointed) complete and locally compact length space. Moreover,
it is easy to see that $\dgh(B_r(Q),B_r({\bf Q}))\leq 1$ for every
$r\geq 0$. 
 \\
\newline {\noindent\bf Notation:} Given a pointed metric space $\mathbf{E}=(E,d,\rho)$ and
$\lambda >0$, we write $\lambda\cdot\mathbf{E}$ for the dilated (or rescaled) space
$(E,\lambda\cdot d,\rho)$. In particular, if $\lambda,\delta>0$,
$\lambda\cdot B_\delta(\mathbf{E})=B_{\lambda
  \delta}(\lambda\cdot\mathbf{E})$. 
\begin{remark}
\label{rem:localGH}
From our observation above, we deduce that our limit results for
quadrangulations $Q_n^{\sigma_n}$ in the local Gromov-Hausdorff sense will
follow if we show that for each $r\geq 0$, $B_r(a_n^{-1}\cdot
Q_n^{\sigma_n})$ converges in distribution in $\mathbb{K}$ towards the
ball of radius $r$ in the corresponding limit space (all our limit spaces
are already locally compact length spaces). We therefore do not
have to deal with the more complicated notion of local Gromov-Hausdorff
convergence for general (pointed) metric spaces, see~\cite[Definition
8.1.1]{BuBuIv}.
\end{remark}

\subsubsection{Local limits of maps}
\label{sec:locallimit}
Local limits of maps in the sense of Benjamini and Schramm~\cite{BeSc}
concern the convergence of combinatorial balls. More specifically, given a
rooted planar map $\m$ and $r\geq 0$, write $\cb_r(\m)$ for the combinatorial
of radius $r$, that is the submap of $\m$ formed by all the vertices $v$ of
$\m$ with $\dgr(\varrho,v)\leq r$, together with the edges of $\m$ in between
such vertices. For two rooted maps $\m$ and $\m'$, the local distance between
$\m$ and $\m'$ is defined as
$$
\dmap(\m,\m') = \left(1+\sup\{r\geq 0:\cb_r(\m)=\cb_r(\m')\}\right)^{-1}.
$$
The metric $\dmap$ induces a topology on the space of all finite
quadrangulations (with or without boundary). {\it Infinite
  quadrangulations} are the elements in the completion of this space with
respect to $\dmap$ that are not finite quadrangulations (the $\UIHPQ$ is a
random infinite quadrangulation with an infinite boundary).
See~\cite{CuMeMi} for more on this.

\section{Main results}
\label{sec:mainresults}
We formulate now in a proper way our main results, which cover together
with the results of~\cite{Be3,BeMi} all the convergences listed in the
introduction. The proofs will be given in Section~\ref{sec:proofs}. 

\subsection{Scaling limits of quadrangulations with a boundary}
\label{sec:results-scalinglimits}
  We let $Q_n^{\sigma_n}$ be uniformly distributed over the set of all
  rooted planar quadrangulations with $n$ inner faces and a boundary of
  perimeter $2\sigma_n$, $\sigma_n\in\N$. Recall that we
  write $V(Q_n^{\sigma_n})$ for the vertex set of $Q_n^{\sigma_n}$,
  $\rho_n$ for its root vertex and $\dgr$ for the graph distance on
  $V(Q_n^{\sigma_n})$. Always, $(a_n,n\in\N)$ a sequence of (strictly)
  positive reals. All convergences in this section are in law, with respect
  to the local Gromov-Hausdorff topology. We always let $n\rightarrow\infty$.
\begin{thm}
  \label{thm:BP}
  Assume 
  $\sigma_n\ll\sqrt{n}$. If $\sqrt{\sigma_n} \ll a_n \ll n^{1/4}$, then
$$
(V(Q_n^{\sigma_n}),a_n^{-1} \dgr,\rho_n) \longrightarrow
\BP.
$$  
\end{thm}

\begin{thm}
  \label{thm:IBD}
  Assume $1\ll\sigma_n\ll\sqrt{n}$ and $a_n \sim
  (4/9)^{1/4}\sqrt{\sigma_n/\sigma}$ for some $\sigma \in (0,\infty)$.
  Then
$$
(V(Q_n^{\sigma_n}),a_n^{-1} \dgr,\rho_n) \longrightarrow\IBD_{\sigma}.
$$ 
\end{thm}

\begin{thm}
  \label{thm:BHP1}
  Assume $1\ll \sigma_n\ll n$ and $1\ll a_n \ll \min\{\sqrt{\sigma_n},\,\sqrt{n/\sigma_n}\}$.
  Then
$$
(V(Q_n^{\sigma_n}),a_n^{-1} \dgr,\rho_n) \longrightarrow \BHP.
$$  
\end{thm}

\begin{thm}
  \label{thm:BHP3}
  Assume $\sqrt{n}\ll\sigma_n\ll n$ and $a_n\sim 2\sqrt{\theta n/3\sigma_n}$ for some $\theta\in (0,\infty)$.
  Then
$$
(V(Q_n^{\sigma_n}),a_n^{-1} \dgr,\rho_n) \longrightarrow\BHP_\theta.
$$  
\end{thm}

\begin{thm}
  \label{thm:ICRT}
  Assume $\sigma_n\gg \sqrt{n}$ and $\max\{1,\,\sqrt{n/\sigma_n}\} \ll a_n
  \ll \sqrt{\sigma_n}$.  Then
$$
(V(Q_n^{\sigma_n}),a_n^{-1} \dgr,\rho_n) \longrightarrow\ICRT.
$$
\end{thm}
When the scaling sequence $(a_n,n\in\N)$ satisfies
$a_n\gg\max\{\sqrt{\sigma_n},n^{1/4}\}$, then the limiting space is the
trivial one-point metric space. This is a direct consequence of the results
in~\cite{Be3}, for example.

The Brownian half-plane $\BHP$ does also arise as the weak scaling limit of the
$\UIHPQ$ (similarly, the Brownian plane $\BP$ is the scaling limit of the
so-called uniform infinite planar quadrangulation $\UIPQ$, see the first
part of~\cite[Theorem 2]{CuLG}).
\begin{thm}
\label{thm:UIHPQ-BHP}
$$
\lambda\cdot\UIHPQ \xrightarrow[]{\lambda\rightarrow 0}\BHP.
$$
\end{thm}

\subsection{Couplings and topology}
For proving Theorem~\ref{thm:BHP1}, we follow a strategy similar to
that in Curien and Le Gall~\cite{CuLG}. As an intermediate step, we
establish a coupling between the Brownian disk $\BD_\sigma$ and the
Brownian half-plane $\BHP_\theta$, which we also apply to determine the topology of $\BHP_\theta$.
\begin{thm}
\label{thm:coupling-BD-BHP}
Let $\eps>0$, $r\geq 0$. Let
$\sigma(\cdot):(0,\infty)\rightarrow(0,\infty)$ be a function satisfying
$\lim_{T\rightarrow\infty}\sigma(T)/T=\theta\in[0,\infty)$ and
$\liminf_{T\rightarrow\infty}\sigma(T)/\sqrt{T}>0$. Then there exists
$T_0=T_0(\eps,r,\sigma)$ such that for all $T\geq T_0$, one can construct
copies of $\BD_{T,\sigma(T)}$ and $\BHP_\theta$ on the same probability
space such that with probability at least $1-\eps$, there exist two
isometric open subsets $O_\BD$, $O_{\BHP}$ in these spaces which are both
homeomorphic to the closed half-plane $\overline{\mathbb{H}}$ and contain
the balls $B_r(\BD_{T,\sigma(T)})$ and $B_r(\BHP_\theta)$, respectively.
\end{thm}
We remark that for the proof of Theorem~\ref{thm:BHP1}, it would be
sufficient to show that the balls of radius $r$ around the root in the
corresponding spaces are isometric. From the stronger version of the
coupling stated above, we can however additionally deduce
\begin{corollary}
\label{cor:topology-BHP}
For every $\theta \geq 0$, the space $\BHP_\theta$ is a.s. homeomorphic to
the closed half-plane $\overline{\mathbb{H}}=\R\times\R_+$.
\end{corollary}
Since the Brownian half-plane $\BHP=\BHP_0$ is scale-invariant, i.e.,
$\lambda\cdot\BHP =_d\BHP$ for every $\lambda>0$,
Theorem~\ref{thm:coupling-BD-BHP} moreover implies that $\BHP$ is locally
isometric to the disk $\BD_\sigma(=\BD_{1,\sigma})$.
\begin{corollary}
\label{cor:isometry-BD-BHP}
  Fix $\sigma\in(0,\infty)$, and let $\eps>0$. Then one can find 
  $\delta>0$ and construct on the same probability space copies of
  $\BD_\sigma$ and $\BHP$ such that with probability at least
  $1-\eps$, $B_\delta(\BHP)$ and $B_\delta(\BD_\sigma)$ are
  isometric.
\end{corollary}
The proof of Corollary~\ref{cor:isometry-BD-BHP} is immediate from the
scaling properties of $\BD_{T,\sigma}$ and $\BHP$, whereas
Corollary~\ref{cor:topology-BHP} needs an extra argument, which we give in
Section~\ref{sec:proof-coupling-BD-BHP}.  
\begin{remark}
\label{rem:BHP-char}
  The local isometry between $\BHP$ and $\BD_\sigma$ together with the fact
  that $\BHP$ is scale-invariant uniquely characterizes the law of $\BHP$
  in the set of all probability measures on $\mathbb{K}_{bcl}$. This follows from the
  argument in the proof of~\cite[Proposition 3.2]{CuLG2}, where a
  similar characterization of the Brownian plane is given.
\end{remark}

For establishing Theorem~\ref{thm:BHP1}, we shall also need a coupling
between the $\UIHPQ$ and $Q_n^{\sigma_n}$ when $\sigma_n$ grows slower than
$n$.
\begin{prop}
\label{prop:Qn-UIHPQ}
Assume $1\ll \sigma_n\ll n$, and put $\vartheta_n=\min\left\{\sigma_n,\,
  n/\sigma_n\right\}$.  Given any $\eps>0$, there exist $\delta>0$
and $n_0\in\N$ such that for every $n\geq n_0$, one can construct copies of
$Q_n^{\sigma_n}$ and $\UIHPQ$ on the same probability space such that with
probability at least $1-\eps$, the balls $B_{\delta
  \sqrt{\vartheta_n}}(Q_n^{\sigma_n})$ and $B_{\delta
  \sqrt{\vartheta_n}}(\UIHPQ)$ are isometric. Moreover, we have the local
convergence $$ (V(Q_n^{\sigma_n}),\dgr,\rho_n) \longrightarrow \UIHPQ
$$  
in distribution for the metric $\dmap$, as $n\rightarrow\infty$.
\end{prop}
Note that the above mentioned $\UIPQ$ is in turn the weak limit in the
sense of $\dmap$ for uniform quadrangulations {\it without} a boundary, see
Krikun~\cite{Kr}.

For proving Theorem~\ref{thm:IBD} and showing that the topology of
$\IBD_\sigma$ is that of a pointed closed disk, we couple the Brownian
disk $\BD_{T,\sigma}$ for large volumes $T$ with the infinite-volume
Brownian disk $\IBD_\sigma$.
\begin{thm}
\label{thm:coupling-BD-IBD}
Fix $\sigma\in(0,\infty)$, and let $\eps>0$, $r\geq 0$. There exists
$T_0=T_0(\eps,r,\sigma)$ such that for all $T\geq T_0$, we can
construct copies of $\BD_{T,\sigma}$ and $\IBD_\sigma$ on the same
probability space such that with probability at least $1-\eps$,
there exist two isometric open subsets $A_{\BD}$, $A_{\IBD}$ in these
spaces which are both homeomorphic to the pointed closed disk
$\overline{\mathbb{D}}\setminus \{0\}$ and contain the balls
$B_r(\BD_{T,\sigma})$ and $B_r(\IBD_\sigma)$, respectively.
\end{thm}
It will be straightforward to deduce
\begin{corollary}
\label{cor:topology-IBD}
For each $\sigma\in(0,\infty)$, the space $\IBD_\sigma$ is a.s. homeomorphic to
the pointed closed disk $\overline{\mathbb{D}}\setminus \{0\}$, where $\mathbb{D}=\{z\in \mathbb{C}:|z|<1\}$.
\end{corollary}
In order to prove Theorem~\ref{thm:IBD}, we finally need a coupling of
balls in the quadrangulations $Q_n^{\sigma_n}$ and
$Q_{R\sigma_n^2}^{\sigma_n}$ of a radius of order $\sqrt{\sigma_n}$, when
$1\ll \sigma_n\ll \sqrt{n}$ and $R$ is large.
\begin{prop}
\label{prop:coupling-Qn-largevol}
Assume $1\ll \sigma_n\ll \sqrt{n}$. Given any $\eps>0$ and $r>0$,
there exist $R_0>0$ and $n_0\in\N$ such that for every integer $R\geq R_0$
and every $n\geq n_0$, on can construct copies of $Q_n^{\sigma_n}$ and
$Q_{R\sigma_n^2}^{\sigma_n}$ on the same probability space such that with
probability at least $1-\eps$, the balls
$B_{r\sqrt{\sigma_n}}(Q_n^{\sigma_n})$ and
$B_{r\sqrt{\sigma_n}}(Q_{R\sigma_n^2}^{\sigma_n})$ are isometric.
\end{prop}

  Some of our results involving $\UIHPQ$, $\BHP$ and $\BD_\sigma$ are
  depicted in Figure~\ref{fig:diagram1}, which should be compared
  with~\cite[Figure 1]{CuLG}.
\begin{figure}[ht]
\centering\parbox{9.5cm}{\includegraphics[width=0.6\textwidth]{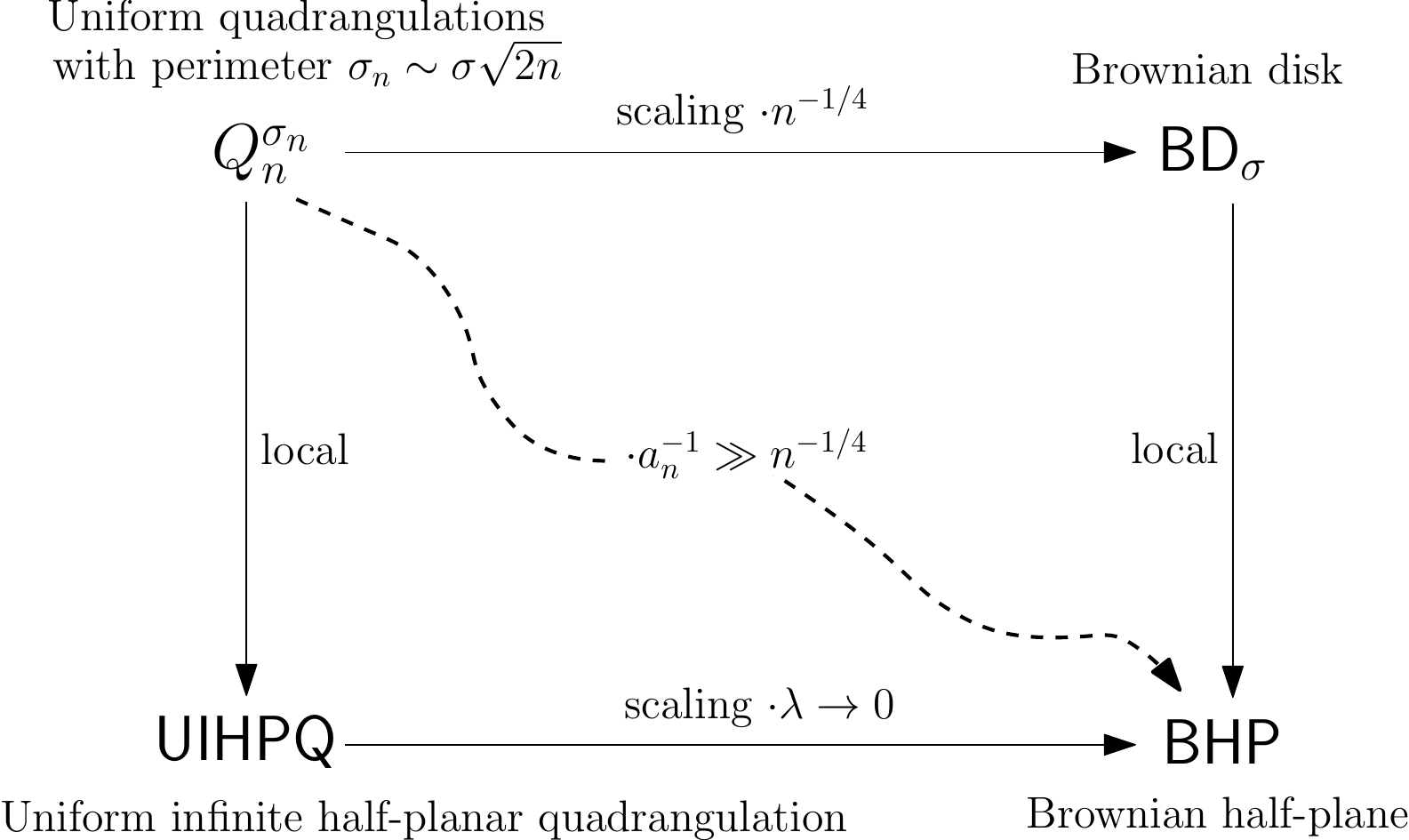}}
 \parbox{6.5cm}{
\caption{Illustration of~\cite[Theorem 1]{BeMi} for the regime
  $\sigma_n\sim\sigma\sqrt{2n}$, Theorem~\ref{thm:BHP1},
  Theorem~\ref{thm:UIHPQ-BHP}, Corollary~\ref{cor:BD2} in the case
  $\sigma(T)\equiv \sigma\in(0,\infty)$, and
  Proposition~\ref{prop:Qn-UIHPQ}. Compare with~\cite[Figure 1]{CuLG}.}
\label{fig:diagram1}}
\end{figure}

\subsection{Limits of the Brownian disk}
\label{sec:results-BDlimits}
Our statements from the last two sections imply various limit results for
the Brownian disk $\BD_{T,\sigma(T)}$ when zooming-in around its root. We let
$\sigma(\cdot):(0,\infty)\rightarrow(0,\infty)$ be a function of the volume
of the Brownian disk that specifies its perimeter. All of the following
convergences hold in distribution with respect to the local
Gromov-Hausdorff topology when the volume $T$ of the disk tends to infinity.

\begin{corollary}
\label{cor:BD1}
Assume $\lim_{T\rightarrow\infty}\sigma(T)=
0$. Then
$$
\BD_{T,\sigma(T)}\longrightarrow\BP.
$$
\end{corollary}
\begin{corollary}
\label{cor:BD4}
Assume $\lim_{T\rightarrow\infty}\sigma(T)=
\varsigma\in(0,\infty)$. Then
$$
\BD_{T,\sigma(T)}\longrightarrow\IBD_\varsigma.
$$
\end{corollary}
\begin{corollary}
\label{cor:BD2}
Assume $\sigma(T)\rightarrow\infty$ and
$\sigma(T)/T\rightarrow\theta\in[0,\infty)$ as $T\rightarrow\infty$. Then
$$
\BD_{T,\sigma(T)}\longrightarrow\BHP_{\theta}.
$$
\end{corollary}
\begin{corollary}
\label{cor:BD3}
Assume $\sigma(T)/T\rightarrow\infty$ as
$T\rightarrow\infty$. Then
$$
\BD_{T,\sigma(T)}\longrightarrow\ICRT.
$$
\end{corollary}
Note that Corollary~\ref{cor:BD2} includes the case where
$\sigma(T)=\sqrt{T}$. Then $\theta=0$, and since by scaling, $T^{1/4}\cdot
\BD_1=_d \BD_{T,\sqrt{T}}$, it follows that $\BHP$ is the tangent cone in
distribution of any disk $\BD_{A,L}$ for fixed $A,L>0$. See~\cite[Section
8.2]{BuBuIv} for an explanation of this terminology in the context of
boundedly compact length spaces, and compare with~\cite[Theorem 1]{CuLG},
where it is shown that the Brownian plane is the tangent cone of the
Brownian map at its root.

For completeness, but without going into details, let us mention that 
identically to the proof of Corollary~\ref{cor:BD1} (or Corollary~\ref{cor:BD3}), a combination
of~\cite[Theorem 1]{BeMi} and~\cite[Theorem 4]{Be3} (or ~\cite[Theorem
4]{Be3}) leads to the convergences
$$
\BD_{T,\sigma}\xrightarrow[]{\sigma\to
  0}\BM_T,\quad\quad \BD_{T,\sigma}\xrightarrow[]{T\to
  0}\CRT_{3\sigma}
$$
in law in the sense of the {\it global} Gromov-Hausdorff topology. The
factor $3$ in $\CRT_{3\sigma}$ stems from the particular normalization of
the Brownian disk.
\begin{remex}
\label{rem:exercises}
  We leave it as an exercise to the reader to find the right combination of
  our (or Bettinelli's, cf.~\cite{Be3}) foregoing
  results to deduce the following additional results on
  tangent cones (in distribution, with respect to the local Gromov-Hausdorff topology):
$$\CRT_T\xrightarrow[]{T\rightarrow\infty} \ICRT,\quad\quad \BHP_\theta\xrightarrow[]{\theta
\to 0} \BHP,\quad\quad 
\IBD_\sigma\xrightarrow[]{\sigma\rightarrow\infty} \BHP.$$ Combining
results from the regime $\sigma_n\ll \sqrt{n}$ in the first and from
$\sigma_n\gg \sqrt{n}$ in the second case, one may also prove the following
scaling results in law:
$$ \BHP_\theta\xrightarrow[]{\theta\rightarrow\infty} \ICRT\quad,\quad\quad 
 \IBD_\sigma\xrightarrow[]{\sigma\rightarrow 0} \BP.$$
\end{remex}

\begin{figure}[!h]
\centering\parbox{8.7cm}{\includegraphics[width=0.45\textwidth]{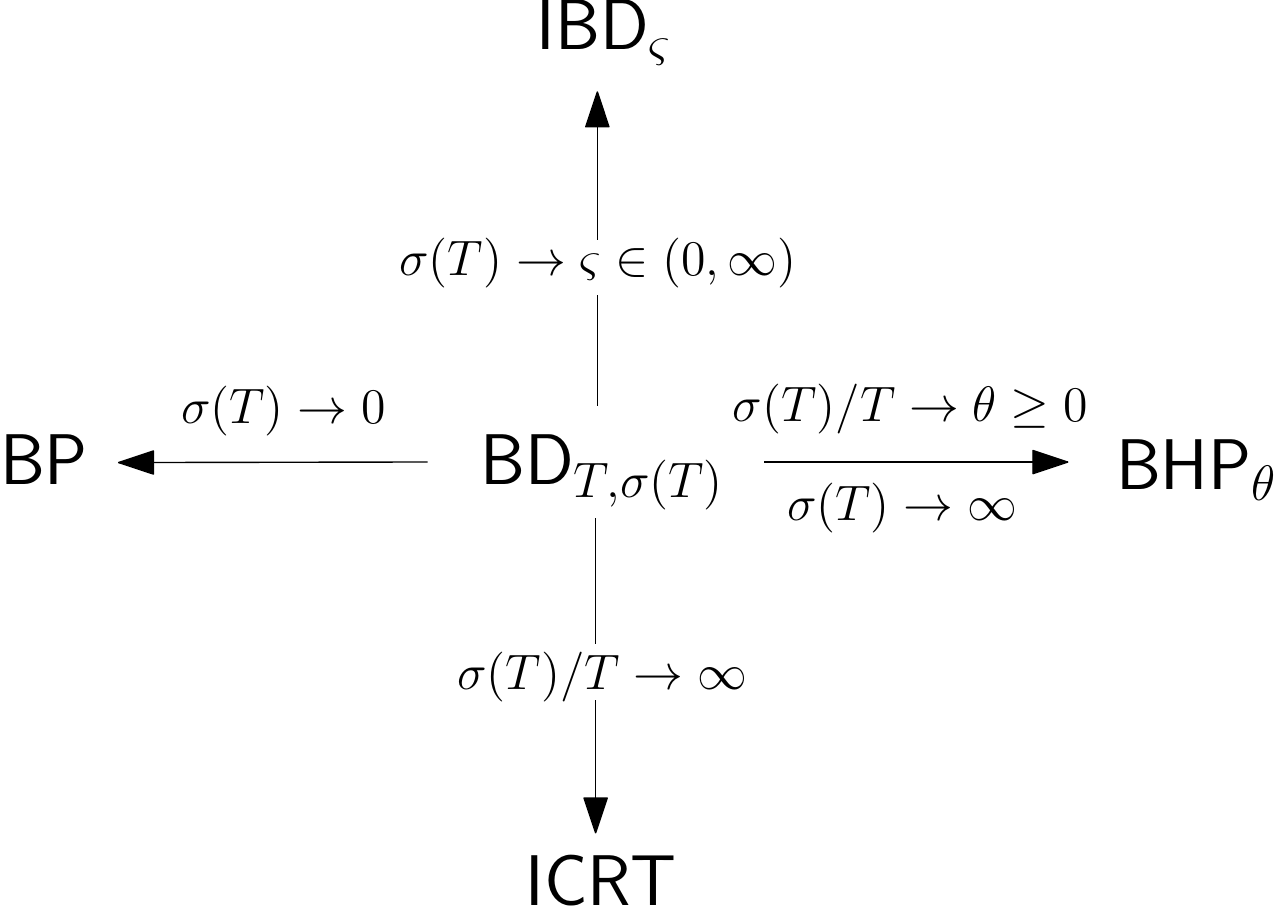}}
 \parbox{7.3cm}{
   \caption{Zooming-in around the root of the Brownian disk
     $\BD_{T,\sigma(T)}$ of volume $T$ and perimeter $\sigma(T)$. The
     figure shows all possible weak limits in the local Gromov-Hausdorff
     sense when $T\rightarrow\infty$
     (Corollaries~\ref{cor:BD1},~\ref{cor:BD4},~\ref{cor:BD2},~\ref{cor:BD3}).}
   \label{fig:diagram2}}
  \end{figure}

\section{Encoding of quadrangulations with a boundary}
\label{sec:encoding}
We will use a variant of the Cori-Vauquelin-Schaeffer~\cite{CoVa,Sc}
bijection developed by Bouttier, Di Francesco and
Guitter~\cite{BoDFGu} to encode quadrangulations with a boundary. More
specifically, we will encode planar quadrangulation of size $n$ with a
boundary of size $2\sigma$ in terms of $\sigma$ trees with $n$ edges in
total, which are attached to a discrete bridge of length $\sigma$. We first
introduce the encoding objects. Our notation is inspired by~\cite{Be2,Be3}.
\subsection{Encoding in the finite case}
\label{sec:encoding-finite}
\subsubsection{Well-labeled tree, forest and bridge}
\label{sec:welllabeledforest}
A {\it well-labeled tree} $(\tau,(\ell(u))_{u\in\tau})$ of size $|\tau|=n$
consists of a rooted plane tree $\tau$ with $n$ edges together with integer
labels $(\ell(u))_{u\in V(\tau)}$ attached to the vertices of $\tau$, such
that the root has label $0$, and $|\ell(u)-\ell(v)| \leq 1$ whenever $u$
and $v$ are neighbors. 

A {\it well-labeled forest} with $\sigma$ trees and $n$ tree edges is a
collection $\f=(\tau_0,\ldots,\tau_{\sigma-1})$ of $\sigma$
trees with $n$ edges in total, together with a labeling of vertices $\la
:\cup_{i=0}^{\sigma-1}V(\tau_i)\rightarrow\mathbb{Z}$, which has the
property that for each $i=0,\ldots,\sigma-1$, the tree $\tau_i$ together
with the restriction $\la\restriction V(\tau_i)$ forms a well-labeled tree.

The vertex set of $\f$ is $V(\f)=\cup_{i=0}^{\sigma-1}V(\tau_i)$. Note that
$|V(\f)|=n+\sigma$. The size of $\f$ is given by $|\f|= n$, i.e., its
number of edges. We write $(0),\ldots,(\sigma-1)$ for the root vertices of
$\tau_0,\ldots,\tau_{\sigma-1}$. If $u$ is a vertex of a tree of $\f$,
$\an(u)$ denotes the root of this tree. In particular, the vertex set of
the $j$th tree of $\f$ is the set $\{u\in V(\f):\an(u)=(j-1)\}$,
$j=1,\ldots,\sigma$.  We write $t(\f)=\sigma$ for the number of trees of
$\f$. We will often identify the root vertices with the integers
$0,\ldots,\sigma-1$ and consequently regard $\an(u)$ as a number.

We call the pair $(\f,\la)$ a {\it well-labeled forest} and denote by
$$\Fo_\sigma^n = \{(\f,\la): t(\f) = \sigma, |\f| = n\}$$
the set of all well-labeled forests of size $n$ with $\sigma$ trees.

A {\it bridge} of length $\sigma\geq 1$ is a sequence of numbers
$(\br(0),\br(1), \ldots ,\br(\sigma))$ with $\br(0)=0$ and such that $\br(i+1) - \br(i) \in
\N_0\cup\{-1\}$, and $\br(\sigma) \le 0$.

By linear interpolation between integer
values, we will view $\br: [0,\sigma]\to\R$ as a
continuous function and write $\Br_\sigma\subset\mathcal{C}([0,\sigma],\R)$
for the set of all possible bridges of
length $\sigma$. 

The terminal value $\br(\sigma)$ of a bridge has a special
interpretation: It keeps the information where to find the root in the
quadrangulation associated to a triplet
$((\f,\la),\br)\in\Fo_\sigma^n\times\Br_\sigma$, see
Section~\ref{sec:BDG-bijection} below.

\subsubsection{Contour pair and label function}
\label{sec:contourlabel-finite}
Consider a well-labeled forest $(\f,\la)$ of size $n$ with $\sigma$ trees. In order to define
its contour pair and label function, it is convenient to associate to $\f$
a representation in the plane, as depicted in Figure~\ref{fig:finiteforest}: We
add $\sigma-1$ edges which link the root vertices $(0),\ldots,(\sigma-1)$,
such that vertex $(i-1)$ gets connected to $(i)$ for $i=1,\ldots,\sigma-1$,
plus an extra vertex $(\sigma)$ and an extra edge linking $(\sigma-1)$ to
$(\sigma)$. We extend $\la$ to $(\sigma)$ by setting $\la((\sigma))=0$.
We refer to the segment connecting the roots of $\f$ and the extra vertex
$(\sigma)$ as the {\it floor} of $\f$.

\begin{figure}[ht]
\begin{center}
\begin{minipage}{1\linewidth}
\parbox{8.5cm}{\center\includegraphics[width=0.44\textwidth]{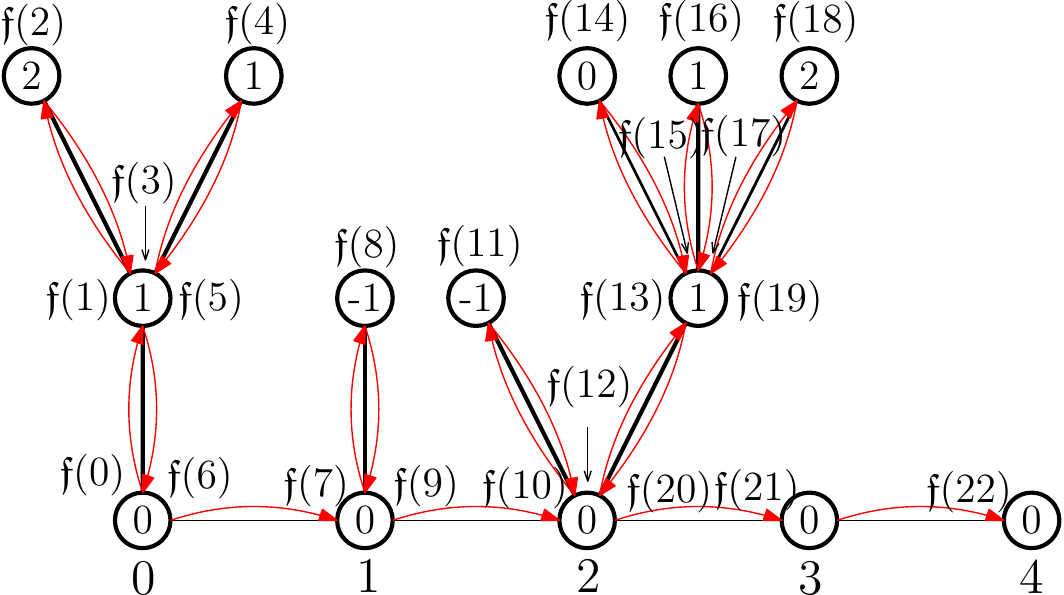}}
\parbox{5cm}{\includegraphics[width=0.44\textwidth]{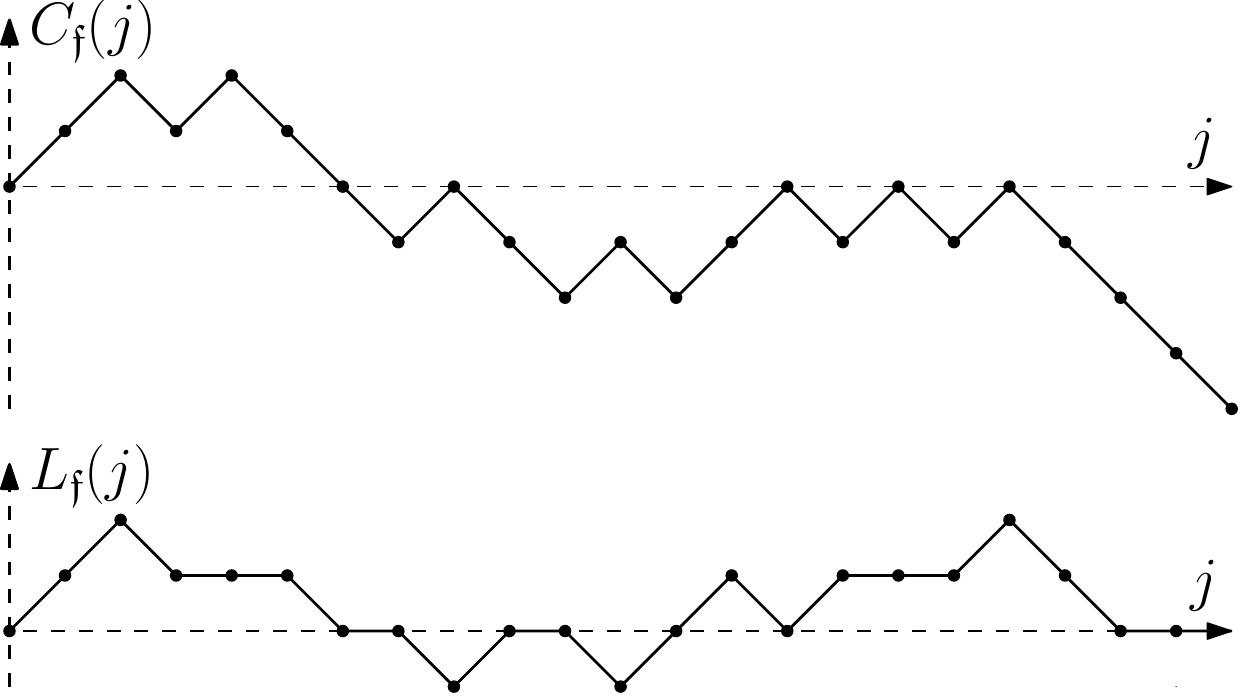}}
\end{minipage}
\end{center}
\caption{On the left: A proper representation of a finite forest $\f$ of
  size $13$ with $4$ trees, together with its facial sequence. The
  rightmost vertex labeled $4$ is the added extra vertex. On the right: Its
  contour pair.}
  \label{fig:finiteforest}
\end{figure}

The {\it facial sequence} 
$\f(0),\ldots,\f(2n+\sigma)$ of $\f$ is the sequence of vertices
obtained from exploring (the embedding of) $\f$ in the contour order, 
starting from vertex $(0)$. In other words,
$\f(0),\ldots,\f(2n+\sigma-1)$ is given by the sequence of vertices of the
discrete contour paths of the trees
$\tau_0,\ldots,\tau_{\sigma-1}$, and the sequence terminates
with value $\f(2n+\sigma)=(\sigma)$. See, e.g.,~\cite[Section 2]{LGMi} for
more on contour paths.

Given a well-labeled forest $(\f,\la)$, we define its {\it
  contour pair} $(C_\f,L_\f)$  by 
$$
C_\f(j) =d_{\f}(\f(j),(\sigma))-\sigma, \quad L_\f(j) = \la(\f(j)),\quad j=0,\ldots,2n+\sigma.
$$
Here, $d_{\f}$ denotes the graph
distance on the representation of $\f$ in the plane.

We call $C_\f$ the {\it contour function} of $\f$, since it is obtained from concatenating the contour
paths of the trees $\tau_0,\ldots,\tau_{\sigma-1}$, with an additional $-1$
step after a tree has been visited. Note that $L_\f(\f(j))=0$ if $\f(j)$
lies on the floor of $\f$.  See again Figure~\ref{fig:finiteforest} for
an illustration.

Now consider additionally a bridge $\br\in\Br_\sigma$. Put 
$\uC_{\f}(j)=\min_{[0,j]}C_{\f}$. The function
$$
\La_{\f}(j) = L_{\f}(j) + \br(-\uC_{\f}(j)),\quad
j=0,\dots,2n+\sigma,
$$
is called the {\it label function} associated to
$((\f,\la),\br)$. The label function plays an important role in measuring distances in the
quadrangulation associated through the Bouttier-Di Francesco-Guitter
bijection, see Section~\ref{sec:distances}.

By linear interpolation between integers, we extend all three functions
$C_\f$, $L_\f$ and $\La_\f$ to continuous real-valued functions on
$[0,2n+\sigma]$. 

\subsection{Encoding in the infinite case}
\label{sec:encoding-infinite}
We next introduce the infinite analogs of the objects from the previous
section. They will encode certain infinite
quadrangulations with an infinite boundary.
\subsubsection{Well-labeled infinite forest and infinite bridge}
A {\it well-labeled infinite forest} is an infinite
collection $\f=(\tau_{i},i\in\Z)$ of finite rooted plane
trees, together with a labeling of vertices $\la
:\cup_{i\in\Z}V(\tau_i)\rightarrow\mathbb{Z}$ such that for each $i\in\Z$, $\tau_i$ together
with the restriction of $\la$ to $V(\tau_i)$ forms a well-labeled tree.

We write again $(k)$ for the root vertex of $\tau_k$ and often identify
$(k)$ with $k\in\Z$.  We call the pair $(\f,\la)$ an {\it well-labeled
  infinite forest} and denote by $\Fo_\infty$ the set of all well-labeled
infinite forests.

An {\it infinite bridge} is a sequence of numbers
$\br=(\br(i),i\in\Z\cup\{\partial\})$ with $\br(0)=0$, $\br(i+1)- \br(i)
\in \N_0\cup\{-1\}$ for all $i\in\Z$ and
$\br(\partial)\in\{\br(-1)-1,\ldots,0\}$.

The extra value $\br(\partial)$ will keep track of the position of the root
in the quadrangulation. Often, we consider only the values $\br(i)$,
$i\in\Z$, and then view $\br$ as a continuous
function from $\R$ to $\R$, by linear interpolation between integer
values. We write $\Br_\infty$ for the set of all infinite bridges $\br$ which
have the property that $\inf_{i\in \N}\br(i)=-\infty$, and $\inf_{i\in \N}\br(-i)=-\infty$.

\subsubsection{Contour pair and label function in the infinite case}
\label{sec:contourlabel-infinite}
We consider a well-labeled infinite forest $(\f,\la)\in\Fo_\infty$. Again, we view
$\f$ as a graph properly embedded in the plane (Figure~\ref{fig:infiniteforest}): We identify
the set of roots of the trees of $\f$ with $\Z$ and connect neighboring
roots by an edge. We obtain what we call the {\it floor} of $\f$. The trees
$\tau_i$ of $\f$ are drawn in the upper half-plane and attached to
the floor.

The {\it facial sequence} $(\f(i),i\in\Z)$ of $\f$ is defined as follows:
$(\f(0),\f(1),\ldots)$ is the sequence of vertices of the contour
paths of the trees $\tau_i, i\in\N_0$, in the contour order, starting from
the root of the tree $\tau_0$, and 
$(\f(-1),\f(-2),\ldots)$ is given by the sequence of vertices of the
contour paths $\tau_{-1},\tau_{-2},\ldots$, in the {\it counterclockwise}
order, starting from the root of the tree $\tau_{-1}$.

\begin{figure}[ht]
\begin{center}
\begin{minipage}{1\linewidth}
\parbox{8.4cm}{\center\includegraphics[width=0.44\textwidth]{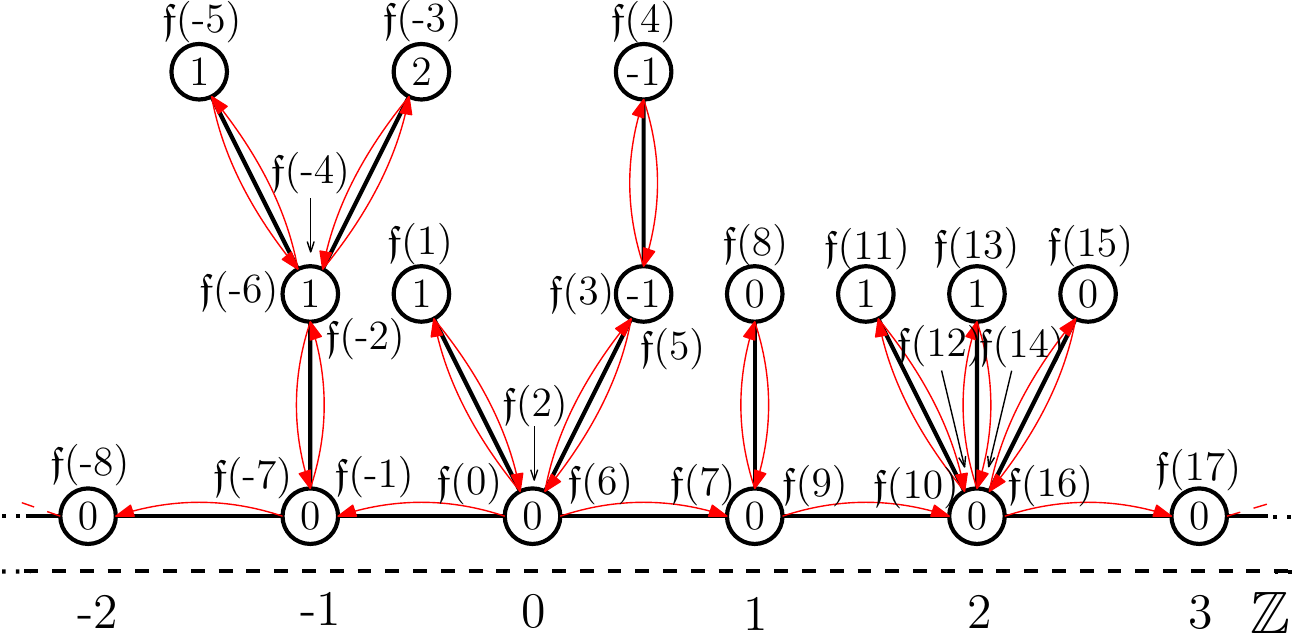}}
\parbox{5cm}{\includegraphics[width=0.44\textwidth]{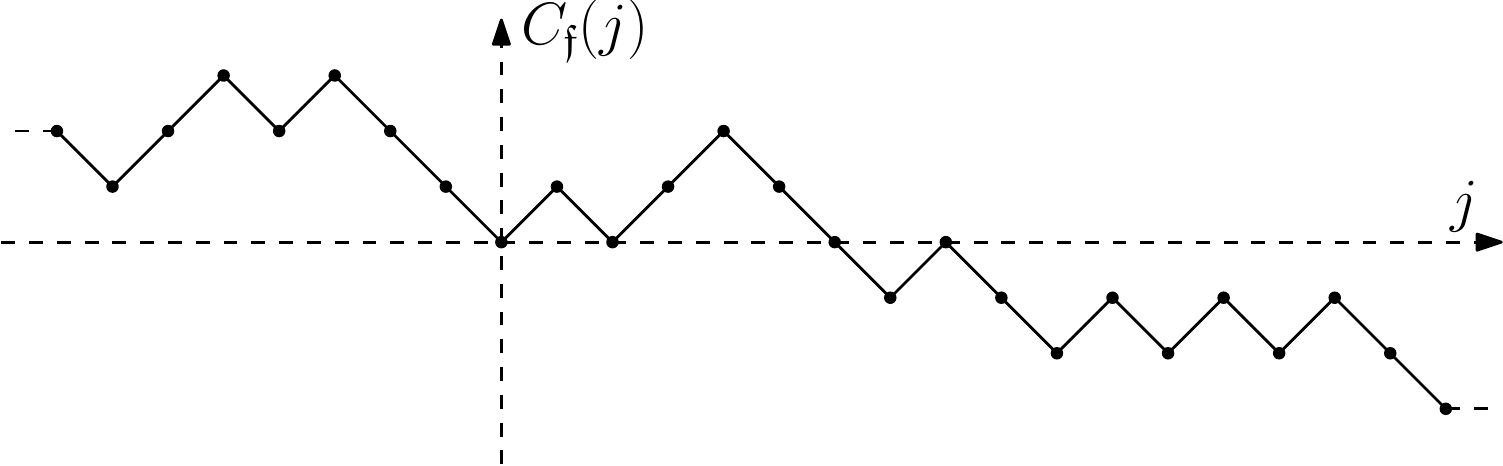}}
\end{minipage}
\end{center}
\caption{On the left: A proper representation of an infinite
  forest $\f$, together with its facial
  sequence. On the right: Its contour function.}
  \label{fig:infiniteforest}
\end{figure}

In analogy to the finite case, given a well-labeled infinite tree $(\f,\la)$, its {\it
  contour pair} $(C_{\f},L_{\f})$ is a tuple functions defined {\it via}  
$$
C_{\f}(j) =d_{\f}(\f(j),\an(\f(j)))-\an(\f(j)), \quad L_{\f}(j) =
\la(\f(j)),\quad j\in\Z,
$$
where $d_{\f}$ is the graph distance on the embedding of $\f$, and
$\an(\f(j))$ denotes the root of the tree $\f(j)$ belongs to. Be aware of
the small abuse of notation: In the expression for $C_{\f}$, $\an(\f(j))$ is first
viewed as a vertex and then as an integer.

Note that $\lim_{j\rightarrow\infty}C_\f(j)\rightarrow -\infty$ and
$\lim_{j\rightarrow -\infty}C_\f(j)\rightarrow +\infty$ for every infinite
forest. As for a finite forest, we call $C_{\f}$ the {\it contour function}
of $\f$.

If additionally $\br\in \Br_\infty$, we define the {\it label function} associated to
$((\f,\la),\br)$ by
$$\La_\f(j)=
L_\f(j) + \br(\uC_\f(j)),\quad j\in\Z,\quad
\La_\f(\partial)=\br(\partial),$$ where $\uC_\f(j)=\inf_{(-\infty,j]}C_\f$
for $j<0$ and $\uC_\f(j)=\min_{[0,j]}C_\f$
for $j\geq 0$, as above.

Again by linear interpolation between integers, we view $C_\f,L_\f$ and
$\La_\f$ as continuous functions on $\R$.

\subsection{Bouttier-Di Francesco-Guitter bijection}
\label{sec:BDG-bijection}
Recall that a rooted quadrangulation with a
boundary comes with a distinguished edge along the boundary, the root edge,
whose origin is the root vertex.  We write $\cQ_n^{\sigma}$ for the set of
all rooted quadrangulations with $n$ inner faces and a boundary of size
$2\sigma$.

A {\it pointed quadrangulation with a boundary} is a pair $(\q,\vd)$, where
$\q$ is a rooted quadrangulation with a boundary and $\vd\in V(\q)$ is a
distinguished vertex. The set of all rooted pointed quadrangulations with
$n$ internal faces and $2\sigma$ boundary edges is denoted by
$$\cQ_{n,\sigma}^\bullet=\left\{(\q,\vd) :
  \q\in\cQ_n^{\sigma}, \vd\in V(\q)\right\}.$$

\subsubsection{The finite case}
The Bouttier-Di
Francesco-Guitter bijection~\cite{BoDFGu} provides us with a
bijection 
$$\Phi_n:\Fo_\sigma^n \times \Br_\sigma\longrightarrow \cQ_{n,\sigma}^\bullet.$$ We shall here
content ourselves with the description of the mapping from the encoding
objects to the quadrangulations. We follow largely the presentation
in~\cite{Be3}, where also a description of the reverse direction can be found.

In this regard, let $((\f,\br),\la)\in\Fo_\sigma^n\times\Br_{\sigma}$. Out
of this triplet, we will now construct a rooted pointed quadrangulation
$(\q,\vd)\in\cQ_{n,\sigma}^\bullet$. Recall the facial sequence $\f(0),\ldots,\f(2n+\sigma)$ of $\f$ obtained
from exploring the trees of $\f$ in the contour order, as well as the
associated label function $\La_\f$. We view $\f$ as embedded in the plane
(as explained above) and add an additional vertex $\vd$ inside the only
face of $\f$, with label $\La_\f(\vd)=-\infty$. 

The vertex set of $\q$ is given by $V(\f)\cup\{\vd\}$. Note that by
definition, the additional vertex $(\sigma)$ which forms part of the
embedding of $\f$ is not an element of $V(\f)$.  In order to specify the
edges between the vertices of $\q$, we define for $i=0,\ldots,2n+\sigma-1$
the {\it successor} $\suc(i)\in\{0,\ldots,2n+\sigma-1\}\cup\{\infty\}$ of
$i$ to be the first number $k$ in the list
$(i+1,\ldots,2n+\sigma-1,0,\ldots,i-1)$ with the property that
$\La_{\f}(k)=\La_{\f}(i)-1$, with $\suc(i)=\infty$ if there is no such
number. Letting $\f(\infty)=\vd$, we now follow the facial sequence of $\f$
and draw for every $i=0,\ldots,2n+\sigma-1$ an arc between $\f(i)$ and
$\f(\suc(i))$, in such a way that it neither crosses arcs that were
previously drawn, nor edges of the embedding of $\f$. Since any vertex of
$\f$ which is not a leaf is visited at least twice in the contour
exploration, there can be several arcs connecting $\f(i)$ and
$\f(\suc(i))$. By a small abuse of language, we therefore speak of the arc
connecting $i$ to $\suc(i)$ and write
$$
i\arcr\suc(i)\quad\hbox{or}\quad i\arcl\suc(i)
$$
for the oriented arc from $i$ towards $\suc(i)$ or from $\suc(i)$ towards $i$, respectively.
\begin{figure}[ht]
\begin{center}
\begin{minipage}{1\linewidth}
  \center\parbox{7.7cm}{\includegraphics[width=0.45\textwidth]{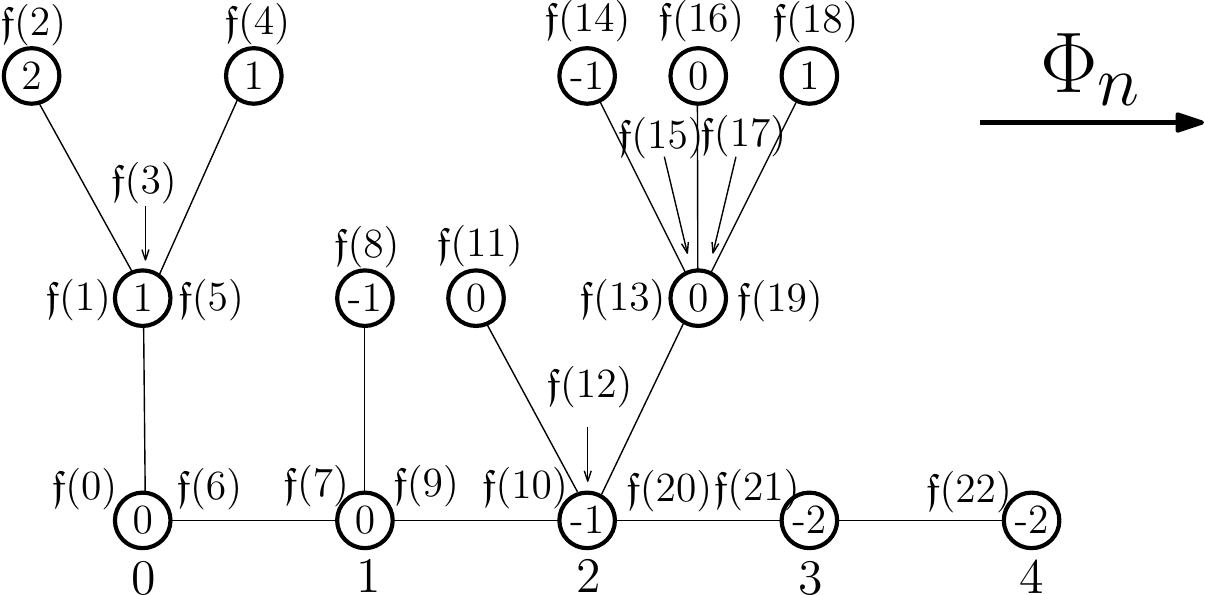}}
\parbox{6.5cm}{\includegraphics[width=0.45\textwidth]{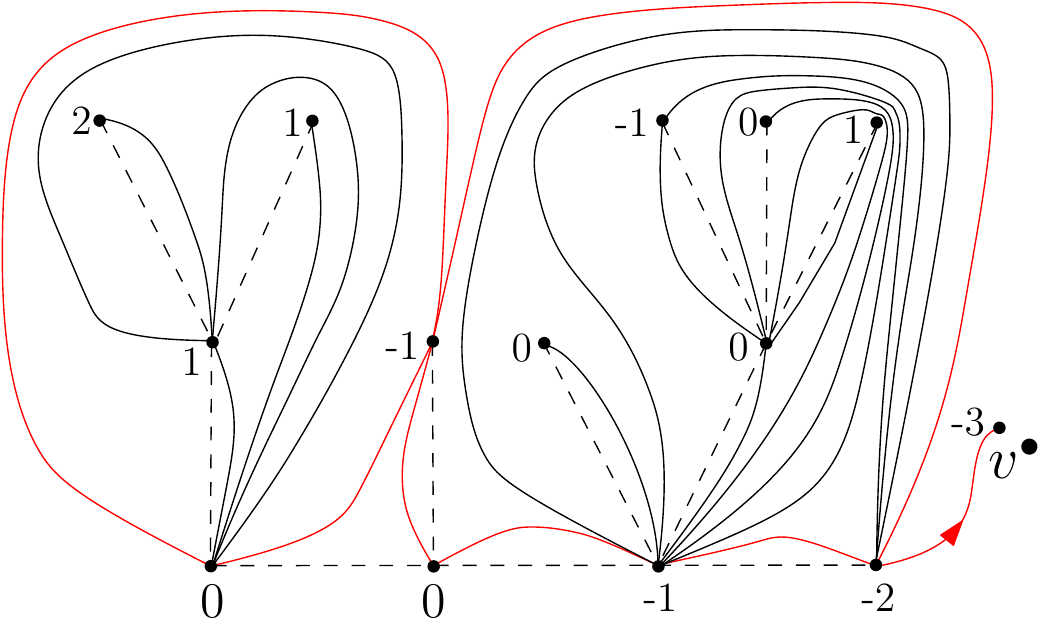}}
\end{minipage}
\end{center}
\caption{The Bouttier-Di Francesco-Guitter bijection $\Phi_n$ applied to an
  element $((\f,\la),\br)\in\Fo_{\sigma}^n\times\Br_\sigma$. The forest
  $\f$ is the same as in Figure~\ref{fig:finiteforest}, but the labels are
  shifted by the values of the bridge $\br$. The (non-simple) boundary of
  the associated quadrangulation on the left is represented in red. Note
  that the extra vertex $\vd$ is in this example a boundary vertex. The
  rightmost vertex $\f(22)=(4)$ on the left is not a vertex of the
  quadrangulation. Its label $-2$ captures the information where to find
  the root edge, which is indicated by an arrow.}
\label{fig:BDFG}
\end{figure}
The arcs between the vertices $V(\f)\cup\{\vd\}$ form the edges of
$\q$, and it remains to specify the root edge of $\q$: The root vertex is
given by $\f(\suc^{-\br(\sigma)}(0))$, and the root edge is in case
$\br(\sigma)>\br(\sigma-1)-1$ given by
$\suc^{-\br(\sigma)}(0)\arcr\suc^{-\br(\sigma)+1}(0)$, and in case
$\br(\sigma)=\br(\sigma-1)-1$ by $2n+\sigma-1\arcl\suc(2n+\sigma-1)$.
Note that in the second case, we have indeed
$\f(\suc(2n+\sigma-1))=\f(\suc^{-\br(\sigma)}(0))$,
i.e., $\suc(2n+\sigma-1)$ is the root vertex.

\subsubsection{The infinite case}
Let $\cQ$ denote the completion of the space of all rooted finite
quadrangulations with a boundary with respect to $\dmap$. We extend
$\Phi_n$ to a mapping
$$\Phi:\left({\cup}_{n,\sigma\in\N}\Fo_{\sigma}^n\times\Br_\sigma\right)\cup\left(\Fo_\infty\times\Br_\infty\right)\longrightarrow \cQ$$
as follows. For elements $((\f,\la),\br)\in
\Fo_{\sigma}^n\times\Br_\sigma$, we let
$\Phi(((\f,\la),\br))=\Phi_n((\f,\la),\br)$, where we view the latter as an
element in $\cQ_n^{\sigma_n}$, by simply forgetting its distinguished
vertex.

Now let $((\f,\la),\br)\in\Fo_\infty\times\Br_\infty$. For $i\in\Z$, we
define the {\it successor} $\suc_\infty(i)$ to be the smallest number $k$
greater than $i$ such that $\La_{\f}(k)=\La_{\f}(i)-1$. Note that since
$\inf_{i\in\N}\br(i)=-\infty$, the definition make sense. We consider a
proper embedding of $\f$ in the plane as described above and draw an arc
between $\f(i)$ and $\f(\suc_\infty(i))$, for any $i\in\Z$, as indicated by
Figure~\ref{fig:uihpq}. Again we can do this in a way such that arcs do not
cross. The vertex set of $\Phi(((\f,\la),\br))$ is given by $V(\f)$,
and the edges are the arcs we constructed. Finally, we follow a rooting
convention which is analogous to the finite case (we adapt the notion
$i\arcr\suc_\infty(i)$ in the obvious way): The root vertex is given by
$\f(\suc_\infty^{-\br(\partial)}(0))$, and the root edge is in case
$\br(\partial)>\br(-1)-1$ given by
$\suc_\infty^{-\br(\partial)}(0)\arcr\suc^{-\br(\partial)+1}(0)$, and in
case $\br(\partial)=\br(-1)-1$ by $-1\arcl\suc_\infty(-1)$.
\begin{figure}[ht]
  \centering
  \includegraphics[width=0.8\textwidth]{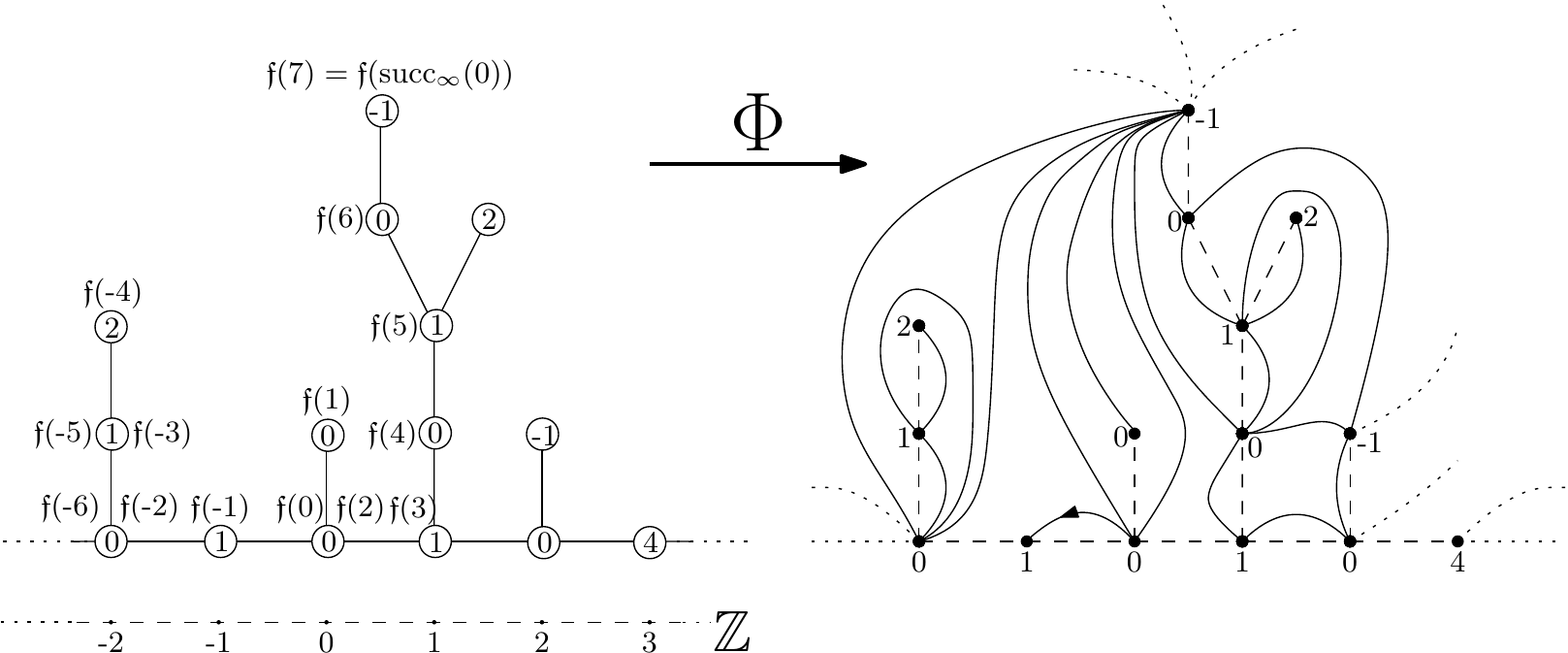}
  \caption{The Bouttier-Di Francesco-Guitter mapping applied to an element
    $((\f,\la),\br)\in\Fo_\infty\times\Br_\infty$. The successor of $0$ is
    $7$, which is also the successor of $-6,-2,1,2,4,6$. The vertex labels
    are given by $\La_{\f}$, as in Figure~\ref{fig:BDFG}. The root edge of
    the map is the oriented arc $-1\arcl\suc_\infty(-1)$ indicated by an
    arrow.}
  \label{fig:uihpq}
\end{figure}
\begin{remark}
  Notice that a triplet $((\f,\br),\la)$ in
  $\Fo_\sigma^n\times\Br_{\sigma}$ or in $\Fo_\infty\times\Br_{\infty}$ is
  uniquely determined by its associated contour and label functions
  $(C_\f,\La_{\f})$. In particular, it makes sense to speak of the
  quadrangulation associated to $(C_\f,\La_{\f})$.  The distinguished
  vertex $\vd$ in the finite case will play no particular role in our
  statements, since we view quadrangulations as metric spaces pointed at
  their root vertices.
\end{remark}

\subsection{Construction of the \normalfont{\UIHPQ}}
\label{sec:constr-UIHPQ}
We first introduce a $\Fo_\infty$-valued random element
$(\f_\infty,\la_\infty)$ together with a $\Br_\infty$-valued random element
$\br_\infty$, which will encode the $\UIHPQ$. 

\subsubsection{Uniformly labeled critical infinite forest}
 Let $\tau$ be a finite random plane tree. Conditionally on $\tau$,
 we assign a sequence of i.i.d random variables with the uniform
 distribution on $\{-1,0,1\}$ to the edges of $\tau$. The label
 $\ell(u)$ of a vertex $u$ of $\tau$ is defined to be the sum of the random
 variables along the edges of the (unique) path from the root to $u$. Such
 a random labeling $\ell:V(\tau)\rightarrow\mathbb{Z}$ is referred to as a
 {\it uniform labeling}. If the tree $\tau$ is a Galton-Watson tree with a 
 geometric offspring distribution of parameter $1/2$, we say that $\tau$ is
 a {\it critical geometric Galton-Watson tree}.  If $\ell$ is a uniform labeling of
 $\tau$, we refer to the pair $(\tau,(\ell(u))_{u\in\tau})$ as a {\it
   uniformly labeled critical geometric Galton-Watson tree}.

 A {\it uniformly labeled critical infinite forest} is a random element 
  $(\f_\infty,\la_\infty)$ taking values in $\Fo_\infty$ such that the pairs
  $(\tau_i,\la_\infty\restriction V(\tau_i))$, $i\in\Z$, are
 independent uniformly labeled critical geometric Galton-Watson trees.
 
 \subsubsection{Uniform infinite bridge}
 \label{sec:infinitebridge}
 Let $\br_\infty=(\br_\infty(i),i\in\mathbb{Z})$ be a two-sided random walk
 starting from $0$ at time $0$, i.e., $\br_\infty(0)=0$, which has independent increments given by
 $$
 \P(\br_\infty(i)-\br_\infty(i-1)=k)=2^{-k-2},\quad k\in\mathbb{N}_0\cup\{-1\},\quad\hbox{for
 }i\in\mathbb{Z}\setminus\{0\},
 $$
 and 
 $$
 \P(-\br_\infty(-1)=k)=(k+2)2^{-(k+3)},\quad k\in\mathbb{N}_0\cup\{-1\}.
 $$
 Note that $-\br_\infty(-1)$ has same law as $G+G'-1$ for $G$ and $G'$ two independent geometric random variables of parameter
 $1/2$. This follows from the well-known fact that $G+G'+1$ is distributed
 as a size-biased geometric random variable. We refer to
 Section~\ref{sec:bridges} for more explanations. Next, given $\br_\infty(-1)$, we let $\br_\infty(\partial)$ be a
 uniformly distributed random variable in $\{\br_\infty(-1)-1,\ldots,0\}$,
 independent of everything else.
 
 We call the random element $\br_\infty=(\br_\infty(i),
 i\in\Z\cup\{\partial\})$ with values in $\Br_\infty$ the {\it uniform
   infinite bridge}.

 We review now the construction of the $\UIHPQ$ given in~\cite{CuMi}. Note
 that there, the encoding is defined in a slightly different (but
 equivalent) manner, and the root edge is oriented in the opposite
 direction.  The following definition is justified by
 Proposition~\ref{prop:Qn-UIHPQ}.
 \begin{defn}
   Let $(\f_\infty,\la_\infty)$ be a uniformly labeled critical infinite
   forest, and let $\br_\infty$ be a uniform infinite bridge independent of
   $(\f_\infty,\la_\infty)$. The uniform infinite half-planar
   quadrangulation $\UIHPQ$ is the (rooted) random infinite quadrangulation
   $Q_{\infty}^{\infty}=(V(Q_{\infty}^{\infty}),\dgr,\rho)$ with an
   infinite boundary obtained from applying the Bouttier-Di
   Francesco-Guitter mapping $\Phi$ to
   $((\f_\infty,\la_\infty),\br_\infty)$.
\end{defn}

In~\cite{CuMi}, it was shown that in the sense of $\dmap$, there are the
weak convergences
$$
Q_n^{\sigma_n}\xrightarrow[]{n \to \infty}Q_{\infty}^{\sigma},\quad
Q_{\infty}^{\sigma}\xrightarrow[]{\sigma \to
\infty}Q_{\infty}^{\infty},$$ where $Q_{\infty}^{\sigma}$ is the
so-called (rooted) uniform infinite planar quadrangulation with a boundary
of perimeter $2\sigma$. We also point at the recent work~\cite{CaCu}, where
a construction of the $\UIHPQ$ with a positivity constraint on labels is
given, similarly to the Chassaing-Durhuus construction~\cite{ChDu} of the
$\UIPQ$.

 \begin{remark}
   We stress that while we use the notation $(\f,\la)$ for both a finite or
   infinite (deterministic) well-labeled forest, and similarly, $\br$
   represents a finite or infinite bridge,
   $(\f_\infty,\la_\infty)\in\Fo_\infty$ and $\br_\infty\in\Br_\infty$ will
   always stand for {\it random} elements with the particular law just
   described. We will implicitly assume that $\br_\infty$ is independent of
   $(\f_\infty,\la_\infty)$. Similarly, for given $\sigma_n$,
   $((\f_n,\la_n),\br_n)$ will denote a random element with the uniform
   distribution on $\Fo_n^{\sigma_n}\times\Br_{\sigma_n}$, see
   Section~\ref{sec:usualsetting}.
\end{remark}

\subsection{Some ramifications}
We gather here some consequences and remarks which we will
tacitly use in the following. We begin with some observations concerning the
Bouttier-Di Francesco-Guitter bijection.  
\subsubsection{Distances}
\label{sec:distances}
Let $(\q,\vd)\in\mathcal{Q}_{n,\sigma}^\bullet$ be a (rooted) pointed
quadrangulations of size $n$ with a boundary of size $2\sigma$.  Then
$(\q,\vd)$ corresponds to a pair $((\f,\la),\br)\in\Fo_\sigma^n
\times\Br_\sigma$ {\it via} the Bouttier-Di Francesco-Guitter
bijection, and the sets $V(\q)\setminus\{\vd\}$ and $V(\f)$ are identified
through this bijection. Recall that the label function $\La=\La_f$
represents the labels in the forest shifted tree by tree according to the
values of the bridge $\br$.  By a slight abuse of notation, we will view
$\La$ also as a function on $V(\q)\setminus\{\vd\}$ (or $V(\f)$): If $v\in
V(\q)\setminus\{\vd\}$, there is at least one
$i\in\{0,\ldots,2n+\sigma-1\}$ such that $v$ is visited in the $i$th step
of the contour exploration, and we let $\La(v)=\La(i)$. Note that this
definition makes sense, since $\La(i) = \La(j)$ if $\f(i)=\f(j)$.

Write $\dq$ for the graph distance on $\bq$.  From the description of the
bijection above, we deduce that
\begin{equation}
\label{eq:distance-vdot}
\dq(u,\vd)=\La(u)-\min\La +1.
\end{equation}
Moreover, if $v_0$ is the root vertex of $\q$, we know that its distance to vertex $\f(0)=(0)$ is 
\begin{equation}
\label{eq:distance-root-0}
\dq(v_0,(0))= -\br(\sigma).
\end{equation}
In general, there is no simple formula for distances in $\q$. However, as
we explain next, there exist lower and upper bounds in terms of $\La$. 

We first discuss a lower bound. If $u,v\in V(\f)$ are vertices of the same
tree $\tau$ of $\f$, i.e., $\an(u)=\an(v)$, we let $[[u,v]]$ be the vertex
set of the unique injective path in $\tau$ connecting $u$ to $v$. If $(i)$,
$(j)$ are two tree roots of $\f$ with $i<j$, we let $[[(i),(j)]]$ denote
the sequence of root vertices $(i),(i+1),\ldots,(j)$. For the remaining
cases, if $\an(u)< \an(v)$, we put
$$[[u,v]] = [[u,\an(u)]]\cup [[
\an(u),\an(v)]] \cup[[v,\an(v)]],$$
whereas if $\an(v)<\an(u)$, we let 
$$[[u,v]] = [[u,\an(u)]]\cup [[
\an(u),(\sigma-1)]]\cup [[(0),\an(v)]] \cup
[[v,\an(v)]].$$  

Now let $u,v\in V(\q)\setminus\{\vd\}$. 
The so-called {\it cactus bound}
states that
\begin{equation}
\label{eq:cactus1}
\dq(u,v) \geq \La(u)+\La(v)-2\max\left\{\min_{[[
    u,v]]}\La, \min_{[[ v,u]]}\La\right\}.
\end{equation}
See~\cite[Proposition 2.3.8]{Mi3} for a proof in a slightly different
context, which is readily adapted to our setting. Since vertex $(0)$ has
label $\La(0)=0$ and $\La$ coincides with the values of the bridge along
the floor of $\f$, the distance $\dq((0),u)$ for $u\in
V(\q)\setminus\{\vd\}$ is lower bounded by
\begin{equation}
\label{eq:cactus2}
\dq((0),u) \geq -\max\left\{\min_{[0,\an(u)]}\br, \min_{[\an(u),\sigma-1]}\br\right\}.
\end{equation}

For an upper bound of $\dq(u,v)$ when $u,v\in V(\q)\setminus\{\vd\}$,
choose $i,j\in\{0,\ldots,2n+\sigma-1\}$ such that $\f(i)= u$ and
$\f(j)=v$. Define
$$
\overrightarrow{[i,j]} = 
\left\{\begin{array}{l@{\quad\mbox{if }}l}
      \{i,\ldots, j\} & i\leq j\\
\{i,\ldots,2n+\sigma-1\} \cup \{ 0,\ldots,j\}
& i>j
\end{array}\right..
$$
Then there is the upper bound (see~\cite[Lemma 3]{Mi1} for a proof)
\begin{equation}
\label{eq:dist-upperbound}
\dq(u,v) \leq
\La(u)+\La(v)-2\max\left\{\min_{\overrightarrow{[ i,j]}}\La(\f),\,\min_{\overrightarrow{[
      j,i]}}\La(\f)\right\}+2.
\end{equation}

Bounds similar to~\eqref{eq:cactus1},~\eqref{eq:cactus2}
and~\eqref{eq:dist-upperbound} can be formulated for infinite
quadrangulations $\q_\infty$ constructed from triplets
$((\f,\br),\la)\in\Fo_\infty\times \Br_\infty$. For example, if $u,v\in
V(\f)$ with $\an(u)\leq \an(v)$, the cactus bound~\eqref{eq:cactus1} reads
\begin{equation}
\label{eq:cactus3}
d_{{\bf q}_\infty}(u,v) \geq \La(u)+\La(v)-2\min_{[[
    u,v]]}\La.
\end{equation}

\subsubsection{Bridges}
\label{sec:bridges}
We will need some properties of elements in $\Br_\sigma$.  Firstly, as
it is shown in~\cite[Lemma 6]{Be3}, by identifying a bridge
$(\br(i),0\leq i\leq \sigma)\in\Br_\sigma$ with the sequence
\begin{equation}
\label{eq:correspondence-bridge}
\big(\underbrace{+1,+1,\dots,+1}_{\br(0)-\br(\sigma) \textup{
    times}},-1,\underbrace{+1,+1,\dots,+1}_{\br(1)-\br(0)+1 \textup{
    times}},-1,\underbrace{+1,+1,\dots,+1}_{\br(2)-\br(1)+1 \textup{
    times}},\dots, -1,
\hspace*{-2mm}\underbrace{+1,+1,\dots,+1}_{\br(\sigma)-\br(\sigma-1)+1
  \textup{ times}}\big),
\end{equation}
one obtains a one-to-one correspondence between $\Br_\sigma$ and the set of
sequences in $\{-1,+1\}^{2\sigma}$ counting exactly~$\sigma$ times the
number $-1$. As a consequence, $|\Br_\sigma|={2\sigma \choose \sigma}$.

It is helpful to adopt the following point of view. Imagine that we mark
$\sigma$ points on the discrete circle $\mathbb{Z}/\textup{mod}\, 2\sigma$
uniformly at random. Marked points obtain label $-1$, unmarked points label
$+1$.  Now choose uniformly at random one of the $2\sigma$ circle points as
the origin. By walking around the circle in the clockwise order starting
from the chosen origin, one observes a sequence of consecutive $+1$ and
$-1$, which is distributed as~\eqref{eq:correspondence-bridge} when $\br$
is chosen uniformly at random in $\Br_\sigma$. In particular,
$(\br(\sigma)-\br(\sigma-1)+1)+(\br(0)-\br(\sigma)+1)=- \br(\sigma-1)+2$
has the law of a size-biased pick among all $\sigma$ consecutive segments
of the form $(+1,+1,\ldots,+1,-1)$. When $\sigma$ tends to infinity, it is
readily seen that $-\br(\sigma-1)$ converges in distribution to $G+G'-1$,
where $G$ and $G'$ are two independent geometric random variables of
parameter $1/2$. This explains the particular law of the increment
$-\br_\infty(-1)$ of a uniform infinite bridge $\br_\infty$ that forms part
of the encoding of the $\UIHPQ$.

Next, let
$(X_i,i\in\N)$ be a sequence of i.i.d. random variables with distribution
$$
\P(X_1=k)=2^{-k-2},\quad k\geq -1.
$$
Put $\Sigma_j = \sum_{i=1}^jX_i$, with $\Sigma_0=0$. Fix $0\leq k\leq
\sigma$, and denote by $S^{(k)}=(S^{(k)}(j),j=0,\ldots,\sigma)$ the discrete
bridge distributed as $(\Sigma_j,j=0,\ldots,\sigma)$ conditioned on
$\{\Sigma_\sigma =-k\}$. Then the above considerations imply that
$S^{(k)}$ is uniformly distributed over the set $\{\br\in\Br_\sigma:
\br(\sigma)=-k\}$.  Secondly, we can compute
\begin{equation}
\label{eq:law-bsigma}\P\left(\br(\sigma)=-k\right)=\frac{1}{2}\frac{(2\sigma-k-1)!}{(2\sigma-1)!}\frac{\sigma!}{(\sigma-k)!}\leq
2^{-k}, 
\end{equation} and $\P\left(\br(\sigma)=-k\right)\rightarrow 2^{-k-1}$ as
$\sigma\rightarrow\infty$. See~\cite[Proof of Proposition 7]{Be3} for a complete
argument.

\subsubsection{Forests}
\label{sec:forests}
In the rest of this paper, we will often use the following well-known fact
(see, e.g.,~\cite[Section 2]{LGMi}): If
$\f=(\tau_0,\ldots,\tau_{\sigma-1})$ is chosen uniformly at random among
all forests with $\sigma$ trees and $n$ edges, then the corresponding
discrete contour path $(C_{\f}(j),j=0,\ldots,\sigma)$, is distributed as a
simple random walk path starting at $0$ and conditioned to end at $-\sigma$
at time $2n+\sigma$. As a consequence, we have for
$j\in\{0,\ldots,\sigma\}$ and positive integers $k_i$,
\begin{equation}
\label{eq:forest-hittingtime}
\P\left(|\tau_0|=k_0,\ldots,
  |\tau_{j-1}|=k_{j-1}\right)=\P\left(T_{-j}=2(k_0+\ldots+k_{j-1})+j\,|\,T_{-\sigma}=2n+\sigma\right),
\end{equation}
where $T_{-i}$ denotes the first hitting time of $-i$ of a simple random
walk started at $0$. Also note that the joint law of the trees
$(\tau_0,\ldots,\tau_{\sigma-1})$ is invariant under permutation of its
components. Moreover, the sequence of trees
$(\tau_0,\ldots,\tau_{\sigma-1})$ has the law of $\sigma$ independent
critical geometric Galton-Watson trees conditioned to have total size
$n$. In this context, we recall (see, e.g.,~\cite[Section 2.2]{LGMi}) that
if $\P_{\textup{GW}}$ is the law of critical geometric Galton-Watson tree
and $\tau$ a given finite tree, then
\begin{equation}
\label{eq:criticalGW}
\P_{\textup{GW}}\left(\tau\right)=(1/2)\,4^{-|\tau|}.
\end{equation}

Probabilities as in~\eqref{eq:forest-hittingtime} can be computed using
Kemperman's formula (see, e.g.,~\cite[Chapter 6]{Pi}). It tells us that if
$(S_i,i\in \N_0)$ is a simple random walk started at $0$, then
\begin{equation}
\label{eq:Kemperman}
\P\left(T_{j}=k\right)=\frac{|j|}{k}\P(S_k=j),\quad j\in\mathbb{Z},\,k\in\N.
\end{equation}

By applying Kemperman's formula to $\P(T_{-\sigma}=2n+\sigma)$ and counting
paths, we obtain
$$
|\Fo_\sigma^n| = 3^n\frac{\sigma}{2n+\sigma}{2n+\sigma\choose n}.
$$
Note that the factor $3^n$ accounts for the $3^n$ possible labelings of a
forest with $n$ tree edges.

For estimating $\P(S_k=j)$ when $k$ and $j$ are large, one typically applies a local central limit
theorem. Setting
$$
\overline{p}(k,j)=\frac{2}{\sqrt{2\pi
    k}}\exp\left(-\frac{j^2}{2k}\right),\quad j\in\mathbb{Z},\,k\in\N, 
$$
and $\overline{p}(0,j)=\delta_0(j)$, one
has (see, e.g.,~\cite[Theorem 1.2.1]{La})
\begin{equation}
\label{eq:localCLT}
\P(S_k=j)=\overline{p}(k,j) + O\left(1/k^{3/2}\right)
\end{equation}
if $k+j$ is even, and $\P(S_k=j)=0$ otherwise.
For us, it will mostly be sufficient to record that
$\P(S_k=j)\leq Ck^{-1/2}$ for some $C>0$ uniformly in $j$ and $k$.

However, in the boundary regime $\sigma_n\gg \sqrt{n}$, we will sometimes find ourselves in an
atypical regime for simple random walk, where the control provided
by~\eqref{eq:localCLT} is not good enough. In this case, we use the following
asymptotic expression due to Bene\v{s}~\cite[Theorem 1.3, first
case]{Be}. For $x\ll m$ such that $x+m$ is even,
\begin{equation}
\label{eq:localCLT2}
  \P(S_m = x) = \sqrt{\frac{2}{\pi m}}\exp\left(-\sum_{\ell = 1}^\infty \frac{1}{2\ell (2 \ell
    -1)} \frac{x^{2 \ell}}{ m^{2 \ell  -1}}\right) \left(1 +
    O\left(\frac{x^2}{m^2} + \frac{1}{m}\right) \right).
\end{equation}
Note that as it is remarked in~\cite{Be}, this expression can also be
obtained from~\cite[Theorem 6.1.6]{BoBo} by an explicit calculation of the
rate function.

\subsubsection{Remarks on notation} 
\label{sec:usualsetting}
We always let $\N=\{1,2,\ldots\}$, $\N_0=\N\cup\{0\}$.  Recall that for
real sequences $(a_n,n\in\N),(b_n,n\in\N)$, $a_n \ll b_n$ or $b_n \gg a_n$
means that $a_n/b_n \to 0$ as $n\to\infty$, and $a_n\sim b_n$ means
$a_n/b_n\to 1$.  Moreover, we write $a_n\lesssim b_n$ if $a_n\leq C b_n$
for some constant $C>0$ independent of $n$. Sometimes, we also use the 
Landau Big-O and Little-o notation, in a way that will be clear from the
context.

Given a random variable (or sequence) $U$ and an event $\mathcal{E}$, we
write $\mathcal{L}(U)$ and $\mathcal{L}(U|\mathcal{E})$ for the law of $U$
and the conditional law of $U$ given $\mathcal{E}$, respectively. The total
variation norm of a probability measure is denoted by
$\|\cdot\|_{\textup{TV}}$.

We now specify a (notational) framework in which we will often work.
\begin{mdframed}
{\bf The usual setting.} For each $n\in\N$, we let $Q_n^{\sigma_n}=(V(Q_n^{\sigma_n}),\rho_n,\dgr)$ be 
uniformly distributed over the set $\mathcal Q_n^{\sigma_n}$ of rooted
quadrangulations with $n$ internal faces and $2\sigma_n$ boundary edges. Given $Q_n^{\sigma_n}$, we choose
$\vd_n$ uniformly at random among the elements of $V(Q_n^{\sigma_n})$,
and then $(Q_n^{\sigma_n},\vd_n)$ is uniformly distributed over
$Q_{n,\sigma_n}^\bullet$ and corresponds through the Bouttier-Di
Francesco-Guitter bijection to a triplet $((\f_n,\la_n),\br_n)$ uniformly
distributed over the set $\Fo_{\sigma_n}^n \times \Br_{\sigma_n}$. We let
$(C_n,L_n)$ be the contour pair corresponding to $(\f_n,\la_n)$ and write
$$
\La_n = \left(L_n(t) + \br_n(-\uC_n(t)), 0\leq t\leq 2n+\sigma_n\right)
$$
for the label function associated to $((\f_n,\la_n),\br_n)$.

The random triplet $((\f_\infty,\la_\infty),\br_\infty)$ represents a
uniformly labeled critical infinite forest and an independent uniform
infinite bridge and encodes the $\UIHPQ$
$Q_\infty^\infty=(V(Q_\infty^{\infty}),\rho,\dgr)$. We write
$(C_\infty,L_\infty)$ for the corresponding contour pair and $\La_\infty$
for the label function.

While $B_r(Q_n^{\sigma_n})$ denotes the closed ball of radius $r$ around
the root $\rho_n$ in $Q_n^{\sigma_n}$, we will also consider the ball
$B_r^{(0)}(Q_n^{\sigma_n})$ around the vertex $\f_n(0)=(0)$, and similarly
for the $\UIHPQ$.
\end{mdframed}

\section{Auxiliary results}
\label{sec:auxiliaryresults}
In this part we collect general results and observations which will be
useful later on. Our statements on Galton-Watson trees might be of some
interest on its own.
\subsection{Convergence of forests}
The first two lemmas in this section provide the necessary control over the
trees of a forest $\f_n$ chosen uniformly at random in $\Fo^{\sigma_n}_n$
in the regime $\sigma_n\ll\sqrt{n}$.
\begin{lemma}
\label{lem:GW1}
  Assume $\sigma_n\ll\sqrt{n}$. Denote by $(\tau_i)_{1\leq i\leq
    \sigma_n}$ a family of $\sigma_n$ independent critical geometric Galton-Watson trees. Then
$$
\liminf_{\delta\downarrow
  0}\liminf_{n\rightarrow\infty}\P\left(\exists !\,
  i\in\{1,\ldots,\sigma_n\}\hbox{ with }|\tau_i|\geq (1-\delta) n
  \Big| \sum_{i=1}^{\sigma_n}|\tau_i|=n\right) = 1.
$$
\end{lemma}
\begin{proof}
  We use the contour function representation of the forest
  $\f_n$ as a simple random walk, conditioned on first hitting
  $-\sigma_n$ at time $2n+\sigma_n$ (and interpolated linearly between
  integer times). We let $(S_i,0\leq i\leq 2n+\sigma_n)$ denote such a
  conditioned random walk. Under our assumptions, it holds that
  \begin{equation}
    \label{eq:1}
    \left(\frac{S_{(2n+\sigma_n)
          t}}{\sqrt{2n}}\right)_{0\leq t\leq 1}\xrightarrow[n\to\infty]{(d)}\mathbbm{e}\, ,
\end{equation}
in distribution in $\mathcal{C}([0,1],\R)$, where $\mathbbm{e}$ is the
normalized Brownian excursion. This ``folklore'' result is implicit in
\cite{Be2}, so we recall quickly how to obtain it, omitting some
details. First, by~\cite{BeChPi}, one can represent the conditioned
random walk as a cyclic shift of a simple random walk that is
conditioned to hit $-\sigma_n$ at time $2n+\sigma_n$, but not
necessarily for the first time. More precisely, calling $S'$ this new
random walk, we let $\nu_n$ be a uniform random variable in
$\{0,1,\ldots,\sigma_n-1\}$, and we let
$$A_n=\inf\left\{i\geq 0: S'_i=\min\{ S'_j: 0\leq j \leq
2n+\sigma_n\}+\nu_n\right\}\, .$$
Then, the sequence $(S''_i,0\leq i\leq 2n+\sigma_n)$ defined by 
$$S''_i=\left\{\begin{array}{lll}S'_{A_n+i}-S'_{A_n} & \mbox{ if }&
    0\leq i\leq 2n+\sigma_n-A_n\\
-\sigma_n+S'_{i-2n-\sigma_n+A_n} -S'_{A_n}& \mbox{ if }&2n+\sigma_n-A_n<i\leq 2n+\sigma_n
  \end{array}\right.
  $$
  has same distribution as $(S_i,0\leq i\leq 2n+\sigma_n)$. Now it is
  classical that under the assumption that $\sigma_n=o(\sqrt{n})$,
$$\left(\frac{S'_{(2n+\sigma_n)t}}{\sqrt{2n}}\right)_{0\leq t\leq 1}\xrightarrow[n\to\infty]{(d)}\mathbbm{b}\, ,$$
where $\mathbbm{b}$ is a standard Brownian bridge. From this, one
deduces that 
$$\left(\frac{S_{(2n+\sigma_n)t}}{\sqrt{2n}}\right)_{0\leq t\leq
  1}=_d\left(\frac{S''_{(2n+\sigma_n)t}}{\sqrt{2n}}\right)_{0\leq
t\leq 1}\xrightarrow[n\to\infty]{(d)}V\mathbbm{b}\, ,$$ where
$V\mathbbm{b}=(\mathbbm{b}_{s+s_*}-\mathbbm{b}_{s_*},0\leq s\leq 1)$ is
the Vervaat transform of $\mathbbm{b}$, that is the cyclic shift of
$\mathbbm{b}$ at the a.s.\ unique time $s_*$ where it attains its
overall minimum. Here, the bridge $\mathbbm{b}$ is extended
periodically by $\mathbbm{b}_{s+1}=\mathbbm{b}_{s}$ for $s\in
[0,1]$. Finally, we use the well-known fact that $\mathbbm{e}$ and
$V\mathbbm{b}$ have the same distribution.

Now notice that the quantities $\tau_1,\ldots,\tau_{\sigma_n}$ are
equal to (half) the lengths of the excursions of $S$ above its infimum
process, in the order in which they appear. Hence, the convergence
\eqref{eq:1} clearly implies that the largest of these quantities
satisfies
$$\frac{\max\{\tau_i: 1\leq
  i\leq\sigma_n\}}{n}\longrightarrow 1$$
in probability as $n\to\infty$, while all the other quantities are
negligible compared to $n$ in probability. 
\end{proof}

\begin{lemma}
\label{lem:GW2}
  Assume $\sigma_n\ll\sqrt{n}$. Denote by $(\tau_i)_{1\leq i\leq
    \sigma_n}$ a family of $\sigma_n$ independent critical geometric Galton-Watson trees. Write $i_\ast$ for
  the smallest index such that $|\tau_{i_\ast}|\geq \max_{1\leq i\leq
    \sigma_n, i\neq i_\ast}|\tau_j|$. Then
$$
\lim_{n\rightarrow\infty}\left\|\mathcal{L}\Big((\tau_i)_{1\leq i\leq
    \sigma_n, i\neq i_\ast}\,\Big|\,\sum_{i=1}^{\sigma_n}|\tau_i|=n\Big) -\mathcal{L}\left((\tau_i)_{1\leq i\leq
    \sigma_n-1}\right)\right\|_{\textup{TV}}=0.
$$
\end{lemma}

\begin{proof}
  For $\delta>0$, and $F$ a bounded and
  measurable function,
\begin{align*}
\lefteqn{\E\left[F((\tau_i)_{i\neq i_\ast})\,\Big|\,\sum_{i=1}^{\sigma_n}|\tau_i|=n\right]}\\
&=\E\left[F((\tau_i)_{i\neq i_\ast})\1_{\{|\tau_{i_\ast}|\geq \delta
    n\}\cap\{|\tau_{i_\ast}|>|\tau_i|\mbox{\small{ for all }}i\neq i_\ast\}}\,\Big|\,\sum_{i=1}^{\sigma_n}|\tau_i|=n\right] +
\|F\|_\infty R_n^{(\delta)},
\end{align*}
where by Lemma~\ref{lem:GW1}, the error term $R_n^{(\delta)}$ satisfies
$\limsup_{\delta\downarrow
  0}\limsup_{n\rightarrow\infty}R_n^{(\delta)}=0$. Therefore it remains to
consider the expectation in the last display for small but fixed
$\delta$. Put $p(k,m)= \P(\sum_{i=1}^{k}|\tau_i|=m)$, and write $\tau$
instead of $\tau_\sigma$. Using exchangeability of the
trees, the expectation becomes
\begin{align*}
\lefteqn{\frac{\sigma_n}{p(\sigma_n,n)}\E\left[F((\tau_i)_{1\leq i\leq
        \sigma_n-1})\1_{\{|\tau_{\sigma_n}|>\max_{1\leq i\leq
        \sigma_n-1}|\tau_i|\}\cap\{|\tau_{\sigma_n}|\geq \delta
      n\}\cap\{\sum_{i=1}^{\sigma_n}|\tau_i|=n\}}\right] = } \\
&\E\Bigg[F((\tau_i)_{1\leq i\leq
        \sigma_n-1})\underbrace{\frac{\sigma_n}{p(\sigma_n,n)}\P\Big(|\tau_{\sigma_n}|>\max_{1\leq i\leq
        \sigma_n-1}|\tau_i|; |\tau_{\sigma_n}|\geq \delta
      n; \sum_{i=1}^{\sigma_n}|\tau_i|=n\,\Big|\,(\tau_i)_{1\leq
        i\leq \sigma_n-1}\Big)}_{=Z_n}\Bigg].
\end{align*}
In order to conclude, it suffices to show that
\begin{equation}
\label{eq:GW2-1}
\limsup_{n\rightarrow\infty}\E\left[\left|Z_n-1\right|\right]=0.
\end{equation}
Let $K\in\N$. We split into
\begin{equation}
\label{eq:GW2-2}
\E[|Z_n-1|]=\E\left[|Z_n-1|\1_{\{\sum_{i=1}^{\sigma_n-1}|\tau_i|\leq K\sigma_n^2\}}\right] +\E\left[|Z_n-1|\1_{\{\sum_{i=1}^{\sigma_n-1}|\tau_i|> K\sigma_n^2\}}\right].
\end{equation}
We first show that
\begin{equation}
\label{eq:GW2-3}
\limsup_{K\rightarrow\infty}\limsup_{n\rightarrow\infty}\E\left[|Z_n-1|\1_{\{\sum_{i=1}^{\sigma_n-1}|\tau_i|> K\sigma_n^2\}}\right]=0.
\end{equation}
We estimate
$$
\E\left[|Z_n-1|\1_{\{\sum_{i=1}^{\sigma_n-1}|\tau_i|>
    K\sigma_n^2\}}\right]\leq
\E\left[Z_n\1_{\{\sum_{i=1}^{\sigma_n-1}|\tau_i|>
    K\sigma_n^2\}}\right]+\P\left(\sum_{i=1}^{\sigma_n-1}|\tau_i|>
  K\sigma_n^2\right).
$$
Recall that the last term is equal to $\P(T_{-\sigma_n+1}>2K\sigma_n^2+\sigma_n-1)$,
where $T_{k}$ is as above the first hitting time of $k$ of a simple random
walk $(S_i,i\in\N_0)$ started at zero. Standard random walk estimates (e.g., Kemperman's
formula~\eqref{eq:Kemperman} together with~\eqref{eq:localCLT}) entail that
$$
\limsup_{K\rightarrow\infty}\limsup_{n\rightarrow\infty}\P\left(T_{-\sigma_n+1}>
  2K\sigma_n^2+\sigma_n-1\right)=0.
$$
The first term on the right hand side in the next to
last display is estimated by
\begin{align*}
\E\left[Z_n\1_{\{\sum_{i=1}^{\sigma_n-1}|\tau_i|> K\sigma_n^2\}}\right]&\leq\frac{\sigma_n}{p(\sigma_n,n)}
\P\left(\sum_{i=1}^{\sigma_n-1}|\tau_i|>
  K\sigma_n^2;\,\sum_{i=1}^{\sigma_n}|\tau_i|=n;\,|\tau_{\sigma_n}|\geq \delta
  n\right)\\
&\leq \frac{\sigma_n}{p(\sigma_n,n)}\sum_{m=
  K\sigma_n^2}^{n-\lceil\delta n\rceil}p(\sigma_n-1,m)p(1,n-m)\\
&\lesssim n^{3/2}\sum_{m=
  K\sigma_n^2}^{n-\lceil\delta
  n\rceil}\frac{\sigma_n}{(2m+\sigma_n-1)^{3/2}}\frac{1}{(2(n-m)+1)^{3/2}}\\
&\lesssim \sigma_n\sum_{m=
  K\sigma_n^2}^{\lceil n/2\rceil}\frac{1}{m^{3/2}} +  \sigma_n
\sum_{m=\lceil n/2\rceil+1}^{n-\lceil\delta n\rceil}\frac{1}{(n-m)^{3/2}}\,\lesssim\,
\frac{1}{K^{1/2}}+ \frac{\sigma_n}{\sqrt{\delta n}}.
\end{align*}
Recalling that $\sigma_n\ll \sqrt{n}$, this finishes the proof
of~\eqref{eq:GW2-3}. We turn to the first term on the right hand side
of~\eqref{eq:GW2-2}. First note that on the event 
$$
\left\{\sum_{i=1}^{\sigma_n-1}|\tau_i|\leq K\sigma_n^2;\,\sum_{i=1}^{\sigma_n}|\tau_i|=n\right\},
$$
we have $|\tau_{\sigma_n}|\geq \delta n$ and $|\tau_{\sigma_n}|>\max_{1\leq i\leq
        \sigma_n-1}|\tau_i|$ almost surely, provided $n$ is large
      enough. Therefore,
\begin{align*}
  \lefteqn{\E\left[|Z_n-1|\1_{\{\sum_{i=1}^{\sigma_n-1}|\tau_i|\leq
        K\sigma_n^2\}}\right]}\\
  &=
  \E\left[\left|\frac{\sigma_n}{p(\sigma_n,n)}\P\Big(|\tau_{\sigma_n}|=n-\sum_{i=1}^{\sigma_n-1}|\tau_i|\,\Big|\,(\tau_i)_{1\leq
      i\leq
      \sigma_n-1}\Big)-1\right|\1_{\{\sum_{i=1}^{\sigma_n-1}|\tau_i|\leq
      K\sigma_n^2\}}\right]\\
  &=\sum_{m=0}^{K\sigma_n^2}\left|\frac{\sigma_np(1,n-m)}{p(\sigma_n,n)}-1\right|p(\sigma_n-1,m).
\end{align*}
We now show that the terms inside the absolute value in the last display
are of order $o(1)$ as $n$ tends to infinity, uniformly in $m$ with $m\leq
K\sigma_n^2$. First, by Kemperman's formula~\eqref{eq:Kemperman},
$$
\frac{\sigma_np(1,n-m)}{p(\sigma_n,n)}=
\frac{2n+\sigma_n}{(2(n-m)+1)}\frac{\P\left(S_{2(n-m)+1}=1\right)}{\P\left(S_{2n+\sigma_n}=\sigma_n\right)}.
$$
Since $\sigma_n^2\ll n$, we have $2n+\sigma_n/(2(n-m)+1)\sim 1$. For the
fraction of the two probabilities involving simple random walk, we apply
the local central limit theorem~\eqref{eq:localCLT} and obtain  
$$
\limsup_{n\rightarrow\infty}\sup_{m\leq K\sigma_n^2}\left|\frac{\P\left(S_{2(n-m)+1}=1\right)}{\P\left(S_{2n+\sigma_n}=\sigma_n\right)}-1\right|=0.
$$
This shows
$$
\limsup_{n\rightarrow\infty}\E\left[|Z_n-1|\1_{\{\sum_{i=1}^{\sigma_n-1}|\tau_i|\leq
    K\sigma_n^2\}}\right]=0.
$$
With~\eqref{eq:GW2-3}, we have proved that~\eqref{eq:GW2-1}
holds, completing thereby the proof of the lemma.
\end{proof}
The next statement will prove useful for the regimes
$1\ll\sigma_n\ll\sqrt{n}$ and $\sigma_n\sim \sigma\sqrt{2n}$,
$\sigma\in(0,\infty)$, as well as for the local convergence of
$Q_n^{\sigma_n}$ towards the $\UIHPQ$ when $1\ll\sigma_n\ll n$.
We stress that if $\sigma_n\ll \sqrt{n}$, the following lemma is already a
corollary of Lemmas~\ref{lem:GW1} and~\ref{lem:GW2}.
\begin{lemma}
\label{lem:GW3}
  Assume $1\ll\sigma_n\ll n$. Denote by $(\tau_i)_{1\leq i\leq
    \sigma_n}$ a family of $\sigma_n$ independent critical geometric Galton-Watson trees. If $k_n$ is a
  sequence of positive integers with $k_n\leq \sigma_n$ and
  $k_n=o(\sigma_n\wedge (n/\sigma_n))$ as $n\rightarrow\infty$, then
$$
\limsup_{n\rightarrow\infty}\left\|\mathcal{L}\Big((\tau_i)_{1\leq i\leq
    k_n}\,\Big|\,\sum_{i=1}^{\sigma_n}|\tau_i|=n\Big)
  -\mathcal{L}\left((\tau_i)_{1\leq i\leq
      k_n}\right)\right\|_{\textup{TV}}=0.
$$
\end{lemma}

\begin{proof}
  The arguments are similar to those in the proof of
  Lemma~\ref{lem:GW2}. We set again 
  $p(k,m)= \P(\sum_{i=1}^{k}|\tau_i|=m)$. We have for bounded and
  measurable $F$
\begin{align*}
  &\E\left[F((\tau_i)_{1\leq i\leq k_n})\,\Big|\,\sum_{i=1}^{\sigma_n}|\tau_i|=n\right]\\
  &= \E\Bigg[F((\tau_i)_{1\leq i\leq k_n})\underbrace{\frac{1}{p(\sigma_n,n)}\P\left(\sum_{i=1}^{\sigma_n}|\tau_i|=n\,
      \,\Big|\,(\tau_i)_{1\leq i\leq k_n}\right)}_{=Z_n}\Bigg],
\end{align*}
and the claim follows if we show that 
$$\limsup_{n\rightarrow\infty}\E\left[|Z_n-1|\right]=0.$$
We argue now similarly to Lemma~\ref{lem:GW2}. With $K\in\N$, we split into
$$
  \E[|Z_n-1|]=\E\left[|Z_n-1|\1_{\{\sum_{i=1}^{k_n}|\tau_i|\leq Kk_n^2\}}\right] +\E\left[|Z_n-1|\1_{\{\sum_{i=1}^{k_n}|\tau_i|> Kk_n^2\}}\right]
$$
and bound the second term by
$$
\E\left[|Z_n-1|\1_{\{\sum_{i=1}^{k_n}|\tau_i|>
    Kk_n^2\}}\right]\leq
\E\left[Z_n\1_{\{\sum_{i=1}^{k_n}|\tau_i|>
    Kk_n^2\}}\right]+\P\left(\sum_{i=1}^{k_n}|\tau_i|>
  Kk_n^2\right).
$$
The last term in the above display is estimated in the same way as the analogous term in
Lemma~\ref{lem:GW2}. For the first term, we have 
\begin{align*}
\lefteqn{\E\left[Z_n\1_{\{\sum_{i=1}^{k_n}|\tau_i|> Kk_n^2\}}\right]
\leq \frac{1}{p(\sigma_n,n)}\sum_{m=
  Kk_n^2}^{n}p(k_n,m)p(\sigma_n-k_n,n-m)}\\
&\lesssim \frac{n^{3/2}}{\sigma_n}\sum_{m=
  Kk_n^2}^{\lceil
  n/2\rceil}\frac{k_n}{m^{3/2}}\frac{\sigma_n-k_n}{(n-m)^{3/2}} +  
\frac{k_n}{\sigma_n}\sum_{m=\lceil n/2\rceil+1}^{n}p(\sigma_n-k_n,n-m)\lesssim \frac{1}{\sqrt{K}}+ \frac{k_n}{\sigma_n}.
\end{align*}
Recalling that $k_n\ll\sigma_n$, we obtain
$$
\limsup_{K\rightarrow\infty}\limsup_{n\rightarrow\infty}\E\left[|Z_n-1|\1_{\{\sum_{i=1}^{k_n}|\tau_i|> Kk_n^2\}}\right]=0.
$$
It remains to show that for fixed $K$,
\begin{equation}
\label{eq:GW3-1}
\limsup_{n\rightarrow\infty}\E\left[|Z_n-1|\1_{\{\sum_{i=1}^{k_n}|\tau_i|\leq
    K k_n^2\}}\right]=0.
\end{equation}
We write 
\begin{align*}
  \lefteqn{\E\left[|Z_{n}-1|\1_{\{\sum_{i=1}^{k_n}|\tau_i|\leq Kk_n^2\}}\right]}\\
  &= \E\left[\left|\frac{1}{p(\sigma_n,n)}\P\Big(\sum_{i=
      k_n+1}^{\sigma_n}|\tau_i|=n-\sum_{i=1}^{k_n}|\tau_i|\,\Big|\,(\tau_i)_{1\leq
      i\leq k_n}\Big)-1\right|\1_{\{\sum_{i=1}^{k_n}|\tau_i|\leq K k_n^2\}}\right]\\
  &=\sum_{m=0}^{K k_n^2}\left|\frac{p(\sigma_n-k_n,n-m)}{p(\sigma_n,n)}-1\right|p(k_n,m).
\end{align*}
Again, our proof will be complete if we show that the terms inside the
absolute value are of order $o(1)$, uniformly in $m$ with $m\leq K
k_n^2$. For such $m$, let
$$x_n=\sigma_n -  k_n, \quad y_n = 2(n-m)+ \sigma_n -k_n.$$
Since the case where $\sigma_n$ is much larger than $\sqrt{n}$ is also
included in our statement, we apply the refined
version~\eqref{eq:localCLT2} (and first Kemperman's formula), which gives
\begin{align}
\frac{p(\sigma_n-k_n,n-m)}{p(\sigma_n,n)}&\sim\frac{\P\left(S_{y_n}=x_n\right)}{\P\left(S_{2n+\sigma_n}=\sigma_n\right)}\nonumber\\
&\sim \exp\left(-\sum_{\ell = 1}^\infty \frac{1}{2\ell (2 \ell
    -1)} \left(\frac{x_n^{2 \ell}}{ {y_n}^{2 \ell  -1}}
  -\frac{\sigma_n^{2\ell}}{ (2n+\sigma_n)^{2\ell-1}}\right)\right),
\label{eq:GW3-2}
\end{align}
everything uniformly in $m$ with $m\leq K k_n^2$. By Taylor's expansion, we
obtain
\begin{multline*}
  \frac{x_n^{2 \ell}}{ {y_n}^{2 \ell -1}}
  -\frac{\sigma_n^{2\ell}}{ (2n+\sigma_n)^{2\ell-1}} = \\
  \frac{\sigma_n^{2\ell}}{(2n+\sigma_n)^{2 \ell -1}}\left[ -2 \ell
    \frac{k_n}{\sigma_n} + (2 \ell -1) \frac{2(m+k_n)}{2n+\sigma_n} + O\left(\left(\frac{k_n}{\sigma_n}\right)^2 \right
    ) + O\left(\left(\frac{m+k_n}{2n+\sigma_n}\right)^2\right)
  \right].
\end{multline*}
Since $k_n=o(\sigma_n\wedge (n/\sigma_n))$, we have for $m\leq K k_n^2$
\begin{align*}
  -2\ell\frac{\sigma_n^{2\ell}}{(2n+\sigma_n)^{2 \ell -1}} \frac{ 
k_n}{\sigma_n} 
  & =  \frac{\sigma_n^{2(\ell-1)}}{(2n+\sigma_n)^{2 
(\ell -1)}}\,o(1) ,\quad\hbox{ and}\\
(2\ell-1)  \frac{\sigma_n^{2\ell}}{(2n+\sigma_n)^{2 \ell -1}}  \frac{2(m+k_n)}{2n+\sigma_n}  &  = \frac{\sigma_n^{2(\ell-1)}}{(2n+\sigma_n)^{2 
(\ell -1)}}\,o(1).
\end{align*}
In particular, all terms of the sum inside the exponential in~\eqref{eq:GW3-2} tend to zero as
$n\rightarrow\infty$. Moreover, for $n$ large, each term is bounded by $2^{-2(\ell-1)}$,
which is summable. We finish the proof of the lemma by an application of 
dominated convergence, giving
$$
\limsup_{n\rightarrow\infty}\sup_{m\leq Kk_n^2}\left|\frac{p(\sigma_n-k_n,n-m)}{p(\sigma_n,n)}-1\right|=0.
$$
\end{proof}

\subsection{Convergence of bridges}
Here, we collect two convergence results of a bridge $\br_n$ uniformly
distributed in $\Br_{\sigma_n}$ which are valid in all regimes $\sigma_n\gg 1$. The
first lemma follows from~\cite[Lemma 10]{Be1} (recall the remarks above on
the distribution of $\br_n$).
\begin{lemma}
\label{lem:bridge0}
Assume $\sigma_n\rightarrow\infty$, and let $\br_n$ be a bridge of length $\sigma_n$ uniformly distributed in
$\Br_{\sigma_n}$.  Then $(\br_n(\sigma_n s)/\sqrt{2\sigma_n},0\leq s\leq 1)$ converges as $n\rightarrow\infty$ to a
standard Brownian bridge $\mathbbm{b}$,
and the convergence holds in distribution in the space $\mathcal{C}([0,1],\mathbb{R})$.
\end{lemma}

The next lemma provides a finer convergence without normalization for the
bridge restricted to the first and last $k_n$ values when
$k_n=o(\sigma_n)$.
\begin{lemma}
\label{lem:bridge1}
Assume $\sigma_n\rightarrow\infty$. Let $\br_n$ be uniformly distributed in
$\Br_{\sigma_n}$, and let $\br_\infty$ be a uniform infinite bridge as
defined under Section~\ref{sec:infinitebridge}. Then, if $k_n$ is a
sequence of positive integers with $k_n\leq \sigma_n$ and $k_n=o(\sigma_n)$
as $n\rightarrow\infty$,
\begin{equation*}
\begin{split}\limsup_{n\rightarrow\infty}&\left\|\mathcal{L}((\br_n(\sigma_n-k_n),\ldots,\br_n(\sigma_n-1),\br_n(0),\br_n(1),\ldots,\br_n(k_n)))\right.\\
&\left. -\,\mathcal{L}((\br_\infty(-k_n),\ldots,\br_\infty(-1),\br_\infty(0),\br_\infty(1),\ldots,\br_\infty(k_n)))\right\|_{\textup{TV}}=0.\end{split}
\end{equation*}
\end{lemma}
\begin{proof}
Let $(b(i) : -k_n\leq i\leq k_n)$ be a $\mathbb{Z}$-valued
  sequence with $b(0)=0$ and $b(i+1)-b(i)\in\mathbb{N}_0\cup\{-1\}$. Note
  that by definition, both
  $(\br_n(\sigma_n-k_n),\ldots,\br_n(0),\ldots,\br_n(k_n))$ and
  $(\br_\infty (-k_n),\ldots,\br_\infty (0),\ldots,\br_\infty (k_n))$ are
  only supported on such sequences. By
  definition of $\br_\infty$, we obtain
$$
\P\left(\br_\infty(i)=b(i),\, -k_n\leq i\leq k_n\right)=(-b(-1)+2)2^{-(b(k_n)-b(-k_n))-4k_n-1}.
$$
Next recall the interpretation of the increments of $\br_n$ explained in
Section~\ref{sec:bridges}. We get
\begin{align*}
  \lefteqn{\P\left(\br_n(i)=b(i),\,\br_n(\sigma_n-i)=b(-i),\,1\leq i\leq
      k_n\right)}\\
  &=\sum_{j=b(-1)-1}^0\P\left(\br_n(i)=b(i),\,\br_n(\sigma_n-i)=b(-i),\,1\leq i\leq
      k_n,\,\br_n(\sigma_n)=j\right)\\
    &=(-b(-1)+2)\P\left(\br_n(i)=b(i),\,\br_n(\sigma_n-i)=b(-i),\,1\leq
      i\leq k_n,\,\br_n(\sigma_n)=0\right)\\
    &=(-b(-1)+2){2\sigma_n-(b(k_n)-b(-k_n))-4k_n-1\choose
        \sigma_n-2k_n-1}\Big/{2\sigma_n\choose\sigma_n}.
\end{align*} 
Here, the next to last line follows from the fact that $\br_n$ is uniformly
distributed in $\Br_{\sigma_n}$, and the last line follows from counting
the possibilities to put $(\sigma_n-2k_n-1)$ times the number $-1$ in the
remaining  $(2\sigma_n-(b(k_n)-b(-k_n))-4k_n-1)$ spots.

We now concentrate on $b$ such that $|b(k_n)-b(-k_n)|\leq K\sqrt{k_n}$ for
some fixed constant $K>0$.  We put $B_n=b(k_n)-b(-k_n)+1$.
An application of Stirling's formula shows that
$$
\frac{{2\sigma_n-4k_n-B_n\choose
    \sigma_n-2k_n-1}}{{2\sigma_n\choose\sigma_n}}\sim
2^{-4k_n-B_n}\left(\frac{\sigma_n-2k_n-B_n/2}{\sigma_n-2k_n-B_n}\right)^{\sigma_n-2k_n-B_n}\left(\frac{\sigma_n-2k_n-B_n/2}{\sigma_n-2k_n}\right)^{\sigma_n-2k_n}
$$
as $n\rightarrow\infty$, uniformly in $b$ with $|b(k_n)-b(-k_n)|\leq
K\sqrt{k_n}$. Next, observe that
$$
\left(\frac{\sigma_n-2k_n-B_n/2}{\sigma_n-2k_n}\right)^{\sigma_n-2k_n}=\left(1-\frac{B_n/2}{\sigma_n-k_n}\right)^{\sigma_n-k_n}
\sim\exp\left(-\frac{B_n}{2}(1+O\left(B_n/\sigma_n)\right)\right),
$$
and similarly
$$
\left(\frac{\sigma_n-2k_n-B_n/2}{\sigma_n-2k_n-B_n}\right)^{\sigma_n-2k_n-B_n}\sim\exp\left(\frac{B_n}{2}\left(1-O(B_n/\sigma_n)\right)\right).
$$
Note that $B^2_n/\sigma_n=o(1)$ as $n\rightarrow\infty$ uniformly in the sequences $b$ under
consideration. Now let $0<\eps<1$. Putting the above estimates
together, we deduce that there exist $n'=n'(K,\eps)$ sufficiently large such that
for all $n\geq n'$,
\begin{align*}
(1-\eps)(2^{-(b(k_n)-b(-k_n))-4k_n-1})&\leq{2\sigma_n-(b(k_n)-b(-k_n))-4k_n-1\choose
  \sigma_n-2k_n-1}\Big/{2\sigma_n\choose\sigma_n}\nonumber\\
&\leq (1+\eps)(2^{-(b(k_n)-b(-k_n))-4k_n-1}).
\end{align*}
Using the aforementioned interpretation of $\br_n$ (or
Lemma~\ref{lem:bridge0}), it is immediate to check that both
$\br_n(-k_n)$ and $\br_n(k_n)$ are of order $\sqrt{k_n}$ for large $n$,
i.e., we find $K>0$ such that $|\br_n(k_n)-\br_n(-k_n)|\leq K\sqrt{k_n}$
with probability at least $1-\eps$, provided $n$ is sufficiently
large. By Donsker's invariance principle, we see that a
similar bound holds for $\br_\infty$. For any set
$\mathcal{E}_n$ of $\mathbb{Z}$-valued sequences of length $2k_n+1$, we thus
obtain
\begin{align*}
  \begin{split}&|\,\P\left(\left(\br_n(\sigma_n-k_n),\ldots,\br_n(0),\ldots,\br_n(k_n)\right)\in
      \mathcal{E}_n\right)\\
    &\,-\,\P\left(\left(\br_\infty (-k_n),\ldots,\br_\infty
        (0),\ldots,\br_\infty(k_n)\right)\in
      \mathcal{E}_n\right)|\end{split}\\
  \quad &\leq \P\left(|\br_n(k_n)-\br_n(-k_n)|\geq K\sqrt{k_n}\right) + \P\left(|\br_\infty(k_n)-\br_\infty(-k_n)|\geq K\sqrt{k_n}\right)\\
  &\quad + \sum_{b\in\mathcal{E}_n:\atop |b(k_n)-b(-k_n)|\leq
    K\sqrt{k_n}}\Bigg[\Big|\P\left(\br_\infty(i)=b(i),\, -k_n\leq i\leq k_n\right)\cdot\\
  &\quad\quad\quad
  \frac{\P\left(\br_n(i)=b(i),\,\br_n(\sigma_n-i)=b(-i),\,1\leq i\leq k_n\right)}{\P\left(\br_\infty(i)=b(i),\,
      -k_n\leq i\leq k_n\right)}-1\Big|\Bigg]\\
  &\quad \leq 2\eps + (1+\eps)-1\leq 3\eps.
\end{align*}
This finishes the proof.
\end{proof}

\subsection{Root issues}
We work in the usual setting introduced in Section~\ref{sec:usualsetting}.
As the next lemma shows, instead of showing  distributional convergence of
balls in $Q_n^{\sigma_n}$ or $Q_\infty^\infty$ around
the roots, we can as well consider the corresponding balls around $(0)$. 
 
\begin{lemma}
\label{lem:ball0}
  Let $(a_n)_{n\geq 1}$ be a sequence of reals with
  $a_n\rightarrow\infty$ as $n\rightarrow\infty$. Let $r\geq 0$. Then, in the notation
  from above, we have the following convergences in probability as $n\rightarrow\infty$.
\begin{enumerate}
\item
$\dgh\left(B_r\left((a_n^{-1}\cdot
    Q_n^{\sigma_n}\right),B_r^{(0)}\left((a_n^{-1}\cdot
    Q_n^{\sigma_n}\right)\right)\rightarrow 0,
$
\item $\dgh\left(B_r\left((a_n^{-1}\cdot
      Q_\infty^{\infty})\right),B_r^{(0)}\left((a_n^{-1}\cdot
      Q_\infty^{\infty}\right)\right)\rightarrow 0.
$
\end{enumerate}
\end{lemma}
The proof will be a consequence of the following general lemma.
\begin{lemma}
\label{lem:localGHconv}
  Let $r\geq 0$, and let $\mathbf{E}=(E,d,\rho)$ and
  $\mathbf{E}'=(E',d',\rho')$ be two pointed complete and locally compact length
  spaces. Let $\cR\subset E\times E'$ be a subset with the following
  properties:
\begin{itemize}
\item $(\rho,\rho')\in\cR$,
\item for all $x \in B_r(\mathbf{E})$, there exists $x'\in E'$ such that
  $(x,x')\in\cR$,
\item for all $y' \in B_r(\mathbf{E}')$, there exists $y\in E$ such that $(y,y')\in\cR$.
\end{itemize}
Then, $\dgh(B_r(\mathbf{E}),B_r(\mathbf{E}'))\leq (3/2)\dis(\cR)$.
\end{lemma}
\begin{remark}
  Note that $\cR$ is not necessarily a correspondence; nonetheless,
  the definition of the distortion $\dis(\cR)$ from
  Section~\ref{S-notionconvergence} makes sense (we allow it to take the
  value $+\infty$).
\end{remark} 
\begin{proof}[Proof of Lemma~\ref{lem:localGHconv}]
   We construct a correspondence $\tilde{\cR}$ between $B_r(\mathbf{E})$
   and $B_r(\mathbf{E}')$. For each $x\in B_r(\mathbf{E})$, there exists by
   assumption $x'=x'_{x}\in E'$ such that $(x,x')\in\cR$. Since
   $d'(x',\rho')\leq d(x,\rho) + \dis(\cR)$, we see that in fact $x'\in
   B_{r+\dis(\cR)}(\mathbf{E}')$.  We choose $z'=z'(x)\in B_r(\mathbf{E}')$
   that minimizes $d'(x',z')$. Note that such a $z'$ exists in a complete
   and locally compact length space. Then $d'(x',z')\leq \dis(\cR)$. In an
   entirely similar way, using the third property of $\cR$ instead of the
   second, we assign to each $y'\in B_r(\mathbf{E}')$ an element
   $z=z(y')\in B_r(\mathbf{E})$. In this notation, we now define
$$
\tilde{\cR}=\left\{(x,z'(x)) : x\in B_r(\mathbf{E})\right\}\cup
\left\{(z(y'),y') : y'\in B_r(\mathbf{E}')\right\}.
$$
Clearly, $\tilde{\cR}$ is a correspondence between $B_r(\mathbf{E})$ and
$B_r(\mathbf{E}')$, and a straightforward application of the triangle
inequality shows that in fact $\dis(\tilde{\cR}) \leq 3\dis(\cR)$. This
proves our claim and hence the lemma.
\end{proof}

\begin{proof}[Proof of Lemma~\ref{lem:ball0}]
  We show only (a), the proof of (b) is similar. We apply
  Lemma~\ref{lem:localGHconv} as follows. Instead of considering the pointed
  quadrangulation $(V(Q_n^{\sigma_n}),\dgr,\rho_n)$, we may work with the
  corresponding pointed length space $\mathbf{E}_n=(E_n,d,(0))$ obtained
  from replacing edges by Euclidean segments of length one, as explained
  in Section~\ref{sec:locGH} (the distance $d$ between two points is given
  by the length of a shortest path between them). Similarly,
  we replace $(V(Q_n^{\sigma_n}),\dgr,(0))$ by
  $\mathbf{E}'_n=(E_n,d,\rho_n)$. Define
$$
\cR_n=\{(\rho_n,(0))\}\cup\{(x,x) : x\in E_n\}.
$$
Then $\cR_n$ fulfills trivially the properties of
Lemma~\ref{lem:localGHconv}, and we have dis$(\cR_n)\leq d(\rho_n,(0))
=-\br_n(\sigma_n)$ by~\eqref{eq:distance-root-0}. Since $\br_n(\sigma_n)$ is
stochastically bounded, see~\eqref{eq:law-bsigma}, the claim follows.
\end{proof}

\section{Main proofs}
\label{sec:proofs}
We start now with the proofs of the main results.  To facilitate the
reading, we will sometimes include a paragraph ``Idea of the proof'', where
we informally explain the basic strategy.
\subsection{Brownian plane (Theorem~\ref{thm:BP})}
\label{sec:proof-thmBP}
Recall that Theorem~\ref{thm:BP} deals with the regime
$\sigma_n\ll\sqrt{n}$ and $\sqrt{\sigma_n} \ll
a_n$.
\begin{mdframed}
{\bf Idea of the proof.} Let $((\f_n,\la_n),\br_n)$ be uniformly distributed over the set
$\Fo_{\sigma_n}^n \times \Br_{\sigma_n}$, and let $\La_n$ be the associated
label function. Thanks to Lemmas~\ref{lem:GW1} and~\ref{lem:GW2}, we know
that for large $n$, $\f_n$ has a unique largest tree $\tau$ of a size of
order $n$, and all the other $\sigma_n-1$ trees behave as independent
critical geometric Galton-Watson trees. As a consequence, both the maximal
and minimal values of the label function $\La_n$ restricted to these
$\sigma_n-1$ non-largest trees are of order $\sqrt{\sigma_n}$, see the
proof of Lemma~\ref{lem:BP2}. Under a rescaling of distances by the factor
$a_n^{-1}$, this implies by a result of Bettinelli~\cite[Lemma 23]{Be3}
that the part of the quadrangulation encoded by the forest without its
largest tree $\tau$ is negligible in the limit $n\rightarrow\infty$ for the
local Gromov-Hausdorff topology. Conditionally on its size, $\tau$ is
uniformly distributed among all plane trees, and (up to the removal of a
single edge) so is the associated quadrangulation among all
quadrangulations with $|\tau|$ faces and no boundary. This allows us to
apply the second part of~\cite[Theorem 3]{CuLG}, which states that the
Brownian plane appears as the scaling limit $m\rightarrow\infty$ of uniform
quadrangulations with $m$ faces when the scaling factor approaches zero
slower than $m^{-1/4}$.
\end{mdframed}

To make things precise, we recall
\begin{lemma}[Lemma 23 of~\cite{Be3}]
\label{lem:BP1}
Let $\sigma\in\N$. Let $((\f,\la),\br)\in \Fo_{\sigma}^n \times \mathcal
B_{\sigma}$. Fix any tree $\tau$ of $\f$. Let $b\in\{-1,0\}$. We view
$(\tau,\la_{|\tau})$ as an element of $\Fo_1^{|\tau|}$ and denote by
$\q_\f\in\mathcal{Q}_n^\sigma$ and $\q_{\tau}\in\mathcal{Q}_{|\tau|}^1$ the
quadrangulations associated to $((\f,\la),\br)$ and
$((\tau,\la_{|\tau}),(0,b))$, respectively, through the Bouttier-Di
Francesco-Guitter bijection (the distinguished vertices are omitted). Then
$$
\dgh\left(V(\q_\f),V(\q_\tau)\right)\leq 2\left
(\max_{\f\setminus \mathring{\tau}}\La_\f-\min_{\f\setminus \mathring{\tau}}\La_\f +1\right),
$$
where $\mathring{\tau}$ stand for the tree $\tau$ without its root vertex, and
$\La_\f$ is the label function associated to $((\f,\la),\br)$ as
defined in Section~\ref{sec:contourlabel-finite}.
\end{lemma}
\begin{remark}
  As always, we interpret $V(\q_\f)$ and $V(\q_\tau)$ as pointed metric
  spaces (pointed at their root vertices and endowed with the graph
  distance). Note that~\cite[Lemma 23]{Be3} is formulated in terms of the
  unpointed Gromov-Hausdorff distance, but the proof carries over to the
  pointed version used here.
\end{remark}
Let $r\geq 0$. For the balls $B_r(\q_\f)$ and $B_r(\q_\tau)$ around
the root vertices, we claim that 
\begin{equation}
\label{eq:BP-1}
\dgh\left(B_r(\q_\f,B_r(\q_\tau\right) \leq 3
\dgh\left(V(\q_\f),V(\q_\tau)\right) +8.
\end{equation}
Indeed, we may first replace both $V(\q_\f)$ and $V(\q_\tau)$ by the
corresponding length spaces ${\mathbf Q}_\f$ and ${\mathbf Q}_\tau$ as
explained in Section~\ref{sec:locGH}. We obtain
$$
\left|\dgh\left(B_r(\q_\f),B_r(\q_\tau)\right)- \dgh\left(B_r({\mathbf
    Q}_\f),B_r({\mathbf Q}_\tau)\right)\right|\leq 2.
$$
For estimating the Gromov-Hausdorff distance on the right, we note that
every correspondence between ${\mathbf Q_\f}$ and
${\mathbf Q}_\tau$ satisfies the requirements of Lemma~\ref{lem:localGHconv}, so that by this lemma
$$
\dgh\left(B_r({\mathbf Q}_\f),B_r({\mathbf Q}_\tau)\right)\leq
(3/2)\inf_\cR\dis(\cR) = 3\dgh\left({\mathbf Q_\f},{\mathbf
    Q}_\tau\right)\leq 3\dgh\left(V(\q_\f),V(\q_\tau)\right)+ 6,
$$
where the infimum is taken over all
correspondences between ${\mathbf Q_\f}$ and ${\mathbf Q}_\tau$, and the
equality follows from the alternative description of the Gromov-Hausdorff distance
in terms of correspondences.
We are now in position to prove Theorem~\ref{thm:BP}.
\begin{proof}[Proof of Theorem~\ref{thm:BP}]
Recall from\cite[Theorem 4]{Be3} that 
$$
(V(Q_n^{\sigma_n}),(8/9)^{-1/4}n^{-1/4} \dgr,\rho_n) \xrightarrow[n \to \infty]{(d)}
\BM
$$  
in the Gromov-Hausdorff topology, where $\BM$ is the Brownian map. This
result immediately implies that when $a_n\gg n^{1/4}$, then
$(V(Q_n^{\sigma_n}),a_n^{-1}\dgr,\rho_n)$ converges to the trivial metric
space consisting of a single point, which proves the second part of
the theorem.

For the first part and the rest of this proof, we assume $\sqrt{\sigma_n}
\ll a_n \ll n^{1/4}$. We have to show that for each
  $r\geq 0$,
\begin{equation}
\label{eq:BP-2}
B_r\left(a_n^{-1}\cdot Q_n^{\sigma_n}\right)\xrightarrow[n \to \infty]{(d)}B_r(\BP)
\end{equation}
in distribution in $\mathbb{K}$. Let $((\f_n,\la_n),\br_n)$ be uniformly distributed in
$\Fo_{\sigma_n}^n \times \Br_{\sigma_n}$, and write $Q_n^{\sigma_n}$
for the (rooted and pointed) quadrangulation associated through the Bouttier-Di
Francesco-Guitter bijection, as usual. We denote by $\tau_\ast^{(n)}$
the largest tree of $\f_n$ (we take that with the smallest index if several trees attain
the largest size). We let $b_n\in\{-1,0\}$ be uniformly distributed and
independent of everything else and denote by $\hat{Q}_n$ the quadrangulation
encoded by $((\tau_\ast^{(n)},{\la_n}_{|\tau_\ast^{(n)}}),(0,b_n))$, in the same way as in Lemma~\ref{lem:BP1}.

We obtain from~\eqref{eq:BP-1} together with
Lemma~\ref{lem:BP1} that
\begin{equation}
\label{eq:BP-3}
\dgh\left(B_r(a_n^{-1}\cdot Q_n^{\sigma_n}),B_r(a_n^{-1}\cdot
        \hat{Q}_n)\right)\leq \frac{6}{a_n}\left(\max_{\f_n\setminus
      \mathring{\tau}_{\ast}^{(n)}}\La_{n}-\min_{\f_n\setminus
      \mathring{\tau}_{\ast}^{(n)}}\La_n\right)+o(1)
\end{equation}
as $n\rightarrow\infty$, where in the notation of Lemma~\ref{lem:BP1},
$\mathring{\tau}_{\ast}^{(n)}$ stands for the tree ${\tau}_{\ast}^{(n)}$
without its root, and $\La_n$ is the label function of
$((\f_n,\la_n),\br_n)$. We claim that the right hand side in the last
display converges to zero in probability.  In this regard, recall that
$$\La_n = \left(L_n(t) + \br_n(-\uC_n(t)),\,0\leq t\leq 2n+\sigma_n\right).$$
By Lemma~\ref{lem:bridge0}, the values of $\br_n$ are of order
$\sqrt{\sigma_n}\ll a_n$, so that we may replace $\La_n$ by $L_n$
in~\eqref{eq:BP-3}.  Denote by $\f'_n=\f_n\setminus \tau_{\ast}^{(n)}$ the
forest obtained from $\f_n$ by removing $\tau_{\ast}^{(n)}$, i.e., if
$\tau_{\ast}^{(n)}$ is the tree of $\f_n$ with index $i$, then
$\f'_n=(\tau_0^{(n)},\ldots,\tau_{i-1}^{(n)},\tau_{i+1}^{(n)},\ldots,\tau_{\sigma_n-1}^{(n)})$. We
let $\la'$ be the labeling of $\f_n$ restricted to $\f'_n$, and write
$(C'_n,L'_n)$ for the contour pair corresponding to $(\f'_n,\la'_n)$. We
view both $C'_n$ and $L'_n$ as continuous functions on $[0,\infty)$ by
letting $C'_n(s)=C'_n(s\wedge (2(n-|\tau_{\ast}^{(n)}|)+\sigma_n-1))$, and
similarly with $L'_n$. The convergence to zero of the right hand side
in~\eqref{eq:BP-3} is now a consequence of the following lemma.
\begin{lemma}
\label{lem:BP2}
In the notation from above, we have for sequences $a_n$ satisfying  $a_n\gg \sqrt{\sigma_n}$,
$$
\left(\frac{1}{a^2_n}C'_n,\frac{1}{a_n}L'_n\right)\xrightarrow[n\to
  \infty]{(p)} (0,0)\quad\textup{in }\mathcal{C}([0,\infty),\mathbb{R})^2.
$$
\end{lemma}
\begin{proof}
  Let $(\tilde{\tau}_i,(\tilde{\ell}_i(u))_{u\in\tau_i})$,
  $i=0,\ldots,\sigma_n-2$, be a sequence of $\sigma_n-1$ uniformly labeled
  critical geometric Galton-Watson trees. Consider the forest
  $\tilde{\f}_n=(\tilde{\tau}_0,\ldots,\tilde{\tau}_{\sigma_n-2})$ together
  with the labeling $\tilde{\la}_n$ given by $\tilde{\la}_n\restriction
  V(\tilde{\tau_i})=\tilde{\ell}_i$, for all $i$.  Let
  $(\tilde{C}_n,\tilde{L}_n)$ denote the contour pair associated to
  $(\tilde{\f}_n,\tilde{\la}_n)$, continuously extended to $[0,\infty)$
  outside $[0,2\sum_{i=0}^{\sigma_n-2}|\tilde{\tau}_i|+\sigma_n-1]$ as
  described above.

By Lemma~\ref{lem:GW2}, we can for each $\eps>0$ couple the pairs
$(C_n',L_n')$ and $(\tilde{C}_n,\tilde{L}_n)$ on the same probability space
such that with probability at least $1-\eps$, we have the equality
$$
(C'_n,L'_n)=(\tilde{C}_n,\tilde{L}_n)
$$
as elements of $\mathcal{C}([0,\infty),\mathbb{R})^2$, provided $n$ is
sufficiently large. Our claim therefore follows if 
\begin{equation}
\label{eq:BP-3a}
\left(\frac{1}{a^2_n}\tilde{C}_n,\frac{1}{a_n}\tilde{L}_n\right)\xrightarrow[n\to
  \infty]{(p)} (0,0).
\end{equation}
From Section~\ref{sec:forests}, we know that the law of $\tilde{C}_n$
agrees with that of a simple random walk started from $0$ and stopped upon
hitting $-(\sigma_n-1)$, with linear interpolation between integer
values. Donsker's invariance principle thus shows that
$((1/\sigma_n)\tilde{C}_n(\sigma_n^2t),t\geq 0)$ converges in distribution
to a standard Brownian motion $(B_{t\wedge T_{-1}}, t\geq 0)$ stopped upon
hitting $-1$. Arguments like in~\cite[Proof of Theorem 4.3]{LGMi} then
imply convergence of the finite-dimensional laws on
$\mathcal{C}([0,\infty),\R)^2$ of the tuple
$((1/\sigma_n)\tilde{C}_n(\sigma_n^2\cdot),(1/\sqrt{\sigma_n})\tilde{L}_n(\sigma_n^2\cdot))$,
and tightness of the second component follows {\it via} Kolmogorov's
criterion from moment bounds on $\tilde{C}_n$ as in~\cite[Lemma
2.3.1]{LGMi} (in our case, these bounds are in fact easier to establish,
since we consider an unconditioned random walk). We do not repeat the
arguments here, but refer the reader to~\cite{LGMi} or~\cite[Section
5]{Be1} for more details. We obtain the convergence in distribution
$$
\left(\frac{1}{\sigma_n}\tilde{C}_n(\sigma^2_n\cdot),\frac{1}{\sqrt{\sigma_n}}\tilde{L}_n(\sigma^2_n\cdot)\right)
\xrightarrow[n \to \infty]{(d)} \left(B_{\cdot\wedge T_{-1}}, Z\right)\quad\textup{in }\mathcal{C}([0,\infty),\mathbb{R})^2,
$$
where $Z=(Z_t,t\geq 0)$ is the Brownian snake driven by $(B_{t\wedge
  T_{-1}}, t\geq 0)$. Since $a^2_n\gg \sigma_n$, this last result implies
clearly~\eqref{eq:BP-3a} and hence the assertion of the lemma.
\end{proof}

Going back to~\eqref{eq:BP-3}, it remains to show that for
$\eps>0$, $F:\mathbb{K}\rightarrow\mathbb{R}$ continuous and bounded and $n\geq n_0$,
\begin{equation}
\label{eq:BP-4}
\left|\E\left[F\left(B_r\left(a_n^{-1}\cdot
        \hat{Q}_n\right)\right)\right]-\E\left[F(B_r(\BP))\right]\right|\leq \eps.
\end{equation}
Let $\delta>0$. We estimate
\begin{align*}
\lefteqn{\left|\E\left[F\left(B_r\left(a_n^{-1}\cdot
          \hat{Q}_n\right)\right)\right]-\E\left[F(B_r(\BP))\right]\right|
\leq 2\sup|F|\,\P\left(|\tau_{\ast}^{(n)}|\leq \delta n\right)}\\
&\quad + \sum_{k=\lceil
  \delta n\rceil}^n\P\left(|\tau_{\ast}^{(n)}|=k\right)\left|\E\left[F\left(B_r\left(a_n^{-1}\cdot
          \hat{Q}_n\right)\right)\,\Big|\,|\tau_{\ast}^{(n)}|=k\right]-\E\left[F(B_r(\BP))\right]\right|.
\end{align*}
For $\delta=\delta(F,\eps)>0$ small and $n=n(\delta,\eps)\in\N$
sufficiently large, we have by Lemma~\ref{lem:GW1}
$$
2\sup|F|\,\P\left(|\tau_{\ast}^{(n)}|\leq \delta n\right) \leq \eps/2.
$$
Concerning the summands in the second term, we note that conditionally on
$|\tau_{\ast}^{(n)}|=k$, the quadrangulation $\hat{Q}_n^{\sigma_n}$ is
uniformly distributed among all quadrangulations in $\cQ_k^{1}$, i.e., those
with $k$ inner faces and a boundary of size $2$. Removing the only edge of the
boundary which is not the root edge, we obtain a quadrangulation uniformly
distributed among all quadrangulations with $k$
faces and no boundary. Clearly, the removal of this edge does not change
the underlying metric space. By~\cite[Theorem
2]{CuLG}, we therefore get for $k\geq \lceil \delta n\rceil$ and
$n$ sufficiently large, recalling that $a_n \ll n^{1/4}$, 
$$\left|\E\left[F\left(B_r\left(a_n^{-1}\cdot
          \hat{Q}_n)\right)\right)\,\Big|\,|\tau_{\ast}^{(n)}|=k\right]-\E\left[F(B_r(\BP))\right]\right|\leq
  \eps/2.
$$
This shows~\eqref{eq:BP-4} and hence~\eqref{eq:BP-2}.
\end{proof}

\subsection{Coupling Brownian disk \& half-planes
  (Theorem~\ref{thm:coupling-BD-BHP} and Corollary~\ref{cor:topology-BHP})}
\label{sec:proof-coupling-BD-BHP}
Let us first show how Corollary~\ref{cor:topology-BHP} follows from
Theorem~\ref{thm:coupling-BD-BHP}.
\begin{proof}[Proof of Corollary~\ref{cor:topology-BHP}]
  Theorem~\ref{thm:coupling-BD-BHP} implies that with probability $1$, for
  every $r>0$, the ball $B_r(\BHP_\theta)$ is included in an open set of
  $\BHP_\theta$ homeomorphic to $\overline{\mathbb{H}}$. This shows that
  $\BHP_\theta$ is a simply connected topological surface with a boundary,
  and that this boundary is connected and non-compact: it must therefore be
  homeomorphic to $\R$. We construct a surface $S$ without boundary by
  gluing a copy $H$ of the closed half-plane $\overline{\mathbb{H}}$ to
  $\BHP_\theta$ along the boundary. This non-compact surface is still
  simply connected by van Kampens' Theorem, and in particular, it is
  one-ended. Therefore, it must be homeomorphic to $\R^2$,
  see~\cite{Ri}. Now if $\phi$ is a homeomorphism from the boundary of
  $\BHP_\theta$ to $\R$, then the Jordan-Schoenflies Theorem (in fact a
  simple variation of the latter) implies that $\phi$ can be extended to a
  homeomorphism $\overline{\phi}$ from $S$ to $\R^2$, and the two halves
  $\BHP_\theta$ and $H$ of $S$ must be sent {\it via} $\overline{\phi}$ to
  the two half-spaces $\overline{\mathbb{H}}$ and
  $-\overline{\mathbb{H}}$. In particular, $\overline{\phi}$ induces a
  homeomorphism from $\BHP_\theta$ to a closed half-plane, as wanted.
 \end{proof}
 We turn to Theorem~\ref{thm:coupling-BD-BHP}, and in this regard, we begin
 with proving the following weaker statement (compare with Proposition 4
 of~\cite{CuLG} for the Brownian map and plane).
\begin{prop}
  \label{prop:isometry-BD-BHP}
  Let $\eps>0$, $r\geq 0$. Let
  $\sigma(\cdot):(0,\infty)\rightarrow(0,\infty)$ be a function satisfying
  $\lim_{T\rightarrow\infty}\sigma(T)/T=\theta\in[0,\infty)$ and
  $\liminf_{T\rightarrow\infty}\sigma(T)/\sqrt{T}>0$. Then there exists
  $T_0=T_0(\eps,r,\sigma)$ such that for all $T\geq T_0$, one can construct
  copies of $\BD_{T,\sigma(T)}$ and $\BHP_\theta$ on the same probability
  space such that with probability at least $1-\eps$, the balls
  $B_r(\BD_{T,\sigma(T)})$ and $B_r(\BHP_\theta)$ of radius $r$ around the
  respective roots are isometric.
\end{prop}
Before proving Proposition~\ref{prop:isometry-BD-BHP}, we recapitulate for the
reader's convenience in the next section the definitions of $\BD_{T,\sigma}$
and $\BHP_\theta$.

\subsubsection{Brownian half-plane and disk}
\label{sec:recapBHPBD}
In order to ease the reading of the proofs which follow, we use a notation
which differs slightly from that in Section~\ref{sec:def}.  Let
$\sigma(\cdot):(0,\infty)\rightarrow(0,\infty)$ be a perimeter function as
given in the statement of Proposition~\ref{prop:isometry-BD-BHP}. Recall
that $\sigma(\cdot)$ is a function of the volume $T$ of the disk.

For defining the {\bf Brownian disk }$\BD_{T,\sigma(T)}$ of volume $T$ and boundary
length $\sigma(T)$, we consider a contour function $F=(F_t,0\leq t\leq T)$
and a label function $W=(W_t,0\leq t\leq T)$ given as follows.
\begin{itemize}
\item $F$ has the law of a first-passage Brownian
bridge on $[0,T]$ from $0$ to $-\sigma(T)$.
\item Given $F$, the function $W$ has same distribution as 
$(b_{-\underline{F}_t}+Z_t,0\leq t\leq T)$, where
\begin{itemize}
\item $(Z_t,t\in\R)=Z^{F-\underline{F}}$ is
      a continuous modification of the centered Gaussian process with covariances given by
      $$
      \E\left[Z_sZ_t\right]=\min_{[s\wedge t,s\vee t]}F-\underline{F},
      $$  
      with $\underline{F}_t=\inf_{[0,t]}F$.
\item $(b_x,0\leq x\leq \sigma(T))$ is a standard Brownian bridge
  with duration $\sigma(T)$ and scaled by the factor $\sqrt{3}$,
  independent of $Z^{F-\underline{F}}$.
\end{itemize}
\end{itemize}
The pseudo-metrics $d_{F}$ and $d_{W}$
on $[0,T]$ are given by
$$
d_{F}(s,t)=F_s+F_t-2\min_{[s\wedge t,s\vee t]}F,
$$
and
$$ 
d_{W}(s,t)=W_s+W_t-2\max\left(\min_{[s\wedge
    t,s\vee t]}
  W,\,\min_{[0,s\wedge t]\cup[s\vee t,T]}W\right).
$$
We shall write $D$ instead of $D_{F,W}$, i.e.,
$$
D(s,t)=\inf\left\{\sum_{i=1}^kd_{W}(s_i,t_i):\begin{array}{l}
k\geq 1, \, s_1,\ldots,s_k,t_1,\ldots,t_k\in I,s_1=s,t_k=t,\\ 
d_{F}(t_i,s_{i+1})=0\mbox{ for every }i\in \{1,\ldots,k-1\}
    \end{array}
    \right\}\, .
$$
The Brownian disk $\BD_{T,\sigma(T)}$ has the law of the pointed metric
space $([0,T]/\{D=0\},D,\rho)$, with $\rho$ being the equivalence class of
$0$.

\paragraph{}
The {\bf Brownian half-plane }$\BHP_{\theta}$, $\theta\in[0,\infty)$ is
given in terms of contour and label processes $\Xha=(\Xha_t,t\in\R)$ and
$\Wha=(\Wha_t,t\in\R)$ specified as follows:
\begin{itemize}
\item $(\Xha_t,t\geq 0)$ has the law of a one-dimensional Brownian motion
  $B=(B_t,t\geq 0)$ with
  drift $-\theta$ and $B_0=0$, and $(\Xha_{-t},t\geq 0)$ has the law of the Pitman
  transform of an independent copy of $B$.
\item Given $\Xha$, the (label) function $\Wha$ has same distribution as
    $(\gamma_{-\underline{X}_t^{\theta}}+\Zha_t,t\in \R)$, where
    \begin{itemize}
   \item given $\Xha$, $\Zha=(\Zha_t,t\in\R)=Z^{\Xha-\underline{X}^{\theta}}$ is a continuous
     modification of the centered Gaussian process with covariances given by
      $$
      \E\left[\Zha_s\Zha_t\right]=\min_{[s\wedge t,s\vee t]}\Xha-\underline{X}^{\theta},
      $$  
     with $\underline{X}_t^{\theta}=\inf_{[0,t]}\Xha$ for $t\geq 0$, and
     $\underline{X}_t^{\theta}=\inf_{(-\infty,t]}\Xha$ for $t<0$.
   \item $(\gamma_x,x\in \R)$ is a two-sided Brownian motion with
     $\gamma_0=0$ and scaled by the factor $\sqrt{3}$, independent of $\Zha$. 
     \end{itemize}
\end{itemize}
For notational simplicity, we include here the scaling factor
$\sqrt{3}$ already in the definition of $\gamma$. 
The pseudo-metrics
$d_{\Xha}$ and $d_{\Wha}$
on $\mathbb{R}$ are given by
$$
d_{\Xha}(s,t)=\Xha_s+\Xha_t-2\min_{[s\wedge t,s\vee t]}\Xha\quad d_{\Wha}(s,t)=\Wha_s+\Wha_t-2\min_{[s\wedge t,s\vee t]}\Wha,
$$
and we write $D_{\theta}$ instead of $D_{\Xha,\Wha}$, cf.~\eqref{eq:Dfg}, i.e.,
$$
D_{\theta}(s,t)=\inf\left\{\sum_{i=1}^kd_{\Wha}(s_i,t_i):\begin{array}{l}
k\geq 1, \, s_1,\ldots,s_k,t_1,\ldots,t_k\in I,s_1=s,t_k=t,\\ 
d_{\Xha}(t_i,s_{i+1})=0\mbox{ for every }i\in \{1,\ldots,k-1\}
    \end{array}
    \right\}\, .
$$
Then the Brownian half-plane $\BHP_\theta$ has the law of the pointed
metric space $(\R/\{D_{\theta}=0\},D_{\theta},\rho_{\theta})$, with 
$\rho_{\theta}$ being the equivalence class of $0$.

\begin{remark}
  Be aware that all the quantities in the definition of $\BD_{T,\sigma(T)}$
  depend on $T$ or $\sigma(T)$ (like $F,b,W,Z$ or the pseudo-metric $D$). The real
  $T$ measuring the volume will be chosen sufficiently large later on, but
  for the ease of reading, we mostly suppress $T$ from the
  notation.
\end{remark}

\subsubsection{Absolute continuity relation between contour functions}
A key step in proving Proposition~\ref{prop:isometry-BD-BHP} is to relate
the contour function $\Xha$ for $\BHP_\theta$ to the contour function $F$ for
$\BD_{T,\sigma(T)}$, in spirit of~\cite[Proposition 3]{CuLG}.

We fix once for all a perimeter function
$\sigma(\cdot):(0,\infty)\rightarrow(0,\infty)$ as given in the statement of
Proposition~\ref{prop:isometry-BD-BHP}, and let $\theta=\lim_{T\rightarrow\infty}\sigma(T)/T\in[0,\infty)$.

For given $T>0$, which we will choose large enough later on, we let $F$ be
a first-passage Brownian bridge on $[0,T]$ from $0$ to $-\sigma(T)$, $B$ a
one-dimensional Brownian motion on $[0,\infty)$ with drift $-\theta$ and
$B_0=0$, and, by a small abuse of notation, $\Pi$ the Pitman transform of
an independent copy of $B$.

Now assume $\alpha,\beta>0$ with $\alpha+\beta<T$. We consider the pair
$((F_t)_{0\leq t\leq \alpha},(F_{T-t})_{0\leq t\leq \beta})$ as an
element of the space $\mathcal{C}([0,\alpha],\mathbb{R})\times
\mathcal{C}([0,\beta],\mathbb{R})$. We write $(\omega,\omega')$ for the
generic element of this space.

We next introduce some probability kernels. Let $t>0$. For $x\in\mathbb{R}$, the
heat kernel is denoted
$$
p_t(x)=\frac{1}{\sqrt{2\pi t}}\exp\left(-\frac{x^2}{2t}\right).
$$
For $x,y>0$, the transition density of Brownian motion killed upon hitting
$0$ is given by 
$$
p_t^\ast(x,y) = p_t(y-x)-p_t(y+x).
$$
The density of the first hitting time of level $x>0$ of Brownian motion
started at $0$ is
$$
g_t(x) = \frac{x}{t}p_t(x).
$$
The transition density of a three-dimensional Bessel process takes the form
\begin{equation}
\label{eq:kernel-bessel}
r_t(x,y)=\left\{\begin{array}{l@{\quad\mbox{if }}l}
2yg_t(y)&  x=0\\
x^{-1}p_t^{\ast}(x,y)y& x,y>0
\end{array}\right..
\end{equation}
In~\cite[Theorem 1]{PiRo}, Pitman and Rogers show that the Pitman transform
of a one-dimensional Brownian motion with drift $-\theta$ has the law of the
radial part of a three-dimensional Brownian motion with a drift of
magnitude $\theta$. In particular, if $\theta=0$, it has the law of a
three-dimensional Bessel process, and for all $\theta \geq 0$, it is a
transient process. In~\cite[Theorem 3]{PiRo}, it is moreover shown that its
transition density is given by
\begin{equation}
\label{eq:kernel-qttheta}
q_t^{(\theta)}(x,y)= \exp\left(-(t/2)\theta^2\right)h^{-1}(x\theta)r_t(x,y)h(y\theta),
\end{equation}
where 
$$
h(x)=\left\{\begin{array}{l@{\quad\mbox{if }}l}
x^{-1}\sinh x&  x>0\\
1& x=0
\end{array}\right..
$$

Recall that $\sigma(\cdot)$ satisfies
$\lim_{T\rightarrow\infty}\sigma(T)/T=\theta\in[0,\infty)$,
$\liminf_{T\rightarrow\infty}\sigma(T)/\sqrt{T}>0$.
\begin{lemma}
\label{lem:abs-cont-F}
In the notation from above, the law of 
$$
\left((F_t)_{0\leq t\leq \alpha},(F_{T-t})_{0\leq t\leq \beta}\right)
$$
is absolutely continuous with respect to the law of
$$
\left((B_t)_{0\leq t\leq \alpha},(\Pi_t-\sigma(T))_{0\leq t\leq \beta}\right),
$$
with density given by the function
$$\varphi_{T,\alpha,\beta}(\omega,\omega') = \mathbbm{1}_{\{\omega_s >
  -\sigma(T)\textup{ for
  }s\in[0,\alpha]\}}(\omega)\,\,\frac{p^\ast_{T-(\alpha+\beta)}(\omega_\alpha
  +\sigma(T),
  \omega'_\beta+\sigma(T))}{2(\omega'_\beta+\sigma(T))g_T(\sigma(T))}\,\frac{\exp\left(\omega_\alpha\theta+\frac{\alpha+\beta}{2}\theta^2\right)}{h((\omega'_\beta+\sigma(T))\theta)}.
$$
Moreover, with $\P_{\alpha,\beta}$ denoting the joint (product) law of
$((B_t)_{0\leq t\leq \alpha},(\Pi_t)_{0\leq t\leq \beta})$, the following
holds true: For each $\eps >0$, there exists $T_0>0$ and a
measurable set $E=E(\eps,T_0)\subset
\mathcal{C}([0,\alpha],\mathbb{R})\times \mathcal{C}([0,\beta],\mathbb{R})$
with $\P_{\alpha,\beta}(E)\geq 1-\eps$ such that for $T\geq
T_0$,
$$
\sup_{(\omega,\omega'+\sigma(T))\in E}\left|\varphi_{T,\alpha,\beta}(\omega,\omega')-1\right| \leq \eps.
$$
\end{lemma}
Note that $\varphi_{T,\alpha,\beta}$ depends on the second
  coordinate $\omega'$ only through its endpoint $\omega'_\beta$. 
\begin{proof}
  We show that the finite-dimensional distributions of $F$ agree with
  those
  of $$\varphi_{T,\alpha,\beta}(\omega,\omega')\P_{\alpha,\beta}(\textup{d}\omega,\textup{d}\omega').$$
  Note that the law of the first-passage Brownian bridge $F$ is
  specified by $F_T=-\sigma(T)$ and
\begin{equation}
\label{eq:fpb}
\E\left[f\left((F_t)_{0\leq t\leq T'}\right)\right] 
= \E\left[f\left((\tilde{\gamma})_{0\leq t\leq
      T'}\right)\mathbbm{1}_{\{\underline{\tilde{\gamma}}_{T'}>-\sigma(T)\}}\frac{g_{T-T'}(\tilde{\gamma}_{T'}+\sigma(T))}{g_{T}(\sigma(T))}\right]
\end{equation}
for all $0\leq T'<T$ and all functions $f\in
\mathcal{C}([0,T'],\mathbb{R})$, where $\tilde{\gamma}$ is a standard
one-dimensional Brownian motion started from zero (without drift).

Let us next simplify notation. For $x\in\mathbb{R}$, write
$\tilde{x}=x+\sigma(T)$. For $0<t_1<t_2<\dots<t_p$ and $x_1,\dots,x_p>-\sigma(T)$, let
$$
G_{t_1,\dots,t_p}(x_1,\dots,x_p)=p_{t_1}^\ast(\sigma(T),\tilde{x}_1)p_{t_2-t_1}^\ast(\tilde{x}_1,\tilde{x}_2)\dots
p_{t_p-t_{p-1}}^\ast(\tilde{x}_{p-1},\tilde{x}_p).
$$
For $0<t'_1<t'_2<\dots<t'_q$ and $x_{p+1},\dots,x_{p+q}>-\sigma(T)$, let
$$
H_{t'_1,\dots,t'_q}(x_{p+q},\dots,x_{p+1})=g_{t'_1}(\tilde{x}_{p+q})p_{t'_2-t'_1}^\ast(\tilde{x}_{p+q},\tilde{x}_{p+q-1})\dots
p_{t'_q-t'_{q-1}}^\ast(\tilde{x}_{p+2},\tilde{x}_{p+1}).
$$
Now fix $0<t_1<t_2<\dots<t_p=\alpha$ and $0<t'_1<t'_2<\dots<t'_q=\beta$.
We infer from~\eqref{eq:fpb} that the density of the $(p+q)$-tuple
$(F_{t_1},\dots,F_{t_p},F_{T-t'_q},\dots,F_{T-t'_1})$ is given
by the function
\begin{equation}
\label{eq:abs-cont-F-density}
f_{t_1,\dots,t_p,t'_1,\dots,t'_q}(x_1,\dots,x_{p+q})=
G_{t_1,\dots,t_p}(x_1,\dots,x_p)H_{t'_1,\dots,t'_q}(x_{p+q},\dots,x_{p+1})\cdot
\frac{p^\ast_{T-(\alpha+\beta)}(\tilde{x}_p,\tilde{x}_{p+1})}{g_T(\sigma(T))}.
\end{equation}
From Girsanov's theorem, we know that the finite-dimensional laws 
$(B_{t_1},\dots,B_{t_p})$ of a one-dimensional Brownian motion $B$ with drift
$-\theta$ are absolutely continuous with respect to those of a standard
Brownian motion $\gamma$ without drift, with a density given by
$$
\exp\left(-\theta \gamma_{t_p}-\alpha\theta^2/2\right).
$$
Next, we see from~\eqref{eq:kernel-qttheta} that the
finite-dimensional laws of $(\Pi_{t_1'}-\sigma(T),\dots, \Pi_{t_q'}-\sigma(T))$ have
density
$$
\pi_{t'_1,\dots,t'_q}(x_{p+q},\dots,x_{p+1})=
2\tilde{x}_{p+1}\exp\left(-(\beta/2)\theta^2\right)h(\tilde{x}_{p+1}\theta)H_{t'_1,\dots,t'_q}(x_{p+q},\dots,x_{p+1}),
$$
for $x_{p+q},\dots,x_{p+1}>-\sigma(T)$. By~\eqref{eq:abs-cont-F-density} and
the last two observations, the first claim of the
statement follows. 

For the second, for every $\delta>0$, by continuity of $B$ and $\Pi$, we
can find a constant $K=K(\delta,\alpha,\beta)>0$ such that
\begin{equation}
\label{eq:Aalphabeta}
\P_{\alpha,\beta}\left(\min_{[0,\alpha]}B>-K,\,\max_{[0,\beta]}\Pi<K\right)\geq
1-\delta.
\end{equation}
The second claim now follows from~\eqref{eq:Aalphabeta} and the fact that
for every $\delta'>0$, if $T$ is large enough, we have
\begin{equation}
\label{eq:density1}
\left|\frac{p^\ast_{T-(\alpha+\beta)}(x+\sigma(T),
    y+\sigma(T))}{2(y+\sigma(T))g_T(\sigma(T))}\,\frac{\exp\left(x\theta+\frac{\alpha+\beta}{2}\theta^2\right)}{h((y+\sigma(T))\theta)}
-1\right|\leq \delta'
\end{equation}
uniformly in $x\in\mathbb{R}$ with $|x| \leq K$ and
$y\geq -\sigma(T)$ with $|y+\sigma(T)|\leq K$. The last display in turn
follows from a straightforward but somewhat tedious calculation; we give
some indication for the case
$\lim_{T\rightarrow\infty}\sigma(T)/T=\theta>0$. First, as $T\rightarrow\infty$,
\begin{align*}
  \lefteqn{\frac{p^\ast_{T-(\alpha+\beta)}(x+\sigma(T),
      y+\sigma(T))}{2(y+\sigma(T))g_T(\sigma(T))}\,\frac{\exp\left(x\theta+\frac{\alpha+\beta}{2}\theta^2\right)}{h((y+\sigma(T))\theta)}}\\
  &\sim\left(\frac{\exp\left(-\frac{(y-x)^2}{2(T-(\alpha+\beta))}\right)-\exp\left(-\frac{(x+y+2\sigma(T))^2}{2(T-(\alpha+\beta))}\right)}{\exp\left(-\theta^2T/2\right)}\right)\frac{\exp\left(x\theta+\frac{\alpha+\beta}{2}\theta^2\right)}{2\sinh\left(\theta(y+\sigma(T))\right)}.
\end{align*}
Then, uniformly in $x$ and $y$ as specified above, we find 
$$
\exp\left(-\frac{(y-x)^2}{2(T-(\alpha+\beta))}+\theta^2T/2\right)\sim \exp\left((-x+y+\sigma(T))\theta-(\alpha+\beta)\theta^2/2\right),
$$
and
$$
\exp\left(-\frac{(x+y+2\sigma(T))^2}{2(T-(\alpha+\beta))}+\theta^2T/2\right)
\sim \exp\left((-x-y-\sigma(T))\theta-(\alpha+\beta)\theta^2/2\right).
$$
Putting these three estimates together,~\eqref{eq:density1} follows. The
case $\lim_{T\rightarrow\infty}\sigma(T)/T=0$ with $\liminf_{T\rightarrow\infty}\sigma(T)/\sqrt{T}>0$
is similar but easier (note that the expression for
$\varphi_{T,\alpha,\beta}$ simplifies when $\theta=0$).
\end{proof}

We need a similar absolute continuity property for the Brownian bridge
$b$ on $[0,\sigma(T)]$ from $0$ to $0$ with respect to two
independent linear Brownian motions $\gamma$ and $\gamma'$ scaled by $\sqrt{3}$.  Let $\alpha,\beta>0$
such that $\alpha+\beta<\sigma(T)$.
\begin{lemma}
\label{lem:abs-cont-b}
The law of 
$$
\left((b_t)_{0\leq t\leq \alpha},(b_{L-t})_{0\leq t\leq \beta}\right)
$$
is absolutely continuous with respect to the law of
$$
\left((\gamma_t)_{0\leq t\leq \alpha},(\gamma'_{t})_{0\leq t\leq \beta}\right),
$$
with density given by the function
$$\tilde{\varphi}_{T,\alpha,\beta}(\omega,\omega') = \frac{p_{\sigma(T)-(\alpha+\beta)}(\omega'_\beta-\omega_\alpha)}{p_{\sigma(T)}(0)}.
$$
Moreover, with $\P_{\alpha,\beta}$ denoting the joint law of
$((\gamma_t)_{0\leq t\leq \alpha},(\gamma'_{t})_{0\leq t\leq \beta})$, the
following holds true: For each $\eps >0$, there is $T_0>0$ and a
measurable set $E=E(\eps,T_0)\subset \mathcal{C}([0,\alpha],\mathbb{R})\times
\mathcal{C}([0,\beta],\mathbb{R})$ with $\P_{\alpha,\beta}(E)\geq 1-\eps$ such that
for $T\geq T_0$,
$$
\sup_{(\omega,\omega')\in E}\left|\tilde{\varphi}_{T,\alpha,\beta}(\omega,\omega')-1\right| \leq \eps.
$$
\end{lemma}
\begin{proof}
The first part is immediate from the fact that the law of the Brownian bridge $b$ is
  specified by $b_{\sigma(T)}=0$ and
\begin{equation}
\E\left[f\left((b_t)_{0\leq t\leq T'}\right)\right] 
= \E\left[f\left((\gamma_t)_{0\leq t\leq
      T'}\right)\frac{p_{\sigma(T)-T'}(\gamma_{T'})}{p_{\sigma(T)}(0)}\right]
\end{equation}
for all $0\leq T'<\sigma(T)$ and all $f\in
\mathcal{C}([0,T'],\mathbb{R})$. The proof of the second part is very similar to that
of Lemma~\ref{lem:abs-cont-F}. We omit the details.
\end{proof}

\subsubsection{Isometry of balls in {\normalfont$\
    \BD_{T,\sigma(T)}$} and {\normalfont $\BHP_\theta$}}
Recall the definition of $\sigma(T)$ and $\theta$ from the statement of the
proposition. In the proof that follows, we will choose $T>0$ sufficiently
large. We work with the following processes:
\begin{itemize}
\item $F$ a first-passage
Brownian bridge on $[0,T]$ from $0$ to $-\sigma(T)$;
\item $b$ a Brownian bridge on
  $[0,\sigma(T)]$ from $0$ to $0$, multiplied by $\sqrt{3}$, independent of $F$;
\item $B$ a Brownian motion on
  $[0,\infty)$ with drift $-\theta$, started from $B_0=0$;
\item $\Pi$ the Pitman transform of an independent copy of $B$;
\item $\gamma$ a two-sided Brownian motion on $\mathbb{R}$ with
  $\gamma_0=0$, scaled by the factor $\sqrt{3}$, independent of $(B,\Pi)$;
\item $Z,W$ and $\Zha,\Wha$ the random processes associated with 
  $F,b$ and $B,\Pi,\gamma$ as described in Section~\ref{sec:recapBHPBD}. 
\end{itemize}

\begin{proof}[Proof of Proposition~\ref{prop:isometry-BD-BHP}]
For $x\in\mathbb{R}$, let
$$\eta_{\textup{l}}(x)=\inf\{t\geq 0:B_t\leq -x\},\quad\eta_{\textup{r}}(x)=\sup\{t\geq
0:\Pi_t=x\}.$$
We fix $\eps>0$ and $r\geq 0$ and first introduce some auxiliary events.
For $A>0$, define 
\begin{equation*}
\mathcal{E}^1(A)=\left\{\begin{split}&\min_{[0,A]}\gamma <
    -6r,\,\min_{[A,A^2]}\gamma <
    -6r,\,\min_{[A^2,A^3]}\gamma < -6r,\\
    &\min_{[-A,0]}\gamma < -6r,\,\min_{[-A^2,-A]}\gamma <
    -6r,\,\min_{[-A^3,-A^2]}\gamma < -6r\end{split}\right\}.
\end{equation*}
Next, for $u_0>0$, $A>0$, let
$$\mathcal{E}^2(A,u_0)=\left\{\eta_{\textup{l}}(A^3)\leq
  u_0\right\},\quad \mathcal{E}^3(A,u_0)=\left\{\eta_{\textup{r}}(A^3)\leq u_0\right\}.
$$
For $u_2\geq u_1>0$, let
$$
\mathcal{E}^4(u_1,u_2)=\left\{\inf_{[u_2,\infty)}\Pi>\min_{[u_1,u_2]}\Pi\right\}.
$$
For $u_4\geq u_3>0$ and $T\geq u_4$, let
$$
\mathcal{E}^5(u_3,u_4,T)=\left\{\min_{[0,T-u_4]}F>\min_{[T-u_4,T-u_3]}F\right\}.
$$
Standard properties of Brownian motion imply that there exist $A>0$
such that $\P(\mathcal{E}^1)\geq 1-\eps/10$, and we fix $A$
accordingly. Then, we can find $u_0>0$ such that
$\P(\mathcal{E}^2)\geq 1-\eps/10$ and $\P(\mathcal{E}^3)\geq
1-\eps/10$, due to the fact that $\Pi$ is
transient.  Then, we can find $u_1$ and $u_2$ with $u_2\geq u_1\geq
u_0$ such that $\P(\mathcal{E}^4)\geq 1-\eps/10$. 

At last,we claim that we can find $u_4$, $u_3$ satisfying $u_4\geq u_3\geq
u_2$ and $T'_0$ with $T'_0\geq 2u_4$ such that for $T\geq T'_0$,
$\P(\mathcal{E}^5)\geq 1-\eps/10$. To see this, one can use the fact that
$F+\sigma(T)$ is a bridge of a three-dimensional Bessel process from
$\sigma(T)$ to $0$ with duration $T$, see, e.g.,~\cite{BeChPi}.  Its
time-reversal $\tilde{F}=F(T-\cdot)+\sigma(T)$ is then a Bessel bridge from
$0$ to $\sigma(T)$ with duration $T$, see~\cite[Chapter XI $\S 3$]{ReYo}. Write
$\tilde{F}^1$ for a Bessel bridge from $0$ to $\sigma(T)/\sqrt{T}$ with
duration $1$. Using this representation we have
\begin{align}
\label{eq:E5}
1-\P(\mathcal{E}^5(u_3,u_4,T))&=\P\left(\min_{[u_3,u_4]}\tilde{F}\geq\min_{[u_4,T]}\tilde{F}\right)
=  \P\left(\min_{[u_3/T,u_4/T]}\tilde{F}^1\geq\min_{[u_4/T,1]}\tilde{F}^1\right)\nonumber\\
&=\P\left(\min_{[u_3/T,u_4/T]}\tilde{F}^1\geq\min_{[u_4/T,1/2]}\tilde{F}^1\right)+o(1)\end{align}
as $T\rightarrow\infty$, where we used scaling at the second step and for
the last equality that $u_4/T\rightarrow 0$ for a fixed $u_4$ as
$T\to\infty$. Note that the law of
$(\tilde{F}^1_{t_1},\ldots,\tilde{F}^1_{t_k})$ for $0<t_1<\dots<t_k<1$ has
density
$$r_{t_1}(0,x_1)r_{t_2-t_1}(x_1,x_2)\dots
r_{t-t_k}\left(x_k,\sigma(T)/\sqrt{T}\right)\Big/r_1\left(0,\sigma(T)/\sqrt{T}\right),$$
see again~\cite[Chapter XI $\S 3$]{ReYo}. For a moment, let us denote by
$\tilde{\Pi}^1$ the Pitman transform of a Brownian motion with drift
$-\sigma(T)/\sqrt{T}$. Its transition kernel is given by
$q_t^{(\sigma(T)/\sqrt{T})}$, see~\eqref{eq:kernel-qttheta} above. A small
calculation involving the last display and the explicit form of
$q_t^{(\sigma(T)/\sqrt{T})}$ shows that the Bessel bridge $\tilde{F}^1$
restricted to $[0,1/2]$ is absolutely continuous with respect to
$\tilde{\Pi}^1$, with a density that can be written as
$\Xi(\tilde{\Pi}_{1/2}^1)$ for some measurable and bounded function
$\Xi$. The probability on the right hand side in~\eqref{eq:E5} is therefore
bounded from above by
\begin{align*}
  \|\Xi\|_\infty\P\left(\min_{[u_3/T,u_4/T]}\tilde{\Pi}^1\geq\min_{[u_4/T,1/2]}\tilde{\Pi}^1\right)+o(1)
&=\|\Xi\|_\infty\P\left(\min_{[u_3,u_4]}\Pi\geq\min_{[u_4,T/2]}\Pi\right)+o(1)\\
&=\|\Xi\|_\infty\P\left(\min_{[u_3,u_4]}\Pi\geq\inf_{[u_4,\infty)}\Pi\right)+o(1),
\end{align*}
where for the first equality, we used scaling again and the fact that
$\sigma(T)/T\rightarrow\theta$, and for the second equality transience of $\Pi$. Identically to the event
$\mathcal{E}^4$, transience of $\Pi$ implies that the last probability on the right can be made as small as we wish if we
choose $u_4$ large enough. This shows that for each choice of $u_3\geq
u_2$, we find $u_4\geq u_3$ and $T'_0$ such that $\mathcal{E}^5(u_3,u_4,T)\geq 1-\eps$
for all $T\geq T'_0$.

We now fix numbers $A,u_4\geq u_3\geq\dots\geq u_0$ and
$T'_0$ as discussed above. By Lemmas~\ref{lem:abs-cont-F} and~\ref{lem:abs-cont-b},
we deduce that we can find $T_0>T'_0$ such that every for
$T\geq T_0$, the processes $F,b,B,\Pi,\gamma$ can be
coupled on the same probability space such that the event
$$
\mathcal{E}^6(T)=\left\{
  \begin{split}&
    F_t=B_t,\,F_{T-t}=\Pi_t-\sigma(T) &\hbox{
  for }t\in[0,u_4]\\
&b_x =
\gamma_x,\,b_{L-x}=\gamma_{-x} &\hbox{ for }x\in [0,A^3]
\end{split}\right\}.
$$
has probability at least $1-\eps/2$, $F$ is independent of $b$,
$B$ is independent of $\Pi$, and $\gamma$ is independent of $(B,\Pi)$.

Now fix $T\geq T_0$. We assume
that $F,b,B,\Pi,\gamma$ have been coupled as above. Recall that the
snake $(Z_t,0\leq t\leq T)$ and the label function $(W_t,0\leq t\leq
T)$ of the Brownian disk $\BD_{T,\sigma(T)}$ are defined in terms of $F$ and
$b$, see Section~\ref{sec:recapBHPBD}. We put
$$
\mathbb{W}_t=\left\{\begin{array}{l@{\quad\mbox{if }}l}
    W_t& t\in[0,u_1]\\
    W_{T+t} & t\in[-u_1,0]
\end{array}\right..
$$
Given $F$, the process
$(\mathbb{W}_t)_{t\in[-u_1,u_1]}$ is Gaussian; moreover, if we restrict
ourselves to the event $\mathcal{E}^5$, we have
$${\underline F}_{T-t}=\min_{[T-u_4,T-t]}F\quad\hbox{for }t\in [0,u_1].$$
Hence the covariance of $(\mathbb{W}_t)_{t\in[-u_1,u_1]}$
is on $\mathcal{E}^5$ a function of the tuple
\begin{equation}
\label{eq:tuple1}
\left((F_t)_{0\leq t\leq u_4},(F_{T-t})_{0\leq t\leq u_4}\right).
\end{equation}
We turn to the Brownian half-plane and its label function
$\Wha=(\Wha_t,t\in\mathbb{R})$, which we define in terms of $B$, $\Pi$ and
$\gamma$, see again Section~\ref{sec:recapBHPBD}. Conditionally on $(B,
\Pi)$, $\Wha$ is a Gaussian process. Moreover, since on the event
$\mathcal{E}^4$,
$$
\inf_{[t,\infty)}\Pi = \min_{[t,u_4]}\Pi\quad\hbox{for }t\in[0,u_1],
$$
the covariance of the restriction of $\Wha$ to $[-u_1,u_1]$ is on
$\mathcal{E}^4$ given by exactly the same 
function of the tuple $((B_t)_{0\leq t\leq u_4},(\Pi_t)_{0\leq t\leq
  u_4})$ as the covariance of
$(\mathbb{W}_t)_{t\in[-u_1,u_1]}$. Since a shift of $\Pi$ does not
affect the covariance, we can instead consider the tuple
\begin{equation}
\label{eq:tuple2}
\left((B_t)_{0\leq t\leq u_4},(\Pi_t-\sigma(T))_{0\leq t\leq u_4}\right).
\end{equation}
On the event $\mathcal{E}^5$, both tuples of processes~\eqref{eq:tuple1}
and~\eqref{eq:tuple2} coincide. On the event
$\mathcal{E}^4\cap\mathcal{E}^5\cap\mathcal{E}^6$, we can therefore
construct $W$ and $\Wha$ such that
\begin{equation}
\label{eq:idWWla}
W_t = \Wha_t,\quad W_{T-t} = \Wha_{-t}\quad\hbox{ for all }t\in[0,u_1],\quad
\end{equation}
We shall now work on the event
$\mathcal{F}=\mathcal{E}^1\cap\mathcal{E}^2\cap\mathcal{E}^3\cap\mathcal{E}^4\cap\mathcal{E}^5\cap\mathcal{E}^6$,
which has probability at least $1-\eps$, and assume that the
identity~\eqref{eq:idWWla} holds true.

We follow a strategy similar to~\cite{CuLG}. Let $s,t\in [0,T]$. If
either $s,t\in[0,T/2]$, or $s,t\in[T/2,T]$, we let
$$
\tilde{d}_{W}(s,t) = W_s +W_t -2\min_{[s\wedge t,
  s\vee t]}W.
$$
Otherwise, we set
$$
\tilde{d}_{W}(s,t) = W_s +W_t -2\min_{[0,s\wedge
  t]\cup[s\vee t,T]}W.
$$
We shall need the continuous analog of the cactus bound~\eqref{eq:cactus1}
for the pseudo-metric $D$ belonging to the Brownian disk
$\BD_{T,\sigma(T)}$. It reads
\begin{equation}
\label{eq:cactusDBD}
D(s,t) \geq W_s+W_t-2\max\left\{\min_{[s\wedge t,s\vee
t]}W,\min_{[0,s\wedge t]\cup[s\vee t,T]}W\right\},\quad{s,t\in[0,T]}.
\end{equation}
See, for example,~\cite{CuLG} for a proof of the corresponding bound in the
context of the Brownian map, which can easily be adapted to the Brownian
disk. In the notation from above, we have the following
\begin{lemma}
\label{lem:DBD}
 Assume $\mathcal{F}$ holds.  
\begin{enumerate}
\item For every $t\in [\eta_{\textup{l}}(A),T-\eta_{\textup{r}}(A)]$, $D(0,t) >
  r$.
\item For every $s,t\in [0,\eta_{\textup{l}}(A)]\cup[0,T-\eta_{\textup{r}}(A)]$ with
  $\max\{D(0,s), D(0,t)\}\leq r$, it holds that
\begin{equation}
\label{eq:DBD-rhs}
D(s,t) =\inf_{s_1,t_1,\dots,s_k,t_k}\sum_{i=1}^k \tilde{d}_{W}(s_i,t_i),
\end{equation}
where the infimum is over all possible choices of $k\in\N$ and reals
$s_1,\dots,s_k,t_1,\dots,t_k\in
[0,\eta_{\textup{l}}(A^2)]\cup[T-\eta_{\textup{r}}(A^2),T]$ such that
$s_1=s,t_k=t$, and $d_{F}(t_i,s_{i+1})=0$ for $1\leq i\leq k-1$.
\end{enumerate}
\end{lemma}
\begin{proof}
  (a) If $t\in [\eta_{\textup{l}}(A),T-\eta_{\textup{r}}(A)]$, then by the cactus
  bound~\eqref{eq:cactusDBD} in the first inequality,
\begin{equation}
\label{eq:bound-DBD1}
D(0,t)\geq W_t-2\max\left\{\min_{[0,t]}W,\min_{[t,T]}W\right\}\geq
  -\max\left\{\min_{[0,\eta_{\textup{l}}(A)]}W,\min_{[T-\eta_{\textup{r}}(A),T]}W\right\}.
\end{equation}
Let us show how to bound the first minimum on the right hand side. On the
event $\mathcal{E}^6$, $b_x=\gamma_x$ for $x\in[0,A^3]$ and $F_t=B_t$ for $t\in[0,u_4]$. On
the event $\mathcal{E}^2$, we know that $\eta_{\textup{l}}(A)$, the first
instant when $B$ attains the value $-A$, is bounded from above by $u_0$,
which satisfies $u_0 \leq u_1\leq u_4$. Moreover, on $\mathcal{E}^1$,
$\min_{[0,A]}\gamma < -6r$. Going back to the definition of $W$ (and using
the fact that $Z_t$ equals zero if $\underline{F}$ attains a new minimum at
$t$), we obtain that the first minimum on the right hand side is bounded
from above by $-6r$.

For the second minimum on the right of~\eqref{eq:bound-DBD1}, we first
observe that on the event $\mathcal{E}^6$, we have also
$b_{L-x}=\gamma_{-x}$ for $x\in[0,A^3]$ and $F_{T-t}=\Pi_t-\sigma(T)$
for $t\in [0,u_4]$. Now on $\mathcal{E}^5\cap\mathcal{E}^6$, we have
$$\underline{F}_{T-t}=\min_{[T-u_4,T-t]}F=\min_{[t,u_4]}(\Pi-\sigma(T))\quad\hbox{
  for } t\in[0,u_1].$$ But on $\mathcal{E}^3$, $\eta_{\textup{r}}(A)\leq u_0\leq
u_1$, so that in particular 
$$
\underline{F}_{T-\eta_{\textup{r}}(A)}
=\min_{[\eta_{\textup{r}}(A),u_4]}(\Pi-\sigma(T))\geq A-\sigma(T),$$ where
for the last inequality we used the fact that $\Pi_t\geq A$ for $t\geq
\eta_{\textup{r}}(A)$. Since on $\mathcal{E}^1$, also $\min_{[-A,0]}\gamma <
-6r$, the second minimum is bounded above again by $-6r$. This proves
$D(0,t)\geq 6r$ whenever $t\in [\eta_{\textup{l}}(A),T-\eta_{\textup{r}}(A)]$,
which is more than we claimed.\\
(b) Recall that
\begin{equation}
\label{eq:dist-DBD}
D(s,t)=\inf\left\{\sum_{i=1}^k d_{W}(s_i,t_i) :\begin{split}& k\geq 1,
    s_1,\dots,s_k,t_1,\dots,t_k\in[0,T],s_1=s, t_k=t,\\
    & d_{F}(t_i,s_{i+1})=0\hbox{ for every
    }i\in\{1,\ldots,k-1\}\end{split}\right\}.
\end{equation}
Since $D(s,t)\leq D(0,s)+D(0,t)\leq 2r$ for
$s,t$ as in the statement, it suffices to look at
$s_1,\dots,s_k,t_1,\dots,t_k\in[0,T]$ with
\begin{equation}
\label{eq:bound-Dcirc}
\sum_{i=1}^k d_W(s_i,t_i) \leq 3r.
\end{equation}
We now argue that on the right hand side of~\eqref{eq:dist-DBD}, we can
restrict ourselves to reals
$s_1,\dots,s_k,t_1,\dots,t_k\in[0,\eta_\ell(A^2)]\cup[T-\eta_{\textup{r}}(A^2),T].$ Suppose that there is
$i\in\{1,\ldots,k\}$ such that $t_i$ is not included in $[0,\eta_\ell(A^2)]\cup[T-\eta_{\textup{r}}(A^2),T]$.
Note that from the cactus bound and the fact that $W_0=W_T=0$, we have 
$|W_s|\leq r$ whenever $D(0,s)\leq r$. Therefore, the cactus bound gives
$$
D(s,t_i) \geq -r -\max\left\{\min_{[s\wedge {t_i},s\vee
t_i]}W,\min_{[0,s\wedge t_i]\cup[s\vee t_i,T]}W\right\}.
$$
Recall that by assumption $s \in [0,\eta_\ell(A)]\cup[T-\eta_{\textup{r}}(A),T]$. If 
$t_i\in[0,T]\setminus([0,\eta_\ell(A^2)]\cup[T-\eta_{\textup{r}}(A^2),T])$, then both minima
on the right hand side are taken over subsets which include either
$[\eta_\ell(A),\eta_\ell(A^2)]$ or $[T-\eta_{\textup{r}}(A^2),T-\eta_{\textup{r}}(A)]$. Therefore,
\begin{equation}
\label{eq:bound-DBD2}
D(s,t_i) \geq -r -
\max\left\{\min_{[\eta_\ell(A),\eta_\ell(A^2)]}W,\min_{[T-\eta_{\textup{r}}(A^2),T-\eta_{\textup{r}}(A)]}W\right\}.
\end{equation}
We can now argue similarly to (a) to show that both minima are bounded
from above by $-6r$. Concerning the first minimum, we know on
$\mathcal{E}^6$ that $b_x=\gamma_x$ for $x\in[0,A^3]$ and $F_t=B_t$ for $t\in[0,u_4]$. On
$\mathcal{E}^2$, we have $\eta_{\textup{l}}(A^2)\leq u_0$. Since
$[-A^2,-A]\subset B([\eta_{\textup{l}}(A),\eta_{\textup{l}}(A^2)]$, the
bound on the first minimum follows from the fact that on $\mathcal{E}^1$,
$\min_{[A,A^2]}\gamma < -6r$. The second minimum is treated similarly and
left to the reader.

With these bounds, we obtain $D(s,t_i) \geq 5r$. On the other hand,
we know from~\eqref{eq:bound-Dcirc} that $D(s,t_i) \leq 3r$, a
contradiction. The case where $s_i$ is not included in
$[0,\eta_\ell(A^2)]\cup[T-\eta_{\textup{r}}(A^2),T]$ for some $i\in\{1,\ldots,k\}$ is analogous.

Therefore, we can restrict ourselves in~\eqref{eq:dist-DBD} to reals
$s_1,\dots,s_k,t_1,\dots,t_k\in[0,\eta_\ell(A^2)]\cup[T-\eta_{\textup{r}}(A^2),T]$. We
still have to show that we can replace $d_{W}$
in~\eqref{eq:dist-DBD} by $\tilde{d}_{W}$.  Let
$s_1,\dots,s_k,t_1,\dots,t_k\in[0,\eta_\ell(A^2)]\cup[T-\eta_{\textup{r}}(A^2),T]$
with $s_1=s$, $t_k=t$ and such that~\eqref{eq:bound-Dcirc} holds. Assume first
that there is $i\in\{1,\ldots,k\}$ such that $s_i\in[0,\eta_\ell(A^2)]$ and
$t_i\in[T-\eta_{\textup{r}}(A^2),T]$, and let us show that then
$d_{W}(s_i,t_i)=\tilde{d}_{W}(s_i,t_i)$. First,
by~\eqref{eq:bound-Dcirc} in the first inequality,
$$
 3r\geq d_{W}(s,s_i)\geq W_s-W_{s_i}.
$$
Since $W_s\geq -r$, this shows $W_{s_i}\geq -4r$, and identically one
obtains $W_{t_i}\geq -4r$. Using again~\eqref{eq:bound-Dcirc}, 
\begin{align*}
3r\geq d_{W}(s_i,t_i)&=W_{s_i} +W_{t_i} - 2\max\left\{\min_{[s_i,t_i]}
  W,\,\min_{[0,s_i]\cup[t_i,T]}W\right\}\\
&\geq -8r-2\max\left\{\min_{[s_i,t_i]}
  W,\,\min_{[0,s_i]\cup[t_i,T]}W\right\}
\end{align*}
We claim that this last inequality can only hold if the maximum is attained
at the second minimum (which means precisely
$d_{W}(s_i,t_i)=\tilde{d}_{W}(s_i,t_i)$). Indeed, if
$s_i\in[0,\eta_\ell(A^2)]$ and $t_i\in[T-\eta_{\textup{r}}(A^2),T]$, then
$[s_i,t_i]$ contains the interval
$[\eta_\ell(A^2),\eta_\ell(A^3)]$. Arguing in the same way as for
the first minimum in~\eqref{eq:bound-DBD2}, we deduce that
$\min_{[s_i,t_i]}W\leq -6r$, which proves our claim.

The case where $t_i\in[0,\eta_\ell(A^2)]$ and
$s_i\in[T-\eta_{\textup{r}}(A^2),T]$ is treated by symmetry. Assume now
both $s_i,t_i$ lie in $[0,\eta_\ell(A^2)]$. Then the interval $[s_i\vee
t_i,T]$ contains the interval $[\eta_\ell(A^2),\eta_\ell(A^3)]$, so that
$\min_{[s_i\vee t_i,T]}W\leq -6r$ by the same reasoning, which gives again
$d_{W}(s_i,t_i)=\tilde{d}_{W}(s_i,t_i)$. If both $s_i,t_i$ lie in
$[T-\eta_{\textup{r}}(A^2),T]$, then $[0,s_i\wedge t_i]$ contains
$[T-\eta_{\textup{r}}(A^3),T-\eta_{\textup{r}}(A^2)]$, and the minimum of
$W$ over this interval is again bounded from above by $-6r$, using
arguments as for the second minimum in~\eqref{eq:bound-DBD1}
(or~\eqref{eq:bound-DBD2}). This leads to
$d_{W}(s_i,t_i)=\tilde{d}_{W}(s_i,t_i)$ also in this case, which completes
the proof of (b).
\end{proof}

We turn to the analogous statement for the pseudo-distance function
$D_{\theta}$ of the Brownian half-plane. Recall the definition of
$(\Xha,\Wha)$ for the case of the Brownian half-plane $\BHP_\theta$,
cf. Section~\ref{sec:recapBHPBD}.
\begin{lemma} 
\label{lem:DH}
Assume $\mathcal{F}$ holds.
\begin{enumerate}
\item For every $t'\in (-\infty,-\eta_{\textup{r}}(A)]\cup[\eta_{\textup{l}}(A),\infty)$,
  $D_{\theta}(0,t') > r$.
\item For every $s',t'\in [-\eta_{\textup{r}}(A),\eta_{\textup{l}}(A)]$ with
  $\max\{D_{\theta}(0,s'), D_{\theta}(0,t')\}\leq r$, it holds that
\begin{equation}
\label{eq:DH-rhs}
D_{\theta}(s',t') =\inf_{s'_1,t'_1,\dots,s'_k,t'_k}\sum_{i=1}^k d_{\Wha}(s'_i,t'_i),
\end{equation}
where the infimum is over all possible choices of $k\in\N$ and reals
$s'_1,\dots,s'_k,t'_1,\dots,t'_k\in [-\eta_{\textup{r}}(A^2),\eta_{\textup{l}}(A^2)]$ such that
$s'_1=s',t'_k=t'$, and $d_{\Xha}(t'_i,s'_{i+1})=0$ for $1\leq i\leq k-1$.
\end{enumerate}
\end{lemma}
\begin{proof}
Essentially, one can rely on the identity~\eqref{eq:idWWla} and then follow
the proof of Lemma~\ref{lem:DBD}. The starting point is the cactus
bound for $D_{\theta}$, which reads
$$
D_{\theta}(s',t')\geq \Wha_{s'}+\Wha_{t'}-2\min_{[s'\wedge t',s'\vee
  t']}\Wha,\quad s',t'\in\mathbb{R}.
$$
We refer again to the proof in~\cite{CuLG}, which can be transferred to
this setting.

Let us now sketch how to prove (a); the proof of (b) is left to the
reader. If $t'\in (-\infty,-\eta_{\textup{r}}(A)]$, the cactus bound gives
$$
D_{\theta}(0,t')\geq -\min_{[t',0]}\Wha.
$$
The very definitions of $\Wha$ and $\eta_{\textup{r}}(A)$ together with the fact that on
$\mathcal{E}^1$, $\min_{[-A,0]}\gamma<-6r$, entail that the minimum is
bounded from above by $-6r$. The same bound holds if $t'\in
[\eta_{\textup{l}}(A),\infty)$, which proves (a).
\end{proof}
Combining Lemmas~\ref{lem:DBD} and~\ref{lem:DH}, we obtain the
following corollary. For the rest of the proof of
Proposition~\ref{prop:isometry-BD-BHP}, we put for a point $u\in[0,T]$
$$
I(u)= \left\{\begin{array}{l@{\quad\mbox{if }}l}
u&  u\in[0,T/2]\\
u-T& u\in[T/2,T]
\end{array}\right..
$$
\begin{corollary}
\label{cor:DBD-DH}
Assume $\mathcal{F}$ holds. Let $s,t\in
[0,\eta_\ell(A)]\cup[T-\eta_{\textup{r}}(A),T]$. Then it holds that
$\max\{D(0,s),D(0,t)\}\leq r$ if and only if
$\max\{D_{\theta}(0,I(s)),D_{\theta}(0,I(t))\}\leq r$. Under these conditions,
$$
D(s,t)=D_{\theta}(I(s),I(t)).
$$
\end{corollary}
\begin{proof}
  Take $s,t\in [0,\eta_\ell(A)]\cup[T-\eta_{\textup{r}}(A),T]$ with
  $\max\{D(0,s),D(0,t)\}\leq r$. We claim that the right
  hand side of the expression~\eqref{eq:DBD-rhs} for $D(s,t)$
  agrees with the right hand side of~\eqref{eq:DH-rhs} for $s'_1=I(s),
  t'_k=I(t)$. First note that we have
  $\max\{\eta_{\textup{l}}(A^2),\eta_{\textup{r}}(A^2)\}\leq T/2$ and thus
  $u\in[0,\eta_{\textup{l}}(A^2)]\cup[T-\eta_{\textup{r}}(A^2),T]$ if and only if
  $I(u)\in[-\eta_{\textup{r}}(A^2),\eta_{\textup{l}}(A^2)]$.  Now let
  $s_1,\dots,s_k,t_1,\dots,t_k\in[0,\eta_{\textup{l}}(A^2)]\cup[T-\eta_{\textup{r}}(A^2),T]$
  such that $s_1=s$ and $t_k=t$.  On $\mathcal{F}$, we have
  $d_{F}(t_i,s_{i+1}) = 0$ if and only if
  $d_{\Xha}(I(t_i),I(s_{i+1}))=0$, and $D(s_i,t_i)=
  D_{\theta}(I(s_i),I(t_{i+1}))$ for each $i\in\{1,\ldots,k\}$, which proves
  our claim.  

  In order to see that the right hand side of~\eqref{eq:DH-rhs} for
  $s'_1=I(s), t'_k=I(t)$ agrees with $D_{\theta}(I(s),I(t))$, we still have to
  show that 
  \begin{equation}
  \label{eq:condI}
   I(s),I(t)\in[-\eta_{\textup{r}}(A),\eta_{\textup{l}}(A)]\quad\hbox{and}\quad
  \max\{D_{\theta}(0,I(s)),D_{\theta}(0,I(t))\}\leq r.
  \end{equation}
  The first statement is clear since $s,t\in
  [0,\eta_\ell(A)]\cup[T-\eta_{\textup{r}}(A),T]$. For the second statement, the right
  hand side of~\eqref{eq:DH-rhs} specialized to $s'_1=I(s), t'_k=0$ yields an upper
  bound on $D_{\theta}(0,I(s))$, and then the equality of the right hand
  sides~\eqref{eq:DBD-rhs} and~\eqref{eq:DH-rhs} shows
  $D_{\theta}(0,I(s))\leq D(0,s)\leq r$. Entirely similar, one sees
  $D_{\theta}(0,I(t))\leq r$. A symmetry argument shows that~\eqref{eq:condI}
  implies 
  $\max\{D(0,s),D(0,t)\}\leq r$. Finally, invoking again the equality of the
  right hand sides~\eqref{eq:DBD-rhs} and~\eqref{eq:DH-rhs} shows
  $D(s,t)=D_{\theta}(I(s),I(t))$.
\end{proof}
We finish now the proof of Proposition~\ref{prop:isometry-BD-BHP} by
showing that the balls $B_r(\BD_{T,\sigma(T)})$ and $B_r(\BHP_\theta)$ are
isometric on the event~$\mathcal{F}$. Write $\Y$ for the pointed metric
space $([0,T]/\{D=0\},D,\rho)$, so that $\Y$ has the law of
$\BD_{T,\sigma(T)}$. By Lemma~\ref{lem:DBD} (a), points in $B_r(\Y)$ are on
$\mathcal{F}$ equivalence classes of the form $[s]$ for
$s\in[-\eta_{\textup{r}}(A),\eta_{\textup{l}}(A)]$. By the last statement
of Corollary~\ref{cor:DBD-DH}, we deduce that the map $I$ from above can be
viewed as an isometric map from $B_r(\Y)$ to the quotient
$\Zsf=(\R/\{D_\theta=0\},D_\theta,\rho_\theta)$, which has the law of
$\BHP_\theta$. Moreover, from Lemma~\ref{lem:DH} (a) we see that $I$ maps
$B_r(\Y)$ onto $B_r(\Zsf)$, and it sends $\rho$, the equivalence
class of $0$ in $\Y$, to $\rho_{\theta}$, the equivalence class of $0$ in
$\Zsf$. This completes the proof of the proposition.
\end{proof}

We end this section by improving Proposition~\ref{prop:isometry-BD-BHP} to
the statement of Theorem~\ref{thm:coupling-BD-BHP}.
\subsubsection{Proof of Theorem~\ref{thm:coupling-BD-BHP}}
We will need some known facts about the Brownian disks of finite volume,
mostly from Bettinelli~\cite{Be3,Be4}. With
$\Y=([0,T]/\{D=0\},D,\rho)$ as in the previous section, we let
$p_\Y:[0,T]\to \Y$ be the canonical projection.

\begin{lemma}[Proposition 17 in~\cite{Be4}]
  \label{lem:proof-prop-refpr-1}
Let $s,t\in [0,T]$ with $s\neq t$ be such that
$p_\Y(s)=p_\Y(t)$ (equivalently $D(s,t)=0$). Then either 
$d_{F}(s,t)=0$ or $d_{W}(s,t)=0$. Moreover, the
topology of $\Y$ is equal to the quotient topology of
$[0,T]/\{D=0\}$. 
\end{lemma}

\begin{lemma}[Theorem 2 and Proposition 21 in~\cite{Be3}] 
  \label{lem:proof-prop-refpr}Almost surely, the space
  $\Y$ is homeomorphic to the closed unit disk
  $\overline{\mathbb{D}}$, and the boundary of $\Y$ as a
  topological surface is determined by 
$$p_\Y^{-1}(\partial\Y)=\{s\in [0,T]:F_s=\underline{F}_s\}\,
.$$
\end{lemma}

Let a real-valued function $f$ be defined on an interval $J$, and
$t\in J$. We say that $t$ is a right-increase point of $f$ if there
exists $\eps>0$ such that $[t,t+\eps]\subset J$ and $f(s)\geq f(t)$
for every $s\in [t,t+\eps]$. Left-increase points are defined similarly,
and a unilateral increase point is a time $t$ which is either a
left-increase point or a right increase point.  Note for instance that
the preceding lemma implies that a point of $\partial \Y$ is
necessarily of the form $p_\Y(s)$ where $s$ is a unilateral increase point of $F$.

\begin{lemma}[Lemma 12 in~\cite{Be3}]
  \label{lem:proof-prop-refpr-3}
Almost surely, the sets of unilateral increase points of $F$
and $W$ are disjoint. 
\end{lemma}

The following lemma is not strictly needed but useful. See also~\cite{LGWe}. 

\begin{lemma}[Lemma 11 in~\cite{Be3}]
  \label{lem:proof-prop-refpr-2}
Almost surely, there exists a unique $s_\ast\in (0,T)$ such that
$W_{s_\ast}=\min_{[0,T]}W$. 
\end{lemma}
For $t\in[0,T)$, define 
$$\Phi_t(r)=\inf\{s\geq_\circ t:W_s=W_t-r\}\, ,\qquad
0\leq r\leq W_t-W_{s_\ast}\, ,$$ where in the notation $\geq _\circ$, it
should be understood that we consider the cyclic order in $[0,T]$ when $T$
and $0$ are identified.  More precisely, identifying $[0,T)$ with $\R/T\Z$,
for $s,t\in [0,T)$, let $[s,t]_\circ$ be the cyclic interval from $s$ to
$t$, namely, $[s,t]_\circ=[s,t]$ if $s\leq t$ and
$[s,t]_\circ=[s,T)\cup[0,t]$ if $t<s$. Then $\Phi_t(r)=s$ if and only if
$W_u>W_t-r$ for every $u\in [t,s]_\circ\setminus\{s\}$, and $W_s=W_t-r$.
The properties that we will need are summarized in the following
statement. For the rest of the proof, the time $s_\ast\in (0,T)$ is
specified as in Lemma~\ref{lem:proof-prop-refpr-2}.

\begin{lemma}
  \label{lem:proof-prop-refpr-4}
The following properties hold almost surely. 
\begin{enumerate}
\item For every $t\in [0,T]$, the path $\Gamma_t=p_\Y\circ \Phi_t$ is a
  geodesic path from $p_\Y(t)$ to $x_\ast=p_\Y(s_\ast)$.
  \item For every geodesic path $\Gamma$ to $x_\ast$ in $\Y$,
    there exists a unique $t\in [0,T)$ such that $\Gamma_t=\Gamma$.
\item For every $t\in [0,T]$, the path $\Gamma_t$ intersects $\partial
  \Y$, if at all, only at its origin $\Gamma_t(0)$.
  \item Let $s,t\in [0,T]$. Then the intersection of the images
    $\{\Gamma_s(r):0<r<D(s_\ast,s)\}$ and $\{\Gamma_t(r):0<r\leq
    D(s_\ast,t)$ (excluding the starting point) is the set 
    $$\left\{\Gamma_s(D(s_\ast,s)-r):0\leq r\leq \max\left\{\inf_{[s,t]_\circ}
    W,\inf_{[t,s]_\circ}W\right\}-W_{s_\ast}\right\}.$$
In particular, as soon as $s,t\in [0,T)$ and $s \neq t$, there exists
$\eps>0$ such that $\{\Gamma_s(r):0<r\leq \eps\}$ and
$\{\Gamma_t(r):0<r\leq \eps\}$ are disjoint. 
\end{enumerate}
\end{lemma}

\begin{proof}
  (a) and (b) are proved in~\cite[Proposition 23]{Be4}.  To prove (c), we
  notice that from the definition of $\Phi_t$, every point in the set
  $\{\Gamma_t(r):0<r\leq D(s_\ast,t)\}$ must be of the form $p_\Y(s)$,
  where $s$ is a left-increase point of $W$. By
  Lemma~\ref{lem:proof-prop-refpr-3} it cannot be a unilateral increase
  point of $F$, and thus $p_\Y(s)$ is not in $\partial \Y$ by
  Lemma~\ref{lem:proof-prop-refpr}.

  To prove (d) we first note that whenever $a<\max\left\{\inf_{[s,t]_\circ}
    W,\,\inf_{[t,s]_\circ}W\right\}$, it must hold that $\inf\{u\geq_\circ
  s:W_u=a\}=\inf\{u\geq_\circ t:W_u=a\}$, and the fact that
  $\Gamma_s(D(s_\ast,s)-r)=\Gamma_t(D(s_\ast,t)-r)$ for $r$ in the range
  given in the statement is a simple rewriting of this property and of the
  fact that $D(s_\ast,s)=W_s-W_{s_\ast}$, which is the length of the
  geodesic $\Gamma_s$. On the other hand, if
  $a>\max\left\{\inf_{[s,t]_\circ} W,\,\inf_{[t,s]_\circ}W\right\}$, then
  it is simple to see that $s_a=\inf\{u\geq_\circ s:W_u=a\}$ and
  $t_a=\inf\{u\geq_\circ t:W_u=a\}$ are such that $d_{W}(s_a,t_a)>0$, and
  since both points are left increase points for $W$, this implies that
  $p_\Y(s_a)\neq p_\Y(t_a)$ by Lemmas~\ref{lem:proof-prop-refpr-3}
  and~\ref{lem:proof-prop-refpr-1}. We leave to the reader to check that
  this implies (d).
\end{proof}

Let $a_0>0$, which will be fixed later on, and let
$O^0_{\BD}=[0,\eta_{\textup{l}}(a_0)]\cup[T-\eta_{\textup{r}}(a_0),T]$,
where $\eta_{\textup{l}}$ and $\eta_{\textup{r}}$ are defined as in the
proof of Proposition~\ref{prop:isometry-BD-BHP}. We reason on the event
that $s_\ast\notin O^0_\BD$, which will later be granted (with high
probability) by the fact that $T$ is bound to go to $\infty$. For now we
only assume that $\sigma(T)>2a_0$ so that by Lemma
\ref{lem:proof-prop-refpr}, the points
$x_{\textup{l}}=p_\Y(\eta_{\textup{l}}(a_0))$ and
$x_{\textup{r}}=p_\Y(T-\eta_{\textup{r}}(a_0))$ are distinct elements of
$\partial\Y$ (outside an event of zero probability). Let $t_\ast\in
O^0_\BD$ be such that $W_{t_\ast}=\min_{O^0_\BD}W$ (this defines $t_\ast$
uniquely a.s.\,, but we are not going to need this fact explicitly.) By (d)
in Lemma~\ref{lem:proof-prop-refpr-4}, together with the fact that
$s_\ast\notin O^0_\BD$, the paths $\Gamma_{\eta_{\textup{l}}(a_0)}$ and
$\Gamma_{T-\eta_{\textup{r}}(a_0)}$ are disjoint until they first meet at
the point $y_\ast=p_\Y(t_\ast)$.  We let $P$ be the union of the segments
of $\Gamma_{\eta_{\textup{l}}(a_0)}$ and
$\Gamma_{T-\eta_{\textup{r}}(a_0)}$ between $x_{\textup{l}},x_{\textup{r}}$
and $y_\ast$.

\begin{lemma}
  \label{lem:proof-prop-refpr-5}
  The set $P$ is a simple curve in $\Y$ from $x_{\textup{l}}$ to $x_{\textup{r}}$,
  that intersects $\partial\Y$ only at $x_{\textup{l}}$ and
  $x_{\textup{r}}$. Letting $O_\BD$ be the connected component of
  $\Y\setminus P$ that contains $p_\Y(0)$, then $O_\BD$
  is a.s. homeomorphic to the closed half-plane $\overline{\mathbb{H}}$,
  and is the interior of the set $p_\Y(O^0_\BD)$ in
  $\Y$. 
\end{lemma}

\begin{proof}
The fact that $P$ is a simple path follows from the discussion
around its definition, and the fact that it intersects the boundary only at
its extremities follows at once from (c) in Lemma
\ref{lem:proof-prop-refpr-4}. The fact that $O_\BD$ is a.s. homeomorphic to
$\overline{\mathbb{H}}$ follows at once from this and the fact that
$\Y$ is homeomorphic to $\overline{\mathbb{D}}$. It remains to show that
$O_\BD$ is the interior of the set $p_\Y(O^0_\BD)$.

Note that the curve $\beta:x\mapsto p_\Y(\inf\{s\in
[0,T]:F_s=-x\})$ is a continuous curve from
$[0,\sigma(T)]$ to $\partial \Y$ with same starting and
ending point, and taking distinct values otherwise. If we view $\beta$
as defined on a circle $\R/\sigma(T)\Z$, then it realizes a
homeomorphism onto $\partial \Y$. In particular,
$p_\Y(O^0_\BD)$ contains the segment $S$ of $\partial \Y$
between $x_{\textup{l}}$ and $x_{\textup{r}}$ that contains $\rho=p_\Y(0)$
(including $x_{\textup{l}}$, $x_{\textup{r}}$), while
$p_\Y([0,T]\setminus O^0_\BD)$ contains the other segment
which is equal to $S'=\partial \Y\setminus S$. 
For every $s\in [0,T]$, let  
$$\Xi_s(r)=\sup\{t\leq s:F_t=F_s-r\}\, ,\qquad 0\leq
r\leq F_s-\underline{F}_s\, .$$ Then
$p_\Y\circ\Xi_s$ defines a continuous path in $\Y$ from
$p_\Y(s)$ to the point $\pi(s)=p_\Y(\sup\{t\leq
s:F_t=\underline{F}_t\})$ which is in $\partial
\Y$. Moreover, for every $r\in
(0,F_s-\underline{F}_s]$, the point $\Xi_s(r)$ is a
right-increase point of $F$, so by Lemma
\ref{lem:proof-prop-refpr-3} it does not belong to
$P\setminus\{x_{\textup{l}},x_{\textup{r}}\}$, since the latter set contains only points of
the form $p_\Y(t)$ where $t$ is a unilateral increase point of
$W$. Clearly, $\pi(s)\in S$ if $s\in O^0_\BD$, while
$\pi(s)\in S'$ otherwise. We have proved that there exists a
continuous path from $x$ to $p_\Y(0)$ not intersecting $P$ for
every $x\in O_\BD$, while there exists a continuous path from $x$ to
$S'$ not intersecting $P$ for every $x\in \Y\setminus
p_\Y(O^0_\BD)$, and this shows that $O_\BD$ and $\Y\setminus
p_\Y(O^0_\BD)$ are the two connected components of
$\Y\setminus P$. 
\end{proof}
\begin{figure}[ht]
  \centering
  \includegraphics[width=0.6\textwidth]{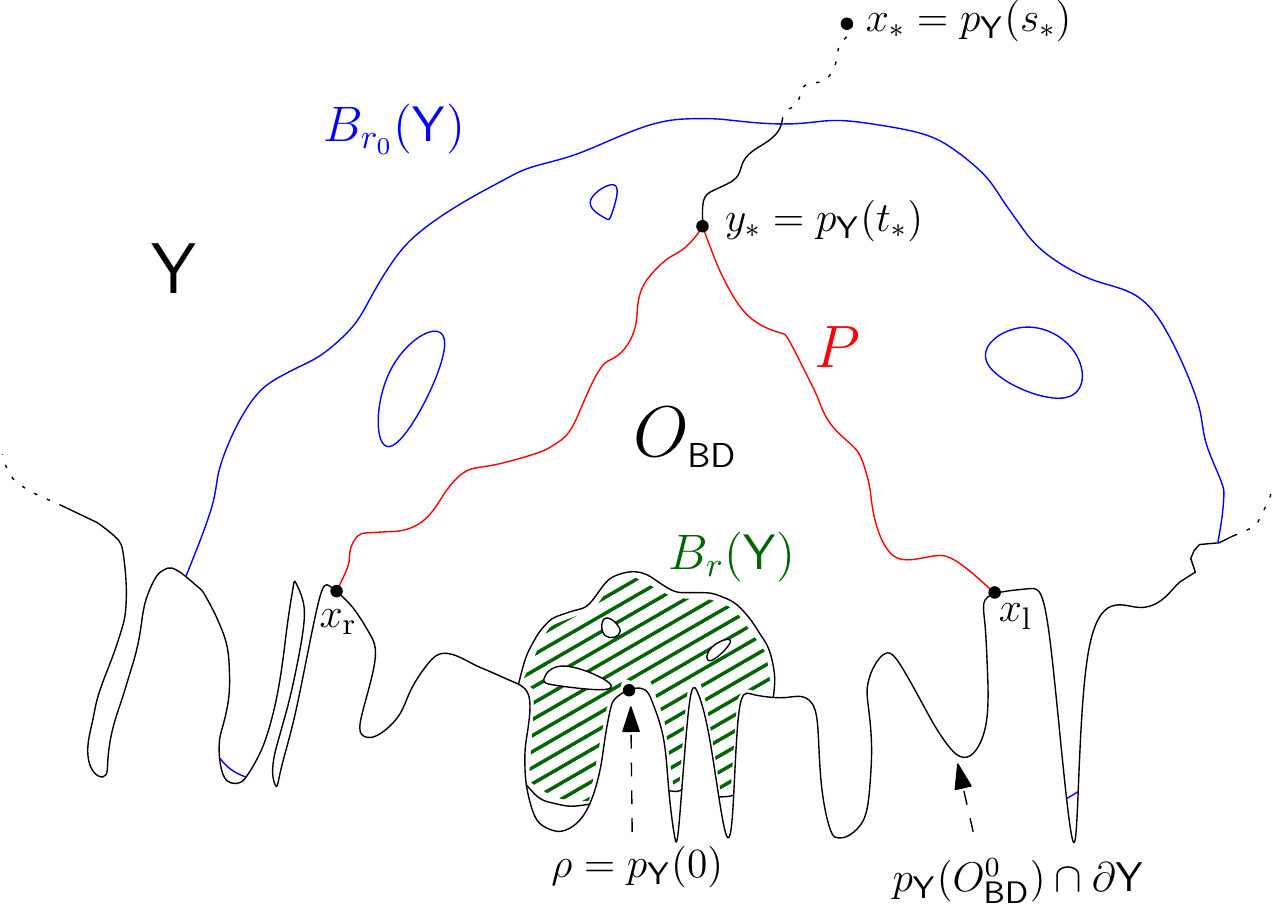}
  \caption{Proof of Theorem~\ref{thm:coupling-BD-BHP}: The ball
    $B_{r}(\Y)$ (shaded in green) itself is not simply
    connected, but it is included in a simply connected and open set
    $O_{\BD}\subset \Y$, which is homeomorphic to the closed half-plane
    $\overline{\mathbb{H}}$. The set $O_{\BD}$ is bordered by the simple
    curve $P$ (in red) in $\Y$, which is the union of two geodesic segments
    starting from $x_{\textup l}$ and $x_{\textup r}$, respectively, and by
    the boundary segment $S=p_\Y(O_{\BD}^0)\cap\partial \Y$. The larger
    ball $B_{r_0}(\Y)$ encompasses $O_{\BD}$.}
  \label{fig:topo-BHP}
\end{figure}

To finish the proof of Theorem~\ref{thm:coupling-BD-BHP}, fix
$r>0$, and let $a_0$ be large enough so that
$$\P\left(\min_{[0,a_0]}\gamma<-2r,\min_{[-a_0,0]}\gamma<-2r\right)>1-\eps/4\, .$$
Then,
we choose $r_0>r$ such that, with $\Wha$ the label function of $\BHP_\theta$,
$$\P(\omega(\Wha,[-\eta_{\textup{r}}(a_0),\eta_{\textup{l}}(a_0)])\leq r_0/2)\geq 1-\eps/4\, ,$$
where $\omega(f,I)=\sup_I f-\inf_I f$ is the modulus of continuity of $f$
over the set $I$. We use this value of $r_0$ to apply Proposition
\ref{prop:isometry-BD-BHP}. Fix $\eps>0$.  Then for every $T\geq
T_0(\eps/2)$ large enough, we can construct copies of $\Y=\BD_{T,\sigma(T)}$
and $\BHP_\theta$ such that $B_{r_0}(\Y)$ and $B_{r_0}(\BHP_\theta)$
are isometric with probability at least $1-\eps/2$. More precisely, we are
going to use the event $\mathcal{F}$ specified in the proof of
Proposition~\ref{prop:isometry-BD-BHP} on which this property holds (in the
definition of $\mathcal{F}$ we have to make sure that $A$ is chosen so that
$A>a_0$), and which implies the property that $s_\ast\notin O^0_\BD$, on
which our analysis so far is based. The probability of
$$\mathcal{F}'=\mathcal{F}\cap
\left\{\min_{[0,a_0]}\gamma<-2r,\min_{[-a_0,0]}\gamma<-2r\right\}
\cap \{\omega(\Wha,[-\eta_{\textup{r}}(a_0),\eta_{\textup{l}}(a_0)])\leq
r_0/2\}$$ is then at least $1-\eps$. On this event we claim that
$$B_{r}(\Y)\subset O_\BD\subset
B_{r_0}(\Y)\, .$$ The second inclusion comes from the
fact that $D(0,s)\leq d_{W}(0,s)\leq
2\omega(\Wha,[-\eta_{\textup{r}}(a_0),\eta_{\textup{l}}(a_0)])$ for every $s\in
[0,\eta_{\textup{l}}(a_0)]\cup [T-\eta_{\textup{r}}(a_0),T]$ (recall that
$\mathbb{W}=\Wha$ on the set $[-\eta_{\textup{r}}(A^3),\eta_{\textup{l}}(A^3)]$ on
$\mathcal{F}$).  The first inclusion comes from the cactus
bound, with the fact that $\min_{[0,a_0]}\gamma<-2r$ and
$\min_{[-a_0,0]}\gamma<-2r$, just as in the proof of (a) in Lemma
\ref{lem:DBD}. To be more precise, this shows that
$B_{2r}(\Y)\subset p_\Y([0,\eta_{\textup{l}}(a_0)]\cup
[T-\eta_{\textup{r}}(a_0),T])$, and since $O_\BD$ is equal to the interior of the
latter set, the wanted inclusion follows. We refer to
Figure~\ref{fig:topo-BHP} for an illustration.

Finally, we prove that $O_\BHP=I(O_\BD)$ is an open subset of $\BHP_\theta$,
where $I$ is the isometry defined before Corollary
\ref{cor:DBD-DH}, or more precisely the induced map on
$\Y$. Let $x\in O_\BHP$, so that $x=I(y)$ for some
$y\in O_\BD\subset B_{r_0}(\Y)$. Let $\delta>0$ be
such that $y+B_\delta(\Y)\subset O_\BD$, and $x'\in
\BHP_\theta$ be such that $D_{\theta}(x,x')<\delta$. Then
$D_{\theta}(0,x')\leq D_{\theta}(0,x)+\delta<r_0$, so that
$x'\in B_{r_0}(\BHP_\theta)$. Therefore, there exists $y'\in
B_{r_0}(\Y)$ with $x'=I(y')$, and one has
$D(y,y')=D_{\theta}(x,x')<\delta$, so that $y'\in O_\BD$,
and thus $x'\in O_\BHP$. This proves that $O_\BHP$ is open and
concludes the proof of Theorem~\ref{thm:coupling-BD-BHP}.

\subsection{Coupling quadrangulations of large volumes
  (Propositions~\ref{prop:Qn-UIHPQ} and~\ref{prop:coupling-Qn-largevol})}
We begin with the proof of Proposition~\ref{prop:Qn-UIHPQ}, which is in spirit
of~\cite[Lemma 8 and Proposition 9]{CuLG}.
\begin{proof}[Proof of Proposition~\ref{prop:Qn-UIHPQ}]
  Assume $1\ll \sigma_n\ll n$. Set $\vartheta_n=\min\{\sigma_n,\,
  n/\sigma_n\}$, and let $\eps>0$.  Let $((\f_n,\la_n),\br_n)$ be uniformly
  distributed over the set $\Fo_{\sigma_n}^n \times \Br_{\sigma_n}$, and
  consider a triplet $((\f_\infty,\la_\infty),\br_{\infty})$ of a uniformly
  labeled critical infinite forest and a uniform infinite bridge.

  We first argue that we can find $\delta>0$ and $n_0$ such that for all
  $n\geq n_0$, we can construct $((\f_n,\la_n),\br_n)$ and
  $((\f_\infty,\la_\infty),\br_{\infty})$ on the same probability space
  such that on an event of probability at least $1-\eps$, the
  corresponding balls of radius $2\delta \sqrt{\vartheta_n}$ around the
  vertices $\f_n(0)$ and $\f_\infty(0)$ in the associated quadrangulations
  are isometric.
  
  For $0\leq k\leq\sigma_n-1$, write $\tau(\f_n,k)$ for the tree of $\f_n$
  rooted at $(k)$, and put $\tau(\f_n,\sigma_n)=\tau(\f_n,0)$.
  Similarly, define $\tau(\f_\infty,k)$ to be the tree of $\f_\infty$
  rooted at $(k)$, where now $k\in\mathbb{Z}$.

  As a consequence of Lemmas~\ref{lem:GW3} and~\ref{lem:bridge1}, there exist $\delta'>0$ and $n'_0\in\N$ such
  that for $n\geq n'_0$, with $A_n=\lfloor \delta'\vartheta_n\rfloor$,
  we can construct $((\f_n,\la_n),\br_n)$ and
  $((\f_\infty,\la_\infty),\br_{\infty})$ on the same probability space
  such that if we let
  \begin{align*}
    \mathcal{E}^1(n,\delta') &= \quad\left\{\tau(\f_n,i)=\tau(\f_\infty,i),\,\tau(\f_n,\sigma_n-i)=\tau(\f_\infty,-i),\,0\leq
      i\leq A_n\right\}\\
    &\,\quad \cap\left\{\br_n(i)=\br_\infty(i),\,\br_n(\sigma_n-i)=\br_\infty(-i),\,1\leq
      i\leq A_n\right\}\\
    &\,\quad\cap\left\{\la_n|_{\tau(\f_n,i)}=
      \la_\infty|_{\tau(\f_\infty,i)},\,\la_n|_{\tau(\f_n,\sigma_n-i)}=\la_\infty|_{\tau(\f_\infty,-i)},\,0\leq i
      \leq A_n\right\},
\end{align*}
then $\mathcal{E}^1(n,\delta')$ has probability at least
$1-\eps/3$. We fix such a $\delta'$. For $\delta>0$ and $n\in\N$, put
 $$\mathcal{E}^2(n,\delta)=\left\{\min_{[0,\,A_n]}\br_\infty
   <-5\delta \sqrt{\vartheta_n},\,
   \min_{[-A_n,\,0]}\br_\infty<-5\delta\sqrt{\vartheta_n}\right\}\cap\left\{-\br_\infty(-1)<\delta^{-1}\right\},$$
 and let
 $$\mathcal{E}^3(n,\delta)=\left\{\min_{[A_n+1,\,\sigma_n-(A_n+1)]}\br_n
  <-5\delta\sqrt{\vartheta_n}\right\}.$$
Donsker's invariance principle applied to
$(\br_{\infty}(i),i\in\mathbb{Z})$ guarantees that we can find
$\delta>0$ such that for all sufficiently large $n$,
$$
\P\left(\mathcal{E}^2(n,\delta)\right)\geq 1-\eps/3.
$$
Moreover, provided $n$ is large enough and $\delta>0$ is
sufficiently small, Lemma~\ref{lem:bridge0} ensures that 
$$
\P\left(\mathcal{E}^3(n,\delta)\right)\geq 1-\eps/3.
$$
We fix $n_0\geq n'_0$ and $\delta>0$ such that for all $n\geq
n_0$, the bounds in the last two displays hold.  

From now on, we work on the event
$\mathcal{E}^1(n,\delta')\cap\mathcal{E}^2(n,\delta)\cap\mathcal{E}^3(n,\delta)$.
Let $(Q_n^{\sigma_n},\vd)=\Phi_n(((\f_n,\la_n),\br_n))$ and
$Q_\infty^\infty=\Phi_{\infty}(((\f_\infty,\la_\infty),\br_{\infty}))$ be
the quadrangulations constructed from the triplets $((\f_n,\la_n),\br_n)$
and $((\f_\infty,\la_\infty),\br_{\infty})$ {\it via} the Bouttier-Di
Francesco-Guitter mapping. We denote by $d_n$ and $d_{\infty}$ the graph
distances on $V(Q_n^{\sigma_n})$ and $V(Q_\infty^\infty)$. We write
$$\f'_n=\left(\tau(\f_n,\sigma_n-A_n),\ldots,\tau(\f_n,\sigma_n-1),\tau(\f_n,0),\ldots,\tau(\f_n,A_n)\right)$$
for the forest obtained from restricting $\f_n$ to 
the last $A_n$ and the first $A_n+1$ trees, and identically
$$\f'_\infty=\left(\tau(\f_{\infty},-A_n),\ldots,\tau(\f_{\infty},-1),\tau(\f_\infty,0),\ldots,\tau(\f_\infty,A_n)\right).$$
Recall the cactus bounds~\eqref{eq:cactus1} and~\eqref{eq:cactus3} for
$Q_n^{\sigma_n}$ and $Q_\infty^\infty$, respectively. For
vertices $v\in V(\f_n)\setminus V(\f'_n)$, we obtain, with $(0)=\f_n(0)$,
$$
d_n((0),v) \geq -\max\left\{\min_{[0,A_n]}\br_n, \min_{[-A_n,0]}\br_n\right\} \geq 5\delta\sqrt{\vartheta_n},
$$
and identically, for vertices $v\in V(\f_\infty)\setminus V(\f'_\infty)$,
writing now $(0)$ for $\f_\infty(0)$, 
$
d_{\infty}((0),v) \geq 5\delta\sqrt{\vartheta_n}.
$
We argue now similarly to the second part in the proof of~\cite[Lemma
8]{CuLG}. Firstly, if $u\in V(\f_n)$ is any vertex with $d_n((0),u)\leq
5\delta\sqrt{\vartheta_n}-1$, then any vertex on a geodesic path from $(0)$
to $u$ in $Q_n^{\sigma_n}$ satisfies the same bound and must therefore
belong to one of the trees in $\f'_n$. From the construction of edges in
the Bouttier-Di Francesco-Guitter mapping, we deduce that any edge of
$Q_n^{\sigma_n}$ on such a geodesic path corresponds to an edge of
$Q_\infty^{\infty}$ with the same endpoints, provided none of these edges
in $Q_n^{\sigma_n}$ connect two vertices $w$ and $w'$ such that the set of
vertices between $w$ and $w'$ in the cyclic contour order around the forest
$\f_n$ contains the vertices of $\f_n\setminus \f_n'$. But on the event
$\mathcal{E}^3(n,\delta)$, the set of vertices between $w$ and $w'$ would
in particular contain a (root) vertex of $\f_n\setminus\f_n'$ with label
less than $-5\delta\sqrt{\vartheta_n}$. This would imply that both vertices
$w$ and $w'$ of such an edge have a label which is smaller than
$-5\delta\sqrt{\vartheta_n}$, too, in contradiction to the fact that
$d_n((0),v)\leq 5\delta\sqrt{\vartheta_n}-1$ for all vertices $v$ on a
geodesic between $(0)$ and $u$. We deduce that if $u\in V(\f_n)$ satisfies
$d_n((0),u)\leq 5\delta\sqrt{\vartheta_n}-1$, then $d_\infty((0),u)\leq
d_n((0),u)$. Since in turn any edge of $Q_\infty^{\infty}$ on a geodesic
between $(0)$ and a vertex $u\in V(\f_\infty)$ with $d_\infty((0),u)\leq
5\delta\sqrt{\vartheta_n}-1$ does correspond to an edge of $Q_n^{\sigma_n}$
with the same endpoints, we obtain also $d_n((0),u)\leq
d_\infty((0),u)$. Therefore, we have that vertices with distance at most
$5\delta\sqrt{\vartheta_n}-1$ from $(0)$ are the same in $Q_n^{\sigma_n}$
and $Q_\infty^{\infty}$.  We claim that
\begin{equation}
\label{eq:prop-BHP1-eq1}
d_n(u,v)=d_{\infty}(u,v)\quad\hbox{whenever }u,v\in B^{(0)}_{2\delta\sqrt{\vartheta_n}}(Q_n^{\sigma_n}).
\end{equation}
Indeed, if $u,v$ are vertices in $B^{(0)}_{2\delta\sqrt{\vartheta_n}}(Q_n^{\sigma_n})$, then any geodesic
connecting $u$ to $v$ in $Q_n^{\sigma_n}$ must lie entirely in
$B^{(0)}_{4\delta\sqrt{\vartheta_n}}(Q_n^{\sigma_n})$, and any edge on such a geodesic
corresponds to an edge in $Q_\infty^\infty$. Since the converse is also
true, we obtain~\eqref{eq:prop-BHP1-eq1}, and with the correspondence of edges
between $Q_n^{\sigma_n}$ and $Q_\infty^\infty$ alluded to above we deduce
that the balls $B^{(0)}_{2\delta\sqrt{\vartheta_n}}(Q_n^{\sigma_n})$ and
$B^{(0)}_{2\delta\sqrt{\vartheta_n}}(Q_\infty^{\infty})$ are isometric on an event of
probability at least $1-\eps$.

Finally, recall from the Bouttier-Di Francesco-Guitter bijection that the
root vertex $\rho_n$ of $Q_n^{\sigma_n}$ is given by
$\f_n(\suc^{-\br_n(\sigma_n)}(0))$, where conditionally on
$\br_n(\sigma_n-1)$, $\br_n(\sigma_n)$ is uniformly distributed on
$\{\br_n(\sigma_n-1)-1,\ldots,0\}$. Similarly, the root vertex $\rho$ of
$Q_\infty^\infty$ is given by
$\f_\infty(\suc^{-\br_\infty(\partial)}(0))$, where conditionally on
$\br_\infty(-1)$, $\br_\infty(\partial)$ is uniformly distributed on
$\{\br_\infty(-1)-1,\ldots,0\}$. On the event
$\mathcal{E}^{1}(n,\delta')\cap \mathcal{E}^{2}(n,\delta)$, we can
couple $\br_n(\sigma_n)$ and $\br_\infty(\partial)$ such that
$\br_n(\sigma_n)=\br_\infty(\partial)$. Moreover, for $n$ large enough, we
have on this event 
$B_{\delta\sqrt{\vartheta_n}}(Q_n^{\sigma_n})\subset
B^{(0)}_{2\delta\sqrt{\vartheta_n}}(Q_n^{\sigma_n})$ and
$B_{\delta\sqrt{\vartheta_n}}(Q_{\infty}^{\infty})\subset
B_{2\delta\sqrt{\vartheta_n}}^{(0)}(Q_{\infty}^{\infty})$. Therefore, we
have equality of $B_{\delta\sqrt{\vartheta_n}}(Q_n^{\sigma_n})$ and
$B_{\delta\sqrt{\vartheta_n}}(Q_{\infty}^{\infty})$ on the event
$\mathcal{E}^{1}(n,\delta')\cap
\mathcal{E}^{2}(n,\delta)\cap\mathcal{E}^{3}(n,\delta)$. Local convergence of $Q_n^{\sigma_n}$
towards $\UIHPQ$ in the sense of $\dmap$ is a direct consequence of
this, and the proposition is proved.
\end{proof}

We now turn to the proof of Proposition~\ref{prop:coupling-Qn-largevol}.
We will adopt the notion of~\cite[Section 4.3.1]{CuLG} concerning pruned
(pointed) trees. More precisely, a (finite) pointed tree consists of a pair
$\boldsymbol{\tau}=(\tau,\xi)$, where $\tau$ is a tree of finite size and
$\xi$ is a distinguished vertex of $\tau$. Given such a pointed tree
$\boldsymbol{\tau}=(\tau,\xi)$ and $h$ an integer with $0\leq h<|\xi|$,
$\mathscr{P}(\boldsymbol{\tau},h)$ represents the subtree of $\tau$
containing all the vertices $u\in V(\tau)$ such that the height of the most
recent common ancestor of $u$ and $\xi$ is strictly less than $h$, together
with the ancestor $\left[\xi\right]_h$ of $\xi$ at height exactly $h$. By
pointing $\mathscr{P}(\boldsymbol{\tau},h)$ at $\left[\xi\right]_h$, this
subtree is itself considered as a pointed tree. If $h\geq |\xi|$, we agree
that $\mathscr{P}(\boldsymbol{\tau},h)=(\{\emptyset\},\emptyset)$, where
$\emptyset$ represents the root vertex of $\tau$. It is straightforward to
see that if $\boldsymbol{\tau}=(\tau,\xi)$ is a pointed tree and $h$ and
$h'$ are two integers with $h'\geq h\geq 0$, then
\begin{equation}
\label{eq:pruningconsistent}
\mathscr{P}\left((\boldsymbol{\tau},h'),h\right)=\mathscr{P}\left(\boldsymbol{\tau},h\right).
\end{equation}

\begin{proof}[Proof of Proposition~\ref{prop:coupling-Qn-largevol}]
  We assume $1\ll \sigma_n\ll \sqrt{n}$ and fix
  $\eps>0$ and $r>0$ for the rest of this proof.  We let
  $((\f_n,\la_n),\br_n)$ be uniformly distributed over the set
  $\Fo_{\sigma_n}^n \times \Br_{\sigma_n}$, and for given $R\in\N$, we let
  $((\f'_n,\la'_n),\br'_n)$ be uniformly distributed over
  $\Fo_{\sigma_n}^{R\sigma_n^2} \times \Br_{\sigma_n}$. Identically to
  the proof of Proposition~\ref{prop:Qn-UIHPQ}, it suffices to show that we can
  find $R_0>0$ and $n_0\in\N$ such that for all integers $R\geq R_0$ and
  all $n\geq n_0$, we can construct $((\f_n,\la_n),\br_n)$ and
  $((\f'_n,\la'_n),\br'_{n})$ on the same probability space such that on an
  event of probability at least $1-\eps$, the corresponding balls of
  radius $2r\sqrt{\sigma_n}$ around the vertices $\f_n(0)$ and $\f'_n(0)$
  in the associated quadrangulations are isometric.
  
  For $0\leq k\leq\sigma_n-1$, we let $\tau(\f_n,k)$ be the tree of $\f_n$
  rooted at $(k)$ and denote by $i_\ast$ the smallest index such that
  $|\tau(\f_n,i_\ast)|\geq |\tau(\f_n,k)|$ for all $0\leq k\leq
  \sigma_n-1$. We shall point the tree $\tau(\f_n,i_\ast)$, by choosing
  conditionally on $\tau(\f_n,i_\ast)$ a vertex $\xi_n\in
  V(\tau(\f_n,i_\ast))$ uniformly at random. We write
  $(\tau(\f_n,i_\ast),\xi_n)$ for the pointed tree obtained in this way,
  and for $h\in\N$, we write
  $\la_n|_{\mathscr{P}((\tau(\f_n,i_\ast),\xi_n),h)}$ for the restriction
  of the labels $\la_n$ of $\f_n$ to the subtree
  $\mathscr{P}((\tau(\f_n,i_\ast),\xi_n),h)$ of $(\tau(\f_n,i_\ast),\xi_n)$
  pruned at height $h$, see the notation above. Finally, we let
  $(\tau_i,\ell_i)$, $0\leq i\leq \sigma_n-1$, be a sequence of independent
  uniformly labeled critical geometric Galton-Watson trees.

  For $H\in\N$, set $H_n=H\sigma_n$. Recall that the law of
  $((\f'_n,\la'_n),\br'_n)$ depends on $R\in\N$. We claim that for each
  fixed integer $H\in\N$, provided $n$ and $R$ are sufficiently large, we
  can construct $((\f_n,\la_n),\br_n)$, $((\f'_n,\la'_n),\br'_n)$, $\xi_n$,
  $\xi'_n$ and $(\tau_i,\ell_i)$ for $0\leq i\leq \sigma_n-1$ on the same
  probability space such that the event
  \begin{align*}
    \mathcal{E}_1(n,R,H) &=
    \left\{i_\ast=i'_\ast\right\}\cap\left\{\tau(\f_n,i)=\tau(\f'_n,i)=\tau_i,\,0\leq
      i\leq \sigma_n-1,\, i\neq i_\ast\right\}\\
    &\,\quad\cap
    \left\{\mathscr{P}\left((\tau(\f_n,i_\ast),\xi_n),H_n\right)=\mathscr{P}\left((\tau(\f'_n,i'_\ast),\xi'_n),H_n\right)\neq
      (\{\emptyset\},\emptyset)\right\}\\
    &\,\quad\cap\left\{\br_n(i)=\br'_n(i),\,0\leq i\leq \sigma_n\right\}\\
    &\,\quad\cap \left\{\la_n|_{\tau(\f_n,i)}=
      \la'_n|_{\tau(\f'_n,i)}=\ell_i,\,0\leq i \leq \sigma_n-1, i\neq
      i_\ast\right\}\\&\,\quad
    \cap\left\{\la_n|_{\mathscr{P}\left((\tau(\f_n,i_\ast),\xi_n),H_n\right)}=\la'_n|_{\mathscr{P}\left((\tau(\f'_n,i'_\ast),\xi'_n),H_n\right)}\right\}
  \end{align*}
  has probability at least $1-\eps/2$. Let us look separately at the
  different sets on the right hand side. Firstly, from Lemma~\ref{lem:GW1}
  we know that $\f_n$ has with high probability a unique largest tree of
  order $\sigma_n^2$, and its index is uniform in
  $\{0,\ldots,\sigma_n-1\}$. Moreover, Lemma~\ref{lem:GW2} asserts that the
  other trees of $\f_n$ are close in total variation to $\sigma_n-1$
  critical geometric Galton-Watson trees. The same holds for $\f'_n$, from
  which we deduce that $\f_n,\f'_n$ and $\tau_i$, $0\leq i\leq \sigma_n-1$,
  can be coupled such that the intersection of the first two events on the
  right hand side has probability at least $1-\eps/3$, say. For the
  event on the second line concerning the pruned trees, we use that fact
  that conditionally on $|\tau(\f_n,i_\ast)|=m_n$,
  $(\tau(\f_n,i_\ast),\xi_n)$ is uniformly distributed over the set of all
  pointed trees of size $m_n$. A similar statement holds for
  $\tau(\f'_n,i'_\ast)$. Now by Lemma~\ref{lem:GW1}, for any $K>0$, the
  probability that $|\tau(\f_n,i_\ast)|\geq K\sigma^2_n$ tends to one with
  increasing $n$, since $n\gg \sigma_n^2$. Similarly, for any given $K>0$,
   by choosing $R$ large enough, we can ensure that
  $|\tau(\f'_n,i'_\ast)|\geq K\sigma^2_n$ holds with a probability as close
  to one as we wish for large $n$. An application of Proposition $7$
  of~\cite{CuLG} therefore shows that both
  $\mathscr{P}\left((\tau(\f_n,i_\ast),\xi_n),H_n\right)$ and
  $\mathscr{P}\left((\tau(\f'_n,i'_\ast),\xi'_n),H_n\right)$ are for large
  $R$ and $n$ close in total variation to the so-called uniform infinite
  tree (or Kesten's tree) pruned at height $H_n$. Applying the
  triangle inequality, we see that the total variation distance between
  $\mathscr{P}\left((\tau(\f_n,i_\ast),\xi_n),H_n\right)$ and
  $\mathscr{P}\left((\tau(\f'_n,i'_\ast),\xi'_n),H_n\right)$ can be made as
  small as we wish, provided $R$ and $n$ are taken sufficiently large.

  Combining the above coupling with this last observation, we infer that we
  can in fact couple $\f_n$, $\f'_n$, $\xi_n$, $\xi_n'$ and $\tau_i$ for
  $0\leq i\leq \sigma_n-1$ such that the intersection of the first three
  events on the right hand side has probability at least $1-\eps/2$
  for large $R$ and $n$. Since the bridges $\br_n$ and $\br'_n$ have both
  the same law and are independent of the trees, we can additionally assume
  that the probability space carries realizations of $\br_n$ and $\br'_n$
  such that $\br_n\equiv \br'_n$. A similar arguments allows us to couple
  the labelings $\la_n$, $\la'_n$, and $\ell_i$ such that the last two
  events on the right hand side in the definition of $\mathcal{E}_1(n,R,H)$
  hold true.  This proves the claim about $\mathcal{E}_1(n,R,H)$.

  We will now work on the event $\mathcal{E}_1(n,R,H)$ and let
  $(Q_n^{\sigma_n},\vd)=\Phi_n(((\f_n,\la_n),\br_n))$ and
  $(Q_{R\sigma_n^2}^{\sigma_n},w^{\bullet})=\Phi_{R\sigma_n^2}(((\f'_n,\la'_n),\br'_n))$
  be the quadrangulations constructed from the triplets
  $((\f_n,\la_n),\br_n)$ and $((\f'_n,\la'_n),\br'_n)$,
  respectively. Recall that $[\xi_n]_{H_n}$ denotes the ancestor of $\xi_n$
  in $\tau(\f_n,i_\ast)$ at height $H_n$. Let
  $$M_n=-\min_{[[\emptyset,[\xi_n]_{H_n}]]}\la_n,$$ 
  where $[[\emptyset,[\xi_n]_{H_n}]]$ is the vertex set of the unique
  injective path in $\tau(\f_n,i_\ast)$ connecting the (tree) root
  $\emptyset$ to $[\xi_n]_{H_n}$. By definition of the labeling $\la_n$,
  conditionally on the tree, $M_n$ has the law of the maximum attained by a
  random walk started at zero and stopped after $H_n$ many steps, with
  increments uniformly distributed in $\{-1,0,1\}$. Setting
$$\mathcal{E}_2(n,H)=\left\{M_n\geq 5r\sqrt{\sigma_n}\right\},
$$
we can ensure by an application of Donsker's invariance principle that for
$H\in\N$ sufficiently large (recall that $r$ was fixed at the beginning,
and $H_n=H\sigma_n$), the event $\mathcal{E}_2(n,H)$ has probability at
least $1-\eps/2$. In particular, by choosing 
$H\in\N$ large enough, we obtain that 
$\mathcal{E}_1(n,R,H)\cap\mathcal{E}_2(n,H)$ has probability at least
$1-\eps$ for all $R$, $n\in\N$ sufficiently large. 

It remains to convince ourselves that on the event
$\mathcal{E}_1(n,R,H)\cap \mathcal{E}_2(n,H)$, the balls
$B^{(0)}_{2r\sqrt{\sigma_n}}(Q_n^{\sigma_n})$ and
$B^{(0)}_{2r\sqrt{\sigma_n}}(Q_{R\sigma_n^2}^{\sigma_n})$ are
isometric. Since the arguments are very close to those provided in the
proofs of Proposition~\ref{prop:Qn-UIHPQ} above and~\cite[Lemma 8]{CuLG}, we
only sketch them in order to avoid too much repetition. Write
$\emptyset=u_0,u_1,\ldots,u_{H_n}=[\xi_n]_{H_n}$ for the vertices of the
non-backtracking path connecting $\emptyset$ to $[\xi_n]_{H_n}$ in
$\tau(\f_n,i_\ast)$.  Let $k_n\in\{0,\ldots,H_n\}$ such that
$$\la_{n}(u_{k_n})=-\min_{[[\emptyset,[\xi_n]_{H_n}]]}\la_n.$$
Recall the identification of $V(Q_n^{\sigma_n})\setminus\{\vd\}$ with
$V(\f_n)$. Denote by $d_n$ the graph distance on
$V(Q_n^{\sigma_n})$. If $v$ is a vertex of $\tau(\f_n,i_\ast)$ that
does not belong to the subtree
$\mathscr{P}\left((\tau(\f_n,i_\ast),\xi_n),k_n\right)$, then the ancestral
lines of $v$ and $\xi_n$ coincide at least up to level $k_n$. In
particular, they both contain the vertex $u_{k_n}$. For such vertices $v$,
we obtain from the cactus bound~\eqref{eq:cactus1} on the event
$\mathcal{E}_1(n,R,H)\cap \mathcal{E}_2(n,H)$ the bound
$$
d_n((0),v)\geq 5r\sqrt{\sigma_n},
$$
with $(0)=\f_n(0)$. See~\cite[Proof of Lemma 8]{CuLG} for the complete
argument (note however that $(0)$ might be the root of a tree different
from $\tau(\f_n,i_\ast)$). On $\mathcal{E}_1(n,R,H)$, using
additionally~\eqref{eq:pruningconsistent},
$$
\mathscr{P}\left((\tau(\f_n,i_\ast),\xi_n),k_n\right)=
\mathscr{P}\left((\tau(\f'_n,i'_\ast),\xi'_n),k_n\right),
$$
and the labelings $\la_n$ and $\la'_n$ restricted to the subtrees on the
left and right, respectively, agree. Therefore, a similar inequality holds
for $Q_{R\sigma^2_n}^{\sigma_n}$, for vertices $v'$ of
$\tau(\f'_n,i'_\ast))$ which do not belong to the subtree
$\mathscr{P}\left((\tau(\f'_n,i'_\ast),\xi'_n),k_n\right)$. Adapting now
the reasoning of~\cite[Proof of Lemma 8]{CuLG} to our situation (see also
the proof of Proposition~\ref{prop:Qn-UIHPQ} above), we obtain that vertices
with distance at most $5r\sqrt{\sigma_n}-1$ from $(0)$ are the same in
$Q_n^{\sigma_n}$ and $Q_{R\sigma^2_n}^{\sigma_n}$ on the event
$\mathcal{E}_1(n,R,H)\cap \mathcal{E}_2(n,H)$, and, with $d_n'$ being the
graph distance in $Q_{R\sigma^2_n}^{\sigma_n}$,
$$
d_n(u,v)=d'_n(u,v)\quad\hbox{whenever }u,v\in B^{(0)}_{2r\sqrt{\sigma_n}}(Q_n^{\sigma_n}).
$$
This finishes our proof.
\end{proof}

\subsection{Brownian half-plane with zero skewness (Theorems~\ref{thm:BHP1}
  and~\ref{thm:UIHPQ-BHP})}

We work in the usual setting from Section~\ref{sec:usualsetting}. In particular,
$Q_n^{\sigma_n}$ denotes a random variable with the uniform distribution
over the set $\mathcal Q_n^{\sigma_n}$ of rooted quadrangulations with $n$
internal faces and $2\sigma_n$ boundary edges. 

Our proofs of Theorems~\ref{thm:BHP1} and~\ref{thm:UIHPQ-BHP} are
essentially consequences of the coupling of balls between the Brownian disk
$\BD_{T,\sqrt{T}}$ and the Brownian half-plane $\BHP$
(Proposition~\ref{prop:isometry-BD-BHP}), of the fundamental convergence
\begin{equation}
\label{eq:BeMi}
\left(V(Q_n^{\sigma_n}),(8/9)^{-1/4}n^{-1/4}\dgr,\rho_n\right) \xrightarrow[n \to \infty]{(d)}
\BD_\sigma=\BD_{1,\sigma}
\end{equation}
proved in~\cite[Theorem 1]{BeMi} for the regime
$\sigma_n\sim\sigma\sqrt{2n}$ when $\sigma\in(0,\infty)$ is a fixed real
and of the coupling between $Q_n^{\sigma_n}$ and the $\UIHPQ$
$Q_\infty^\infty$ (Proposition~\ref{prop:Qn-UIHPQ}).

\begin{proof}[Proof of Theorem~\ref{thm:UIHPQ-BHP}]
  In view of  Remark~\ref{rem:localGH}, the result follows if we show that for every
  $r\geq 0$ and every sequence of positive reals $a_n\rightarrow\infty$,
$$
B_r\left(a_n^{-1}\cdot Q_\infty^{\infty}\right)\xrightarrow[n\to
  \infty]{(d)}B_{r}(\BHP)$$ in distribution in $\mathbb{K}$. For notational simplicity, we
restrict ourselves to the case $r=1$. Fix $\eps>0$. By
Proposition~\ref{prop:isometry-BD-BHP}, we find
$T_0=T_0(\eps)>0$ such that for all $T\geq
T_0$, we can construct copies of $\BD_{T,\sqrt{T}}$ and $\BHP$ on
the same probability space such that 
\begin{equation}
\label{eq:BHP-coupling1}
B_1\left(\BD_{T,\sqrt{T}}\right)=B_1(\BHP)
\end{equation}
with probability at least $1-\eps$.  

Let $\sigma_n=\lceil \sqrt{2n}\rceil$. By
Proposition~\ref{prop:Qn-UIHPQ}, there exists $\delta>0$ such that
for $n$ large enough, we can couple $Q_{n}^{\sigma_{n}}$ and
$Q_{\infty}^{\infty}$ on the same probability space such that with
probability at least $1-\eps$,
$
B_{\delta\sqrt{\sigma_n}}\left(Q_{n}^{\sigma_{n}}\right) =B_{\delta\sqrt{\sigma_n}}\left(Q_{\infty}^{\infty}\right).
$
We can and will assume that $\delta< 2T_0^{-1/4}$. We put $m_n=\lceil \delta^{-4}a_n^4\rceil$. Then
$a_n\leq \delta\sqrt{\sigma_{m_n}}$. With $m_n$
taking the role of $n$, the last observation enables us to find a coupling between $Q_{m_n}^{\sigma_{m_n}}$ and
$Q_{\infty}^{\infty}$ on the same probability space such that for large
$n$, we have with
probability at least $1-\eps$
\begin{equation}
\label{eq:BHP-coupling2}
B_{a_n}\left(Q_{m_n}^{\sigma_{m_n}}\right) =B_{a_n}\left(Q_{\infty}^{\infty}\right).
\end{equation}

Let $F:\mathbb{K}\rightarrow\mathbb{R}$ be bounded and continuous, and put
$T=\delta^{-4}(8/9)$. Note that $T\geq T_0$. We work
with a coupling of $Q_{m_n}^{\sigma_{m_n}}$ and $Q_{\infty}^{\infty}$ as
well as with a coupling of $\BD_{T,\sqrt{T}}$ and $\BHP$ such that the
properties just mentioned hold. Then
\begin{align*}
  \lefteqn{\left|\E\left[F\left(
          B_1\left(a_n^{-1}\cdot Q_\infty^{\infty}\right)\right)\right]-\E\left[F\left(B_1(\BHP)\right)\right]\right|}\\
  &\quad\quad \leq \left|\E\left[F\left( a_n^{-1}\cdot
        B_{a_n}\left(Q_\infty^{\infty}\right)\right)-F\left(
        a_n^{-1}\cdot B_{a_n}\left(Q_{m_n}^{\sigma_{m_n}}\right)\right)\right]\right|\\
  &\quad\quad\quad + \left|\E\left[F\left(
        B_{1}\left(a_n^{-1}\cdot Q_{m_n}^{\sigma_{m_n}}\right)\right)\right]-\E\left[F\left(B_{1}\left(\BD_{T,\sqrt{T}}\right)\right)\right]\right|\\
  &\quad\quad\quad +
  \left|\E\left[F\left(B_{1}\left(\BD_{T,\sqrt{T}}\right)\right)-F\left(B_1\left(\BHP\right)\right)\right]\right|.
\end{align*}
Using the coupling~\eqref{eq:BHP-coupling2} for the first and the
coupling~\eqref{eq:BHP-coupling1} for the third summand on the right hand
side, we see that both of them are bounded from above by $2\eps\sup F$. The
second summand converges to zero as $n\rightarrow\infty$,
using~\eqref{eq:BeMi} and the scaling relation $\BD_{T,\sqrt{T}}=_d
T^{1/4}\BD_{1}$. This concludes the proof of Theorem~\ref{thm:UIHPQ-BHP}.
\end{proof}
\begin{proof}[Proof of Theorem~\ref{thm:BHP1}]
  We have to show that when $1\ll\sigma_n\ll n$, we have for every $r\geq
  0$ and any sequence $1\ll
  a_n\ll\min\{\sqrt{\sigma_n},\,\sqrt{n/\sigma_n}\}$, $
  B_r(a_n^{-1}\cdot Q_n^{\sigma_n})\longrightarrow B_{r}(\BHP)$
  in distribution in $\mathbb{K}$ as $n\rightarrow\infty$. Let $\eps>0$ and
  $r\geq 0$. By Proposition~\ref{prop:Qn-UIHPQ}, we can couple
  $Q_{n}^{\sigma_{n}}$ and $Q_{\infty}^{\infty}$ on the same probability
  space such that with probability at least $1-\eps$, for $n\geq n_0$, $
  B_{ra_n}\left(Q_{n}^{\sigma_{n}}\right)
  =B_{ra_n}\left(Q_{\infty}^{\infty}\right).  $ Then, for
  $F:\mathbb{K}\rightarrow\mathbb{R}$ bounded and continuous,
\begin{align*}
  \left|\E\left[F\left(
          B_r\left(a_n^{-1}\cdot Q_n^{\sigma_n}\right)\right)-F\left(B_r(\BHP)\right)\right]\right|
  &\leq \left|\E\left[F\left(a_n^{-1}\cdot
        B_{ra_n}\left(Q_n^{\sigma_n}\right)\right)-F\left(a_n^{-1}\cdot
        B_{ra_n}\left(Q_\infty^\infty\right)\right)\right]\right|\\
  &\quad + \left|\E\left[F\left(
        B_r\left(\BHP\right)\right)-F\left(a_n^{-1}\cdot
        B_{ra_n}\left(Q_\infty^\infty\right)\right)\right]\right|.
 \end{align*}
 Under our coupling, the first summand behind the inequality is bounded by
 $2\eps\sup|F|$ provided $n\geq n_0$. By
 Theorem~\ref{thm:UIHPQ-BHP}, the second summand converges to zero as
 $n\rightarrow\infty.$
\end{proof}

\begin{remark}
\label{rem:jointconv-CLBr}
Notice that in our proofs of the couplings
Proposition~\ref{prop:isometry-BD-BHP} (between $\BD_\sigma$ and $\BHP$)
and Proposition~\ref{prop:Qn-UIHPQ} (between $Q_n^{\sigma_n}$ and
$\UIHPQ$), we construct in fact joint couplings of contour functions, label
functions and balls in the corresponding metric spaces. As a consequence,
the theorems proved in this section can be strengthened in a way we now 
exemplify based on Theorem~\ref{thm:UIHPQ-BHP}.

Recall that we view the contour and label functions $C_\infty$ and
$\La_\infty$ which specify the $\UIHPQ$ $Q_\infty^\infty$ as (random)
continuous functions on $\R$. The Brownian half-plane $\BHP$ is constructed
from contour and label functions $X^0=(X^0_t,t\in\R)$ and $W^0=(W^0_t,t\in\R)$ as
specified in Section~\ref{sec:recapBHPBD}.

We now claim that for each $r\geq 0$ and any positive sequence $a_n\rightarrow\infty$,
\begin{equation}
\label{eq:jointconv-CLBr}
\left(\frac{C_\infty((9/4)a_n^4\cdot)}{(3/2)a_n^2},\frac{\La_\infty\left((9/4)a_n^4\cdot\right)}{a_n},
  B_r\left(a_n^{-1}\cdot Q_\infty^\infty\right)\right)
\xrightarrow[n\to\infty]{(d)}\left(X^0,W^0,B_{r}(\BHP)\right)
\end{equation}
jointly in the space
$\mathcal{C}(\R,\R)^2\times\mathbb{K}$. The convergence does also hold with
$B_r(a_n^{-1}\cdot Q_\infty^\infty)$ replaced by $B_r^{(0)}(a_n^{-1}\cdot Q_\infty^\infty)$.

To see why~\eqref{eq:jointconv-CLBr} holds, one has to slightly enhance the
proof of Theorem~\ref{thm:UIHPQ-BHP}. Since all the necessary arguments
were already given, we restrict ourselves to a sketch proof and leave it to
the reader to fill in the details.  We assume $r=1$ for simplicity. Let
$T>0$, denote by $(F,W)$ the contour and label function of
$\BD_{T,\sqrt{T}}$, and set $F(-t)=F(T-t)+\sqrt{T}$ and $W(-t)=W(T-t)$ for
$t\in[-T,0]$. Now fix $K>0$.

Firstly, the arguments in the proof of
Proposition~\ref{prop:isometry-BD-BHP} show that for $T>0$ large, one
can construct a coupling such that with high probability,
Equality~\eqref{eq:BHP-coupling1} holds jointly with an equality of
$(F,W)$ and $(\Xha,\Wha)$ on $[-K,K]^2\subset [-T,T]^2$.

Secondly, let $m_n=\lceil\delta^{-4}a_n^4\rceil$ and
$\sigma_{m_n}=\lceil\sqrt{2m_n}\rceil$ as in the proof of
Theorem~\ref{thm:UIHPQ-BHP}. We extend the contour function $C_{m_n}$ of
$Q_{m_n}^{\sigma_{m_n}}$ to $t\in [-(2m_n+\sigma_{m_n}),0]$ by setting
$C_{m_n}(t)=C_{m_n}(2m_n+\sigma_{m_n}+t)+\sigma_{m_n}$. Similarly, we
extend the label function $\La_{m_n}$ by letting
$\La_{m_n}(t)=\La_{m_n}(2m_n+\sigma_{m_n}+t)$ for
$t\in[-(2m_n+\sigma_{m_n}),-1]$, and then by linear interpolation on
$[-1,0]$ between $\La_{m_n}(-1)$ and $\La_{m_n}(0)=0$.

From the proof of Proposition~\ref{prop:Qn-UIHPQ} we see that for
$\delta$ small and $n$ large, one can construct a coupling such that with
high probability, Equality~\eqref{eq:BHP-coupling2} holds jointly with an
equality of $(C_{m_n},\La_{m_n})$ and $(C_\infty,\La_\infty)$ on
$[-Ka_n^4,Ka_n^4]^2$.

Thanks to~\cite{Be3,BeMi}, we already know that the convergence
of $a_n^{-1}\cdot Q_{m_n}^{\sigma_{m_n}}$ to $\BD_{T,\sqrt{T}}$ (with
$T=\delta^{-4}(8/9)$) holds jointly with the convergence
$$\left(\frac{C_{m_n}((9/4)a_n^4\cdot)}{(3/2)a_n^{2}},
    \frac{\La_{m_n}((9/4)a_n^4\cdot)}{a_n}\right)
\xrightarrow[n\to\infty]{(d)}\left(F,W\right)
$$
on $\mathcal{C}([-T, T]^2,\R)^2$. Putting these observations
together,~\eqref{eq:jointconv-CLBr} follows.

We come back to Display~\eqref{eq:jointconv-CLBr} in the proof of
Theorem~\ref{thm:BHP3} below.
\end{remark}

\subsection{Brownian half-plane with non-zero skewness
  (Theorem~\ref{thm:BHP3})}
\label{sec:brownian-half-plane-1}

Theorem~\ref{thm:BHP3} covers the regime $\sqrt{n}\ll \sigma_n\ll n$ when
$a_n\sim 2\sqrt{\theta n/3\sigma_n}$ for some $\theta\in (0,\infty).$ The
parameter $\theta$ measures the skewness of the limiting Brownian
half-plane. Note that the regimes where $\BHP$ corresponding to the choice
$\theta=0$ appears is already treated in Theorem~\ref{thm:BHP1}.

We work in the usual setting introduced Section~\ref{sec:usualsetting}; in
particular, the pair $(Q_n^{\sigma_n},\vd)$ consisting of a quadrangulation
and a distinguished vertex is uniformly distributed over
$\cQ_{n,\sigma_n}^\bullet$ and encoded by a triplet
$((\f_n,\la_n),\br_n)\in\Fo_{\sigma_n}^n \times \mathcal B_{\sigma_n}$. The
associated contour pair is denoted $(C_n,L_n)$, and the corresponding label
function takes the form $\La_n(t) = L_n(t) + \br_n(-\uC_n(t))$, $0\leq
t\leq 2n+\sigma_n$.

It will be convenient to view both $C_n$ and $\La_n$ as continuous
functions on $\R$. Let $N=2n+\sigma_n$. We extend $C_n$ first to $[-N,N]$
by $C_n(t)=C_n(N+t)+\sigma_n$ for $t\in[-N,0]$, and then to all reals $t$
by setting $C_n(t)=C_n(t\vee (-N)\wedge N)$. Similarly, we let
$\La_n(t)=\La_n(N+t)$ for $t\in[-N,-1]$, with linear interpolation on
$[-1,0]$ between $\La_n(-1)$ and $0$. Outside $[-N,N]$, we set
$\La_n(t)=\La_n(t\vee (-N)\wedge N)$.  In this way, we interpret $C_n$ and
$\La_n$ as functions in $\mathcal{C}(\R,\R)$. Recall that they completely
determine $(Q_n^{\sigma_n},\vd)$.

\begin{mdframed} {\bf Idea of the proof.} Fix $r\geq 0$. The ball
  $B_{ra_n}(Q_n^{\sigma_n})$ of radius $ra_n$ around the root in
  $Q_n^{\sigma_n}$ is with high probability encoded by the union of the
  first $ca_n^2$ and last $ca_n^2$ trees of $\f_n$ for some $c>0$, together
  with their labels and the corresponding bridge values along the floor of
  $\f_n$. In Lemma~\ref{lem:RN-Deriv}, we calculate the Radon-Nikodym
  derivative of the law of these $2ca_n^2$ trees with respect to the law of
  $2ca_n^2$ independent critical geometric Galton-Watson trees. In this
  way, we explicitly relate the laws of $B_{ra_n}(Q_n^{\sigma_n})$ and
  $B_{ra_n}(Q_\infty^{\infty})$ to each other. Since we already know that
  $a_n^{-1}\cdot B_{ra_n}(Q_\infty^\infty)$ converges to $B_r(\BHP_0)$
  jointly with its properly rescaled contour and label functions, see
  Remark~\ref{rem:jointconv-CLBr}, it remains to identify the limiting
  Radon-Nikodym derivative, which we find to be the Radon-Nikodym
  derivative of a (two-sided) Brownian motion with drift $-\theta$ with
  respect to standard Brownian motion. An application of the Pitman
  transform then concludes the proof.
\end{mdframed}

Let us now give the details and first introduce some supplementary notation.
Given a continuous function
$f:\mathbb{R}\rightarrow\mathbb{R}$ and $x\in\mathbb{R}$, let
$$
U_x(f)=\inf\{t\leq 0: f(t)=x\}\in [-\infty,0],\quad T_x(f)=\inf\{t\geq 0: f(t)=x\}\in [0,\infty].
$$
In words, $U_x(f)$ is the time of the last visit to $x$ to the left of $0$,
with $U_x(f)=-\infty$ if there is no such time, and $T_x(f)$ is the first
time $f$ visits $x$ to the right of $0$, with $T_x(f)=\infty$ if there is
no such time.  Of course, we can also apply $T_x$ to functions in
$\mathcal{C}([0,\infty),\R)$, and $U_x$ to functions in
$\mathcal{C}((-\infty,0],\R)$.

For $f\in\mathcal{C}(\R,\R)$ and $x>0$, set
$$
v(f,x) = \frac{1}{2}\left(T_{-x}(f)-U_{x}(f) -2x\right)
$$
whenever all terms on the right hand side are finite, and $v(f,x)=\infty$
otherwise. Note that if $x$ is an integer and $f$ is the contour path of an
infinite forest, then $v(f,x)$ is the total number of edges of the $2x$
trees that are encoded by $f$ along the interval $[U_{x}(f),T_{-x}(f)]$.

Let $s>0$ be given. For the rest of this section, we will always set
$s_n=\lfloor (3/2)sa_n^2\rfloor$. Since $a_n^2\ll\sigma_n\ll n$, we  will
implicitly assume that $n$ is so large such that $s_n<\sigma_n<n$.

We first prove an absolute-continuity relation on the interval
$[U_{s_n},T_{-s_n}]$ between $C_n$ and the contour function $C_\infty$ of a
critical infinite forest. For that purpose, we define two probability laws
$\P_{n,r}$, $\Q_{n,r}$ on $\mathcal{C}(\R,\R)$ as follows:
\begin{align*}
\P_{n,s}&=\mathcal{L}\left(\left(C_n(t\vee U_{s_n}(C_n)\wedge
      T_{-s_n}(C_n)),t\in\R\right)\right),\\
\Q_{n,s}&=\mathcal{L}\left(\left(C_\infty(t\vee
    U_{s_n}(C_\infty)\wedge T_{-s_n}(C_\infty)),t\in\R\right)\right).
\end{align*}

\begin{lemma}
\label{lem:RN-Deriv}
Let $s>0$. The laws $\P_{n,s}$ and $\Q_{n,s}$ are absolutely continuous with respect to each
other. Moreover, given $\eps>0$, there exists $n_0\in\N$
such that for all $n\geq n_0$, with $s_n=\lfloor (3/2)sa_n^2\rfloor$,
$$
\sum_{f\in\textup{supp}(\P_{n,s})}\left|\P_{n,s}(f)-\e^{2s\theta-\frac{v(f,s_n)}{(9/4)a_n^4}\theta^2}\Q_{n,s}(f)\right|\leq \eps,
$$
where \textup{supp}$(\P_{n,s})\subset \mathcal{C}(\R,\R)$ denotes the
support of $\P_{n,s}$ (which is equal to \textup{supp}$(\Q_{n,s})$).
\end{lemma}

\begin{proof}
  From the constructions of $C_n$ and $C_\infty$, it is clear that each
  realization of $\P_{n,s}$ is a realization of $\Q_{n,s}$, and vice versa.
  Now let $s>0$ and $\eps>0$. We first show that there exists
  $c_v>0$ such that for $n$ sufficiently large,
\begin{equation}\label{eq:RN-Deriv-toshow1}
\sum_{f\in\textup{supp}(\P_{n,s}):\atop v(f,s_n)>c_v a_n^4}\left|\P_{n,s}(f)-\e^{2s\theta-\frac{v(f,s_n)}{(9/4)a_n^4}\theta^2}\Q_{n,s}(f)\right|\leq \eps/2.
\end{equation}
Since $\theta$ and $s$ are fixed, the last display follows if we show that
for some $c_v>0$,
$$
\P_{n,s}\left(f\in \mathcal{C}(\R,\R): v(f,s_n)>c_v a_n^4\right)\leq
\eps/4,\quad\Q_{n,s}\left(f\in \mathcal{C}(\R,\R): v(f,s_n)>c_v
  a_n^4\right)\leq \eps/4. 
$$
Write $T_k$ for the first hitting time of $k$ of a simple random
walk started at zero. By construction of $C_\infty$, we have
$$
\Q_{n,s}\left(\{v(f,s_n)>c_v
  a_n^4\}\right)= \P\left(T_{-2s_n}>2c_v a_n^4+2s_n\right),
$$
and standard random walk estimates give the existence of
$n_0\in \N$ and $c_v>0$ (depending on $s$, but $s$ is fixed) such that for
$n\geq n_0$, $\Q_{n,s}(\{v(f,s_n)>c_v
  a_n^4\})\leq \eps/4.$ Similarly,
$$
\P_{n,s}\left(\{v(f,s_n)>c_v
  a_n^4\}\right) = \P\left(T_{-2s_n}>2c_va_n^4+2s_n\,|\,T_{-\sigma_n}=2n+\sigma_n\right),
$$
and since $\sigma_n\gg \sqrt{n}$, it is easy to check that the probability
on the right is bounded by the unconditioned probability
$\P\left(T_{-2s_n}>2c_v a_n^4+2s_n\right)\leq \eps/4$. This shows~\eqref{eq:RN-Deriv-toshow1}.

It remains to argue that for fixed $c_v$ and all $n$ large enough, we have also
\begin{equation}\label{eq:RN-Deriv-toshow2}
\sum_{f\in\textup{supp}(\P_{n,s}):\atop v(f,s_n)\leq c_v a_n^4}\left|\P_{n,s}(f)-\e^{2s\theta-\frac{v(f,s_n)}{(9/4)a_n^4}\theta^2}\Q_{n,s}(f)\right|\leq \eps/2.
\end{equation}
In this regard, consider a sequence
$f_n\in\mathcal{C}(\R,\R)$ of functions in the support of $\P_{n,s}$ such
that $v_n=v(f_n,s_n)\leq c_v a_n^4$.  Let
$$x_n=\sigma_n -  2s_n, \quad y_n = 2(n-v_n)+ \sigma_n -2s_n.$$ We can
assume that both $x_n$ and $y_n$ are positive numbers. Let $(S(i), i \in
\N_0)$ denote a simple random walk started at $S(0)=0$. The probability
$\P_{n,s}(f_n)$ is given by the
probability to observe $2s_n$ particular trees of total size $v_n$ in a
forest of size $n$ with $\sigma_n$ trees. By Kemperman's
formula~\eqref{eq:Kemperman}, we obtain
\begin{align} 
      \P_{n,s}(f_n) &=
   \frac{\frac{x_n}{y_n }2^{y_n} \P(S(y_n) =
     x_n)}{\frac{\sigma_n}{2n+\sigma_n}
     2^{2n+\sigma_n}  \P(S(2n+\sigma_n)= \sigma_n)}\nonumber\\
   &= \frac{x_n}{y_n}\frac{2n+\sigma_n}{\sigma_n}2^{-2(v_n+s_n)}
   \frac{\P(S(y_n) =
     x_n)}{\P(S(2n+\sigma_n)= \sigma_n)}\,.  
\label{eq:RN-Deriv-eq1}
 \end{align}
 By definition of $C_\infty$, $\Q_{n,s}(f_n)$ is given by a
 particular realization of $2s_n$ independent critical geometric Galton-Watson trees with $v_n$
 edges in total. Therefore, by~\eqref{eq:criticalGW}, 
\begin{equation}
\label{eq:RN-Deriv-eq2}
\Q_{n,s}(f_n)=2^{-2(v_n + s_n)}.
\end{equation}
 Moreover, by assumption on $\sigma_n$ and $a_n$, we have uniformly in all
 possible choices of $f_n$ that satisfy $v_n\leq c_va_n^4$,
\begin{equation}
  \label{eq:RN-Deriv-eq3}
  \left|\frac{x_n}{y_n}\frac{2n+\sigma_n}{\sigma_n}-1\right| = o(1).
\end{equation}
Since $\sigma_n \gg \sqrt n$, the fraction of random walk probabilities
in~\eqref{eq:RN-Deriv-eq1} is not controlled well-enough by a standard
local central limit theorem as formulated in~\eqref{eq:localCLT}. Instead, we
use~\eqref{eq:localCLT2} and obtain
\begin{equation}
\label{eq:RN-Deriv-eq4}
  \frac{\P(S(y_n) = x_n)}{\P(S(2n+\sigma_n)= \sigma_n )} = \exp\left(-\sum_{\ell = 1}^\infty \frac{1}{2\ell (2 \ell
    -1)} \left(\frac{x_n^{2 \ell}}{ {y_n}^{2 \ell  -1}}
  -\frac{\sigma_n^{2\ell}}{ (2n+\sigma_n)^{2\ell-1}}\right)\right)\left(1+o(1)\right).
\end{equation}
We now analyze the terms in the sum inside the exponential in the last
display, similarly to the proof of Lemma~\ref{lem:GW3}. Firstly,
\begin{multline}
  \frac{x_n^{2 \ell}}{ {y_n}^{2 \ell -1}}
  -\frac{\sigma_n^{2\ell}}{ (2n+\sigma_n)^{2\ell-1}} = \\
  \frac{\sigma_n^{2\ell}}{(2n+\sigma_n)^{2 \ell -1}}\left[ -2 \ell
    \frac{2s_n}{\sigma_n} + (2 \ell -1) \frac{2(v_n +
      s_n)}{2n+\sigma_n} + O\left(\left(\frac{s_n}{\sigma_n}\right)^2 \right
    ) + O\left(\left(\frac{v_n+ s_n}{2n+\sigma_n}\right)^2\right)
  \right].
\label{eq:RN-Deriv-eq5}
\end{multline}
We now observe that 
\begin{align*}
  -2\ell\frac{\sigma_n^{2\ell}}{(2n+\sigma_n)^{2 \ell -1}} \frac{ 
2s_n}{\sigma_n} 
  & =  (-4\ell s\theta+o(1))\frac{\sigma_n^{2(\ell-1)}}{(2n+\sigma_n)^{2 
(\ell -1)}},\quad\hbox{ and}\\
(2\ell-1)  \frac{\sigma_n^{2\ell}}{(2n+\sigma_n)^{2 \ell -1}}  \frac{2(v_n +
        s_n)}{2n+\sigma_n}  &  = (2\ell-1)\frac{2v_n}{(9/4)a_n^4}\left(\theta^2+o(1)\right)\frac{\sigma_n^{2(\ell-1)}}{(2n+\sigma_n)^{2 
(\ell -1)}}.
\end{align*}
Since $\sigma_n\ll n$, we deduce from the last display that if
$\ell\geq 2$, all the terms in~\eqref{eq:RN-Deriv-eq5} converge to $0$ as $n\rightarrow\infty$. If 
$\ell =1$,
\begin{align*}
  -2\ell\frac{\sigma_n^{2\ell}}{(2n+\sigma_n)^{2 \ell -1}} \frac{ 
2s_n}{\sigma_n} 
  & =  -4s\theta+o(1),\quad\hbox{ and}\\
(2\ell-1)  \frac{\sigma_n^{2\ell}}{(2n+\sigma_n)^{2 \ell -1}}  \frac{2(v_n +
        s_n)}{2n+\sigma_n}  &  = 
\frac{2v_n}{(9/4)a_n^4}\theta^2 +o(1).
\end{align*}
For $n$ large enough, $\sigma_n/(2n+\sigma_n) <1/2$, so that each
 term in the sum in~\eqref{eq:RN-Deriv-eq4} is bounded by $C(1/2)^{2(\ell-1)}$ for some
 universal constant $C>0$, which is summable. Therefore, by dominated
 convergence
$$  \frac{\P(S(y_n) = x_n)}{\P(S(2n+\sigma_n)= \sigma_n )}  = \exp\left(
    2s \theta - \frac{v_n}{(9/4)a_n^4}\theta^2\right) +o(1).
$$ 
Note that all the error terms above do depend on $f_n$ only through the
constant $c_v$. Combining the last display with~\eqref{eq:RN-Deriv-eq2}
and~\eqref{eq:RN-Deriv-eq3},~\eqref{eq:RN-Deriv-toshow2} and hence the
claim of the lemma follow.
\end{proof}
\begin{remark}
  Note that $C_\infty$ is a discrete analog of the contour function of the
  Brownian half-plane $\BHP$: The process $(C_\infty(i),i\in \N_0)$ is a
  simple random walk, and if $S=(S(i),i\in\N_0)$ denotes another
  (independent) simple random walk, then it is straightforward to check
  that
$$
(C_\infty(-i),i\in\N)=_d \left(S(i+1)-2\min_{0\leq \ell\leq i+1}S(\ell)
  +1,\,i\in\N\right),
$$
i.e., $(C_\infty(-i),i\in\N)$ is a discrete Pitman-type transform of a
simple random walk. In particular, $-U_k(C_\infty)=_d T_{-k}(S)$.
\end{remark}

For proving Theorem~\ref{thm:BHP3}, it is convenient to introduce
some more notation. Let us first define
rescaled versions of the contour and label functions $C_n$ and $\La_n$ that
capture the information encoded by the first $s_n=\lfloor (3/2)sa_n^2\rfloor$ trees
$(\tau_0,\ldots,\tau_{s_n-1})$ and the last $s_n$ trees
$(\tau_{\sigma_n-s_n},\ldots,\tau_{\sigma_n-1})$ of $\f_n$,
\begin{align*}
C_{n,s}=\left(C_{n,s}(t),t\in\R\right)&=\left(\frac{1}{(3/2)a_n^2}C_{n}\left((9/4)a_n^4t\vee
  U_{s_n}(C_n)\wedge T_{-s_n}(C_n)\right),t\in \R\right),\\
\La_{n,s}=\left(\La_{n,s}(t),t\in\R\right)&=\left(\frac{1}{a_n}\La_{n}\left((9/4)a_n^4t\vee
  U_{s_n}(C_n)\wedge T_{-s_n}(C_n)\right),t\in \R\right).
\end{align*}

Let $((\f_{\infty},\la_\infty),\br_\infty)$ encode the $\UIHPQ$, with
$C_\infty$ and $\La_\infty$ denoting the associated contour and label functions.
In analogy to the last display, we set
\begin{align*}
  C^{\infty}_{n,s}=\left(C^{\infty}_{n,s}(t),t\in\R\right)&=\left(\frac{1}{(3/2)a_n^2}C_{\infty}\left((9/4)a_n^4t\vee
      U_{s_n}(C_{\infty})\wedge T_{-s_n}(C_{\infty})\right),t\in \R\right),\\
  \La_{n,s}^\infty=\left(\La_{n,s}^\infty(t),t\in\R\right)&=\left(\frac{1}{a_n}\La_{\infty}\left((9/4)a_n^4t\vee
      U_{s_n}(C_{\infty})\wedge T_{-s_n}(C_{\infty})\right),t\in \R\right).
\end{align*}
We recapitulate the definition of the contour and label functions
$X^\theta=(X^\theta(t),t\in\R)$ and $W^\theta=(W^\theta(t),t\in\R)$ which
encode the Brownian half-plane $\BHP_\theta$: $(X^\theta(t),t\geq 0)$ is given by a Brownian
motion with linear drift $-\theta$, and $(X^\theta(-t),t\geq 0)$ is the
Pitman transform of an (independent) copy of $(X^\theta(t),t\geq
0)$. Moreover, conditionally on $X^\theta$, the label function
$W^\theta=(W^\theta(t),t\in\R)$ is given by
$W^\theta(t)=\gamma(-\underline{X}^\theta(t))+Z^\theta(t)$, $t\in \R$, where
$Z^\theta=(Z^\theta(t),t\in\R)$ is the random snake driven by
$X^\theta-\underline{X}^\theta$, and $\gamma=(\gamma(t),t\in\R)$ is a
two-sided Brownian motion with $\gamma(0)=0$ and scaled by the factor
$\sqrt{3}$, independent of $Z^{X^\theta-\underline{X}^\theta}$.

We set
\begin{align*}
X^{\theta,s}=\left(X^{\theta,s}(t),t\in\R\right)&=\left(X^\theta\left(t\vee
  U_{s}(X^\theta)\wedge T_{-s}(X^\theta)\right),t\in \R\right),\\
W^{\theta,s}=\left(W^{\theta,s}(t),t\in\R\right)&=\left(W^\theta\left(t\vee U_{s}(X^\theta)\wedge
T_{-s}(X^\theta)\right),t\in \R\right).
\end{align*}
Finally, for $f\in\mathcal{C}(\R,\R)$, put
$$
\lambda_{n,s}(f)=\exp\left(2s\theta-\frac{v(f,s_n)}{(9/4)a_n^4}\theta^2\right).
$$
\begin{proof}[Proof of Theorem~\ref{thm:BHP3}]
Let $r\geq 0$. By Lemma~\ref{lem:ball0}, our claim follows if we show that 
$$
B_r^{(0)}\left(a_n^{-1}\cdot Q_n^{\sigma_n}\right)\xrightarrow[n \to
\infty]{(d)}B_r(\BHP_{\theta})$$ in distribution in $\mathbb{K}$, where we
recall that $\theta=\lim_{n\rightarrow\infty}(3/2)a_n^2\sigma_n/2n.$ For
$n\in \N$ and $s>0$, define the events
$$
\mathcal{E}^1(n,s)=\left\{\min_{[0,\,s_n]}\br_n <-3ra_n,\,
  \min_{[\sigma_n-s_n,\,\sigma_n-1]}\br_n<-3ra_n\right\}\cap\left\{\min_{[s_n+1,\sigma_n-(s_n+1)]}\br_n<-3ra_n\right\}$$
and similarly
\begin{align*}\mathcal{E}^2(n,s)&=\left\{\min_{[0,\,s_n]}\br_\infty
   <-\,3ra_n,\,
   \min_{[-s_n,\,0]}\br_\infty<-3ra_n\right\},\\
\mathcal{E}^3(s)&=\left\{\min_{[0,s]}\gamma <-3r,\,
   \min_{[-s,0]}\gamma<-3r\right\}.
\end{align*}
Let
$\eps>0$ be given. Applying Lemma~\ref{lem:bridge0}, we find $n_0\in \N$ and $s>0$ sufficiently
large such that for $n\geq n_0$, $\P(\mathcal{E}^1(n,s))\geq
1-\eps$. For possibly larger values of $n$ and $s$, Donsker's
invariance principle shows that also $\P(\mathcal{E}^2(n,s))\geq
1-\eps$, and standard properties of Brownian motion give $\P(\mathcal{E}^3(s))\geq
1-\eps$ for $s$ large enough. We now fix $s>0$ and $n_0\in\N$ such that for all $n\geq n_0$, each of the
events $\mathcal{E}^1,\mathcal{E}^2,\mathcal{E}^3$ has probability at least $1-\eps$.

As in the proof of Proposition~\ref{prop:Qn-UIHPQ}, we write
$\tau(\f_\infty,k)$ for the tree of $\f_\infty$ which is attached to $(k)$,
$k\in\Z$. We identify $V(\f_\infty)$ with $V(Q_\infty^\infty)$, as
usual. Recall that the root $\rho$ of $\UIHPQ$ is at distance at
most $-\br_\infty(-1)+1$ away from $(0)$.  On the event $\mathcal{E}^2(n,s)$,
the cactus bound~\eqref{eq:cactus3} thus gives for vertices $v\in V(Q_\infty^\infty)$ which do
not belong to any of the trees $\tau(\f_\infty,k)$, $k=-s_n,\ldots,s_n$,
$$
d_{\infty}(0,v) \geq -\max\left\{\min_{[0,s_n]}\br_\infty, \min_{[-s_n,0]}\br_\infty\right\} \geq 3ra_n
$$
for large $n$. Since for vertices $u,v$ in $B_{ra_n}^{(0)}(Q_\infty^\infty)$, any geodesic
between $u$ and $v$ in $Q_\infty^\infty$ lies entirely in
$B_{2ra_n}^{(0)}(Q_\infty^\infty)$, we obtain from the construction of edges in
the Bouttier-Di Francesco-Guitter mapping that the submap
$B_{ra_n}^{(0)}(Q_\infty^\infty)$ is a measurable function of
$(C_{n,s}^\infty,\La_{n,s}^\infty)$. A similar argument which we leave to
the reader (see also the first part of the proof of Proposition~\ref{prop:Qn-UIHPQ})
shows that on $\mathcal{E}^1(n,s)$, the submap $B_{ra_n}^{(0)}(Q_n^{\sigma_n})$ is
given by the {\it same} function of $(C_{n,s},\La_{n,s})$. Moreover, on
$\mathcal{E}^3(s)$, $B_r(\BHP)$ is determined by $(X_{0,s},W_{0,s})$.

By Lemma~\ref{lem:bridge1}, recalling that
$a^2_n\ll\sigma_n$, we have for large $n$
\begin{equation*}
  \begin{split}&\left\|\mathcal{L}((\br_n(\sigma_n-s_n),\ldots,\br_n(\sigma_n-1),\br_n(0),\br_n(1),\ldots,\br_n(s_n)))\right.\\
    &\quad\left. -\,\mathcal{L}((\br_\infty(-s_n),\ldots,\br_\infty(-1),\br_\infty(0),\br_\infty(1),\ldots,\br_\infty(s_n)))\right\|_{\textup{TV}}\leq
    \eps.\end{split}
\end{equation*}
Combining this bound with Lemma~\ref{lem:RN-Deriv}, the above
observations entail that for any measurable and bounded $F:
\mathcal{C}(\R,\R)^2\times\mathbb{K}\rightarrow\R$ and $n$ large enough
\begin{multline}
\label{eq:BHP3-mainproof-eq1}
\left|\E\left[F\left(C_{n,s},\La_{n,s},B_r^{(0)}\left(a_n^{-1}\cdot
      Q_n^{\sigma_n}\right)\right)\1_{\mathcal{E}^1(n,s)}\right] -\right.\\
\left.\E\left[\lambda_{n,s}(C_\infty)F\left(C^\infty_{n,s},
    \La_{n,s}^\infty,B_r^{(0)}\left(a_n^{-1}\cdot
      Q_\infty^\infty\right)\right)\1_{\mathcal{E}^2(n,s)}\right]\right|\leq C\eps,
\end{multline}
where $C>0$ is a constant that depends only on $F$ and $\theta,s$, which
are fixed. Recall from the proof of Lemma~\ref{lem:RN-Deriv} that for each $\delta>0$, we find
$c_\delta>0$ such that $\P(v(C_\infty,s_n)>c_\delta a_n^4)\leq \delta$. The
joint convergence~\eqref{eq:jointconv-CLBr} thus implies
$$
\left(C^\infty_{n,s}, \La_{n,s}^\infty,B_r^{(0)}\left(a_n^{-1}\cdot
      Q_\infty^\infty\right)\right)\xrightarrow[n\to\infty]{(d)} \left(X^{0,s},W^{0,s},B_r(\BHP)\right)
$$
in $\mathcal{C}(\R,\R)^2\times \mathbb{K}$, and
$$\frac{v(C_\infty,s_n)}{(9/4)a_n^4}\xrightarrow[n
\to \infty]{(d)}\frac{1}{2}\left(T_{-s}-U_s\right)(X^{0}),
$$
where, in hopefully obvious notation, $X^0$ stands for the contour function
of the Brownian half-plane with zero skewness, and $X^{0,s}$,$W^{0,s}$ were
defined above in terms of $\BHP$. For large $n$, we can therefore ensure
that
\begin{multline}
\label{eq:BHP3-mainproof-eq2}
  \Big|\E\left[\lambda_{n,s}(C_\infty)F\left(C^\infty_{n,s},
        \La_{n,s}^\infty,B_r^{(0)}\left(a_n^{-1}\cdot
          Q_\infty^\infty\right)\right)\right] -\\
  \E\left[\exp\left(2s\theta-(T_{-s}-U_s)(X^0)\theta^2/2\right)F\left(X^{0,s},W^{0,s},B_r\left(\BHP\right)\right)\right]\Big|\leq
  \eps.
\end{multline}
We will now rewrite the second expectation in the last display using
Girsanov's (and implicitly Pitman's) transform. More specifically, an application of
Girsanov's theorem for Brownian motion with drift $-\theta$ (see, e.g.,~\cite[Chapter
3.5]{KaSc}) shows that for
$G:\mathcal{C}(\R,\R)\rightarrow\R$ continuous and bounded,
$$
\E\left[\exp\left(2s\theta-(T_{-s}-U_s)(X^0)\theta^2/2\right)G\left(X^{0,s}\right)\right]=\E\left[G\left(X^{\theta,s}\right)\right].
$$
Since on the event $\mathcal{E}^3(s)$,
$B_r(\BHP)$ is a measurable function of $(X^{0,s},W^{0,s})$ (and
$B_r(\BHP_\theta)$ is given by the {\it same} measurable function of
$(X^{\theta,s},W^{\theta,s})$), we obtain
\begin{multline}
\label{eq:BHP3-mainproof-eq3}
\E\left[\exp\left(2s\theta-(T_{-s}-U_s)(X^0)\theta^2/2\right)F\left(X^{0,s},W^{0,s},B_r\left(\BHP\right)\right)\1_{\mathcal{E}^3(s)}\right]\\
=\E\left[F\left(X^{\theta,s},W^{\theta,s},B_r\left(\BHP_\theta\right)\right)\1_{\mathcal{E}^3(s)}\right].
\end{multline}
Using that the three events $\mathcal{E}^1(n,s)$, $\mathcal{E}^2(n,s)$ and
$\mathcal{E}^3(s)$ have all probability at least $1-\eps$, a
combination of~\eqref{eq:BHP3-mainproof-eq1},~\eqref{eq:BHP3-mainproof-eq2}
and~\eqref{eq:BHP3-mainproof-eq3} shows that for large $n$
$$
\left|\E\left[F\left(C_{n,s},\La_{n,s},B_r^{(0)}\left(a_n^{-1}\cdot
      Q_n^{\sigma_n}\right)\right)\right] -
  \E\left[F\left(X^{\theta,s},W^{\theta,s},B_r\left(\BHP_\theta\right)\right)\right]\right|\leq
  C'\eps $$
for some $C'$ depending only on $F$ and $s,\theta$. Clearly, this implies our claim.
\end{proof}

\subsection{Coupling Brownian disks (Theorem~\ref{thm:coupling-BD-IBD} and Corollary~\ref{cor:topology-IBD})}
\label{sec:proof-coupling-BD-IBD}
The main ideas are similar to those of
Section~\ref{sec:proof-coupling-BD-BHP}, but closer in spirit to those
of~\cite{CuLG}. We begin
with showing how Theorem~\ref{thm:coupling-BD-IBD} implies that
$\IBD_\sigma$ is homeomorphic to the pointed closed disk
$\overline{\mathbb{D}}\setminus \{0\}$.
\begin{proof}[Proof of Corollary~\ref{cor:topology-IBD}] The arguments are similar
  to the proof of Corollary~\ref{cor:topology-BHP}. First,
  Theorem~\ref{thm:coupling-BD-IBD} shows that with probability $1$, for
  every $r>0$, the ball $B_r(\IBD_\sigma)$ is contained in a set
  homeomorphic to $\overline{\mathbb{D}}\setminus \{0\}$. In particular,
  $\IBD_\sigma$ is a non-compact surface with a boundary homeomorphic to
  the circle $\mathbb{S}^1$, and it has only one end. Let us glue a copy
  $D$ of $\overline{\mathbb{D}}$ along the boundary of $\IBD_\sigma$, hence
  obtaining a non-compact surface $S$ without boundary, which is now simply
  connected. This surface is thus homeomorphic to $\R^2$. Again, the
  Jordan-Schoenflies theorem shows that any homeomorphism from the boundary
  of $\IBD_\sigma$ to $\mathbb{S}^1$ can be extended to a homeomorphism
  from $S$ to $\R^2$, and this homeomorphism must send $\IBD_\sigma$ to the
  unbounded region $\{z:|z|\geq 1\}$, which in turn is homeomorphic to
  $\overline{\mathbb{D}}\setminus\{0\}$, as wanted.
\end{proof}
As in Section~\ref{sec:proof-coupling-BD-BHP}, we first prove the following
simplification of Theorem~\ref{thm:coupling-BD-IBD}.
\begin{prop}\label{prop:isometry-BD-IBD}
  Fix $\sigma \in (0,\infty)$, and let $\eps > 0$, $r \geq 0$. There
  exists $T_0 = T_0(\eps, r, \sigma)$ such that for all $T \geq T_0$, we
  can construct copies of $\BD_{T,\sigma}$ and $\IBD_\sigma$ on the same
  probability space such that with probability at least $1-\eps$,
  the balls $B_r(\BD_{T,\sigma})$ and $B_r(\IBD_\sigma)$
  of radius $r$ around the respective roots are isometric.
\end{prop}

The crucial step in the proof of the proposition is to show how one can
couple the processes encoding $\BD_{T,\sigma}$ and $\IBD_\sigma$. The rest
of the proof then uses arguments very close to those given in~\cite[Section
3.2]{CuLG} and in Section~\ref{sec:proof-coupling-BD-BHP} above.

\subsubsection{Coupling of contour and label functions} 
Throughout this section, $\sigma\in(0,\infty)$ is fixed, and $T$ denotes
always a strictly positive real. We recall that the main building block of
the Brownian disk $\BD_{T,\sigma}$ is a first-passage Brownian bridge from
$0$ to $-\sigma$ and duration $T$. Let $B=(B_t,t\geq 0)$ be a standard
Brownian motion. In this section, it will be convenient to write
$T_x=\inf\{t\geq 0: B_t<-x\}$ for the first hitting time of $(-\infty,-x)$
of the process $B$, so that $(T_x,0\leq x\leq \sigma)$ is a stable
subordinator of index $1/2$ and Laplace exponent $-\log\E[\exp(-\lambda
T_1)]=\sqrt{2\lambda}$.

Let us write the jump sizes of $(T_x,0\leq x\leq \sigma)$, together
with the times in $[0,\sigma]$ at which they occur, as a point measure
$$\mathcal{M}=\sum_{i\geq 1}\delta_{(\Delta_i,U_i)}\, ,$$
so that $T_{U_i}-T_{U_i-}=\Delta_i$. By well-known properties of
subordinators, this measure is Poisson with intensity measure $(2\pi
y^3)^{-1/2}\d y\otimes \d u\1_{[0,\sigma]}(u)$. As a
consequence, the random variables $U_i,i\geq 1$, are i.i.d.\ uniform in
$[0,\sigma]$ and independent of $(\Delta_1,\Delta_2,\ldots)$. This
property will remain true when we condition the measure $\mathcal{M}$
on events that involve only the sequence
$(\Delta_1,\Delta_2,\ldots)$. 

The first-passage bridge consists in the process $(B_t,0\leq t\leq T)$
conditioned on the event $\{T_{\sigma}=T\}=\{\sum_i\Delta_i=T\}$.  In
order to describe the conditional law of $\mathcal{M}$, we follow
Pitman~\cite[Chapter 4]{Pi} and fix the ordering
$\Delta_1,\Delta_2,\ldots$ as the {\em size-biased ordering} of the
jumps, so that conditionally given $(\Delta_1,\ldots,\Delta_i)$,
$\Delta_{i+1}$ is chosen from all the remaining jumps with probability
that is proportional to its size. 

\begin{lemma}[\cite{Pi}]
  \label{lem:coupl-brown-disks}
  Conditionally given $\{T_{\sigma}=T\}$,
  the law of $\Delta_1$ is
$$\P(\Delta_1\in \d y| T_{\sigma}=T)=\frac{\sigma \d y}{T\, (2\pi
  y)^{1/2}}\frac{q_\sigma(T-y)}{q_\sigma(T)}=\e^{\sigma^2/2T}\sqrt{\frac{T}{y}}q_\sigma(T-y)\d
y\, ,$$
and conditionally given $\{T_{\sigma}=T\, ,\, \Delta_1=y\}$, the
remaining jumps $(\Delta_2,\Delta_3,\ldots)$ have the same
distribution as $(\Delta_1,\Delta_2,\ldots)$ conditionally given
$\{T_{\sigma}=T-y\}$. 
\end{lemma}

This allows us to obtain the main technical lemma of this section,
which one should see as the continuum version of Lemmas~\ref{lem:GW1}
and~\ref{lem:GW2}: it says that given, $T_{\sigma}=T$, the jumps
behave as those of the unconditioned subordinator $(T_{x},0\leq x\leq
\sigma)$, with the exception of the largest jump of size approximately
$T$.

\begin{lemma}
  \label{lem:coupl-brown-disks-1}
\begin{enumerate}
\item For every $\delta\in (0,1)$, one has
$$\liminf_{T\to\infty}\P\left(\Delta_1>(1-\delta)T\, \bigg|\, \sum_i\Delta_i=T\right)=1\,
.$$
\item
One has
$$\lim_{T\rightarrow \infty}\left\|\mathcal{L}\left(\Delta_2,\Delta_3,\ldots\, \bigg|\, \sum_i\Delta_i=T\right)-\mathcal{L}(\Delta_1,\Delta_2,\ldots)\right\|_{\mathrm{TV}}=0.$$
\end{enumerate}
\end{lemma}
\begin{proof}
  From the description of the conditional law of $\Delta_1$ given in
  Lemma~\ref{lem:coupl-brown-disks}, we obtain
$$\P(\Delta_1>(1-\delta)T\, |\,
\sum_i\Delta_i=T)=\e^{\sigma^2/2T}\int_0^{\delta T}\d
x\sqrt{\frac{T}{T-x}}q_\sigma(x)\d
x\underset{T\to\infty}{\longrightarrow} \int_0^\infty q_\sigma(x)\d x=1\,
,$$
by dominated convergence since $\sqrt{T/(T-x)}\leq
(1-\delta)^{-1/2}$. This proves (a). For (b), one can use the second
part of Lemma~\ref{lem:coupl-brown-disks} to obtain the disintegration 
$$\mathcal{L}\left(\Delta_2,\Delta_3,\ldots\, \bigg|\,
  \sum_i\Delta_i=T\right)
=\int_0^T \d
x\, \e^{\sigma^2/2T}\sqrt{\frac{T}{T-x}}\, q_\sigma(x)\mathcal{L}\left(\Delta_1,\Delta_2,\ldots\,
\bigg|\, \sum_i\Delta_i=x\right)\, .$$
Since $q_\sigma$ is the density function of $T_{\sigma}=\sum_i\Delta_i$, we also
have the disintegration
$$\mathcal{L}\left(\Delta_1,\Delta_2,\ldots\right)
=\int_0^\infty \d
x \, q_\sigma(x)\mathcal{L}\left(\Delta_1,\Delta_2,\ldots\,
\bigg|\, \sum_i\Delta_i=x\right)\, ,$$
which entails that 
\begin{align*}
  \lefteqn{\left\|\mathcal{L}\left(\Delta_2,\Delta_3,\ldots\,
        \bigg|\,
        \sum_i\Delta_i=T\right)-\mathcal{L}(\Delta_1,\Delta_2,\ldots)\right\|_{\mathrm{TV}}}\\
&\leq \int_T^\infty q_\sigma(x)\d x +\int_0^T\left|
  \e^{\sigma^2/2T}\sqrt{\frac{T}{T-x}}-1\right|q_\sigma(x)\, \d x\, .
\end{align*}
The first integral obviously converges to $0$, and we can split that
second integral at $T/2$ and rewrite it, after simple manipulations,
as
$$\int_0^{T/2}\left|\e^{\sigma^2/2T}\sqrt{\frac{T}{T-x}}-1\right|q_\sigma(x)\,
\d x +
T\int_0^{1/2}\left|\e^{\sigma^2/2T}\sqrt{\frac{1}{x}}-1\right|q_\sigma(T(1-x))\d
x\, .$$
The first term converges to $0$ by dominated convergence, and the
second vanishes as well since $q_\sigma(T(1-x))\leq 2\sigma/\sqrt{\pi
T^3}$ for every $x\in [0,1/2]$. 
\end{proof}

In the next proposition, we let $(F_t,0\leq t\leq T)$ be a
first-passage bridge of Brownian motion hitting $-\sigma$ for the
first time at $T$. We will let $T^F(x)=\inf\{t\geq 0:F_t<-x\}\wedge T$
for $0\leq x\leq \sigma$. Similarly, we let
$\Delta_0^F,\Delta_1^F,\Delta_2^F,\ldots$ be the jump sizes of $T^F$
ranked in size-biased order, and $U^F_0,U^F_1,\ldots$ be the
corresponding levels. For $i\geq 0$, we let
$$e^F_i(t)=U_i^F+F(T^F(U_i^F-)+t)\, ,\qquad 0\leq t\leq
\Delta_i^F,$$
be the excursion of $F$ above level $-U_i$ --- note that
$\Delta_i^F=T^F(U_i)-T^F(U_i-)$. 

We also let $B$ be a (unconditioned)
standard Brownian motion, and let $\Delta_1,\Delta_2,\ldots$ be the jump sizes
of the first-hitting time subordinator $(T_{x},0\leq x\leq
\sigma)$. We let $R,R'$ be two independent
three-dimensional Bessel processes, independent of $B$. Finally, we let
$U_0$ be a uniform random variable in $[0,\sigma]$, independent of
$B,R,R'$. Figure~\ref{fig:coupling-contour-IBD} illustrates the following proposition.
\begin{figure}[ht]
  \centering
  \includegraphics[width=0.8\textwidth]{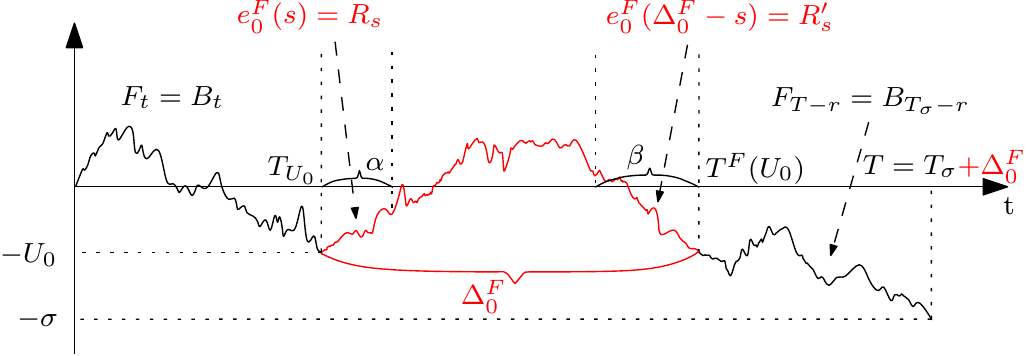}
\caption{The coupling of contour functions stated as
  Proposition~\ref{prop:coupl-brown-disks-2}, with $s=s(t)= t-T_{U_0}$, $r=r(t)=T-t$.}
  \label{fig:coupling-contour-IBD}
\end{figure}

\begin{prop}
  \label{prop:coupl-brown-disks-2}
  For every $\eps\in (0,1)$ and $\alpha,\beta>0$, there exists $T^0>0$ such that for
  every $T>T^0$, it is possible to couple $F,B,R,R',U_0$ on the same
  probability space in such a way that with probability at least
  $1-\eps$, one has $U_0=U_1^F$ and 
$$F_t=B_t,\,0\leq t\leq T^F(U_0-)=T_{U_0}\, ,\qquad
F_{T-t}=B_{T_{\sigma}-t},\,0\leq t\leq T-T^F(U_0)=T_\sigma-T_{U_0}\, ,$$
and
$$e^F_0(t)=R_t,\quad 0\leq t\leq \alpha\, ,\qquad
e^F_0(\Delta_0^F-t)=R'_t,\quad 0\leq t\leq \beta\, ,$$
and finally
$$\inf_{[\alpha,\infty)}R\wedge
\inf_{[\beta,\infty)}R'=\min_{[\alpha,\Delta^F_0-\beta]}e^F_0\, .$$
\end{prop}

\begin{proof}
  By Lemma~\ref{lem:coupl-brown-disks-1}, for $T$ large enough, say
  $T>T^1$, it is
  possible to couple two sequences $\Delta_1,\Delta_2,\ldots$ and
  $\Delta'_0,\Delta'_1,\Delta'_2,\ldots$ on the same probability space such that
\begin{itemize}
\item $(\Delta_1,\Delta_2,\ldots)$ has the law of the jump sizes of
  $(T_{x},0\leq x\leq \sigma)$ ranked in size-biased order, and
\item $(\Delta'_0,\Delta'_1,\Delta'_2,\ldots)$ has the law of
  $(\Delta_1,\Delta_2,\ldots)$ conditionally given $\sum_{i\geq 1}\Delta_i=T$, 
\end{itemize}
in such a way that on an event $\mathcal{E}_1$ of probability at least
$1-\eps/2$, one has
$$\Delta_i=\Delta'_{i}\, ,\qquad i\geq 1\, ,\qquad \mbox{ and
}\qquad \Delta'_0>T/2\, .$$
Extending the probability space if necessary, we can assume that it
also supports an independent family of random variables $e_0,e_1,e_2,\ldots$ that are
independent normalized Brownian excursions, and $U_0,U_1,U_2,\ldots$ that
are independent uniform random variables in $[0,\sigma]$, independent
of all the rest.

By It{\^o}'s synthesis of Brownian motion from its excursions, if we
set, for $i\geq 1$, 
$$B_t=-U_i+\frac{e_i(\Delta_i(t-\sum_{j:U_j<U_i}\Delta_j))}{\sqrt{\Delta_i}}\,
,$$ whenever $\sum_{j\geq 1:U_j<U_i}\Delta_i< t\leq \sum_{j\geq 1:U_j\leq
  U_i}\Delta_j$, then this a.s.\ extends to a continuous path $(B_t,0\leq
t\leq \sum_{i\geq 1}\Delta_i)$ which is a trajectory of Brownian motion
stopped when first hitting $-\sigma$, which occurs at time
$T_{\sigma}=\sum_{i\geq 1}\Delta_i$. Similarly, setting, this time for
$i\geq 0$,
$$F(t)=-U_i+\frac{e_i(\Delta_i'(t-\sum_{j:U_j<U_i}\Delta'_j))}{\sqrt{\Delta'_i}}\,
,$$ whenever $\sum_{j\geq 0:U_j<U_i}\Delta'_i< t\leq \sum_{j\geq 0:U_j\leq
  U_i}\Delta'_j$, this extends to a trajectory of a first-passage bridge
$(F(t),0\leq t\leq T)$ from $0$ to $-\sigma$, as the notation suggests, and
if we set $\Delta^F_i=\Delta'_i$ for $i\geq 0$ then $(\Delta^F_i,i\geq 0)$
is indeed a size-biased ordering of the jumps of the first hitting time
process of negative values of $F$.

On the event $\mathcal{E}_1$, the two processes $B$ and $F$
coincide on the interval $[0,\sum_{j\geq 1:U_j<U_0}\Delta_j]$, and
likewise, $B_{T_{\sigma}-\cdot}$ and $F(T-\cdot)$ coincide on
$[0,\sum_{j\geq 1:U_j>U_0}\Delta_j]$. This yields the first displayed
identity in the statement, since by construction 
$$\sum_{j\geq 1:U_j<U_0}\Delta'_j =T^F(U_0-)\, ,\qquad \sum_{j\geq
  1:U_j<U_0}\Delta_j=\sum_{j\geq 1:U_j\leq U_0}\Delta_j=T_{U_0}\,
, $$
while we have
$$\sum_{j\geq 1:U_j>U_0}\Delta'_j =T-T^F(U_0)\, ,\qquad \sum_{j\geq
  1:U_j>U_0}\Delta_j=\sum_{j\geq 1:U_j\geq
  U_0}\Delta_j=T_{\sigma}-T_{U_0}\, . $$ Finally, in this
construction, and still in restriction to $\mathcal{E}_1$,
$e_0^F=e_0(\Delta_0'\cdot)/\sqrt{\Delta_0'}$ is an excursion of
Brownian motion with duration $\Delta'_0>T/2$. At this point, we can
apply Proposition 3 in~\cite{CuLG}, in the same way as in the proof of
Proposition 4 therein. Up to a further extension of the probability
space, as soon as $T$ is chosen large enough, say $T>T^2$, we can
couple this ``long'' excursion with two independent Bessel processes
$R,R'$ (and independent of all previously defined random variables) in
such a way that the three last identities of the statement are
satisfied on an event $\mathcal{E}_2$ with probability at least
$1-\eps/2$. This yields the wanted result with $T^0=T^1\vee T^2$,
since the intersection $\mathcal{E}_1\cap\mathcal{E}_2$ has
probability at least $1-\eps$.
\end{proof}

\subsubsection{Isometry of balls in {\normalfont $\BD_{T,\sigma}$} and {\normalfont $\IBD_\sigma$}}
We fix $\sigma\in(0,\infty)$, $\eps>0$ and let $r\geq 0$.  With the
coupling from the preceding section at hand, the proof the proposition is a
minor modification of~\cite[Proof of Proposition 4]{CuLG} (see also 
Proposition~\ref{prop:isometry-BD-BHP} and its proof). We will point at the
necessary modifications and then leave it to reader to fill in the
remaining details.

We work in the notation and with the processes of Proposition~\ref{prop:coupl-brown-disks-2} and
denote additionally by $\gamma= (\gamma_u,0\leq u\leq \sigma)$ a standard
Brownian bridge with duration $\sigma$, multiplied by $\sqrt{3}$.

\begin{proof}[Proof of Proposition~\ref{prop:isometry-BD-IBD}]
Let us first introduce a few events. For
$K>0$, put
$$
\mathcal{E}^1(K)=\left\{\max_{[0,\sigma]}\gamma < K\right\}.
$$
Then, given $A>0$, with $\zeta=(\zeta_t,t\geq 0)$ denoting a Brownian
motion started at $0$, let
$$
\mathcal{E}^2(A,K)=\left\{\min_{[0,A]}\zeta<-10r-K,\,\min_{[A,A^2]}\zeta<-10r-K,\,\min_{[A^2,A^4]}\zeta<-10r-K\right\},
$$
and for $A>0$ and $\alpha>0$, set
$$
\mathcal{E}^3(A,\alpha)=\left\{\inf_{[\alpha,\infty)}R\wedge \inf_{[\alpha,\infty)}R'>A^4\right\}.
$$
We first choose $K$ sufficiently large such that $\P(\mathcal{E}^1)\geq
1-\eps/6.$ Then, standard properties of Brownian motion allow us to find
$A>0$ such that $\P(\mathcal{E}^2)\geq 1-\eps/6$ as well, and with such a
fixed $A$, we find by transience of the Bessel process an $\alpha>0$ large
enough such that $\P(\mathcal{E}^3)>1-\eps/3$.

Let us next recall the contour process $Y^\sigma$ of $\IBD_\sigma$
specified just before Definition~\ref{def:IBD} in terms of the Bessel
processes $R$ and $R'$ and the Brownian motion $B$ stopped at times
$T_{U_0}$ and $T_\sigma$. We obtain that on the coupling event
$\mathcal{E}^4=\mathcal{E}^4(\alpha,T)$ described in the statement of
Proposition~\ref{prop:coupl-brown-disks-2} (with $\beta=\alpha$), in the notation from there,
$$ F_t=Y^\sigma_t\quad\mbox{for } t\in[0,T_{U_0}+\alpha]\,,\quad \mbox{ and }\quad 
F_{T-t}+\sigma=Y^\sigma_{-t}\quad\mbox{for }
t\in[0,T_\sigma-T_{U_0}+\alpha]\,.$$ Concerning $\BD_{T,\sigma}$, we use
the notation from Section~\ref{sec:recapBHPBD} (note however that $\sigma$
is now a constant not depending on the volume $T$). We build the 
label process of $\BD_{T,\sigma}$ in the following way: Consider a
Brownian bridge $\gamma$ as specified above, independent of
$(F,B,R,R',U_0)$. Then let $Z=Z^{F-\underline{F}}$ be the random snake
driven by $F-\underline{F}$, and set
$$W_t=\gamma_{-\underline{F}_t}+Z_t,\quad 0\leq t\leq T.$$

Concerning $\IBD^\sigma$, we write $\Zi=Z^{Y^\sigma-\uuY}$ for
the random snake driven by $Y^\sigma-\uuY$, see Definition~\ref{def:IBD}, and 
$\Wi_t=\gamma_{-\uuY_t}+\Zi$ for the label process associated with
$\IBD_\sigma$. Of course, we choose to use
the same bridge $\gamma$ to construct $W$ and $\Wi$. 

We now work always conditionally on $(F,B,R,R',U_0)$.  Similarly to the
considerations around~\eqref{eq:tuple1} and~\eqref{eq:tuple2} in the proof
of Proposition~\ref{prop:isometry-BD-BHP}, one checks that on the event
$\mathcal{E}^4$, the covariance function of
$$(W_t,0\leq t\leq
T_{U_0}+\alpha),(W_{T-t},0\leq t\leq T_\sigma-T_{U_0}+\alpha)$$ on the one
hand, and $$(\Wi_t,0\leq t\leq T_{U_0}+\alpha),(\Wi_{-t},0\leq t\leq
T_\sigma-T_{U_0}+\alpha)$$ on the other hand, are the same. Consequently,
we may assume that $W$ and $\Wi$ are coupled in such a way that, on the
event $\mathcal{E}^4$,
$$
W_t = \Wi_t\quad\hbox{ for all }t\in[0,T_{U_0}+\alpha],\quad W_{T-t} =
\Wi_{-t}\quad\hbox{ for all }t\in[0,T_\sigma-T_{U_0}+\alpha].
$$
From Proposition~\ref{prop:coupl-brown-disks-2}, we derive that for the
choice of $\alpha$ from above, the coupling event $\mathcal{E}^4(\alpha,T)$
has probability at least $1-\eps/3$ provided $T$ is sufficiently large, and
we shall work with such a $T$. The reminder of the proof is now close
to~\cite[Proof of Proposition 4]{CuLG}. For every $x\geq 0$, let
\begin{align*}
\eta_{\textup{l}}(x)&=\sup\{0\leq t\leq \Delta_0^F/2:e_0^F(t)=x\}+T_{U_0},\\
\eta_{\textup{r}}(x)&=\Delta_0^F-\inf\{\Delta_0^F/2\leq t\leq \Delta_0^F:e_0^F(t)=x\}+T_\sigma-T_{U_0},
\end{align*}
and
\begin{align}
\label{eq:etaIBD}
\eta_{\textup{l}}^{\textup{I}}(x)&=\sup\{t\geq 0: R_t=x\}+T_{U_0},\nonumber\\
\eta_{\textup{r}}^{\textup{I}}(x)&=\sup\{t\geq 0: R_t'=x\} +T_\sigma-T_{U_0}.
\end{align}
Then the process
$(Z^\textup{I}_{\eta_{\textup{l}}^\textup{I}(x)}, x\geq 0)$ 
has the law of a standard Brownian motion started at
$Z^\textup{I}_{T_{U_0}}=0$. Choosing this Brownian motion in the definition of the
event $\mathcal{E}^2$ from above, so that on $\mathcal{E}^2$, we have
\begin{equation}
\label{eq:boundZi}
\min_{x\in[0,A]}Z^\textup{I}_{\eta_{\textup{l}}^{\textup{I}}(x)}<-6r-K,\,\min_{x\in[A,A^2]}Z^{\textup{I}}_{\eta_{\textup{l}}^{\textup{I}}(x)}<-6r-K,\,\min_{x\in[A^2,A^4]}Z^{\textup{I}}_{\eta_{\textup{l}}^{\textup{I}}(x)}<-6r-K,
\end{equation}
we shall from now on work on the intersection of events
\begin{equation}
\label{eq:couplingIBD-BD-defF}
\mathcal{F}=\mathcal{E}^1\cap \mathcal{E}^2\cap\mathcal{E}^3\cap\mathcal{E}^4,
\end{equation}
which has probability at least $1-\eps$.

On $\mathcal{E}^3\cap\mathcal{E}^4$, we note that $\min_{[\alpha,\Delta_0^F-\alpha]}e^F_0=\inf_{[\alpha,\infty)}R
\wedge \inf_{[\alpha,\infty)} R' >A^4$, whence for $x\in[0,A^4]$,
$\eta_{\textup{l}}(x)=\eta_{\textup{l}}^{\textup{I}}(x)< T_{U_0}+\alpha$ and 
$\eta_{\textup{r}}(x)=\eta_{\textup{r}}^{\textup{I}}(x)< T_{\sigma}-T_{U_0}+\alpha$. 
It follows that for any $x\in[0,A^4]$,
$$Z_{\eta_{\textup{l}}(x)}= Z_{\eta_{\textup{l}}^\textup{I}(x)}^{\textup{I}}=
Z_{-\eta_{\textup{r}}^\textup{I}(x)}^{\textup{I}} = Z_{T-\eta_{\textup{r}}(x)}.$$

We are now almost in a setting where we can appeal to the reasoning
in~\cite[Section 3.2]{CuLG}. We should still adapt the definition of
$\tilde{d}_W(s,t)$ given just before Lemma~\ref{lem:DBD} to the setting
considered here. Let $s,t\in [0,T]$. If $s,t$ lie both in either
$[0,T_{U_0}+\Delta_0^F/2]$ or in $[T_{U_0}+\Delta_0^F/2,T]$, we let
$$
d'_{W}(s,t) = W_s +W_t -2\min_{[s\wedge t,
  s\vee t]}W.
$$
Otherwise, we set
$$
d'_{W}(s,t) = W_s +W_t -2\min_{[0,s\wedge
  t]\cup[s\vee t,T]}W.
$$

Recall the definition of the pseudo-metric $D(s,t)$ associated to the
Brownian disk $\BD_{T,\sigma}$. The following statement replaces
Lemma~\ref{lem:DBD} and is close to~\cite[Lemma 5(i)]{CuLG}.
\begin{lemma}
\label{lem:DBD2}
 Assume $\mathcal{F}$ holds.  
\begin{enumerate}
\item For every $t\in [\eta_{\textup{l}}(A),T-\eta_{\textup{r}}(A)]$, $D(0,t) >
  r$.
\item For every $s,t\in [0,\eta_{\textup{l}}(A)]\cup[0,T-\eta_{\textup{r}}(A)]$ with
  $\max\{D(0,s), D(0,t)\}\leq r$, it holds that
$$
D(s,t) =\inf_{s_1,t_1,\dots,s_k,t_k}\sum_{i=1}^k d'_{W}(s_i,t_i),
$$
where the infimum is over all possible choices of $k\in\N$ and reals
$s_1,\dots,s_k,t_1,\dots,t_k\in [0,\eta_{\textup{l}}(A^2)]\cup[T-\eta_{\textup{r}}(A^2),T]$ such that
$s_1=s,t_k=t$, and  $d_{F}(t_i,s_{i+1})=0$ for $1\leq i\leq k-1$.
\end{enumerate}
\end{lemma}
\begin{proof}
  One can follow the same line of reasoning as in~\cite[Proof of Lemma
  5(i)]{CuLG}, with one small modification, which is apparent from the
  proof of (a), so let us prove this part. If $t\in
  [\eta_{\textup{l}}(A),T-\eta_{\textup{r}}(A)]$, then by the cactus
  bound~\eqref{eq:cactusDBD}, 
$$
D(0,t)\geq W_t-2\max\left\{\min_{[0,t]}W,\min_{[t,T]}W\right\}\geq
  -\max\left\{\min_{[0,\eta_{\textup{l}}(A)]}W,\min_{[T-\eta_{\textup{r}}(A),T]}W\right\}.
$$
Recalling that $W_t=\gamma_{-\underline{F}_t}+Z_t$, we first remark that on
the event $\mathcal{E}^3\cap\mathcal{E}^4$, since $\eta_{\textup{l}}(A)<
T_{U_0}+\alpha$, we have $Z=\Zi$ on $[0,\eta_{\textup{l}}(A)]$. Since
$\eta_{\textup{l}}([0,A])\subset [0,\eta_{\textup{l}}(A)]$, it follows now
from~\eqref{eq:boundZi} that the minimum of $Z$ on
$[0,\eta_{\textup{l}}(A)]$ is bounded from above by $-6r-K$. But on
$\mathcal{E}^1$, $\max \gamma < K$, so that
$\min_{[0,\eta_{\textup{l}}(A)]}W\leq -6r$. A similar argument holds for
the second minimum, so that in fact $D(0,t)\geq 6r$ whenever
$t\in[\eta_{\textup{l}}(A),T-\eta_{\textup{r}}(A)]$. For (b), one can
follow~\cite[Proof of Lemma 5(i)]{CuLG}, or modify the proof of (b) in
Lemma~\ref{lem:DBD}.
\end{proof}
Entirely similar, one finds the corresponding statement for the
pseudo-metric $\Di$ of $\IBD_\sigma$ that replaces Lemma~\ref{lem:DH}: In
the statement there, $\eta_{\textup{r}}$ and $\eta_{\textup{l}}$ have to be
replaced by $\eta_{\textup{r}}^{\textup{I}}$ and
$\eta_{\textup{l}}^{\textup{I}}$ as defined under~\eqref{eq:etaIBD}, and
$D_\theta$, $d_{W^\theta}$ by $\Di$ and $d_{\Wi}$. Following
again~\cite{CuLG}, or adapting the second part of the proof of
Proposition~\ref{prop:isometry-BD-BHP}, these two lemmas lead to the stated
isometry between $B_r(\BD_{T,\sigma})$ and $B_r(\IBD_\sigma)$ on the event
$\mathcal{F}$ of probability at least $1-\eps$, finishing thereby the proof
of Proposition~\ref{prop:isometry-BD-IBD}. 
\end{proof}

It remains to show how Proposition~\ref{prop:isometry-BD-IBD} can be
improved to the coupling stated in Theorem~\ref{thm:coupling-BD-IBD}.

\subsubsection{Proof of Theorem~\ref{thm:coupling-BD-IBD}}

The proof is close in spirit to that of Theorem~\ref{thm:coupling-BD-BHP}:
for a fixed $r>0$, we must find some $r_0>r$ large enough so that for all
$T$ sufficiently large, the ball $B_{r_0}(\BD_{T,\sigma})$ contains with
high probability an open set $A_{\BD}$ that is homeomorphic to the pointed
closed disk $\overline{\mathbb{D}}\setminus \{0\}$ and that, in turn,
contains the ball $B_r(\BD_{T,\sigma})$ with high probability. Then we will
apply the coupling Proposition~\ref{prop:isometry-BD-IBD} in order to
couple the balls $B_{r_0}(\BD_{T,\sigma})$ and $B_{r_0}(\IBD_{\sigma})$.
The set $A_{\BD}$ will be defined as a region bounded by certain geodesic
paths.

We use a similar notation to that of the proof of
Theorem~\ref{thm:coupling-BD-BHP}, setting $\Y=([0,T]/\{D=0\},D,\rho)$ and
letting $p_\Y$ be the associated canonical projection, so that $Y$ has the
law of the Brownian disk $\BD_{T,\sigma}$ with root $\rho=p_\Y(0)$. We will also
use the geodesic paths $\Gamma_s,s\in [0,T]$, in $\Y$ respectively from
$p_\Y(s)$ to $x_*=p_\Y(s_*)$ defined around
Lemma~\ref{lem:proof-prop-refpr-4}, together with the properties stated
there.

Let $a_0>0$ be a large number to be specified later on. We let
$A_\BD^0=[0,\eta_{\textup{l}}(a_0)]\cup [T-\eta_{\textup{r}}(a_0),T]$,
where $\eta_{\textup{l}},\eta_{\textup{r}}$ are defined in the proof of
Proposition~\ref{prop:isometry-BD-IBD}. We will work on the event that
$s_*\notin A_\BD^0$, a fact that will be granted later with high
probability by the fact that $T$ is going to infinity.  With our notation,
the definition of $A_\BD^0$ is the exact same as $O_\BD^0$ in the proof of
Theorem \ref{thm:coupling-BD-BHP}, but note that by contrast, the points
$p_\Y(\eta_{\textup l}(a_0))$ and $p_\Y(T-\eta_{\textup r}(a_0))$ are now
equal, and we denote it by $x_0$. Note in passing that $x_0\notin \partial
\Y$ by Lemma \ref{lem:proof-prop-refpr}.  We let $t_*\in A^0_\BD$ be such
that $W_{t_*}=\min_{A^0_\BD}W$. The geodesic paths $\Gamma_{\eta_l(a_0)}$
and $\Gamma_{T-\eta_r(a_0)}$ both start from $x_0$, but by (d) in Lemma
\ref{lem:proof-prop-refpr-4}, they become disjoint until they meet again
for the first time at the point $p_\Y(t_*)$. Therefore, the segments of
these geodesics between $x_0$ and $p_\Y(t_*)$ form a simple loop $P$, which
is disjoint from the boundary $\partial \Y$ by (c) in Lemma
\ref{lem:proof-prop-refpr}. We point at Figure~\ref{fig:topo-IBD} for an
illustration. The analog of Lemma \ref{lem:proof-prop-refpr-5} is the
following.

\begin{lemma}
  \label{sec:proof-theor-refthm:c}
The set $P$ is a simple loop in $\Y$ containing $x_0$, and that does
not intersect $\partial \Y$. Letting  $A_\BD$ be the connected
component of $\Y\setminus P$ that contains $p_\Y(0)$, then $A_\BD$ is
almost surely homeomorphic to the pointed closed disk
$\overline{\mathbb{D}}\setminus \{0\}$, and is the interior of the set
$p_\Y(A^0_\BD)$ in $\Y$. 
\end{lemma}

\begin{proof}
  The proof is very similar to that of Lemma
  \ref{lem:proof-prop-refpr-5}. The fact that $A_\BD$ is a.s.\ homeomorphic
  to $\overline{\mathbb{D}}\setminus\{0\}$ is a direct consequence of
  the fact that $\Y$ is homeomorphic to $\overline{\mathbb{D}}$ and
  that $P$ is a simple loop not intersecting $\partial \Y$. 
The only thing that remains to be
  proved given our discussion so far is that $A_\BD$ is indeed the
  interior of $p_\Y(A^0_\BD)$.  However, using the exact same
  definition of the paths $\Xi_s$, it is simple to see that a point in
  $p_\Y(A^0_\BD)$ is linked to $\partial \Y$, and hence to $p_\Y(0)$,
  by a simple path that intersects $P$, if at all, only at its
  starting point. It remains to verify that for every $x\in
  \Y\setminus p_\Y(A^0_\BD)$, we can find a simple path from $x$ to
  $p_\Y(s_*)$ that possibly intersects $P$ only at its origin. Writing
  $x=p_\Y(s)$, we leave it to the reader that such a path can be
  obtained by concatenating segments of the paths $p_\Y\circ
  \Xi_s$ and $p_\Y\circ \Xi_{s_*}$. The conclusion follows.
\end{proof}

Now for a fixed $r>0$ and $\eps>0$, we choose $a_0>0$ large enough so that 
$$\P\left(\min_{[0,a_0]}\Zi_{\eta^\textup{I}_{\textup l}(\cdot)}<-2r\right)\geq
1-\eps/4,$$
and then let $r_0>r$ be large enough so that 
$$\P\left(\omega(\Wi,[-\eta_r^{\textup{I}}(a_0),\eta_l^{\textup{I}}(a_0)])\leq
  r_0/2\right)\geq 1-\eps/4\, .$$ We use this value of $r_0$ to apply the
coupling of $\Y=\BD_{T,\sigma}$ and $\IBD_\sigma$ of Proposition
\ref{prop:isometry-BD-IBD}, guaranteeing that the balls of radius $r_0$ in
these pointed spaces are isometric with probability at least
$1-\eps/2$. From there, we conclude exactly as in the end of the proof of
Theorem \ref{thm:coupling-BD-BHP}, replacing $O_\BD$ and $O_\BHP=I(O_\BD)$
by $A_\BD$ and $A_\BHP=I(A_\BD)$, where $I$ is defined as before Corollary
\ref{cor:DBD-DH} and defines an isometry between $B_{r_0}(\Y)$
and $B_{r_0}(\IBD_{\sigma})$ on the coupling event $\mathcal{F}$ given
by~\eqref{eq:couplingIBD-BD-defF} (note that on this event, one has in
particular $\eta_{\textup l}(x)=\eta^{\textup{I}}_{\textup l}(x)$ and
$\eta_{\textup r}(x)=\eta_{\textup r}^{\textup{I}}(x)$ for every $x\leq
A^4$, and without loss of generality, we can choose $A$ so that $A^4>a_0$.)
This completes the proof of Theorem~\ref{thm:coupling-BD-IBD}.

\begin{figure}[htbp!]
  \centering
  \includegraphics[width=0.5\textwidth]{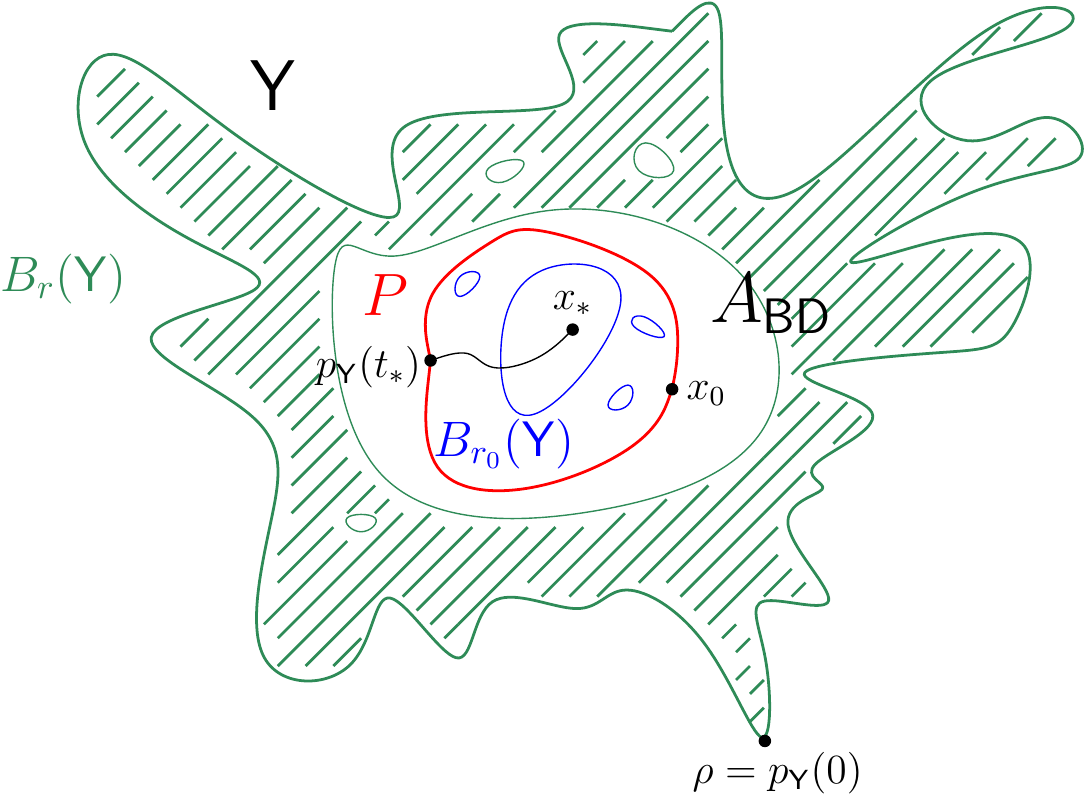}
  \caption{Illustration of the proof of
    Theorem~\ref{thm:coupling-BD-IBD}. Note that in this figure, we look at
    the disk $\Y=\BD_{T,\sigma}$ from above. The ball
    $B_{r}(\Y)$ contains the full boundary of $\Y$ and is
    included in the larger ball $B_{r_0}(\Y)$, whose boundary
    in $\Y$ is indicated by the loops in blue. The ball
    $B_{r_0}(\Y)$ encompasses the open set $A_\BD$, which is
    homeomorphic to the pointed disk
    $\overline{\mathbb{D}}\setminus\{0\}$. The set $A_\BD$ is bordered by
    the boundary of $\Y$ and the simple loop $P$ (in red) in $\Y$. The loop
    $P$ is formed by two segments of geodesics between $x_0$ and
    $p_{\Y}(t_\ast)$.}
  \label{fig:topo-IBD}
\end{figure}

\subsection{Infinite-volume Brownian disk (Theorem~\ref{thm:IBD})}
\label{sec:proof-thmIBD}
For proving Theorem~\ref{thm:IBD}, we will combine the 
convergence towards the Brownian disk $\BD_{T,\sigma}$ proved
in~\cite[Theorem 1]{BeMi} (see Display~\eqref{eq:BeMi}) with the couplings
Theorem~\ref{thm:coupling-BD-IBD} and
Proposition~\ref{prop:coupling-Qn-largevol}. We work in the usual setting
specified in Section~\ref{sec:usualsetting}.
 \begin{proof}[Proof of Theorem~\ref{thm:IBD}]
   Assume $1\ll \sigma_n\ll\sqrt{n}$ and $a_n\sim
   (4/9)^{1/4}\sqrt{\sigma_n/\sigma}$ for some $\sigma\in(0,\infty)$. We
   have to show that for each $r\geq 0$,
$$
B_r\left(a_n^{-1}\cdot Q_n^{\sigma_n}\right)\xrightarrow[n\to
\infty]{(d)}B_{r}(\IBD_\sigma)$$ in distribution in $\mathbb{K}$. We fix
$\eps>0$ and $r\geq 0$. By Theorem~\ref{thm:coupling-BD-IBD}, we find $T_0$
such that for all $T\geq T_0$, we can construct on the same probability
space copies of $\BD_{T,\sigma}$ and $\IBD_\sigma$ such that with
probability at least $1-\eps$, we have an isometry of balls
\begin{equation}
\label{eq:IBD-coupling1}
B_r(\BD_{T,\sigma})=B_r(\IBD_\sigma).
\end{equation}
By Proposition~\ref{prop:coupling-Qn-largevol}, we find $R_0\geq
T_0/(2\sigma^2)$ such that for $R\geq R_0$ and $n$ sufficiently large, we can
construct on the same probability space copies of
$Q_n^{\sigma_n}$ and $Q_{R\sigma_n^2}^{\sigma_n}$ such
that with probability at least $1-\eps$, there is the isometry
\begin{equation}
\label{eq:IBD-coupling2}
B_{ra_n}(Q_n^{\sigma_n})=B_{ra_n}\left(Q_{R\sigma_n^2}^{\sigma_n}\right).
\end{equation}
Now let $F:\mathbb{K}\rightarrow\R$ be a bounded and continuous function
and $R\geq R_0$. We assume that $Q_{n}^{\sigma_n}$ and
$Q_{R\sigma_n^2}^{\sigma_n}$ are constructed on the same probability space
such that~\eqref{eq:IBD-coupling2} holds, and similarly $\BD_{2R\sigma^2,\sigma}$ and
$\IBD_\sigma$ so that~\eqref{eq:IBD-coupling1} is satisfied. We write
\begin{align*}
  \lefteqn{\left|\E\left[F\left(
          B_r\left(a_n^{-1}\cdot Q_n^{\sigma_n}\right)\right)\right]-\E\left[F\left(B_r\left(\IBD_\sigma\right)\right)\right]\right|}\\
  &\quad\quad\leq \left|\E\left[F\left( a_n^{-1}\cdot
        B_{ra_n}\left(Q_n^{\sigma_n}\right)\right)-F\left(
        a_n^{-1}\cdot B_{ra_n}\left(Q_{R\sigma_n^2}^{\sigma_n}\right)\right)\right]\right|\\
  &\quad\quad\quad + \left|\E\left[F\left(
        a_n^{-1}\cdot B_{ra_n}\left(Q_{R\sigma_n^2}^{\sigma_n}\right)\right)\right]-\E\left[F\left(B_r\left(\BD_{2\sigma^2R,\sigma}\right)\right)\right]\right|\\
  &\quad\quad\quad + \left|
    \E\left[F\left(B_r(\BD_{2\sigma^2R,\sigma})\right)-F\left(B_r\left(\IBD_\sigma\right)\right)\right]\right|.
\end{align*}
Using~\eqref{eq:IBD-coupling2} and~\eqref{eq:IBD-coupling1} (note that
$2\sigma^2R\geq T_0$), the first and
third summand on the right hand side are bounded from above by
$2\eps\sup F$. The scaling property $\lambda\cdot \BD_{1,\sigma}=_d
\BD_{\lambda^4 ,\lambda^2\sigma}$ for $\lambda>0$ combined with the
convergence~\eqref{eq:BeMi} implies that the second summand converges to
zero as $n\rightarrow\infty$. This finishes the proof of
Theorem~\ref{thm:IBD}.
\end{proof}

\subsection{Brownian disk limits
  (Corollaries~\ref{cor:BD1},~\ref{cor:BD4},~\ref{cor:BD2}
  and~\ref{cor:BD3}).}
\label{sec:proofs-BDlimits}
\begin{proof}[Proof of Corollaries~\ref{cor:BD1},~\ref{cor:BD4},~\ref{cor:BD2} 
  and~\ref{cor:BD3}.]
  We have to show that for each $r\geq 0$, when $T$ tends to infinity,
  $B_r(\BD_{T,\sigma(T)})$ converges in law to the ball of radius $r$
  around the root in the limit space $\mathcal{X}$ that appears in the
  corresponding corollary. As usual, we consider only the case $r=1$.

  Let $F:\mathbb{K}\rightarrow\mathbb{R}$ be bounded and continuous. For
  $T\in\N$ and $n\in\N$, we set
  $$
  m_n(T)=Tn,\quad
  \sigma_n(T)=\lfloor\sigma(T)\sqrt{2n}\rfloor,\quad a_n=(8/9)^{1/4}n^{1/4}.
  $$
We write 
\begin{align*}
  \lefteqn{\left|\E\left[F\left(
          B_1\left(\BD_{T,\sigma(T)}\right)\right)\right] -
      \E\left[F\left(
          B_1\left(\mathcal{X}\right)\right)\right]\right|}\\
  &\quad\quad \leq \left|\E\left[F\left(
        B_1\left(\BD_{T,\sigma(T)}\right)\right)\right]
    -\E\left[F\left(a_n^{-1}\cdot
        B_{a_n}\left(Q_{m_n(T)}^{\sigma_n(T)}\right)\right)\right)\right|\\
  &\quad\quad\quad +\left|\E\left[F\left(a_n^{-1}\cdot
        B_{a_n}\left(Q_{m_n(T)}^{\sigma_n(T)}\right)\right)\right]
    -\E\left[F\left(B_1\left(\mathcal{X}\right)\right)\right]\right|.
 \end{align*}
 For each fixed $T\in\N$, the convergence~\cite[Theorem
 1]{BeMi} towards the Brownian disk with volume $T$ and perimeter
 $\sigma(T)$ (see~\eqref{eq:BeMi} above) implies that the first summand on
 the right hand side is bounded by $\eps$, provided $n\geq n_0(T)$.

 Concerning the second summand, we argue by contradiction that for large
 enough $T$, there exists $n_0=n_0=(T,\eps)$ such that for any $n\geq n_0$, the
 second summand is bounded by $\eps$ as well.  Indeed, assuming this is not
 the case, we find a sequence of integers $(T_k,k\in\N)$ with
 $T_k\rightarrow\infty$ and a sequence of integers $(n_k,k\in \N)$ with
 $n_k$ depending on $T_k$ and $n_k\rightarrow\infty$, such that
\begin{equation}
\label{eq:BDlimits-eq1}
\left|\E\left[F\left(a_{n_k}^{-1}\cdot
        B_{a_{n_k}}\left(Q_{m_{n_k}(T_k)}^{\sigma_{n_k}(T_k)}\right)\right)\right]
    -\E\left[F\left(B_1\left(\mathcal{X}\right)\right)\right]\right|>\eps.
\end{equation}
We put $T_0=n_0=0$ and for $n\in\N$,
$$
\tilde{a}_n= (8/9)^{1/4}\left(n/T_k\right)^{1/4},\quad
\tilde{\sigma}_n=\lfloor \sigma(T_k)\sqrt{2n/T_k}\rfloor,\quad \textup{ if
} T_{k}n_{k-1}<
n\leq T_k n_k.
$$
For concreteness, we now consider the framework of Corollary~\ref{cor:BD1},
where $\mathcal{X}=\BP$ and $\sigma(T)\rightarrow 0$ as $T$ tends to
infinity. Then $\tilde{\sigma}_n\ll\sqrt{n}$, and
$\sqrt{\tilde{\sigma}_n}\ll \tilde{a}_n\ll n^{1/4}$, so that we can apply
the convergence towards $\BP$ proved in Theorem~\ref{thm:BP} (with
$\sigma_n$ and $a_n$ there replaced by $\tilde{\sigma}_n$ and
$\tilde{a}_n$). However, observing the quadrangulations at sizes
$m_k=T_kn_k$,~\eqref{eq:BDlimits-eq1} contradicts the convergence towards
$\BP$.

In the setting of Corollary~\ref{cor:BD4}, we use Theorem~\ref{thm:IBD}
instead of Theorem~\ref{thm:BP} and the fact that
$\sigma(T)\rightarrow\varsigma\in(0,\infty)$ as $T\rightarrow\infty$. An
identical argument allows us to finish the proof in this
case, with $\mathcal{X}$ given by $\IBD_\varsigma$. In the
framework of Corollary~\ref{cor:BD3} where $\sigma(T)/T\rightarrow\infty$
as $T$ tends to infinity, we apply Theorem~\ref{thm:ICRT} instead.

Let us finally look at Corollary~\ref{cor:BD2}. There,
$\sigma(T)\rightarrow\infty$ and
$\sigma(T)/T\rightarrow\theta\in[0,\infty)$. If $\theta=0$, then, along
sequences $(T_m,m\in\N)$ tending to infinity for which
$\sigma(T_m)/\sqrt{T_m}\rightarrow 0$ as $m\rightarrow\infty$, we follow
the same argumentation by contradiction and use Theorem~\ref{thm:BHP1},
whereas if $\liminf_{m\rightarrow\infty}\sigma(T_m)/\sqrt{T_m}>0$, the
corollary is a direct consequence of Theorem~\ref{thm:coupling-BD-BHP}, and
so is it in the case $\theta>0$.
\end{proof}

\subsection{Infinite continuum random tree (Theorem~\ref{thm:ICRT})}
We use the second construction of $\ICRT=(\cT_Y,d_Y,[0])$ from
Section~\ref{sec:def} in terms of a standard two-sided Brownian motion
$(Y_t,t\in\mathbb{R})$ with $Y_0=0$. Recall that Theorem~\ref{thm:ICRT}
treats the regime $\sigma_n \gg \sqrt{n}$, and we explicitly allow
$\sigma_n$ to grow faster than $n$.

\begin{mdframed}{\bf Idea of the proof.} Let $(\f_n,\la_n)$ be a random
  uniform element of $\Fo_{\sigma_n}^n$. We show that under the rescaling
  by a factor $\max\{1,\sqrt{n/\sigma_n}\} \ll a_n \ll \sqrt{\sigma_n}$,
  the labels of the first and last $ca_n^2$ trees of $\f_n$ converge to
  zero. For the rescaled submap $a_n^{-1}B_{ra_n}(Q_n^{\sigma_n})$, this
  means that for large $n$, the boundary dominates and folds the map into a
  tree-shaped object, which we identify with the $\ICRT$ in the limit
  $n\rightarrow\infty$.
\end{mdframed}

\begin{proof}[Proof of Theorem~\ref{thm:ICRT}]
  We work in the usual setting, cf.~\ref{sec:usualsetting}. We will
  construct a set $\mathcal{R}_n$ which shares the properties of
  Lemma~\ref{lem:localGHconv}, and for that reason, we shall consider both
  $V(Q_n^{\sigma_n})$ and the corresponding length space ${\mathbf
    Q}_n^{\sigma_n}$, cf. Section~\ref{sec:locGH}.

  Let $r\geq 0$. Still the vertex set $V(\f_n)$ is naturally
  identified with points of ${\mathbf Q}_n^{\sigma_n}$, and we may consider
  the ball $B_r^{(0)}({\mathbf Q}_n^{\sigma_n})$ in ${\mathbf
    Q}_n^{\sigma_n}$ of radius $r$ around $\f_n(0)=(0)$ (with respect to
  the shortest-path metric $d$). Since the Gromov-Hausdorff distance
  between $B_r^{(0)}(Q_n^{\sigma_n})$ and $B_r^{(0)}({\mathbf
    Q}_n^{\sigma_n})$ is bounded by one, the theorem follows from
  Lemma~\ref{lem:ball0} if we show that
$$
B_r^{(0)}\left(a_n^{-1}\cdot {\mathbf Q}_n^{\sigma_n}\right)\xrightarrow[n
\to \infty]{(d)}B_r(\cT_Y),$$ for any scaling sequence $a_n$ with
$\max\{1,\sqrt{n/\sigma_n}\} \ll a_n \ll \sqrt{\sigma_n}$.
 
For ease of reading, we restrict ourselves to the case
$r=1$. Define
$$
Y_n(t)= \left\{\begin{array}{l@{\quad\mbox{if }}l}
\br_n(\sigma_n -1 + (a_n^2/2)t)&  -(\sigma_n-2)/a_n^2\leq t<0\\
\br_n((a_n^2/2)t)& 0\leq t\leq \sigma_n/a_n^2
\end{array}\right..
$$
We extend the definition of $Y_n$ to all reals by letting
$Y_n\equiv\br_n(\sigma_n/2)$ on $\{|t|> \sigma_n/a_n^2\}$. Note that $Y_n$
is c\`adl\`ag, with only one possible jump at $t=0$ of height
$-\br_n(\sigma_n-1)$, which is stochastically bounded in probability as
$n\rightarrow \infty$.

From Lemma~\ref{lem:bridge0} we know that $
(1/\sqrt{2\sigma_n})(\br_n(\sigma_nt),0\leq t\leq 1)$ converges to a
standard Brownian bridge $\mathbbm{b}$ in $\mathcal{C}([0,1],\mathbbm{R})$
as $n\rightarrow\infty$. From this and the fact that $a_n\ll
\sqrt{\sigma_n}$, we obtain by standard reasoning (see, e.g., the proof
of~\cite[Lemma 10]{Be1})
\begin{equation}
\label{eq:convYn-Y}
\left(\frac{Y_n(t)}{a_n},t\in\mathbb{R}\right)\xrightarrow[n \to
\infty]{(d)}\left(Y_t,t\in\mathbb{R}\right),
\end{equation}
where $Y=(Y_t,t\in\mathbb{R})$ is a two-sided Brownian motion with $Y_0=0$,
and the convergence holds in the space of c\`adl\`ag functions on
$\mathbb{R}$ equipped with the compact-open topology.

By the Skorokhod representation theorem, we may assume that
\eqref{eq:convYn-Y} holds almost surely uniformly over compacts. Let
$\eps>0$. By standard properties of Brownian motion, we find
$\alpha>0$ and $n_0\in \N$ such that
\begin{equation}
\label{eq:YYn-prop}
\max\left\{\min_{[0,\alpha]}Y_,\min_{[-\alpha,0]}Y\right\}<-1\quad\textup{and}\quad
\max\left\{\min_{[0,\alpha]}Y_n,\min_{[-\alpha,0]}Y_n\right\} < -a_n
\end{equation}
with probability at least $1-\eps$ for $n\geq n_0$. We fix such an
$\alpha$ and work on the event where~\eqref{eq:YYn-prop} holds true.
In the following, we make no difference between the
root vertices $(0),\ldots,(\sigma_n-1)$ and the integers
$0,\ldots,\sigma_n-1$. Moreover, recall that if $v$ is a vertex of a tree of $\f_n$,
$\an(v)$ denotes the root vertex of that tree. For
$v\in V(\f_n)$, the cactus bound~\eqref{eq:cactus2}
gives
\begin{equation}
\label{dgr1v}
\dgr((0),v) \geq -\max\left\{\min_{[
    0,\an(v)]}\br_n, \min_{[\an(v),\sigma_n-1]}\br_n\right\}.
\end{equation}
Every point $v$ in the length space ${\mathbf Q}_n^{\sigma_n}$ lies on a
segment $e$ of length one connecting two vertices in $V(Q_n^{\sigma_n})$.
We associate to each $v\in{\mathbf Q}_n^{\sigma_n}$ that endpoint $v'\in
V(Q_n^{\sigma_n})$ of $e$ which satisfies $d(0,v)\geq \dgr(0,v')$. We then
extend the definition of $\an$ to all points $v$ in ${\mathbf
  Q}_n^{\sigma_n}$ by letting $\an(v)=\an(v')$, where we agree that
$\an(\vd)=\infty$. Next define the subset of vertices
$$
A_n = \left\{v\in {\mathbf Q}_n^{\sigma_n}: \an(v) \in [0,\alpha a_n^2/2] \cup
[\sigma_n-1-\alpha a_n^2/2,\sigma_n-1]\right\}.
$$
On the event where~\eqref{eq:YYn-prop} holds, we have
for vertices $v\in {\mathbf Q}_n^{\sigma_n}\setminus A_n$ by~\eqref{dgr1v}
(and by~\eqref{eq:distance-vdot} for $v$ with $v'=\vd$) for $n$ sufficiently large 
$$d(0,v)\geq \dgr(0,v') >a_n.$$  In words, $A_n$ contains the ball of radius
$a_n$ around $(0)$ in $({\mathbf Q}_n^{\sigma_n},d)$.

We now consider the $\ICRT$ $(\cT_Y, d_Y,[0])$ defined in terms of $Y$ and
write $p_Y:\mathbb{R}\rightarrow \cT_Y$ for the canonical projection. 
Define a subset $\mathcal{R}_n\subset
{\mathbf Q}_n^{\sigma_n}\times\cT_Y$ by putting
\begin{equation*}
  \mathcal{R}_n=\left\{(v,p_Y(t)) :\begin{split}& v\in A_n,
      t\in[0,\alpha]\hbox{ with }\an(v) = \lfloor (a_n^2/2)t\rfloor,\hbox{ or }\\
  &v\in A_n, t\in[-\alpha,0]\hbox{ with
    }\an(v) = \lfloor\sigma_n-1-(a_n^2/2)t\rfloor\end{split}\right\}.
\end{equation*}
On the event where~\eqref{eq:YYn-prop} is true, the set $\mathcal{R}_n$
fulfills the conditions of Lemma~\ref{lem:localGHconv} (with $\rho=(0)$,
$\rho'=p_Y(0)$, $r=a_n$). Thus it remains to show that
$\lim_{n\rightarrow\infty}a_n^{-1}\dis(\mathcal{R}_n)= 0$ in
probability. Note that in the notation from above, we have $|d(v,w)-
\dgr(v',w')| \leq 2$. Recall that the label function $\La_n$ is given by
the labels $\la_n$ of the tree vertices, shifted according to the label of
the corresponding root vertex which is carried by the bridge $\br_n$. By
the distance bounds~\eqref{eq:cactus1} and~\eqref{eq:dist-upperbound}, we
obtain for $v,w\in A_n$ with $\an(v)\leq \an(w)$,
\begin{equation*}
\begin{split}
\left|d(v,w)-\left(\br_n(\an(v))+\br_n(\an(w))-2\max\left\{\min_{[\an(v),\an(w)]}
  \br_n,\min_{[0,\an(v)]\cup[\an(w),\sigma_n-1]}\br_n\right\}\right)\right|\\
\leq 3\left(\sup_{A_n} \la_n-\inf_{A_n} \la_n\right)+4.
\end{split}
\end{equation*}
Since $Y_n$ converges to $Y$,
cf.~\eqref{eq:convYn-Y}, the last display entails that 
\begin{equation}
\label{eq:thmICRT-distortion}
\limsup_{n\rightarrow\infty}\frac{1}{a_n}\dis(\mathcal{R}_n) \leq
\limsup_{n\rightarrow\infty}\frac{3\left(\sup_{A_n} \la_n-\inf_{A_n}
    \la_n\right)}{a_n},
\end{equation}
and we are reduced to show that the right hand side equals zero. For that
purpose, recall that $(C_n,L_n)$ denotes the contour pair associated to
$(\f_n,\la_n)$, where $C_n=(C_n(t),0\leq t\leq 2n+\sigma_n)$ is distributed
as a (linearly interpolated) simple random walk conditioned to first hit
$-\sigma_n$ at time $N=2n+\sigma_n$. For $t\geq 0$, put 
$$
\tilde{C}_n(t) =
\frac{1}{a_n^2}C_n\left(\left(N/\sigma_n\right)a_n^2t\wedge
  N\right),\quad\tilde{L}_n(t)=\frac{1}{a_n}L_n\left(\left(N/\sigma_n\right)a_n^2t\wedge N\right).
$$
The following lemma will complete our proof of Theorem~\ref{thm:ICRT}.
\begin{lemma}
\label{lem:ICRT}
In the regime $\sigma_n\gg \sqrt{n}$, we have for sequences $a_n$ of
positive reals that 
satisfy $\max\{1,\sqrt{n/\sigma_n}\}\ll a_n\ll\sqrt{\sigma_n}$
$$
\left((\tilde{C}_n(t),\tilde{L}_n(t)),t\geq 0\right)\xrightarrow[n\to
\infty]{(p)} \left((-t,0),t\geq 0\right)\quad\textup{in }\mathcal{C}([0,\infty),\mathbb{R})^2.
$$
\end{lemma}
Splitting the set $A_n$ from above into the disjoint sets $A_n=A_n^+\cup
A_n^{-}$, where 
$$
A_n^+=\{v\in {\mathbf Q}_n^{\sigma_n}: \an(v) \in [0,\alpha a_n^2/2]\},\quad
A_n^-=\{v\in {\mathbf Q}_n^{\sigma_n}: \an(v) \in [\sigma_n-1-\alpha a_n^2/2,\sigma_n-1]\},
$$
Lemma~\ref{lem:ICRT} shows that $\sup_{A_n^+}\la_n-\inf_{A_n^+}\la_n =
o(a_n)$ in probability as $n\rightarrow\infty$. By exchangeability of the trees of
$\f_n$, we obtain however also that $\sup_{A_n^-}\la_n-\inf_{A_n^-}\la_n =
o(a_n)$. Thus the right hand side of~\eqref{eq:thmICRT-distortion} is equal
to zero, and the proof of the theorem follows.
\end{proof}
It remains to prove Lemma~\ref{lem:ICRT}.
\begin{proof}[Proof of Lemma~\ref{lem:ICRT}]
  We first prove convergence of the first component. Let
  $0<\eps<1$. We have to show that for each
  $K>0$, as $n\rightarrow\infty$,
  \begin{equation}
\P\left(\sup_{t\in[0,K]}\left|\tilde{C}_n(t)+t\right|\geq
  \eps\right)\longrightarrow 0.
\end{equation}
Set $N=2n+\sigma_n$, and denote by $(S_i,i\in\N_0)$ a simple random walk
started from $S_0=0$. Write $T_{-\sigma_n}$ for its first hitting time of
$-\sigma_n$. Fix $K\geq 1$, and set
$K_n=\lceil(N/\sigma_n)a_n^2K\rceil$. Note that $K_n\ll N$.
By definition of $\tilde{C}_n$, we obtain
\begin{equation}
\label{eq:lemICRT-eq1}
  \P\left(\sup_{t\in[0,K]}\left|\tilde{C}_n(t)+t\right|\geq
    \eps\right) \leq \P\left(\sup_{0\leq i \leq K_n
    }\left|S_i+\frac{\sigma_n}{N}i\right|\geq \eps
    a_n^2\,\big |\,T_{-\sigma_n}=N\right).
  \end{equation}
With the abbreviation
$$\mathcal{E}_n=\left\{\sup_{0\leq i \leq
  K_n}\left|S_i+\frac{\sigma_n}{N}i\right|\geq \eps a_n^2\right\},
$$
we claim that as $n\rightarrow\infty$,
\begin{equation}
\label{eq:lemICRT-eq2}
\P\left(\mathcal{E}_n\,\big |\,T_{-\sigma_n}=N\right)\leq 3
\P\left(\mathcal{E}_n\,\big |\,S_N=-\sigma_n\right) + o(1).
\end{equation}
Indeed, on the one hand, we have by the Markov property at time $K_n$,
\begin{align*}
\lefteqn{\P\left(\mathcal{E}_n\cap\{S_{K_n}< \sigma_n\} \big
    |\,T_{-\sigma_n}=N\right)}\\
&=\E\left[\1_{\mathcal{E}_n\cap\{S_{K_n}< \sigma_n\}\cap\{S_i>-\sigma_n \hbox{\small{ for
    all }}i\leq K_n\}}
  \frac{\P\left(T_{-(\sigma_n+S_{K_n})}=N-K_n\,|\,S_{K_n}\right)}{\P\left(T_{-\sigma_n}=N\right)}\right].
\end{align*}
By Kemperman's formula, we obtain on the event $\{|S_{K_n}| <\sigma_n\}$
for $n$ sufficiently large
\begin{align*}
\frac{\P\left(T_{-(\sigma_n+S_{K_n})}=N-K_n\,|\,S_{K_n}\right)}{\P\left(T_{-\sigma_n}=N\right)}
&=\frac{|\sigma_n+S_{K_n}|}{N-K_n}\frac{N}{\sigma_n}\frac{\P\left(S_{N-K_n}=-(\sigma_n+S_{K_n})\,|\,S_{K_n}\right)}{\P\left(S_N=-\sigma_n\right)}\\
&\leq 3\frac{\P\left(S_{N-K_n}=-(\sigma_n+S_{K_n})\,|\,S_{K_n}\right)}{\P\left(S_N=-\sigma_n\right)}.
\end{align*}
Plugging this into the expression above, we arrive at
$$\P\left(\mathcal{E}_n\cap\{S_{K_n}<\sigma_n\} \big
    |\,T_{-\sigma_n}=N\right)\leq 3\P\left(\mathcal{E}_n\cap\{S_{K_n}< \sigma_n\} \big
    |\,S_N=-\sigma_n\right).
$$
On the other hand, arguing as above,
we obtain
\begin{align*}
  \lefteqn{\P\left(\mathcal{E}_n\cap\{S_{K_n}\geq \sigma_n\} \big
      |\,T_{-\sigma_n}=N\right)=\P\left(S_{K_n}\geq \sigma_n \big
      |\,T_{-\sigma_n}=N\right)}\\
  &\quad= \frac{N}{(N-K_n)\sigma_n}\E\left[(\sigma_n+S_{K_n})\frac{\P\left(S_{N-K_n}=-(\sigma_n+S_{K_n})\,|\,S_{K_n}\right)}{\P\left(S_N=-\sigma_n\right)}\1_{\{S_{K_n}\geq
      \sigma_n\}}\right].
\end{align*}
Clearly, on the event $\{S_{K_n}\geq \sigma_n\}$, the fraction of the two
probabilities inside the expectation is bounded by $1$. Moreover, keeping
in mind that $K_n\ll\sigma_n^2$, we have by the local central limit
theorem~\eqref{eq:localCLT}
$$
\E\left[(\sigma_n+S_{K_n})\1_{\{S_{K_n}\geq
    \sigma_n\}}\right]\leq
2\sum_{\ell=\sigma_n}^{K_n}\ell\,\P\left(S_{K_n}=\ell\right)\lesssim K_n^{1/2}.
$$
We obtain
$\P\left(\mathcal{E}_n\cap\{S_{K_n}\geq \sigma_n\} \big
      |\,T_{-\sigma_n}=N\right)=o(1),$
and~\eqref{eq:lemICRT-eq2} follows. 

In order to conclude the convergence of the first component, we now show
that the probability on the right hand side of~\eqref{eq:lemICRT-eq2} tends
to zero as $n\rightarrow\infty$.  Our proof is based on a change-of-measure
argument in spirit of~\cite[Proof of Lemma 24]{Be3}. We consider the random
walk $\tilde{S}=(\tilde{S}_i,i\in\N_0)$ started at zero with step
distribution
$$
\P\left(\tilde{S}_{i+1}-\tilde{S}_i=1+\frac{\sigma_n}{N}\right)=\frac{1-\sigma_n/N}{2},\quad\P\left(\tilde{S}_{i+1}-\tilde{S}_i=-1+\frac{\sigma_n}{N}\right)=\frac{1+\sigma_n/N}{2}.
$$
The walk $\tilde{S}$ is a martingale, and a direct computation
shows that for every measurable
function $f$ and any $\ell\in\N_0$,
$$
\E\left[f\left(\left(S_i+\frac{\sigma_n}{N}i\right)_{0\leq i\leq
      \ell}\right)\right]=\E\left[\left(1-\frac{\sigma_n^2}{N^2}\right)^{-\ell/2}\left(\frac{1-\sigma_n/N}{1+\sigma_n/N}\right)^{-\frac{1}{2}(\tilde{S}_\ell-\ell\sigma_n/N)}f\left((\tilde{S}_i)_{0\leq
      i\leq \ell}\right)\right].
$$
With the last display,
we can rewrite the probability on the right hand side of~\eqref{eq:lemICRT-eq2} as
$$
\P\left(\sup_{0\leq i \leq
    K_n}\left|S_i+\frac{\sigma_n}{N}i\right|\geq \eps
  a_n^2\,\Big |\,S_{N}=-\sigma_n\right)=\P\left(\sup_{0\leq i \leq
    K_n}\left|\tilde{S}_i\right|\geq \eps
  a_n^2\,\Big |\,\tilde{S}_{N}=0\right).
$$
We estimate
\begin{align}\label{eq:lemICRT-eq3}
\lefteqn{\P\left(\sup_{0\leq i \leq
    K_n}\left|\tilde{S}_i\right|\geq \eps
  a_n^2\,\Big |\,\tilde{S}_{N}=0\right)}\nonumber\\
&\leq \P\left(\inf_{0\leq i \leq
    K_n}\tilde{S}_i\leq -\eps
  a_n^2\,\Big |\,\tilde{S}_{N}=0\right) +\P\left(\sup_{0\leq i \leq
    K_n}\tilde{S}_i\geq \eps
  a_n^2\,\Big |\,\tilde{S}_{N}=0\right),
\end{align}
and we show next that the first summand on the right tends to zero as
$n\rightarrow\infty$.  First, note that for $k\in\Z$, $(\tilde{S}_i,0\leq
i\leq N)$ is under $\P(\cdot\,|\,\tilde{S}_N=2k)$ uniformly distributed
among all paths starting at $0$ at time $0$, ending at $2k$ at time $N$ and
making upward steps of size $ 1+\sigma_n/N$ and downward steps of size
$-1+\sigma_n/N$. Switching an upward step chosen uniformly at random into a
downward step gives a path with law $\P(\cdot\,|\,\tilde{S}_N=2(k-1))$, which
lies below the original one. Therefore, for all $k\leq 0$ such
that $\P(\tilde{S}_{N}=2k) >0$,
$$
\P\left(\inf_{0\leq i \leq K_n}\tilde{S}_i\leq -\eps
  a_n^2\,\Big |\,\tilde{S}_{N}=0\right)\leq\P\left(\inf_{0\leq i \leq
    K_n}\tilde{S}_i\leq -\eps
  a_n^2\,\Big |\,\tilde{S}_{N}=2k\right).
$$
From this inequality, we deduce that
\begin{align*}
  \lefteqn{\P\left(\inf_{0\leq i \leq K_n}\tilde{S}_i\leq -\eps
      a_n^2\,\Big |\,\tilde{S}_{N}=0\right)}\\
  &\leq {\P\left(\tilde{S}_N\leq
      0\right)}^{-1}\sum_{k=0}^\infty\P\left(\tilde{S}_N=-2k\right)\P\left(\inf_{0\leq
      i \leq K_n}\tilde{S}_i\leq -\eps a_n^2\,\Big|\,\tilde{S}_{N}=2k\right)\\
  &\leq {\P\left(\tilde{S}_N\leq 0\right)}^{-1}\P\left(\inf_{0\leq i \leq
      K_n}\tilde{S}_i\leq -\eps a_n^2\right).
\end{align*}
The central limit theorem bounds the probability
$\P(\tilde{S}_N\leq 0)$  from below by $1/3$ for $n$ large
enough. From Doob's inequality, we get 
$$
\P\left(\inf_{0\leq i \leq
    K_n}\tilde{S}_i\leq -\eps
  a_n^2\right)\leq \P\left(\sup_{0\leq i\leq K_n}|\tilde{S}_i|\geq \eps
  a_n^2\right)\leq 
\frac{1}{\eps^2a_n^4}\E\left[\tilde{S}_{K_n}^2\right] \leq
\frac{4K_n}{\eps^2a_n^4} = o(1),
$$
where we have used that $a_n^2\gg \max\{1,n/\sigma_n\}$.  The second term
of~\eqref{eq:lemICRT-eq3} involving the supremum of $\tilde{S}$ up to time
$K_n$ is treated similarly, by switching downward into upward steps.  We
conclude that the probability on the right hand side
of~\eqref{eq:lemICRT-eq1} converges to zero as $n$ tends to infinity.

It remains to prove the joint convergence of $(\tilde{C}_n,\tilde{L}_n)$
stated in the lemma. Let again $K>0$. By what we just have proved and
Skorokhod's theorem, we can assume that $(\tilde{C}_n(t),t\geq 0)$
converges almost surely to $(t, t\geq 0)$ in
$\mathcal{C}([0,K])$. Finite-dimensional convergence of
$(\tilde{C}_n,\tilde{L}_n)$ then follows from standard arguments as for
example given in~\cite[Proof of Theorem 4.3]{LGMi}. We are left with
showing tightness of the laws of $\tilde{L}_n$ on $\mathcal{C}([0,K])$. By
the theorem of Arzel\`a-Ascoli, we have to show that for every $\eps>0$,
\begin{equation}
\label{eq:lemICRT-eq4}
\lim_{\delta\downarrow 0}\limsup_{n\rightarrow\infty}\P\left(\sup_{s,t\in[0,K], |s-t|\leq
    \delta}\left|\tilde{L}_n(s)-\tilde{L}_n(t)\right|\geq \eps\right) =0.
\end{equation}
Recall that $K_n=\lceil(N/\sigma_n)a_n^2K\rceil$. Since by the first part of the lemma,
$$
\P\left(\sup_{i\leq K_n}|C_n(i)| > 2Ka_n^2\right)\longrightarrow 0$$ as $n$
tends to infinity, we may instead show tightness of the laws of
$\tilde{L}_n$ given the event $\mathcal{E}_n'=\{\sup_{i\leq K_n}|C_n(i)|
\leq 2Ka_n^2\}$.  By Kolmogorov's criterion, tightness follows if we show
that there exists a constant $M<\infty$ such that for all $n$ and all $s,t
\in [0,K]$,
\begin{equation}
\label{eq:lemICRT-tightness}
\E\left[|\tilde{L}_n(s)-\tilde{L}_n(t)|^{4}\,\Big |\,\mathcal{E}_n'\right]\leq M|s-t|^2.
\end{equation}
Since $L_n$ is Lipschitz, we can restrict ourselves to the case where $(N/\sigma_n)a_n^2s$ and
$(N/\sigma_n)a_n^2t$ are integers. Let
$$\Delta \tilde{C}_n(s,t)
=\tilde{C}_n(s)+\tilde{C}_n(t)-2\min_{[s,t]}\tilde{C}_n.$$ By definition of
the contour pair $(C_n,L_n)$, conditionally given $C_n$, the difference
$a_n(\tilde{L}_n(s)-\tilde{L}_n(t))$ is distributed as a sum of
i.i.d. variables $\eta_i$ with the uniform law on $\{-1,0,1\}$. By
construction, the sum involves at most $a_n^2\Delta\tilde{C}_n(s,t)$
summands. We thus obtain for some $M'>0$
$$
  \E\left[|\tilde{L}_n(s)-\tilde{L}_n(t)|^{4}\,\Big
    |\,\mathcal{E}_n'\right]\leq
  \frac{1}{a_n^{4}}\E\left[\left(\sum_{i=1}^{a_n^2\Delta\tilde{C}_n(s,t)}\eta_i\right)^{4}\,\Big
    |\,\mathcal{E}_n'\right]
  \leq M'\E\left[|\Delta\tilde{C}_n(s,t)|^2\,\Big |\,\mathcal{E}_n'\right].
$$
By the first part of the lemma, $|\Delta\tilde{C}_n(s,t)|$ converges to
$|s-t|$ in probability, and $|\Delta\tilde{C}_n(s,t)|$ is uniformly bounded
on $\mathcal{E}_n'$ by $8K$. Therefore, the  expectation on the right
converges to $|s-t|^2$, and~\eqref{eq:lemICRT-tightness} follows. 
This completes the proof of the lemma.
\end{proof}

\noindent {\bf Acknowledgments.}
We would like to thank Lo\"ic Richier for helpful discussions, and for
introducing EB to IPE.

\end{document}